\documentclass[a4paper,12pt]{book}
\usepackage{amsmath,amssymb,amsthm,amsfonts,latexsym, enumerate,mathrsfs}
\usepackage[toc,page]{appendix}
\usepackage[active]{srcltx}
\usepackage{graphicx}

\usepackage{hyperref}
 \hypersetup{backref, colorlinks=true}
\hypersetup{backref, colorlinks=true, colorlinks   = true,
	urlcolor     = blue, linkcolor = blue, citecolor   = magenta
}

 \usepackage{makeidx}
\usepackage{index}
\makeindex

\date{}

\let\=\partial

\numberwithin{equation}{section}

\newtheorem {pro}{Proposition}[section]
\newtheorem {thm}[pro]{Theorem}
\newtheorem {cor}[pro]{Corollary}
\newtheorem{lem}[pro]{Lemma}

\theoremstyle{definition}
 \newtheorem{rem}[pro]{Remark}
 
 \newtheorem{rems}[pro]{Remarks}
\newtheorem{dfn}[pro]{Definition}

\newtheorem{exa}[pro]{Example}
\newtheorem{bas}[pro]{Basic properties}

\newcommand{\Q} {\mathbb{Q}}

\newcommand{\dis}{\mathbf{\rho}}
\newcommand{\pr}{\pi}
\newcommand{\aba} {\mathcal{R}_{an}}

\newcommand{\bou}{\mathbf{B}}
\newcommand{\sph}{\mathbf{S}}
\newcommand{\Bb}{\overline{ \mathbf{B}}}

\newcommand{\bir}{\big{)}}
\newcommand{\bil}{\big{(}}
\newcommand{\lepa}{\left(}

\newcommand{\ripa}{\right)}

\newcommand{\eto}{\mbox{ and }}

\newcommand{\pa}{\partial}

\newcommand{\mba}{\overline{M}}

\newcommand{\pai}{\pa^i \hspace{-0.4mm}}

\newcommand{\supp}{\mbox{supp}}

\newcommand{\jac}{\mbox{jac}\,}

\newcommand{\ccc}{ \mathscr{C}}

\newcommand{\lf}{\mathcal{L}}

\newcommand{\G} {\mathbb{G}}
\newcommand{\grr} {\mathcal{G}}

\newcommand{\R} {\mathbb{R}}
\newcommand{\Z} {\mathbb{Z}}
\newcommand{\N} {\mathbb{N}}
\newcommand{\C} {\mathcal{C}}
\newcommand{\A} {\mathcal{A}}

\newcommand{\hn}{\mathcal{H}}
\newcommand{\psit}{\widetilde{\psi}}

\newcommand{\pc}{\widehat{\Psi}}

\newcommand{\St}{\mathcal{S}}

\newcommand{\card}{\mbox{card}}
\newcommand{\dbb}{\mathbf{C}}

\newcommand{\D}{\mathcal{D}}

\newcommand{\M}{\mathcal{M}}
\newcommand{\mfr}{\mathfrak{M}}

\newcommand{\I}{\mathcal{I}}

\newcommand{\F}{\mathcal{F}}

\newcommand{\E}{\mathcal{E}}
\newcommand{\la}{\mathcal{L}}
\newcommand{\eqr}{\sim}

\newcommand{\Qp}{\mathbb{R}_{+}}

\newcommand{\ep}{\varepsilon}

\newcommand{\xo}{{x_0}}

\newcommand{\s}{\mathcal{S}}
\newcommand{\sno}{\mathcal{S}_{n,0}}
\newcommand{\ccn}{\mathcal{C}_n(R)}

\newcommand{\xxc}{\check{X}}

\newcommand{\xt}{\tilde{x}}
\newcommand{\qt}{\tilde{q}}

\newcommand{\fln}{\mathcal{L}_{og}}

\newcommand{\Pa}{\mathcal{P}}

\newcommand{\hh}{\mathcal{V}}
\newcommand{\smn}{\mathcal{S}_{m+n}}
\newcommand{\tim}{{t \in \R^m}}

\newcommand{\ns}{\N^*}

\newcommand{\orn}{{0_{\R^n}}}
\newcommand{\orm}{{0_{\R^m}}}

\title {{\scshape\mdseries  On subanalytic geometry}}

\makeatletter
\usepackage{nomencl}

\makeatother 

\author{Guillaume Valette}

 \makenomenclature

\begin{document}
  \maketitle

%
%
%
 \paragraph{Abstract.}
 These notes constitute a survey on the geometric properties of globally subanalytic sets. We start with their definition and some fundamental  results such as Gabrielov's Complement Theorem or existence of cell decompositions. We then give the main basic tools of subanalytic geometry, such as Curve Selection Lemma, \L ojasiewicz's inequalities, existence of tubular neighborhood, Tamm's theorem (definability of regular points), or existence of regular stratifications (Whitney or Verdier).  We then present the developments of Lipschitz geometry obtained by various authors during the four last decades, giving a proof of existence of metric triangulations, introduced by the author of these notes,   definable bi-Lipschitz triviality, Lipschitz conic structure, as well as invariance of the link under definable bi-Lipschitz mappings. The last chapter is devoted to geometric integration theory, studying the Hausdorff measure of globally subanalytic sets, integrals of subanalytic functions, as well as the density of subanalytic sets (the Lelong number) and Stokes' formula.

\tableofcontents

\cleardoublepage
\phantomsection
\addcontentsline{toc}{chapter}{Introduction}
\chapter*{Introduction}
Semi-analytic sets are the subsets of $\R^n$ that are locally defined by finitely many equalities and
inequalities on analytic functions. The study of the geometry of these sets started\footnote{This introduction includes a few historical facts in order to motivate our approach, but  is by no means exhaustive. Some complements may however be found at the end of each chapter.}
when S. \L ojasiewicz answered positively a question of L. Schwartz about the possibility of dividing a distribution by a nonzero analytic function \cite{lojasiewicz59}. \L ojasiewicz
understood that this required to show that, given an analytic function $f$ on a neighborhood of a compact subset $K$ of $\R^n$, there are a constant $C$ and a positive integer $N$ such
that for all $x \in K$:
$$d(x, f^{-1} (0))^N\le C|f (x)|,$$
 where $d(x, f^{-1}(0))$ stands for the distance from $x$ to $f^{-1} (0)$. He also established many
of the basic properties of semi-analytic sets that will be presented in Chapter \ref{chap_basic}, such
as Curve Selection Lemma (Lemma \ref{curve_selection_lemma}), existence of Whitney regular stratifications (Proposition \ref{pro_w_stratifying}), or
existence of $\ccc^0$ triangulations, which yields that semi-analytic sets are locally contractible. He introduced for this purpose the normal partitions
and showed existence of such a partition for every given semi-analytic set \cite{lojasiewicz59,  lojasiewicz64b, ds}.

A few years later, H. Hironaka’s resolution of singularities \cite{hironakares} offered an
alternative powerful approach of semi-analytic  geometry, on which some authors were nevertheless reluctant
to rely due to the high level of difficulty of the proof of Hironaka’s theorem. The
proof of this theorem was however simplified \cite{bmres}, giving a very adequate
 tool to investigate the geometry of semi-analytic sets. The two approaches
\cite{lojasiewicz64b} and \cite{bmres} have in common that they proceed by induction
on the dimension of the ambient space, applying their induction assumptions to a
suitable submanifold of lower dimension.

One problem to reduce the dimension of the ambient space is that, unlike its complex counterpart,
the real semi-analytic category is not stable under linear projections, even if they are
proper\footnote{
 “Proper” means that the preimage of a compact set is compact.} \cite{osgood}.  A. Gabrielov \cite{gabrielov} thus
 suggested to work with the sets  that are the image of a semi-analytic set under a proper mapping that is the
restriction of a linear projection. Such a set  is now called a {\it subanalytic set}.
 Gabrielov’s idea, which simplified the theory, was relevant because he was able to show that the complement of a subanalytic set is subanalytic (Theorem \ref{thm_gabrielov}), which is not at all an
obvious fact.  \L ojasiewicz's and Hironaka's approaches both turned out to be efficient  to prove Gabrielov's theorem and to investigate the geometry of subanalytic sets \cite{hironaka, bmsemisub, sus90, lojasiewicz93,ds}.

One inconvenience of the subanalytic category was still however that,  the definition of semi-analytic sets being local, the subanalytic
category is only preserved by the projections that are proper. To remediate to this inconvenience, analytic geometers therefore started working with {\it globally subanalytic sets} (see Definition \ref{dfn_glob_sub}). It proved to be a well-behaved category as
J. Denef and L. van den Dries
\cite{denefvandendries}  showed a quantifier elimination result, which made it possible to decompose any globally subanalytic set in basic pieces called cells which are constructed inductively on the dimension by equalities and inequalities on simple functions, which when the domain is $1$-dimensional are expressed locally in terms of Puiseux series (see for instance Theorem \ref{thm_existence_cell_dec} and Proposition \ref{pro_la_and_globally subanalytic}).

Slightly later, A. Parusinski \cite{parusinskiprep,parucamb} (see also \cite[Remark $1.7$]{bmarc}), relying on Hironaka’s techniques, proved the so-called
preparation theorem for subanalytic functions (Theorem \ref{thm_preparation}) that may be regarded as an $n$-variable Puiseux theorem with a description of the leading
coefficient as a normal crossing (called here a reduction, see Definition \ref{dfn_reduced}).
A few years later, J.-M. Lion and J.-P. Rolin \cite{lionrolin} gave an alternative proof of this
preparation theorem  which was inspired by Denef
and van den Dries’ article and the proof of Puiseux Theorem.

%
In Chapter \ref{chap_analytic}, we
shall prove the preparation theorem for globally subanalytic functions (Theorem \ref{thm_preparation}) following very closely Lion and
Rolin’s proof, which yields Gabrielov’s Complement Theorem (Theorem \ref{thm_gabrielov}), as
well as existence of cell decompositions (Theorem \ref{thm_existence_cell_dec}) and description of subanalytic sets and functions in terms of $\la$-functions
(Proposition \ref{pro_la_and_globally subanalytic}). This is certainly not the shortest proof of Gabrielov’s theorem,
but this approach has the advantage to establish simultaneously all the just described results (existence of cell decompositions, description in terms of $\la$-functions, reduction of functions, Puiseux lemma with parameters),
which we shall need in the subsequent chapters.

In Chapter \ref{chap_basic}, we derive from the results of the first chapter the most basic
properties of globally subanalytic sets, such as Quantifier Elimination (Theorem \ref{thm_s_formula}), Curve Selection Lemma (Lemma \ref{curve_selection_lemma}), \L ojasiewicz's inequalities (Theorem \ref{thm_lojasiewicz_inequality} and Corollary \ref{cor_lojasiewicz_gradient}), Tamm’s Theorem (Theorem \ref{thm_X_reg}).
 We will thus include a proof of existence of  Whitney $(b)$ regular stratifications (Definitions \ref{dfn_stratifications}, \ref{dfn_conditon_b_et_w}, and Proposition \ref{pro_existence}) of every subanalytic set, which was established by S. \L ojasiewicz  \cite{lojasiewicz64b},  and discuss several related results, such as existence of stratifications of mappings (Propositions \ref{pro_strat_mapping} and \ref{pro_h_hor_C1}). All these results are needed for the next chapters.

 Many problems of analysis on subanalytic sets, like the above Schwartz's division problem, demand to have a precise description of the geometry of the singularities of the considered set or function.   Many areas of analysis, like for instance the theory of PDE's, require to have bounds for the derivatives of the studied functions,
which often forces to work in the Lipschitz category.  Chapter \ref{chap_lipschitz_geometry} will thus be devoted to Lipschitz geometry. Chapter \ref{chap_gmt} will then give several results of geometric integration theory of subanalytic sets, which is useful to functional analysis as well.

Motivated by the classification of singularities, T. Mostowski \cite{m} introduced stratifications that admit a Lipschitz version of Thom-Mather's First Isotopy Theorem \cite{mather}, which makes these stratifications locally bi-Lipschitz trivial along the strata (see Definition \ref{dfn_stra_bil_triv} and Remark \ref{rem_lipstrat}). Mostowski could show that every complex analytic set can be stratified in this way, result that was later extended to the (real)  subanalytic category  by A. Parusinski \cite{parusinskiprep} (see also \cite{lipsomin, halupczok}).

We will give in Chapter \ref{chap_lipschitz_geometry} an alternative approach of Lipschitz geometry, also providing bi-Lipschitz trivial stratifications (Corollary \ref{cor_stra_bili_triv}).
 We first give some useful results obtained by K. Kurdyka and A. Parusinski  \cite{kurdykawhitney,parusinskiprep,kp}, such as decomposition into Lipschitz cells (Theorem \ref{thm_lipschitz_cells}), and then enter the precise description of the Lipschitz geometry of subanalytic sets by presenting the concept of {\it metric triangulations} (Definition \ref{triangulation}), which appeared in \cite{vlt} (where they are called Lipschitz triangulations). These are triangulations that contain information on the Lipschitz geometry of the triangulated set (see the introductions to Chapter \ref{chap_lipschitz_geometry} and section \ref{sect_metric triangulations} for more). We prove that every globally subanalytic set admits such a triangulation (Theorem \ref{thm existence des triangulations}) and then derive consequences about the classification of globally subanalytic singularities from the metric point of view, showing countability of classes (Corollary \ref{cor_sullivan}), a bi-Lipschitz triviality theorem (Corollary \ref{cor_hardt}), and existence of definably bi-Lipschitz trivial stratifications (Corollary \ref{cor_stra_bili_triv}). We also give a metric triangulation theorem for germs (Theorem \ref{thm_triangulations_locales}), with some extra properties  that enable us to establish a ``Lipschitz conic structure theorem'' (Theorem \ref{thm_local_conic_structure}), which ensures that every globally subanalytic germ is Lipschitz contractible (Corollary \ref{cor_Lipschitz_retract}), which can be regarded as a Lipschitz version of a theorem of \L ojasiewicz \cite{lojasiewicz64a}. We also establish the definable bi-Lipschitz invariance of the link (Corollary \ref{cor_unicite_du_link}).

We finally focus in Chapter \ref{chap_gmt} on geometric measure and integration theory of globally subanalytic sets. We show in particular that the integrals of globally subanalytic functions can be expressed as polynomials of some globally subanalytic functions and of their logarithms (Corollary \ref{cor_integration_fonctions}). We also study the density (sometimes called the Lelong number, see Proposition \ref{pro_density}) of globally subanalytic stratified sets. We finish with some Stokes' theorems on globally subanalytic sets, possibly singular (Theorems \ref{thm_stokes_leaves} and \ref{thm_stokes}).

The results of the last two chapters, especially the Lipschitz Conic Structure Theorem (Theorem \ref{thm_local_conic_structure}) and existence of locally bi-Lipschitz trivial stratifications (Corollary \ref{cor_stra_bili_triv}),  recently turned out to make it possible to carry out a satisfying theory of Sobolev spaces and partial differential equations on subanalytic domains, possibly singular \cite{lebeau, poincwirt, trace, lprime, laplace, depauw, dukulan}. Although these applications to analysis go beyond the scope of this survey, these notes aim at providing the material necessary for this purpose in a way which is accessible to both geometers and specialists of PDE's.

\medskip

 \noindent{\bf Some notations and conventions.} We denote by $\R_+$\nomenclature[aa]{$\R_+$}{set of nonnegative real numbers\nomrefpage} the set of nonnegative real numbers. Throughout these notes $i,j,k,m,$ and $n$ stand for  integers. We denote by $\ns$ the set of all the positive integers.\nomenclature[ab]{$\N^*$}{positive integers\nomrefpage}

 By convention, $\R^0=\{0\}$\nomenclature[ab]{$\R^0$}{by convention $\R^0=\{0\}$\nomrefpage}. The origin of $\R^n$ is denoted $0$ for all $n$ but we will write $0_{\R^n}$\nomenclature[ab]{$0_{\R^n}$}{origin in $\R^n$\nomrefpage} if $n$ is not obvious from the context.

 We denote by $|.|$\nomenclature[ac]{$\vert.\vert$}{Euclidean norm\nomrefpage} the Euclidean norm of $\R^n$.
Given $A\subset \R^n$, we respectively write $cl(A)$\nomenclature[ad]{$cl(A)$}{closure of the set $A$\nomrefpage} and $int(A)$\nomenclature[ad]{$int(A)$}{interior of the set $A$\nomrefpage} for the closure and interior of $A$ (in this norm).  If we say that $M\subset \R^n$ is a manifold, we always mean  that it is a submanifold\footnote{We say that  $M\subset \R^n$ is a {\bf $\ccc^i$ submanifold of $\R^n$} (of dimension $k$) if it is locally the graph of a $\ccc^i$ mapping $\phi:\R^k \to \R^{n-k}$, after a possible linear automorphism of $\R^n$.} of $\R^n$.\index{manifold} The word ``smooth''\index{smooth} will mean $\ccc^\infty$.


Given a mapping $F:A \to B$, with $A \subset \R^n$ and $B \subset \R ^m$, we denote by $\Gamma_F$ {\bf the graph of $F$},\index{graph!of a mapping}\nomenclature[ae]{$\Gamma_F$}{graph of a mapping $F$\nomrefpage} which is the set $\{(x,y)\in A \times B: y=F(x)\}.$  We denote by $F_{|C}$ the restriction\nomenclature[af]{$F_{|C}$}{restriction of the mapping $F$ to $C$\nomrefpage} of $F$ to $C$, if $C\subset A$.

By ``{\bf homeomorphism}'', we mean an invertible continuous  map $h:A \to B$ such that $h^{-1}:B \to A$ is continuous. In particular, we mean that $h$ is an onto map. If it is not onto, we speak about a {\it homeomorphism onto its image}. \index{homeomorphism}

A mapping $\gamma:(a,b)\to \R^n$, $a<b$ in $\R\cup \{\pm\infty\}$, will be sometimes called an {\bf arc}\index{arc}.


As usual, if $U$ is an open subset of $\R^n$, we say that a $\ccc^\infty$ function $f:U\to \R$ is {\bf analytic}\index{analytic function} if for every $x_0\in U$, the Taylor series of $f$ at $\xo$ converges to $f$ pointwise locally near $\xo$. A mapping $F:U\to \R^k$, $x\mapsto (F_1(x),\dots,F_k(x))$ is {\bf analytic} if so is each of its components $F_i$.
More generally, we will say that a mapping $g$ defined  on a subset $A\subset \R^n$ is analytic  if it coincides  with the restriction to $A$ of a mapping $f$ which is analytic on an open neighborhood $U$ of $A$ in $\R^n$.

A {\bf germ}\index{germ} of mapping (resp. set) at $x_0\in \R^n$ is an equivalence class of the equivalence relation that identifies two mappings (resp. sets) that coincide on a neighborhood of $x_0$.   Given a germ of mapping $f:X\to Y$ at $x_0$, we shall write $f:(X,x_0)\to (Y,y_0)$\nomenclature[afa]{$f:(X,x_0)\to (Y,y_0)$}{germ of mapping at $x_0$ satisfying $f(x_0)=y_0$\nomrefpage} as a shortcut to express that $f(x_0)=y_0$.

Given any couple of functions $\zeta$ and $\xi$\nomenclature[afb]{$[\xi,\zeta]$,$(\xi,\zeta)$, $[\xi,\zeta)$, $(\xi,\zeta]$}{\nomrefpage} on a set $A \subset
\R^n$ with $\xi \leq \zeta$ we define the {\bf closed interval $[\xi,\zeta]$} as the set:
\begin{equation}\label{eq_intervals}[\xi,\zeta]:=\{(x,y)\in A \times \R: \xi(x)\leq y \leq \zeta(x)\}.\end{equation}
The open and semi-open intervals $(\xi,\zeta)$, $(\xi,\zeta]$, and $[\xi,\zeta)$ are defined analogously.
 We will sometimes admit $\xi$ or $\zeta$ to be (identically) $\pm \infty$  (the interval will still be a subset of $A\times \R$ however). When $n=0$, by convention, the graph of a function $\xi$ on $A=\{0\}$ will be the singleton $\{\xi(0)\}\subset \R$, and the above intervals will  stand for the corresponding intervals of $\R$.

Given two real valued functions  $f $ and  $g$ on a set $X$ and a subset $A$ of $X$, we write ``$f \lesssim g$ on $A$''\nomenclature[afbf]{$\lesssim$}{inequality up to some constant\nomrefpage}, or  ``for $x\in A$, $f(x) \lesssim g(x)$'', if there exists a positive
real number  $C$ such that $f \leq Cg$ on $A$.  We write  ``$f \sim g$ on $A$'',  or ``for $x\in A$, $f(x)\sim g(x)$''  (and
say that $f$ is {\bf equivalent}\index{equivalent} to $g$ on $A$), if  $f \lesssim g$ and $ g\lesssim f$ on $A$\nomenclature[afbg]{$\sim$}{equivalent functions\nomrefpage}.

Quite often, we will work with a family of sets, say $X_t$, where  $t$ is a parameter that browses a set, say $B$, and families of functions $f_t:X_t\to \R$ and $g_t:X_t \to \R$. It is important to notice that  when we write ``for $t\in B$ and $x\in  X_t$,  $f_t(x) \lesssim g_t(x)$'', the constant of this estimate is assumed to be {\it independent of $t$}. However,  when we write  that ``{\it for each $t$ we have} for  $x\in X_t$, $f_t(x)\lesssim g_t(x)$'', the constant {\it may of course depend on $t$},  except otherwise specified. This is the same difference as between $\exists C,\forall t,\forall x$ and $\forall t,\exists C,\forall x$. We will nevertheless sometimes emphasize the dependence of the constants when important for the theory.

We also wish to stress the fact that the word ``{\bf definable}'' will be used as a shortcut of ``globally subanalytic'' from Chapter \ref{chap_basic} after Proposition \ref{pro_csq_qe}.

     \markboth{G. Valette}{Subanalytic sets}
  \chapter{Subanalytic sets and functions}\label{chap_analytic}

\section{Definitions and basic facts}
\begin{dfn}\label{dfn_semianalytic}
A subset $E\subset \R^n$ is called {\bf semi-analytic}\index{semi-analytic} if it is {\it locally}
defined by finitely many real analytic equalities and inequalities. Namely, for each $a \in   \R^n$, there are
a neighborhood $U$ of $a$ as well as real analytic  functions $f_{ij}$ and $ g_{ij}$ on $U$, where $i = 1, \dots, r\in \N, j = 1, \dots , s_i\in \N$, such that
\begin{equation}\label{eq_definition_semi}
E \cap   U = \bigcup _{i=1}^r\bigcap _{j=1} ^{s_i} \{x \in U : g_{ij}(x) > 0 \mbox{ and } f_{ij}(x) = 0\}.
\end{equation}
\end{dfn}

\begin{exa} In the above definition, a description as displayed in the right-hand-side of (\ref{eq_definition_semi})  is required near each point of $a\in \R^n$ (and not only near the points of $E$).  It thus can be seen that the graph of  $f(x)=\sin \frac{1}{x}$, $x\in (0,1)$, is not a semi-analytic set, although this function is analytic. Condition  (\ref{eq_definition_semi}) fails at the points of the $y$-axis that are in the closure of the graph.\end{exa}


\begin{dfn}\label{dfn_glob_semi}
 A subset $Z$ of $\R^n$ is  {\bf globally semi-analytic}\index{globally semi-analytic set} if $\hh_n(Z)$ is  a semi-analytic subset of $\R^n$, where $\hh_n : \R^n  \to (-1,1) ^n$ is the homeomorphism defined by
  $$\hh_n(x_1, \dots, x_n) :=  (\frac{x_1}{\sqrt{1+|x|^2}},\dots, \frac{x_n}{\sqrt{1+|x|^2}} ).$$\nomenclature[ag]{$\hh_n$}{compactification of $\R^n$\nomrefpage}
\end{dfn}  Of course, globally semi-analytic sets are semi-analytic. The set $\N$ is  an example of set which is analytic but not globally semi-analytic.
Roughly speaking, we can say that a semi-analytic  subset $Z$ of $\R^n$ is   globally semi-analytic if it is still semi-analytic after compactifying $\R^n$. Clearly,  a bounded subset of $\R^n$ is semi-analytic if and only if it is globally semi-analytic.  Unbounded examples are easy to produce:
\begin{exa}\label{exa_glob_semi}
It is easy to see that semi-algebraic sets\index{semi-algebraic set}, that is to say, sets of type $$\bigcup _{i=1}^r\bigcap _{j=1} ^{s_i} \{x \in \R^n : P_{ij}(x) > 0 \mbox{ and } Q_{ij}(x) = 0\},$$ where  $P_{ij}$ and $ Q_{ij}$ are $n$-variable polynomials for all $i$ and $j$, are all  globally semi-analytic.
 Of course, not every  semi-analytic set is semi-algebraic.  It is however worth mentioning here that, thanks to Tarski-Seidenberg's principle (see \cite{bcr, costesemialg}), the semi-algebraic sets constitute a nice class to work, and  most of the results of these notes have their semi-algebraic counterpart  (see \cite{bcr,costeomin, costesemialg, gvhandbook}), with sometimes simpler proofs.
\end{exa}


\medskip

Working with {\it globally} semi-analytic sets  makes it possible to avoid some pathological situations at infinity and provides finiteness properties. 
The flaw of semi-analytic and globally semi-analytic sets is that these classes of sets are not preserved under linear projections. In other words, the projection of a globally semi-analytic set is not always globally semi-analytic:

\medskip

\begin{exa}\label{exa_osgood}(Osgood's example \cite{osgood})
Define a globally semi-analytic subset of $\R^4$ by: $$E:=\{(x,xy,xe^y,y):x \in (0,1) \mbox{ and } y \in (0,1)\}.$$
Let now $\pi:\R^4 \to \R^3$ be the projection omitting the last coordinate.  If $\pi(E)$ were semi-analytic then there would exist a germ of analytic function (at the origin), not identically zero, vanishing at every point of $\pi(E)$ in this neighborhood. Examining the Taylor expansion of this function at the origin quickly leads to a contradiction \cite{bmsemisub,lojasiewicz64b, lojasiewicz93}. 
\end{exa}

To overcome this problem, we will work with a bigger class of sets: the globally subanalytic sets, which are the projections of  globally semi-analytic sets.

\medskip

\begin{dfn}\label{dfn_glob_sub}
 A subset $E\subset \R^n$  is  {\bf globally subanalytic}\index{globally subanalytic!set} if 
there exists a globally semi-analytic
set $Z \subset   \R^{n+p}$, $p\in \N$, such that $E  = \pi(Z)$, where $\pi :    \R^{n+p} \to    \R^n $ is the projection onto the $n$ first coordinates.
We shall denote by $\St_n$\nomenclature[ah]{$\s_n$}{globally subanalytic subsets of $\R^n$\nomrefpage} the set of  globally subanalytic subsets of $\R^n$. 

   We say that {\bf a mapping $f:A \to B$ is globally subanalytic}\index{globally subanalytic!mapping}\index{globally subanalytic!function}, $A \in \s_n$, $B\in \s_m$, if its graph is a globally subanalytic subset of $\R^{n+m}$. In the case $B=\R$, we say that  $f$ is a {\bf globally subanalytic function}\index{globally subanalytic!function}.
\end{dfn}
%



\begin{exa}\label{exa_sub}Globally semi-analytic sets (see example \ref{exa_glob_semi}) provide examples of globally subanalytic sets. The function $\sin x$ is a typical example of a function which is  subanalytic but not globally subanalytic.
The set $E$ of Example \ref{exa_osgood} being globally semi-analytic, its projection $\pi(E)$ (with the notations of the latter example) is  globally subanalytic, although not globally semi-analytic. 
\end{exa}

 We stress the fact that there is no smoothness assumption on subanalytic mappings. Unlike analytic functions, they even can be discontinuous.

\begin{bas}\label{pro_basic_properties_from_dfn}
Below we list some very important properties of globally subanalytic sets and mappings which are direct consequences of their definition.  

{\setlength{\leftmargini}{3pt}
\begin{enumerate} [(1)]
\item\label{item_projection} If $A\in \s_n$ and if $\mu:\R^n \to \R^{m}$, $m\le n$, denotes the projection  onto the $m$ first coordinates then $\mu(A)\in \s_m$.

\item\label{item_union} If $A \in \s_n$ and $B \in \s_n$ then $A \cup B$ and $A \cap B$ both belong to $\s_n$. 

\item\label{item_produit} If $A\in \s_n$ and $B \in \s_m$ then $A \times B \in \s_{n+m}$.

\item \label{item_image} Images and preimages of globally subanalytic sets under globally subanalytic mappings are globally subanalytic.

\item\label{item_F_i} A mapping $F: A \to \R^p$, $F=(F_1,\dots,F_p)$,  $A \in \s_n$, is globally subanalytic if and only if $F_i$ is globally subanalytic for every $i$.

\item\label{item_composition} If $f:A \to B$ and $g:B \to C$ are both globally subanalytic then so is $g\circ f$.

\item \label{item_la_functions} Sums and    products  of globally subanalytic functions are globally subanalytic.
\end{enumerate} }
\end{bas}
\begin{proof}(\ref{item_projection}) is clear from the definition of globally subanalytic sets.
 To prove (\ref{item_union}) take $A$ and $B$ in $\s_n$.  By definition of globally subanalytic sets, there exists  a globally semi-analytic set $Z \subset \R^{n+p}$ (resp.  $Z'\subset \R^{n+p'}$) such that $A=\pi(Z)$ (resp. $B=\pi'(Z')$) where $\pi:\R^{n}\times \R^p  \to \R^n$ (resp. $\pi':\R^{n}\times \R^{p'} \to \R^n$) denotes the projection onto the first factor.  Then $A \cup B=\pi''(Y),$ where $\pi'' :  \R^{n+p+p'} \to \R^n$ is the obvious projection and
 $$Y:=\{(x,z,z')\in \R^n \times \R^p \times\R^{p'}:(x,z)\in Z \mbox{ or } (x,z')\in Z'\}.$$ 
 Since $Z$ and $Z'$ are globally semi-analytic, the sets $\hh_{n+p}(Z)$ and $\hh_{n+p'}(Z)$ can be described by inequalities on analytic functions (as in (\ref{eq_definition_semi})). Consequently, so does $\hh_{n+p+p'}(Y)$, which means that $A\cup B$ is globally subanalytic. Moreover, $A \cap B =\pi''(Y')$ where $$Y' :=\{(x,z,z')\in \R^n \times \R^p \times \R^{p'}:(x,z)\in Z \mbox{ and } (x,z')\in Z'\},$$ which entails that $A \cap B$ is globally subanalytic as well. 
  The proof of (\ref{item_produit}) is similar to the proof of (\ref{item_union}) and is left to the reader.

Proof of (\ref{item_image}). Let $F:A\to B$ be globally subanalytic, with  $A \in \s_n$ and $B \in \s_p$.  Observe that  $F(C)=\pi_2(\pi_1^{-1}(C)\cap \Gamma_F)$ and $F^{-1}(D)=\pi_1(\pi_2^{-1}(D)\cap \Gamma_F)$ where $\pi_1:\R^{n}\times \R^p\to \R^n$ and $\pi_2: \R^{n}\times \R^p\to \R^p$ are the obvious orthogonal projections. Hence, it is enough to consider the case where $F$ is a canonical projection, which follows from (\ref{item_projection}) and (\ref{item_produit}).

Proof of (\ref{item_F_i}). Observe that  $\Gamma_{F_i} =\mu_i(\Gamma_F)$, where $\mu_i:\R^n\times \R^p \to \R^n \times \R$, $i=1,\dots,p$, is the projection defined by $\mu_i(x,v)=(x,v_i)$ if $(x,v)\in \R^n \times \R^p$. By (\ref{item_projection}), it means that $\Gamma_{F_i}$ is globally subanalytic if so is $F$. Conversely, if all the $F_i$'s are globally subanalytic then, in virtue of (\ref{item_produit}), the Cartesian product of their graphs is globally subanalytic, and hence, so is $\Gamma_F$, which can be expressed as a suitable projection of this Cartesian product.
 
Proof of (\ref{item_composition}). If $g$ is globally subanalytic then by (\ref{item_produit})  the map $h:A \times B \to C$ defined by $h(x,y):= g(y)$ is globally subanalytic. Therefore, if  $f$ is  also globally subanalytic,  by (\ref{item_union}), so is the set $E:=\Gamma_h\cap (\Gamma_f\times C)$. But since $\Gamma_{g\circ f}=\nu(E) $, where $\nu:A \times B \times C \to A \times C$ is the projection omitting the second factor, by (\ref{item_projection}), the result follows.

Proof of (\ref{item_la_functions}). It is easily checked that the mappings $(x,y) \mapsto (x+y)$ and $(x,y) \mapsto x\cdot y$ are globally subanalytic. By (\ref{item_composition}),  the sum and product of globally subanalytic functions are thus globally subanalytic.
\end{proof}

We can summarize by saying that globally subanalytic sets and mappings possess all the most basic properties that one would need to perform geometric constructions. Actually,  a very useful one is still missing: the stability under complement. If it was obvious from the definition that the complement of a semi-analytic set is semi-analytic, it is far from being easy to show that the complement of a globally subanalytic set is globally subanalytic. This is nevertheless true and it is generally referred as the Gabrielov's Complement Theorem (Theorem \ref{thm_gabrielov}).  The proof of this theorem  will  use almost all the material introduced in this chapter.

\section{Cell decompositions}
Cell decompositions constitute the central tool of these notes.
\begin{dfn}\label{dfn_cell_decomposition}
 Let us define the  cell decompositions of $\R^n$ inductively on $n$. These are finite partitions  of $\R^n$ into globally subanalytic sets, called  {\bf cells}.

$n = 1:$ A {\bf cell decomposition $\C$ of $\R$}\index{cell!decomposition} is a finite subdivision of $\R$ given by some real numbers $a_1 < .. . < a_l$. The {\bf cells of $\C$}\index{cell} are then the singletons $\{a_i\}$, $0 < i \leq l$, and the intervals $(a_i,
a_{i+1})$,
$0 \leq  i \leq l$, where $a_0 = -\infty$ and $a_{l+1} = +\infty$.

\medskip

$n > 1:$ A {\bf cell decomposition $\C$ of $\R^n$} is given by a cell decomposition $\D$ of $\R^{n-1}$ and,
for each cell $D \in \D$, some globally subanalytic functions, analytic on $D$:
$$\zeta_{D,1} < ... < \zeta_{D,l(D)} : D \to \R.$$
The {\bf cells of $\C$} are then the $\ccc^\infty$ manifolds given by the  graphs 
$$\{(x, \zeta_{D,i}(x)) : x \in D\} ,\; 1 \le i \leq  l(D) ,$$
and the {\bf bands}\index{band}
$$(\zeta_{D,i}, \zeta_{D,i+1}) := \{(x, y) : x \in D
\mbox{ and }\zeta_{D,i}(x) <y<\zeta_{D,i+1}(x)\},$$
for $0 \leq  i \leq l(D)$, where $\zeta_{D,0} \equiv -\infty$ and $\zeta_{D,l(D)+1} \equiv +\infty$. The cell $D$ is then called the {\bf basis}\index{basis of a cell} of the cells defined as above. 

We started this inductive definition at $n=1$ to make it more explicit. It is convenient to set that a cell decomposition of $\R^0=\{0\}$ is constituted (exclusively) by $\{0\}$ (we also will adopt the conventions introduced after (\ref{eq_intervals})).

A cell decomposition is said to be {\bf compatible with finitely many sets $A_1,\dots,A_k$}\index{compatible!cell decomposition} if the
$A_i$'s are unions of cells. A {\bf refinement}\index{refinement! of a cell decomposition} of a cell decomposition $\C$ is a cell decomposition compatible with all the elements of $\C$. 
\end{dfn}

\begin{rem}\label{rem_projection_cell_decomposition} 
 If $\pi:\R^{n+p}\to \R^n$ is the projection onto the $n$ first coordinates and $\C$ is a cell decomposition of $\R^{n+p}$ then 
		the finite  family of sets  $\pi(C)$, $C \in \C$, constitutes a cell decomposition of $\R^n$.  We will denote it by $\pi(\C)$\nomenclature[aj]{$\pi(\C)$}{ projection of a cell decomposition\nomrefpage}.
\end{rem}
It is also worthy of notice that it follows from  this inductive definition that every cell  is $\ccc^\infty$ diffeomorphic to $(0,1)^d$ for some $d$ (with $(0,1)^0=\{0\}$).
The following theorem is fundamental to describe the geometry of globally subanalytic sets:

\begin{thm}\label{thm_existence_cell_dec}
Given  $A_1,\dots, A_k$
in $\St_n$, there is a cell decomposition of $\R^{n}$ compatible with all the $A_i$'s.
\end{thm}

 The proof of this theorem is postponed to section \ref{sect_existence_of_cell_dec}.  We start with an easy example and some hints of proofs.

\begin{exa}\label{exa_cell_dec}
Let $A:=\{(x,y)\in \R^2:f(x,y):=y^2-x^3=0\}$. As $f$ can be factorized $(y-x^{\frac{3}{2}})(y+x^{\frac{3}{2}})$, the set $A\setminus \{(0,0)\}$ is the union of the two half-branches of curves $C_i:=\{(x,(-1)^i x^{\frac{3}{2}}):x>0\}$, $i=1,2$. There  exists a cell decomposition of $\R^2$ compatible with  $A$  and containing the cells $\{0\}$, $C_1$, and $C_2$.
 \end{exa}

In this example, the situation is very simple since the set $A$ is described by an equation which is easily factorized. In particular, the cells have a parametrization as a Puiseux series.
 This example  points out the importance of having a nice factorization of the equations defining our set $A$.

\paragraph{Outline of proof of Theorem \ref{thm_existence_cell_dec}} Let $A\in \s_n$ (we restrict ourselves in this outline to the case of one single set $A$, i.e., we assume $k=1$ for simplicity). Observe first that, since $A$ is the projection of a globally semi-analytic set $Z \subset \R^{n+p}$, it is enough to find a cell decomposition of $\R^{n+p}$ compatible with such a set $Z$ (see Remark \ref{rem_projection_cell_decomposition}).
In other words, we can assume $A$ to be globally  semi-analytic.  

We have to find a cell decomposition such that the functions defining the globally semi-analytic set $A$ (see (\ref{eq_definition_semi})) are of constant sign on every cell.

The proof will be by induction on $n$.
 The basic idea is to proceed in the same way as in Example \ref{exa_cell_dec}: we factorize the analytic functions defining the globally semi-analytic set $A$ until we reach an expression which is sufficiently simple to decompose the set $A$ into graphs and bands. The basic idea of this factorization relies on Weierstrass Preparation Theorem and some related finiteness results of algebraic nature that we present below as preliminaries  (section \ref{sect_preliminaries}, see Theorem \ref{thm_weierstrass} and Proposition \ref{pro_de_finitude}).

 Functions that we have put in such a  nice factorized form will be said to be {\it reduced} (Definition \ref{dfn_reduced}). This form, although a bit more complicated than the expression as a Puiseux series obtained in Example \ref{exa_cell_dec}, is of the same type.

As we will argue inductively on the number of variables, it will not be possible to stay in the semi-analytic category: making use of the inductive assumptions requires to drop some variables and Example \ref{exa_osgood} shows that this forces to exit the semi-analytic category. We thus shall introduce a bigger category of functions than restricted analytic functions: {\it the $\la$-functions} (see section \ref{sect_la_functions}).

We shall show that every $\la$-function can be reduced (Proposition \ref{pro_preparation_la_fonctions}), from which it will follow that we can find a cell decomposition such that finitely many given $\la$-functions have constant sign on  every cell (Lemma \ref{lem_la_cell_decomposition}). The reader is invited to glance at the proof of this lemma which  unravels the close links between reducibility of $\la$-functions (Proposition \ref{pro_preparation_la_fonctions}) and existence of what we call $\la$-cell decompositions (Lemma \ref{lem_la_cell_decomposition}),  accounting for the fact that we prove this lemma and this proposition simultaneously in a joint induction.
To complete this outline, we  detail  separately in section \ref{sect_reduced_functions} the strategy of the proof of this proposition, which is the main technical difficulty of the proof of Theorem \ref{thm_existence_cell_dec} (although our cell decompositions will rather be provided by the closely related Lemma \ref{lem_la_cell_decomposition}).


  \begin{section}{Preliminaries on analytic functions}\label{sect_preliminaries}
  \begin{subsection}{The Weierstrass Preparation Theorem}
 We give a proof this theorem that relies on basic  facts of complex analysis. This is the only place in these notes where  complex numbers will be involved.

\begin{dfn}\label{dfn_x_n_regular} 
 Let $U\subset \R^{m-1}\times \R$ be an open set and let  $(u_0,z_0)  \in U$. An analytic  function  $\psi(u,z) $ on $U$ is  {\bf $z$-regular}\index{z@$z$-regular} at $(u_0,z_0)$ if the function $\varphi:[0,\ep) \to \R$, $\ep>0$ small, defined by  $ \varphi(t):=  \psi(u_0,z_0 +t)$    is not identically zero. It is said to be  {\bf $z$-regular of order $d$}\index{z@$z$-regular!of order $d$} at $(u_0,z_0)$ if $\varphi(t)=a t^d+\dots$ with $a \ne 0$.
\end{dfn}


\begin{thm}\label{thm_weierstrass} (Weierstrass Preparation Theorem)
 Let $\psi$ be an analytic function on a neighborhood of $(u_0,z_0)\in \R^{m-1}\times \R$. If $\psi(u,z)$ is $z$-regular of order $d$ at $(u_0,z_0)$ then there exists a neighborhood of $(u_0,z_0)$ on which  $\psi$ has a representation $$\psi(u,z) = W(u,z)\cdot  f(u,z),$$ where $W$ is an analytic function  satisfying $W(u_0,z_0)\ne 0$, and $$f(u,z)=z^d+a_1(u)z^{d-1}+\dots+a_d(u)$$ is a monic polynomial of degree $d$ in $z$ with analytic coefficients in $u$.
\end{thm}
\begin{proof}We may assume $(u_0,z_0)=(0_{\R^{m-1}},0_\R)$ and $\psi(0,0)=0$. The Taylor series of $\psi$ being convergent near the origin, this function extends to a holomorphic function (still denoted $\psi$) on a neighborhood of this point in $\mathbb{C}^{m-1}\times \mathbb{C}$.

Let $\gamma$ be a small circle around $0$ in $ \mathbb{C} $ (oriented counterclockwise)  and let for $u\in \mathbb{C}^{m-1}$ close to the origin:$$c_0(u):=\frac{1}{2i\pi}\int_\gamma \frac{\frac{\pa \psi}{\pa z}(u,z)}{\psi(u,z)}dz.  $$
By the residue theorem, if $\gamma$ is sufficiently small,  $c_0$ is  a continuous integer valued function on a small neighborhood of the origin in $\mathbb{C}^{m-1}$, which must be equal to $d$.
 For every $u\in \mathbb{C}^{m-1}$ close to $ 0_{\mathbb{C}^{m-1}}$, let $b_1(u),\dots, b_d(u)$ be the (complex) roots of the  function $z\mapsto \psi(u,z)$ (repeated with multiplicity) that lie inside the disk delimited by $\gamma$.  Let then for $j=1,\dots,d$
$$c_j:=b_1^j +\dots+b_d^j\;,  $$ 
 and observe that, again due to the residue theorem, we have $$c_j(u) =\frac{1}{2\pi i}\int_\gamma z^j\frac{\frac{\pa \psi}{\pa z}(u,z)}{\psi(u,z)}dz ,$$
 which shows that the $c_j$'s are holomorphic functions on a neighborhood of the origin in $\mathbb{C}^{m-1}$, real valued on $\R^{m-1}$ (the $b_j$'s are pairwise conjugate). We now set for $(u,z)\in \mathbb{C}^{m-1} \times \mathbb{C}$:
 $$f(u,z):=\prod_{j=1} ^d(z-b_j(u))= z^d+a_1(u)z^{d-1}+\dots+a_d(u),$$
with $a_1,\dots,a_d$ analytic. Since the $a_i$'s are polynomial functions of the $c_j$'s (by Girard-Newton's identities), by the above, these are holomorphic functions. We then set $W(u,z):=\frac{\psi(u,z)}{f(u,z)}$ and notice that this function is real valued on a neighborhood of the origin in $\R^m$. Because for every $u$ near $0$, the  functions $z\mapsto \psi(u,z)$ and $z\mapsto f(u,z)$ have the same zeros with the same multiplicities, $W$ is holomorphic with respect to $z$ and nonzero at the origin. By Cauchy formula, we thus have on a neighborhood of the origin \begin{equation}\label{eq_cauchy}                                                                                                                                                                                                                                                                                                                                                                                                                     W(u,z)=\int_\gamma \frac{W(u,\zeta)}{z-\zeta}d\zeta.\end{equation}
where $\gamma$ is as above. Since $f$ and $\psi$ are both holomorphic functions, $W(u,\zeta)$ is also holomorphic with respect to $u$ on the complement of the zeros of $f$, which, by (\ref{eq_cauchy}), means that it is holomorphic everywhere on a neighborhood of the origin.
\end{proof}

\medskip

Given $y_0\in \R^m$, we write $\A_{y_0}$\nomenclature[ak]{$\A_{y_0}$}{ring of analytic functions germs at $y_0$\nomrefpage}  for the ring of analytic function-germs  at $y_0$.

  
  \begin{cor}\label{cor_noetherian}
   The ring $\A_{y_0}$ is Noetherian.
  \end{cor}
\begin{proof}
Set $y_0=(u_0,z_0)\in \R^{m-1}\times \R$. Arguing inductively, we can suppose that $\A_{u_0}$ is Noetherian, which, by Hilbert's Basis Theorem, means that so is the ring
 $$\A_{u_0}[z]:=\{f\in \A_{y_0} :f(u,z)=\sum_{i=0} ^d a_i(u)z^i, \;\; \mbox{ for some }\;a_i \in \A_{u_0}, d\in \N \}.$$
 Let $\I$ be an ideal of $\A_{y_0}$.  Thanks to Weierstrass Preparation Theorem, we know that $\I$ is generated by $\mathfrak{I}:=\I\cap \A_{u_0}[z]$. As $\mathfrak{I}$ is an ideal of $\A_{u_0}[z]$, which is a Noetherian ring, it is finitely generated. 
\end{proof}

\end{subsection}
\begin{subsection}{Some finiteness properties}
\noindent{\bf Algebraic machinery.} 
 We recall that a ring $R$ is {\bf local}\index{local ring} if it has only one maximal ideal.  It is well-known that the maximal ideal of a local ring is constituted by all the non invertible elements of this ring.


Throughout this section $R$ stands for a Noetherian local ring and $\mfr$ for its maximal ideal.
\begin{dfn}Let $M$ be a finitely generated $R$-module. 
 We say that a decreasing sequence $M_*=\{M_i\}_{i\in \N}$  of submodules of $M$ is an  {\bf $\mfr$-filtration of $M$}\index{m-filtration@$\mfr$-filtration} if $\mfr M_i\subset M_{i+1}$ for all $i\ge 0$.
 
  An $\mfr$-filtration $\{M_i\}_{i\in \N}$ of $M$ is {\bf $\mfr$-stable}\index{m-stable@$\mfr$-stable} if $\mfr M_i= M_{i+1}$ for each integer $i$ sufficiently large.
  
 Given an $\mfr$-filtration $M_*$,  we denote by $\grr(M_*)$, the set of formal polynomials 
 whose $T^i$-coefficient lies in $M_i$, that is to say \nomenclature[al]{$\grr(M_*)$}{graded ring of formal polynomials associated to an $\mfr$-filtration\nomrefpage}
 $$\grr(M_*)=\{ \sum_{i=0}^d a_i T^i: a_i\in M_i, d\in \N\}. $$
In particular, since we can regard $\mfr$ as an $R$-module, $\mfr_*:=(\mfr^i)_{i\in \N}$ is an $\mfr$-filtration of $\mfr$ and 
  $$\grr(\mfr_*)=\{ \sum_{i=0}^d a_i T^i: a_i\in \mfr^i,d\in \N\}. $$
 Clearly, $\grr(M_*)$ is a $\grr(\mfr_*)$-module. 
 

 \end{dfn}
 
\begin{rem}\label{rem_noeth_modules}
 We recall that a module is said to be  Noetherian if every submodule is finitely generated, and that a finitely generated  module over a Noetherian ring is always Noetherian. In particular, in the above definition, all the $M_i$'s are finitely generated.  
\end{rem}
 
\begin{lem}\label{lem_grrM} Let $M$ be a finitely generated $R$-module and let  $M_*=\{M_i\}_{i\in \N}$ be an $\mfr$-filtration of $M$. The  $\grr(\mfr_*)$-module
  $\grr(M_*)$ is finitely generated  if and only if $M_*$ is $\mfr$-stable.
%
 
\end{lem}
\begin{proof}
 Assume that the $\grr(\mfr_*)$-module $\grr(M_*)$ is finitely generated, say by $f_1,\dots f_p$, with for each $j$, $f_j:=\sum_{k=0} ^{d_j} a_{j,k} T^k$, where $a_{j,k} \in M_k$. Set for $i\in \N$ $$Z_i:=\{g\in \grr(M_*):g=c T^i,\mbox{ for some } c\in M_i\} .$$
 
 Note that the elements $a_{j,k} T^k$, $ k\le d_j$, $j\le p$, also generate $\grr(M_*)$. Since $M_*$ is an $\mfr$-filtration, this implies that for $i\ge \max\{ d_j: j=1,,\dots,p\}$
we have $\grr(\mfr_*)Z_i\supset Z_{i+1}$, which means that $\mfr M_i\supset M_{i+1} $, as required.

Conversely, if $M_*$ is $\mfr$-stable then there is $d$ such that $\mfr M_i=M_{i+1}$ for each $i\ge d$, which means that $\mfr^{\kappa}M_i=M_{i+\kappa}$ for every such $i$ and every positive integer $\kappa$. This establishes that $\grr(M_*)$ is generated by the union of the respective generators of $Z_1,\dots, Z_d$, which, since every $M_i$ is finitely generated, is a finite set.
\end{proof}
\begin{lem}\label{lem_subm}Let $M$ be a finitely generated $R$-module and let  $M_*=\{ M_i\}_{i\in \N}$ be an  $\mfr$-stable $\mfr$-filtration of $M$. 
 For every submodule $N$ of $M$, the filtration $N_*:=\{ N \cap M_i\}_{i\in \N}$ of $N$ is $\mfr$-stable. 
\end{lem}
\begin{proof}
 Since $M_*$ is $\mfr$-stable, by Lemma \ref{lem_grrM}, $\grr(M_*)$ is a finitely generated $\grr(\mfr_*)$-module. Notice that $\grr(\mfr_*)$, regarded as an $R$-algebra, is finitely generated.  Since $R$ is Noetherian, by Hilbert's Basis Theorem, this implies that $\grr(\mfr_*)$ is Noetherian as well, which implies in turn that $\grr(M_*)$ is a Noetherian $\grr(\mfr_*)$-module (see Remark \ref{rem_noeth_modules}). This establishes that  $\grr(N_*)$, which is a submodule of $\grr(M_*)$, is a finitely generated $\grr(\mfr_*)$ module, which, by Lemma \ref{lem_grrM}, yields that $N_*$ is $\mfr$-stable.
\end{proof}

\begin{lem}\label{lem_artin_rees}(Artin-Rees) Let $M$ be a finitely generated $R$-module and let $N$ be a submodule. For all $i$ large enough, $$N\cap \mfr^{i+1}M=\mfr(N\cap \mfr^i M).  $$
 In other words, the $\mfr$-filtration $\{N\cap \mfr^i M\}_{i\in \N}$ of $N$  is $\mfr$-stable.
\end{lem}
\begin{proof}
 It suffices to apply Lemma \ref{lem_subm} to the $\mfr$-stable $\mfr$-filtration $\{\mfr^i M\}_{i\in \N}$. 
\end{proof}
%
This leads us to the following famous result which is sometimes rephrased by saying that ideals are
 closed in the $\mfr$-adic topology.
\begin{thm}\label{thm_kit}(Krull's intersection theorem)  Every ideal $\I$ of $R$ satisfies\begin{equation*}\label{eq_krull_int}
                                                                    \bigcap_{i\in \N}(\I +\mfr^i) =\I.
                                                                      \end{equation*}
\end{thm}
\begin{proof}
 We start with the case $\I=0$.  Applying Lemma \ref{lem_artin_rees} to $N:=\bigcap_{i\in \N} \mfr^i$ and $M=R$, we get:
 $$N=\mfr  N .$$
 Assume that $N\ne 0$  and take a minimal system of generators $f_1,\dots, f_p$. We have: $$N=\mfr N=\mfr f_1\oplus \dots\oplus \mfr f_p.$$
 In particular $f_1=\sum_{i=1}^p x_i f_i$ with $x_i\in \mfr$ for all $i$, which implies that $(1-x_1)f_1=\sum_{i=2}^p x_i f_i$. As $R$ is local, $(1-x_1)$ is invertible, which, thanks to the latter equality, means that $f_2,\dots, f_p$ span $N$, in contradiction with our minimality assumption on the system of generators. This yields  $\bigcap_{i\in \N} \mfr^i=\{0\}$, as required.
 
 The general case can now be deduced from this particular case by quotienting by the ideal $\I$. Namely, set $A:=R/\I$  and let $q:R\to A$ be the quotient map. The ring $A$ is local and its maximal ideal is $\overline{\mfr}:=q(\mfr)$.   By the above, we thus have $$\bigcap_{i\in \N} \I+\mfr^i=\bigcap_{i\in \N}q^{-1}( \overline{\mfr}^i)=q^{-1}(\bigcap_{i\in \N} \overline{\mfr}^i)= q^{-1}(\{0\})= \I. $$
\end{proof}

\noindent{\bf Application to rings of analytic function-germs.} Let $y_0\in \R^m$.  We have seen that the ring $\A_{y_0}$ is Noetherian. This ring is moreover local and its  maximal ideal is the ideal $\mfr_{y_0}$ constituted by all the elements of $\A_{y_0}$ that vanish at the origin. 
 The following corollary of the above theorem, sometimes referred as Krull's Theorem, will be useful to us. 

\begin{cor}\label{cor_krull}
Let $h:\A_{y_0}^d\to \A_{y_0}$ be an $\A_{y_0}$-linear mapping and let  $a\in \A_{y_0}$. 
 If the equation $$h(f_1,\dots, f_d)=a$$ has  solutions $f_1,\dots, f_d$ in the ring of formal power series $\R[[Y_1,\dots, Y_m]]$ then it has  solutions $g_1,\dots, g_d$ in the ring $\A_{y_0}$. 
\end{cor}
\begin{proof}Let $f_1,\dots, f_d$ be some formal power series which are solutions and let,  for every $j\le d$ and $i\in \N$,  $f_{j,i}$ denote the polynomial of degree $i$ obtained by truncating the formal series $f_j$ at the order $i$. Clearly $(a-h(f_{1,i},\dots,f_{d,i}))\in \mfr_{y_0}^{i+1}$ which means that $a\in \I +\mfr^i $ for all $i$, where $\I$ is the ideal  $h(\A_{y_0}^d)$. The result thus follows from Theorem \ref{thm_kit}.
\end{proof}
 Artin-Rees' Lemma can be established for a finitely generated module $M$, which entails that Krull's intersection theorem  can actually be proved not only for an ideal, but for every finitely generated $R$-module. Consequently, the above corollary is still valid for a linear mapping   $h:\A_{y_0}^d\to \A_{y_0}^p$, i.e., for a {\it linear system of equations}.
 Furthermore, it is not difficult to see that we can require the $g_i$'s to coincide with the  $f_i$'s at any prescribed order. 
 The just above corollary is however enough for our purpose, as we shall only need the following consequence in the proof of Proposition \ref{pro_preparation_la_fonctions} (which is crucial to establish Theorem \ref{thm_existence_cell_dec}). 


\smallskip

\begin{pro}\label{pro_de_finitude}
Let $y_0=(u_0,z_0)\in \R^{m-1}\times \R$. For each  $\psi \in \A_{y_0}$,  there exist  $A_0,\dots,A_d\in \A_{y_0}$ satisfying $A_i(y_0) \ne 0$ for all $i $, as well as  some $(m-1)$-variable analytic function-germs $c_0,\dots,c_d \in \A_{u_0}$,   such that for all $(u,z)$ near $y_0$:
$$ \psi(u,z)=\sum_{i=0}^d c_i(u)(z-z_0)^i A_i(u,z) .$$
\end{pro}
\begin{proof}We can assume that $y_0$ is the origin.
Write then for $(u,z)$ close to the origin $$\psi(u,z)=\sum_{i\in \N} c_i(u)z^i,$$ with $c_i$ $(m-1)$-variable analytic function-germ for every $i$. Since the ring of germs of analytic functions is Noetherian, the ideal generated by all the $c_i$'s is  generated by the germs at $0$ of a finite family   $c_0,\dots, c_d$,  $d\in \N$. For every $j >d$,  there thus exist  $(d+1)$ germs of analytic functions $b_{0,j},\dots,b_{d,j}$ such that $$c_j(u)=\sum_{i=0}^d b_{i,j}(u)c_i(u).$$
  Hence, as formal power series in the indeterminates $(U,Z)=(U_1,\dots,U_{m-1},Z)$ we have:
\begin{eqnarray*}
\psi(U,Z)&=&\sum_{i=0}^d c_i(U)Z^i+\sum_{j>d} \sum_{i=0}^d b_{i,j}(U) c_i(U)Z^j\\
&=&  \sum_{i=0}^d c_i(U)Z^i(1+Z f_i(U,Z)), 
\end{eqnarray*}
for some formal power series $f_i$, $i=0,\dots,d$. Let us consider the linear equation of formal power series:
$$\psi(U,Z)= \sum_{i=0}^d c_i(U)Z^i(1+Z F_i),$$
in the unknowns $F_0,\dots,F_d$.  This equation has a solution $F_i=f_i$, $i=0,\dots,d$, in the ring of formal power series. By Corollary \ref{cor_krull}, it must have a solution $g_i$, $i=0,\dots,d$, in the ring of convergent power series. It is then enough to set $A_i:=1+z g_i$.  
\end{proof}
\end{subsection}
\end{section}

\section{$\la$-functions and $\la$-cells}\label{sect_la_functions}
As we said in the outline of proof of Theorem \ref{thm_existence_cell_dec}, we will have to work with another class of functions, {\it the $\la$-functions}, that we introduce in this section. It will turn out that every globally subanalytic function is piecewise given by $\la$-functions (Proposition \ref{pro_la_and_globally subanalytic}).

The {\bf cube}\index{cube} of radius $\ep>0$ and centered at $a=(a_1,\dots,a_n) \in \R^n$ is
the set:
 $$\dbb(a,\ep):=\{(x_1,\dots,x_n) \in \R^n:\forall i,\; |x_i-a_i|\leq \varepsilon\}.$$
\nomenclature[am]{$\dbb(a,\alpha)$}{cube centered at $a$ of radius $\alpha$\nomrefpage}

\medskip

A {\bf restricted analytic function}\index{restricted analytic functions} is a function $\psi:C \to \R$, with $C$ cube of $\R^n$, which can be extended  analytically to an open neighborhood of  $C$.
 Let $\aba$ be the set of all the restricted analytic functions (of all the cubes).  

Let now $\lf$ be the set of functions obtained by adding the functions $(a,b)\ni x\mapsto x^\lambda$ (the power functions),
$\lambda \in \Q$ and $a<b \in \R\cup \{\pm \infty\}$ (with $a\ge 0$ if $\lambda\notin \N$ and $a b \ge 0$ if $\lambda<0$), to the family $\aba$. We then introduce the {\bf $\lf$-functions}\index{la-@$\la$-!function} inductively on what we will call the complexity, as follows.
\begin{enumerate}
	\item If $A$ is a globally subanalytic set then the restriction to $A$ of an element of $\lf$ is an $\lf$-function.
	\item If $\psi_1$ and $\psi_2$ are $\lf$-functions then so are $(\psi_1+\psi_2)$ and $\psi_1\cdot \psi_2$.
	\item Any function $\Psi:A\to \R$ of type $x\mapsto \psi(\phi_1(x),\dots,\phi_k(x))$, with $\psi$ as well as $\phi_1,\dots,\phi_k$ $\lf$-functions and $A$ globally subanalytic set, is an $\lf$-function.
\end{enumerate}

 Roughly speaking, the $\lf$-functions are the functions that coincide with the restriction to a globally subanalytic set of a function given by finite sums, products, and
composites of  elements of  $\lf$.
The minimal number of operations (sum, product, composition with an element of $\la$) needed to generate an $\la$-function $f$ is called the {\bf complexity}\index{complexity! of an $\la$-function} of the function $f$.  A mapping whose components are $\lf$-functions is  called an   {\bf $\lf$-mapping}\index{la-@$\la$-!mapping}.

Let us emphasize that $\la$-functions are not only the elements of the class $\la$. They are all what we can generate (using rules (i-iii)) with the elements of $\la$. In (ii) and (iii), we do not stipulate the domain of the functions, which, in virtue of (i), can be any subanalytic set such that the sums, products, and
composites are well-defined.

Since, by Properties \ref{pro_basic_properties_from_dfn}, the class of globally subanalytic mappings is closed under sums, products, and compositions, it directly follows from the above definition that $\la$-mappings are globally subanalytic. Proposition \ref{pro_la_and_globally subanalytic} below can be seen as a partial converse of this fact.  

The class of $\lf$-functions is much bigger than the class of restricted analytic functions, as shown by the following examples.
\begin{exa}
Since every polynomial is an $\la$-function, so is for instance the function $\frac{x-3}{\sqrt{x^2+y^4}}$ (on its domain). The function $e^x$, defined on $\R$, although analytic, is not an $\la$-function.  Its restriction to any bounded interval is however an $\la$-function, which entails for instance that so is the function $\R^2\setminus \{(0,0)\}  \ni (x,y)\mapsto e^{-\frac{x^2}{x^2+y^2}}$. 
\end{exa}

	 We then define the {\bf $\la$-cells}\index{la-@$\la$-!cell} of $\R^n$ by induction on $n$. 
Let $C$ be a cell of $\R^n$ and denote by  $B$ its basis. If $ n=0$ then $C$ is always an $\la$-cell. If $n\ge 1$, we say that $C$ is an  $\la$-cell if so is $B$  and if in addition one of the following properties holds:
	\begin{enumerate}[(i)]
	\item  $C$ is the graph of some  $\la$-function $\xi :B \to \R$. 
	\item  $C$ is a band $(\xi_1,\xi_2)$, $\xi_1<\xi_2$, where  $\xi_1$ is either   $ -\infty$  or an  $\la$-function on $B$, and $\xi_2$ is either $ +\infty$ or  an $\la$-function on $B$.
\end{enumerate}
 A cell decomposition of $\R^n$ consisting exclusively of $\la$-cells is called an {\bf $\la$-cell decomposition}. \index{la-@$\la$-!cell}  \index{la-@$\la$-!cell decomposition}

\section{Reduced functions}\label{sect_reduced_functions}  Roughly speaking, the reduced functions will be the functions that have a nice form, up to an $\la$-unit, which requires to define $\la$-units.

\begin{dfn}Let $C\subset \R^n$ denote an $\la$-cell, $B$ its basis, and let $\theta:B\to \R$ be an $\la$-function  satisfying $\Gamma_\theta\cap C=\emptyset$.

An {\bf $\la$-unit}\index{la-@$\la$-!unit} (in the variable $y$) of $C$ is a function $U$ on $C$ that can be written $\psi (V(x))$ with, for $x:=(\xt,x_n)\in C\subset \R^{n-1} \times \R$: $$V(x)=( b_1(\xt),\dots ,b_k(\xt), u(\xt)\,y^{\frac{1}{s}},  v(\xt)\,
y^{-\frac{1}{s}} ), \quad y:=|x_n-\theta(\xt)|,$$
 where $s \in \N^*$,  $b_1,\dots, b_k, u,v $ are analytic
$\la$-functions  on $B$ such that $V(C)$ is relatively compact,  and where  $\psi$ is an  analytic function on a neighborhood of   $cl(V(C))$  nowhere vanishing on this set.
\end{dfn}

%

We can say that an $\la$-unit is a Puiseux series in $y$ and $\frac{1}{y}$ which has coefficients that are analytic $(n-1)$-variable $\la$-functions on the basis of $C$, and which is bounded away from zero and infinity.
\begin{dfn}\label{dfn_reduced} Let $E\subset \R^n$ and let $C\subset E$ be an $\la$-cell of basis $B$.
A  function $\xi:E\to \R$ is {\bf reduced}\index{reduced function} on  $C$ if
 we can find analytic $\la$-functions   $a:B\to \R$ and $\theta:B\to \R$, as well as  $r \in \Q$, such that
 $\Gamma_\theta\cap C=\emptyset$ and
\begin{equation}\label{eq_preparation_dfn}
\xi(\xt ,x_n)= a(\xt)\cdot y^r\cdot U(\xt,y) , \quad y:=|x_n-\theta(\xt)|,
  \end{equation}
for all  $x=(\xt,x_n) \in C\subset \R^{n-1} \times \R$, where  $U$ is an $\la$-unit of $C$ in the variable $y$.

The function $\theta$ is then called the {\bf $\la$-translation of the reduction}\index{la-@$\la$-!translation}.
If $r$
is nonnegative, we say that the reduction is {\bf nondegenerate}\index{nondegenerate reduction}.

A function $\xi:E \to \R$ is {\bf reducible}\index{reducible function} if there is an  $\la$-cell decomposition compatible with $E$ such that $\xi$ is reduced on every cell $C \subset E$ of this cell decomposition.
\end{dfn}

In short, a  reduced function is merely, up to
a product with an $\la$-unit,  a power of $|x_n-\theta(\xt)|$ times an $(n-1)$-variable analytic $\la$-function $a$.

\begin{rem}\label{remark_unit}
Let a function $\xi$ be  reduced  with $\mathcal{L}$-translation $\theta$ on an $\la$-cell $C$ and let $\theta'$ be an $\la$-function on the basis of $C$. It easily follows from the above definitions that if $(x_n - \theta)$ is reduced on $C$ {\it with $\mathcal{L}$-translation $\theta'$} then so is $\xi$.
\end{rem}

\begin{exa}
In the above definition, the graph of $\theta$ lies outside the cell $C$ and thus can be bigger than the zero locus or the poles of the function.
  Let $f(x,y):=\frac{x^3}{y}$ be defined on the set
  $A:=\{(x,y)\in (0,1]^2: x^2 \le y\le x  \}$.   Although $f$ is smooth and nonzero on a neighborhood of $A$, any reduction of this function will involve on some cell an $\la$-translation $\theta$ whose graph is tangent to the $x$-axis. Note also that,  although $f$ is bounded, it is not possible to reduce $f$ in such a way that the exponent $r$ in (\ref{eq_preparation_dfn}) is positive on every cell.
\end{exa}

The main difficulty of this chapter is to show the following fact.

 \begin{pro}\label{pro_preparation_la_fonctions}
 If $C$ is an $\la$-cell then every $\la$-function on $C$  is reducible.
 \end{pro}


\medskip

The proof of this result occupies the whole of section \ref{sect_reducing_la_functions}.  To motivate the preliminaries that we shall carry out, we sketch the main strategy of the proof.

\paragraph{Strategy of the proof of Proposition \ref{pro_preparation_la_fonctions}.}
 The proof is carried out by induction on the number of variables of the function.
 $\la$-functions are explicitly known: they are finite sums, products, and  composites of power functions and restricted analytic functions. Arguing also by induction on the complexity of the function, we just have to show that each of these  operations (sum, product, power, composition with a restricted analytic function) preserves reducible functions. The main issue is indeed to show that composition with a restricted analytic function preserves reduced functions (Proposition \ref{pro_composition_analytique}). The difficulty in studying  $\la$-functions is that  they involve negative powers. The strategy is  to split the proof of  Proposition \ref{pro_composition_analytique} into two steps: we will first show, relying on Weierstrass' Preparation Theorem, that we can reduce $n$-variable $\la$-functions that are analytic with respect to the last variable $x_n$ (Proposition \ref{pro_prep_A_functions}) and then show that we can handle $n$-variable $\la$-functions that are  analytic functions of both  $x_n$  and $\frac{c(\xt)}{x_n}$, where $c$ is an $\la$-function (Proposition \ref{pro_lionrolin_1}).

\section{Reduction of $\la$-functions}\label{sect_reducing_la_functions}
All this section is devoted to the proof of  Proposition \ref{pro_preparation_la_fonctions}, which will be carried out by induction on the dimension of the ambient space. More precisely, we shall establish the  following facts by induction on  $n$:

\medskip

\noindent $(\mathcal{H}_n) \quad$ If $C$ is an $\la$-cell of $ \R^n$ then every $\la$-function $\xi:C \to \R$ is reducible.
\nomenclature[an]{$(\mathcal{H}_n)$}{induction hypothesis of the proof of proposition \ref{pro_preparation_la_fonctions}\nomrefpage}

\bigskip

The assertion $(\hn_0)$ is vacuous.  Fix $n\ge 1$ and assume that $(\hn_i)$ holds true for all $i<n$.

\medskip

 \noindent {\bf Note.} All the propositions and  lemmas of this section will also be proved by
induction on $n$. Hence, we
 will assume that all of them are true when $n$ is replaced by $(n-1)$ and simply check
them for this fixed value of  $n$ (the case $n=0$ always being trivial).

\begin{dfn}\label{dfn_comparable} We say that two functions $f_1:A \to \R$ and $f_2:A \to \R$ are {\bf comparable}\index{comparable} if either $f_1(x)\le f_2(x)$ for all $x \in A$ or $f_2(x)\le f_1(x)$ for all $x \in A$. 
We say that a finite collection of functions $f_i:A \to \R$, $i=1,\dots,k$, is {\bf totally ordered}\index{totally ordered} if  the $f_i$'s are all pairwise comparable with each other.

Given $x\in \R$\index{sign}\index{sign!constant}, let $sign(x):=1$ if $x$ is positive and $sign(x):=-1$ whenever $x$ is negative, and set  $sign(0):=0$. 
We say that a function $\xi:A \to \R$ {\bf has constant sign}\index{constant sign} on $B \subset A$ if the function $sign(\xi(x))$ is constant on $B$.\footnote{A nonnegative function is thus not always of constant sign (since it can vanish).}

\end{dfn}

\subsection{A few consequences of $(\hn_{n-1})$}
The proof of Proposition \ref{pro_preparation_la_fonctions} requires the following lemmas. As we mentioned, we shall just prove these lemmas for our fixed value of $n$.
\begin{lem}\label{lem_la_cell_decomposition} Let $C_1,\dots,C_k$ be $\la$-cells of $\R^{n}$ and let $\xi_i:C_i \to \R$, $i=1,\dots,k$ be reduced  $\la$-functions.
 There is an $\la$-cell decomposition of $\R^{n}$ compatible with the $C_i$'s  and such that for every $i$ the function $\xi_i$ has constant sign on every cell $D \subset C_i$ of this $\la$-cell decomposition.
\end{lem}

\begin{proof}
 For every $i\leq k$, since  $\xi_i$  is reduced,
 there  are  $(n-1)$-variable $\la$-functions $a_{i}$ and $\theta_{i}$ on $B_i$ (here $B_i$ stands for the basis  of $C_i$) such that:
$$\xi_i(\xt ,x_n)= a_{i}(\xt)\cdot y^{r_i}_i\cdot U_{i}(\xt,y_i) , \quad y_i:=|x_n-\theta_{i}(\xt)|,$$
where  $U_{i}(x,y_i)$ is an $\la$-unit in the variable $y_i$ and $r_i\in \Q$.
 Let also $\zeta_{1},\dots,\zeta_p$ be all the $(n-1)$-variable $\la$-functions defining the cells  $C_1,\dots, C_k$ (see Definition \ref{dfn_cell_decomposition}).

 By  $(\mathcal{H}_{n-1})$, the $a_i$'s are reducible.  Moreover, by induction on $n$, this lemma holds true if $n$ is replaced with $(n-1)$. We therefore can find an $\la$-cell decomposition $\C$ of $\R^{n-1}$ compatible with  the $B_i$'s and  such that for every $i\le k$, the $(n-1)$-variable function $a_{i}$ has constant sign on every cell (included in $B_i$).  For the same reason, we may also assume that the family constituted by the respective restrictions of the functions $\theta_{i}$, $i \le k$,  together with the functions $\zeta_i$, $i\le p$,  to each cell of $ \C$ (on which these functions are defined), is totally ordered.

 The cells of the desired $\la$-cell decomposition $\C'$ of $\R^n$ are now given by the  graphs and bands defined by  the respective restrictions   of the functions $\theta_{i}$, $i=1,\dots, k$,   and $\zeta_i$, $i=1,\dots, p$,   to the cells of $\C$  on which they are defined. 

Fix  $i \le k$ and let  $D$ be a cell of $\C'$ included in  $C_i$.
On  $D$,  since the functions $y_i$, $U_{i}$ (by definition of $\la$-units), and $a_{i}$ have constant sign, we see that  $\xi_i$ has constant sign as well.
\end{proof}

\begin{rem}\label{rem_la_cell}This lemma entails that,
given an $\la$-cell $C$ of $\R^n$ of basis $B$ and an $\la$-function $\phi:B\to \R$, we can find an $\la$-cell decomposition of $\R^n$ compatible with $C$ such that on every cell $E\subset C$ we have for all $(\xt,x_n)\in E$
either  $|x_n|\le |\phi(\xt)|$ or $|x_n| \ge |\phi(\xt)|$. 
Indeed, it is enough to apply Lemma \ref{lem_la_cell_decomposition} to the reduced $\la$-functions $\phi$, $x_n$, $(x_n+\phi)$, and $(x_n-\phi)$.
\end{rem}

\begin{rem}\label{rem_common_refinement} Applying the above lemma to the case where the $C_i$'s are the cells of two given cell $\la$-decompositions of $\R^n$, we can conclude that two $\la$-cell decompositions of $\R^n$ have a common refinement.
\end{rem}

\begin{rem}
 A reduced function on a cell is analytic on this cell. By $(\hn_{n-1})$ and the preceding remark, it means that, up to a refinement of the cell decomposition, we can always assume that some given $(n-1)$-variable $\la$-functions are analytic on every cell. 
\end{rem}

\begin{lem}\label{lem_image_inverse}
Let $C$ be an $\la$-cell of $\R^{n}$, $\phi$ be an  $\la$-function on the basis $B$ of $C$, $p\in\Z$,  and define a function on $C$ by $\xi(x):=\phi(\xt) x_n^p$, where $x=(\xt,x_n)$. Given two real numbers $a <b$, there is an $\la$-cell
 decomposition compatible with $\xi^{-1}([a,b])$.
\end{lem}
\begin{proof}
By Lemma \ref{lem_la_cell_decomposition}, there is an $\la$-cell decomposition compatible with $C$ such that $\phi$ and $x_n$ have constant sign on every cell. 
If $E\subset C$ is a cell of this cell decomposition, it is easily checked that $\xi^{-1}_{|E}([a,b])$ can be described by sign conditions on reduced functions. The result thus  follows after applying again Lemma \ref{lem_la_cell_decomposition}.
\end{proof}

\begin{lem}\label{lem_meme_morphisme}
 Given finitely many
reduced functions $\xi_1,\dots,\xi_k$ on an $\la$-cell $C$ of $\R^{n}$, we can find an $\la$-cell decomposition compatible with $C$
 such that on every cell $E \subset C$, all these functions are reduced {\it with the same $\la$-translation}.
\end{lem}
\begin{proof}Let $\theta_1,\dots,\theta_k$ denote the respective $\la$-translations of the reductions of the $\xi_i$'s. 
 As a consequence of Lemma \ref{lem_la_cell_decomposition} (see Remark \ref{rem_la_cell}), there is an $\la$-cell decomposition of $\mathbb{R}^n$ compatible with $C$ such that on each cell $D \subset C$, the functions $|x_n - \theta_i|,  |\theta_i - \theta_j|$, $i<j\le k$, are comparable with each other (see Definition \ref{dfn_comparable}) and the functions $(x_n - \theta_i)$,  $(\theta_i - \theta_j)$, $i<j\le k$, have constant sign.  Fix a cell $D$ and choose $j$ such that on $D$ we have for all $i\le k$:
\begin{equation}\label{fm2.3}
|x_n - \theta_j| \leq |x_n - \theta_i|.
\end{equation}
We are going to show that for all $i$, the function  $(x_n - \theta_i)$ is reduced on  $ D$ with $\mathcal{L}$-translation  $\theta_j$, so that the statement of the lemma will then follow from Remark \ref{remark_unit}. Fix $i \le k$.
The proof now breaks down into two cases.

\medskip

\underline{Case $1$}: $|x_n - \theta_j| \leq |\theta_j - \theta_i|$ on $D$.

\medskip 

If there is $\xt$ in the basis of  $D$ such that $\theta_j(\xt) = \theta_i(\xt)$ then $x_n \equiv \theta_j \equiv \theta_i$ on $D$ (since $(\theta_i - \theta_j)$ has constant sign), and the result is trivial. Otherwise, either the two functions $(x_n - \theta_j)$ and $(\theta_j - \theta_i)$ have the same sign or $\big|\dfrac{x_n - \theta_j}{\theta_j- \theta_i}\big| \leq \dfrac{1}{2}$ (by (\ref{fm2.3})). It means that $(1+\dfrac{x_n - \theta_j}{\theta_j- \theta_i} )$ is an $\mathcal{L}$-unit in the variable $(x_n-\theta_j)$ and hence, the function $(x_n - \theta_i)$ can be reduced by
\begin{equation}\label{eq_reduced_theta12}
x_n - \theta_i = (\theta_j - \theta_i)(1+\dfrac{x_n - \theta_j}{\theta_j- \theta_i}).
\end{equation}

\medskip

\underline{Case $2$}:  $|x_n - \theta_j| \geq |\theta_j - \theta_i|$ on $D$.

\medskip

If there is $x=(\xt,x_n) \in C$ such that $x_n = \theta_j(\xt)$ then $x_n\equiv \theta_j$ on $C$,  and the result is trivial. Otherwise, by (\ref{fm2.3}), $(x_n - \theta_j)$ and $(\theta_j - \theta_i)$ are of the same sign, and hence the function $(1+\dfrac{ \theta_j - \theta_i}{x_n -\theta_j}) $ is an $\mathcal{L}$-unit. As a matter of fact, the function $(x_n - \theta_i)$ can be reduced by writing
\begin{equation}\label{eq_reduced_theta21}
x_n - \theta_i= (x_n - \theta_j)(1 + \dfrac{\theta_j- \theta_i}{x_n -\theta_j}).
\end{equation}
\end{proof}

\begin{rem}\label{rem_produit}
A direct consequence of Remark \ref{rem_common_refinement} and  Lemma \ref{lem_meme_morphisme} is that the product of two reducible functions $\xi_1:C\to \R$ and $\xi_2:C\to \R$, where $C$ is an $\la$-cell of $\R^n$, is reducible. The quotient, if well-defined, is also reducible.
\end{rem}

\begin{lem}\label{lem_refinement_fonctions_dominantes}
Given two  reduced functions $f$ and $g$  on an $\la$-cell  $C$ of $\R^n$ and a positive constant $\ep$, there is an $\la$-cell decomposition compatible with $C$ such that on every cell $E \subset C$ either $|f| \le \ep |g|$, or $|g| \le \ep |f|$, or $|f|\sim |g|$.
\end{lem}
\begin{proof}
By Lemma \ref{lem_meme_morphisme}, we may assume that $f$ and $g$ are reduced on the cells of an $\la$-cell decomposition with {\it the same} $\la$-translation on every cell.
Since it is enough to establish the lemma for arbitrarily small values of $\ep>0$ and because $\la$-units are bounded away from zero and infinity, it is enough to prove the result for some functions of the form $f(x)=a(\xt)|x_n-\theta(\xt)|^r$ and $g(x)=b(\xt)|x_n-\theta(\xt)|^s$, with $a$, $\theta$, and $b$ $\la$-functions on the basis $B$ of an $\la$-cell $C$, $r$ and $s$ in $\Q$, and $x=(\xt,x_n)\in C$.

 By Lemma \ref{lem_la_cell_decomposition}, we may assume that $a$,  $(x_n-\theta)$, and $b$ have constant sign on $C$. The inequality $|f|\leq \ep |g|$ on a cell $C$ now amounts to inequalities of type $x_n \le \phi(\xt)$ or $x_n\ge \phi(\xt)$ with $\phi$ $\la$-function on its basis $B$. But, again by Lemma \ref{lem_la_cell_decomposition}, given any $\la$-function $\phi$ on $B$, there is a refinement of our $\la$-cell decomposition such that the function $(x_n-\phi)$ has constant sign on every cell.
\end{proof}

\begin {lem}\label{lem_xn_et_theta}
Given an $n$-variable reducible function $\xi$, we may always assume, up to a refinement of the $\la$-cell decomposition, that on every cell
either $x_n \sim \theta(\xt)$ or $\theta(\xt)\equiv 0 $, where $\theta$ is the $\la$-translation of the reduction on the cell. 
\end {lem}
\begin{proof} Take a cell decomposition such that $\xi$ is reduced on every cell. 
 By Lemma \ref{lem_refinement_fonctions_dominantes} and Remark \ref{rem_common_refinement}, up to a refinement of the $\la$-cell decomposition we may assume that on every cell either $|x_n| \sim |\theta(\xt)|$ or $|x _n| \leq \frac
 {1}{2} |\theta(\xt)|$ or $|\theta (\xt)|\leq  \frac
 {1}{2}  |x_n| $.  
By  Lemma \ref{lem_la_cell_decomposition}, $x_n$ and $\theta$ may be assumed to  be of constant  sign on every cell, and we will assume them to be nonzero (since otherwise we are done).
If on a cell  $|\theta (\xt)|\leq  \frac
 {1}{2}  |x_n| $ then writing \begin{equation}\label{eq_x_n_moins_theta}
x_n-\theta(\xt)=x_n (1-\frac{\theta(\xt)}{x_n}),
  \end{equation}
we see that, since $ (1-\frac{\theta(\xt)}{x_n})$ is an $\la$-unit, we can assume the $\la$-translation of the reduction to be identically $ 0$. 

 Similarly, in the case where  $|x _n| \leq \frac
 {1}{2} |\theta(\xt)|$,  writing  $x_n-\theta(\xt)=\theta(\xt) (\frac{x_n}{\theta(\xt)} -1)$ also  immediately reduces to the case where the $\la$-translation of the reduction is identically $ 0$.

Finally, assume that  $|x_n|\sim |\theta(\xt)|$. If  these two functions are of the same sign, then $x_n\sim \theta(\xt)$ and we are done. Otherwise,   by  (\ref{eq_x_n_moins_theta}), 
we see that since $ (1-\frac{\theta(\xt)}{x_n})$ is an $\la$-unit when $x_n$ and $\theta$ have opposite signs and are equivalent, we can assume the $\la$-translation of the reduction to be identically $ 0$.
\end{proof}

\begin{lem}\label{lem_up_to_H}  Let $C\subset \R^{n} $ be an $\la$-cell of basis $B$ and  $c:B\to \R$  an $\la$-function. Let  $p\in \Z$ be such that the   $\la$-mapping $$H(x):=(\xt,c(\xt)| x_n|^{1/p}), \quad x=(\xt,x_n)\in C,$$ is well-defined on $C$.
If an $\la$-function $\zeta:H(C)\to \R$ is reducible then so is $\xi:=\zeta\circ H:C\to \R$. 
\end{lem}
\begin{proof}
By Lemma \ref{lem_la_cell_decomposition}, we can  assume that $x_n$ and $c$ are of constant (nonzero) sign on $C$  (we will assume that they are positive for simplicity). Observe that  $H(C)$ is then an $\la$-cell and that the inverse image under  $H$  of an $\la$-cell included in $H(C)$ is an $\la$-cell.
If $\zeta$ is reducible then $\xi$ may  be written on every cell $E \subset C$ of a suitable  $\la$-cell decomposition compatible with $C$: $$\xi(x)=\zeta(H(x))=b(\xt)\cdot y(x)^r \cdot U(\xt,y(x)), \quad y(x):=|c(\xt)\cdot x_n^{1/p}-\theta(\xt)|,$$
where $U(\xt,y(x))$ is an $\la$-unit in the variable $y$ on $E$, $r \in \Q$, and $\theta$ as well as $b$ are $\la$-functions on the basis of $E$. It suffices to show that $y$ is reducible. 

By Lemma \ref{lem_la_cell_decomposition}, refining the $\la$-cell decomposition if necessary, we may assume that  $\theta$ has constant sign on every cell included in $C$. 
If  $\theta \equiv 0$ on a cell   then the result is clear on this cell. Moreover,
possibly rewriting $y$ as $x_n^{1/p} \cdot |c(\xt)-\theta(\xt)\cdot x_n^{-1/p}|$, we see that it suffices to address the case where $p$ is positive (see Remark \ref{rem_produit}).
As $c$ nowhere vanishes, $y$ can be factorized $c(\xt)|x_n^{1/p}-\frac{\theta(\xt)}{c(\xt)}| $, which means that, up to a change of $\theta$, we may assume that $c\equiv 1$. %

 Thanks to Lemma \ref{lem_xn_et_theta}, we can suppose that $x_n\sim \theta(\xt)$ on $H(E)$. As we assume $c\equiv 1$, it means that we have on $E$ 
\begin{equation}\label{eq_c_x_theta}
 x_n\sim \theta(\xt)^p   .                                                                                                                                                                                                                                                                                                                                                                                                                                                                                                                                                                                                                                                                                                                                      \end{equation}
Write then \begin{equation}\label{eq_fractionx_n_theta}
y= x_n^{1/p}-\theta(\xt)=\frac{x_n-\theta(\xt)^p}{x_n^{\frac{p-1}{p}}+x_n^{\frac{p-2}{p}}\theta(\xt)+\dots + \theta(\xt)^{p-1}}.
\end{equation}
Let $D(x)$ denote the denominator of this fraction. As $D$ is a sum of positive terms which are all $\sim$ to $ \theta(\xt)^{p-1}$ (by (\ref{eq_c_x_theta})), we clearly have $D(x)\sim \theta(\xt)^{p-1}$. Hence,   the function $W(x):=\theta(\xt)^{1-p}D(x)$ is  bounded away from zero and infinity. It is therefore an $\la$-unit, which means in particular that it is reduced. 
 Therefore, by Lemma \ref{lem_meme_morphisme} (see Remark \ref{rem_produit}) and  (\ref{eq_fractionx_n_theta}), $y$ is reducible. 
\end{proof}

 \subsection{Reduction of $\mathcal{L}$-functions}
The first step of the reduction process deals with functions which, roughly speaking, are ``analytic in the last variable'':
\begin{pro}\label{pro_prep_A_functions} Let $C\subset \R^n$ be an $\la$-cell of basis $B$ and let $\phi_1,\dots,\phi_k$ be $\la$-functions on $B$. Set  $$\Phi(x):=(\phi_1(\xt),\dots,\phi_k(\xt),x_n), \quad x=(\xt,x_n) \in C,$$  and let  $\psi$ be an analytic function on a neighborhood of $cl(\Phi(C))$.  If $\Phi(C)$ is bounded then the function $\xi:=\psi\circ \Phi$ admits a nondegenerate reduction.
 \end{pro}
\begin{proof}
\noindent {\bf Step 1.} Reduction to the case where $\psi(u_1,,\dots,u_k,z)$ is $z$-regular on a neighborhood of $cl(\Phi(C))$.

By Proposition \ref{pro_de_finitude}, for every  $y_0=(u_0,z_0) \in cl(\Phi(C))\subset \R^k \times \R$,  there is $\ep_{y_0}>0$ and $d\in \N$ for which $\psi$ has the following form on the cube $\dbb(y_0,\ep_{y_o})\subset \R^k \times \R$:
\begin{equation}\label{eq_psi_horm}
 \psi(u,z)=\sum_{i=0}^d c_i(u)(z-z_0)^i A_i(u,z) ,
\end{equation}
where the $c_i$'s are  analytic functions on $\dbb(u_0,\ep_{y_0})$ and the $A_i$'s are  analytic functions on $\dbb(y_0,\ep_{y_0})$ nowhere zero on this set.

 As $cl(\Phi(C))$ is compact, we can extract a finite covering by such cubes. By Remark \ref{rem_common_refinement} and Lemma \ref{lem_image_inverse}, there is an $\la$-cell decomposition compatible with the inverse images under $\Phi$ of the elements of this finite covering.

 Fix a cell $E\subset C$ (to show that $\xi$ is reducible, we may focus on one single cell, by Remark \ref{rem_common_refinement}). By construction, $\Phi(E)$ fits in some cube $\dbb(y_0,\ep_{y_0})$, with $y_0=(u_0,z_0) \in cl(\Phi(E))$, on which (\ref{eq_psi_horm}) holds.  Define some  $\la$-functions  on the basis $D$ of $E$ by: $$g_i(\xt):=c_i(\phi_1(\xt)\dots,\phi_k(\xt)), \quad   i=0,\dots, d.$$ By $(\hn_{n-1})$ and Lemma \ref{lem_la_cell_decomposition},  up to a refinement of the $\la$-cell decomposition, we can assume that the family $|g_0|,\dots, |g_d|$ is totally ordered on $D$, and that  if some $g_i$  vanishes on  $D$ then it is identically zero on this set.

 Let $m\leq d$ be the smallest integer  such that $|g_m|=\max_{i\leq d} |g_i|$ on $D$.  If all the $g_i$'s  are zero on $D$, the function $\xi$ is reduced on $E$, since  identically zero. We thus will assume that  $g_m$ is nowhere zero.

Define now a function $\varphi$ by setting for $(t,u,z)\in \R^{d+1}\times \dbb(y_0,\ep_{y_0})$:
 $$\varphi(t,u,z):=\sum_{i=0} ^d t_i\, (z-z_0)^i\, A_i(u,z).$$
 A straightforward computation of partial derivatives shows that this analytic function is $z$-regular of order at most $m$ at any $(t,u,z)$, with $t \in [0,1]^{d+1}$ satisfying $t_m=1$ and $(u,z) \in \dbb(y_0,\ep_{y_0})$ (if $\ep_{y_0}$ was chosen small enough). To complete  Step $1$, we are going to show that it suffices to work with $\varphi$ instead of $\psi$.

 As $g_m$ nowhere vanishes on $D$, by (\ref{eq_psi_horm}), for every $x\in E$ we have  $\xi(x)=g_m(\xt) \cdot \varphi( \Theta(x))$, where $\Theta$ is  the bounded  $\la$-mapping $$\Theta(x):=(g_0(\xt)/g_m(\xt),\dots,g_d(\xt)/g_m(\xt),\Phi(x)).$$
As $g_m$ is an $(n-1)$-variable function, it is enough to reduce $\varphi \circ \Theta$. As $\varphi(u,z)$ is $z$-regular at any point of $cl(\Theta(E))$, this completes Step $1$.

  \medskip

  \noindent {\bf Step 2.} Proof in the case where $\psi(u_1,\dots,u_k,z)$ is   $z$-regular near  $cl(\Phi(C))$.

   Let $d$ be the greatest order of  $z$-regularity of $\psi$  near $cl(\Phi(C))$ and let us argue by induction on $d$ (the case $d=0$ being trivial).

 By Weierstrass Preparation Theorem, given  a point of $cl(\Phi(C))$, there is a cube centered at this point such that the function $\psi$ is,
 up to a unit, a polynomial with analytic coefficients.
 As $cl(\Phi(C))$ is compact, we can extract a finite covering by such cubes.
 By  Lemma \ref{lem_image_inverse} (and Remark \ref{rem_common_refinement}), there is  an $\la$-cell decomposition $\E$ of $\R^n$ compatible with $C$ and the respective preimages under $\Phi$ of all these cubes. Fix a cell $E \in \E$ included in $C$.
Since we can argue up to a unit and up to a translation, we will assume that $cl(\Phi(E))$ fits in a  cube centered at the origin on which $\psi$ coincides with a monic polynomial in $z$ with analytic coefficients in $u$:
$$\psi(u,z)=z^d+a_1(u)z^{d-1}+ \dots+a_d(u), \quad (u,z)\in\R^k\times \R.$$
If we make the change of variable $z \mapsto z -\frac{a_1(u)}{d}$, the coefficient of  $z^{d-1}$ of this polynomial  becomes zero.
 Consequently, since it is enough to show that the  function $\xi(\xt,x_n+\frac{a_1(\phi(\xt))}{d})$, where $\phi=(\phi_1,\dots,\phi_k)$, is reducible, we will assume that $a_1 \equiv 0$.
For simplicity, set for $\xt$ in the basis of $E$
 $$b_i(\xt):=a_i(\phi_1(\xt),\dots,\phi_k(\xt)), \quad i=2,\dots, d.$$  
By Lemma \ref{lem_la_cell_decomposition}  and $(\hn_{n-1})$,  refining the  $\la$-cell decomposition if necessary, we can assume  the $b_i$'s to have constant sign  and    the family of $\la$-functions $| b_i|^{\frac{1}{i}}$, $i=2,\dots,d$, to be totally ordered
   on  $E$.
 Let $j$ be such that on $E$:\begin{equation}\label{eq_dfn_phi_j}
                                                                | b_j|^{1/j}=\max_{2\le i\le d} |b_i|^{1/i}.
                                                               \end{equation}
If all the $b_i$'s are identically zero, then $\xi$ is indeed already reduced. Since $b_j$ is of constant sign on $E$, we will therefore assume below that it nowhere vanishes on $E$.
Up to one more refinement of the $\la$-cell decomposition (see Remark \ref{rem_la_cell}),
we may assume that one of the following two cases occurs on $E$:

  \begin{equation}\label{eq_caseI}
   \hskip -7.5cm
                 \mbox{\bf Case I: } |x_n |\leq 2 |b_j(\xt)|^{\frac{1}{j}}, \;\mbox{for } (\xt,x_n)\in
E .
  \end{equation}

In this case, we first are going to change  $\xi$ for a function  $\hat{\xi}=\hat{\psi}\circ \hat{\Phi}$, with $\hat{\psi}$ $z$-regular of order less than $d$ (and $\hat{\Phi}$ as in (\ref{eq_phic}) below).
Let  $$\hat{\xi}(\xt,x_n):=|b_j(\xt)|^{-d/j}\cdot \xi(\xt,|b_j(\xt)|^{1/j}\cdot x_n).$$
This function is defined on the cell $E'$ which is the image of $E$ under the mapping $(\xt,x_n)\mapsto (\xt,|b_j(\xt)|^{-1/j}\cdot x_n)$. 
It is clearly enough to show that $\hat{\xi}$ is reducible. Observe that $\hat{\xi}$ is nothing but the composite of the $d$-variable polynomial function $$\hat{\psi}(v_2,\dots,v_d,z):= z^d+v_2z^{d-2}+ \dots+ v_d$$ with the $\la$-mapping on $E'$:
\begin{equation}\label{eq_phic}\hat{\Phi}(x)=\lepa\frac{b_2(\xt)}{|b_j(\xt)|^{2/j}}\,,\dots,\, \frac{b_d(\xt)}{|b_j(\xt)|^{d/j}},\, x_n\ripa .\end{equation}
By (\ref{eq_dfn_phi_j}) and (\ref{eq_caseI}), the set $\hat{\Phi}(E')$ is bounded ((\ref{eq_caseI}) entails that the coordinate $x_n$ of $x \in E'$, is bounded away from infinity on $E'$). As
  $\hat{\psi}(v,z)$ is $z$-regular of order at most $(d-1)$ at every $(v,z) \in cl(\hat{\Phi}(E'))$ (since $v_j \neq 0$  for all $(v,z)=(v_2,\dots,v_{d},z)$ in this set and $\frac{\pa^{d-1} \hat{\psi}}{\pa z^{d-1}}(v,z)=d!\, z$), the result follows by induction on $d$.

  \begin{equation}\label{eq_caseII}
   \hskip -7.3cm
                 \mbox{\bf Case II: }  |x_n|  \geq 2 |b_j(\xt)|^{\frac{1}{j}}, \; \mbox{ for  }  (\tilde{x},x_n)\in
E.
  \end{equation}

In this case,    $\xi$ can be easily reduced as follows.
By (\ref{eq_dfn_phi_j}) and (\ref{eq_caseII}), $
|\frac{b_k(\xt)}{x_n^k}|\le \frac{1}{2^k}$ on $E$, for every $k $.
 As a matter of fact,
 $$|\frac{b_2(\xt)}{x_n^2}+\dots+\frac{b_{d}(\xt)}{x_n^{d}}|
<\frac{1}{2},$$
so that  the equality
$$\xi(\xt,x_n)=x_n^d  \lepa 1+\frac{b_2(\xt)}{x_n^2}+\dots+\frac{b_{d}(\xt)}{x_n^d}\ripa $$
reduces  $\xi$ on the given
cell.
\end{proof}

We now are going to deal with $\la$-functions that are analytic in both  $x_n$  and $\frac{c(\xt)}{x_n}$, where $c$ is an $\la$-function (see Proposition \ref{pro_lionrolin_1} below). The strategy is to split the considered function into two functions, one analytic in $x_n$ and one analytic in  $\frac{c(\xt)}{x_n}$, in order to apply Proposition \ref{pro_prep_A_functions}. To this end, the following lemma will be needed.

\begin{lem}\label{lem_splitting}
Let $\psi$ be an analytic function on a neighborhood of $(a,0,0)$ in $\R^{k+2}$, $a \in \R^k$. There exist $\ep>0$ and two analytic functions $\psi_1$ and $\psi_2$ on some compact cubes  such that
\begin{equation}\label{eq_lem_splitting} \psi (u,z,\frac{c}{z})=z\psi_1(u,c,z)+\psi_2(u,c,\frac{c}{z}),\end{equation}
for every  $(u,c,z)$  satisfying $(u,z,\frac{c}{z})\in \dbb(a,\ep)\times [-\ep,\ep]^2$.
\end{lem}
\begin{proof}
If $\sum b_{m,i,j} (u-a)^m z^i t^j $ denotes  the Taylor expansion at $(a,0,0)$ of the function $\psi(u,z,t)$, it suffices to set:
    $$\psi_1(u,c,z):=\sum_{\underset{m\in \N^k}{0\le j< i} } b_{m,i,j} (u-a)^m c^j z^{i-j-1}\,  \mbox{ and }\, \psi_2(u,c,t):=\sum_{\underset{m\in \N^k}{0\le i\le j}}  b_{m,i,j} (u-a)^m c^{i}t^{j-i} .$$
\end{proof}

\begin{pro}\label{pro_lionrolin_1}
 Let $C\subset \R^n$ be an $\la$-cell of basis $B$ and let $c,\phi_1,\dots,\phi_k$ be  $\la$-functions on $B$. Set  $$\Phi(x):=\big{(}\phi_1(\xt),\dots,\phi_k(\xt),x_n,\frac{c(\xt)}{x_n}\big{)}, \quad x=(\xt,x_n) \in C,$$  and let  $\psi$ be an analytic function on a neighborhood of $cl(\Phi(C))$.  If $\Phi(C)$ is bounded then the function $\xi:=\psi\circ \Phi$  is reducible.
\end{pro}
\begin{proof}{\bf Step 1.}
We  show that the restriction of $\xi$ to an $\la$-cell $C$ on which  $\frac{x_n}{c(\xt)}$ is bounded (assuming that $c(\xt)$ nowhere vanishes on $C$) is reducible.

 On such a cell $C$, as by assumption $\Phi(C)$ is bounded, we have  $|c(\xt)|\sim |x_n|$.
 Consequently, as $x_n$ is bounded on $C$, so is $c$, as well as the $\la$-mapping
$$\Phi'(x):= \big{(}\phi_1(\xt),\dots,\phi_k(\xt), c(\xt), \frac{x_n}{c(\xt)}\big{)},\quad x=(\xt,x_n)\in C. $$
Since $\frac{x_n}{c(\xt)}$ is bounded away from zero on $C$, the function $$\psi'(u_1,\dots,u_m,z,t):=\psi(u_1,\dots,u_m,zt,\frac{1}{t})$$ is analytic on a neighborhood of $cl(\Phi'(C))$
($1/t$ is analytic on the complement of the origin).  By Lemma \ref{lem_la_cell_decomposition}, we may assume that $x_n$ is of constant sign on  $C$.
 As  $\xi=\psi' \circ \Phi'$,  thanks to Lemma \ref{lem_up_to_H},  if we set $H(\xt,x_n):=(\xt,\frac{x_n}{c(\xt)})$, it suffices to show that   $\zeta:=\psi' \circ \Phi'\circ H^{-1}$ is reducible. Since on $H(C)$: $$\Phi'\circ H^{-1}(x)=(\phi_1(\xt),\dots,\phi_k(\xt), c(\xt),x_n),$$ 
the result follows from Proposition \ref{pro_prep_A_functions}.

\noindent {\bf Step 2.} We  show the proposition in its full generality.

For $\xt$ in  $B$, let $\phi(\xt):=(\phi_1(\xt),\dots,\phi_k(\xt))$. By Lemma \ref{lem_splitting}, for every $a \in cl(\phi(B))$, there are $\ep>0$ and two analytic functions $\psi_1$ and $\psi_2$ satisfying  (\ref{eq_lem_splitting}) on $\dbb(a,\ep)\times [-\ep,\ep]^2$ (for our function $\psi$). As $cl(\phi(B))$ is compact, it can be covered by finitely many such cubes $\dbb(a_i,\ep)$, $i=1 ,\dots ,l$. 

By   Lemma \ref{lem_image_inverse} (and Remark \ref{rem_common_refinement}), we can find an $\la$-cell decomposition $\C$  compatible with the sets 
 $\phi^{-1}(\dbb(a_i,\ep))$,  $i=1,\dots,l$. 
Refining $\C$  if necessary, we can assume it to be compatible with the respective inverse images of $[-\ep,\ep]$ under the functions $c(\xt)$, $x_n$, and $\frac{c(\xt)}{x_n}$.  By Lemma \ref{lem_la_cell_decomposition}, we also can suppose that $c(\xt)$ and $x_n$ are of constant sign on every cell of $\C$ included in $C$. Fix  $E \in \C$ included in $C$. 

 If  $|x_n| \ge \ep$ on  $E$, by Proposition \ref{pro_prep_A_functions}, we are done since $t \mapsto \frac{1}{t}$ is analytic on the complement of the origin.  

 If $|\frac{c(\xt)}{x_n}|\ge \ep$ on  $E$ then $|\frac{x_n}{c(\xt)}|$ is bounded and the result directly follows from Step 1.

If $|c(\xt)|\ge \ep$ on  $E$ then (since $x_n$ is bounded on $C$ by assumption) $|\frac{x_n}{c(\xt)}|$ is still bounded and the result also follows from Step $1$.

We thus can assume that $|c|$, $|x_n|$, and $|\frac{c(\xt)}{x_n}|$ are all smaller than $\ep$ on $E$. 
For  $x=(\xt ,x_n)\in C$, let $$\Phi_1(\xt,x_n):=(\phi_1(\xt),\dots,\phi_k(\xt),c(\xt),x_n ), $$
as well as $$\Phi_2(\xt,x_n):=(\phi_1(\xt),\dots,\phi_k(\xt),c(\xt),\frac{c(\xt)}{x_n}).$$
 By construction,  the basis of $E$ is comprised in $ \phi^{-1}(\dbb(a_i,\ep))$ for some $i\le l$, which means that $E\subset\Phi^{-1}_j(\dbb(a_i,\ep)\times [-\ep,\ep]^2)$, for
$j=1,2$. 
As a matter of fact,   (\ref{eq_lem_splitting}) holds on $\Phi(E)$.

 By Proposition \ref{pro_prep_A_functions} and Remark \ref{rem_produit}, the function $\xi_1(x):=x_n\cdot \psi_1 \circ \Phi_1(x)$ is reducible. We claim that  $\xi_2:=\psi_2 \circ \Phi_2$ is  reducible as well  (note that by (\ref{eq_lem_splitting}) we have $\xi=\xi_1+\xi_2$).
Indeed, let $H(x):=(\xt,\frac{c(\xt)}{x_n})$,  and define a bounded $\la$-mapping on $H(E)$ by:$$\Phi_2'(\xt,x_n):=(\phi_1(\xt),\dots,\phi_k(\xt),c(\xt),x_n).$$ By Proposition \ref{pro_prep_A_functions}, $\xi'_2:=\psi_2\circ \Phi'_2$, defined on $H(E)$, is reducible, so that, by  Lemma \ref{lem_up_to_H}, $\xi_2=\xi_2'\circ H $ is reducible, as claimed.

We are ready to check that $\xi$ is reducible. By Lemma \ref{lem_la_cell_decomposition}, up to a refinement of $\C$, we may assume that $\xi_1$ and $\xi_2$ are of constant sign on $E$ (recall that so are $x_n$ and $c(\xt)$). By Lemma \ref{lem_refinement_fonctions_dominantes}, refining again the obtained cell
decomposition of $\R^n$,
 we may assume in addition that one of the two following situations occurs:

\noindent {\bf First Case.}  $|\xi_1| \leq \frac{1}{2}|\xi_2|$ or  $
|\xi_2|\leq \frac{1}{2}  |\xi_1|$ on $E$.

For simplicity, we will assume that the first inequality holds.
It means that $(1+\frac{\xi_1}{\xi_2})$ is  
bounded away from zero and infinity (if  $\xi_2\equiv 0$   on $E$, the result is clear). Since, thanks to Lemma \ref{lem_meme_morphisme}, $\xi_1$ and $\xi_2$ can be reduced with the same $\la$-translation on every cell, the function  $(1+\frac{\xi_1}{\xi_2})$ induces an $\la$-unit on every cell included in $E$ of some $\la$-cell decomposition. Hence, it suffices to rewrite $\xi$ as
$\xi_2\cdot (1+\frac{\xi_1}{\xi_2})$.

\noindent{\bf Second case:}  $\label{eq_casII_proof_prep}
 |\xi_1|\sim |\xi_2|$ on $E$.

 As it will make no difference, we will assume that $\xi_1$, $\xi_2$, $c(\xt)$, and $x_n$ are positive on $E$ (if one of them is zero the result is clear).
We first establish the following

\noindent {\bf Claim.} $x_n$ is $\sim$ on $E$ to an $\la$-function $b(\xt)$.

To check this  claim, observe that Proposition \ref{pro_prep_A_functions} actually ensures that  both $\psi_1 \circ \Phi_1$ and $\xi_2'$ admit {\it nondegenerate} reductions. We thus have  (recall that $\xi_1=x_n \psi_1\circ \Phi_1$)$$\xi_1(\xt,x_n) \sim x_n\cdot a_1(\xt)|x_n-\theta_1(\xt)|^r \;\; \mbox{ and } \; \;\xi_2(\xt,x_n) \sim a_2(\xt)|\frac{c(\xt)}{x_n}-\theta_2(\xt)|^s,$$ for some $\la$-functions $a_1,a_2,\theta_1$, and $\theta_2$ on the basis $E'$ of $E$ and some {\it nonnegative}  rational numbers $r$ and  $s$.   If $x_n \sim \theta_1(\xt)$ or $\frac{c(\xt)}{x_n} \sim \theta_2(\xt)$ on $E$ then the claim clearly holds true. Otherwise, by Lemma \ref{lem_xn_et_theta}, we may assume $\theta_1=\theta_2=0$, so that  $$a_1(\xt) x_n ^{r +1}\sim \xi_1(x) \sim  \xi_2(x) \sim a_2(\xt) \frac{c(\xt)^s}{x_n^s}  ,$$
which  entails that $x_n\sim b(\xt):=(\frac{a_2(\xt)c(\xt)^s}{a_1(\xt)})^{\frac{1}{s+r+1}}$ (here $(s+r+1)$ is nonzero for $s$ and $r$ are both nonnegative), yielding the claim.

Let now $$\psi'(u_1,\dots,u_k, w, z,t):=\psi (u_1,\dots,u_k,w,zt),$$
and $$\Phi'(\xt,x_n):=(\phi_1(\xt),\dots,\phi_k(\xt),\,x_n\,,\, \frac{c(\xt)}{b(\xt)},\, \frac{b(\xt)}{x_n}).$$
 Since $b(\xt)\sim x_n$ and $\Phi(E)$ is bounded, the set $\Phi'(E)$ is bounded as well. By Step $1$,  as $\frac{b(\xt)}{x_n}$ is bounded below away from zero on the cell $E$,  the function $\psi'\circ \Phi'$ must be reducible. As $\xi=\psi\circ \Phi=\psi' \circ \Phi'$, we are done.
\end{proof}

\begin{pro}\label{pro_composition_analytique}
Let $g_1,\dots,g_m$ be  reducible functions on an $\la$-cell $C \subset \R^n$. Set $$G(x):=(g_1(x),\dots ,g_m(x)),$$ and let $f$ be a function which is analytic on a neighborhood of   $cl(G(C))$.  If  $G(C)$ is bounded then $\xi:=f\circ G$ is reducible.
\end{pro}
\begin{proof}By Remark \ref{rem_common_refinement}, we can assume that all the $g_i$'s are reduced on the cells of one single $\la$-cell decomposition $\C$ compatible with $C$.
  Moreover, by Lemma \ref{lem_meme_morphisme},  we can assume that on every cell of  $\C$, the $g_i$'s are reduced with the same  $\la$-translation $\theta$, which,   up to a change $(\xt,x_n)\mapsto (\xt,x_n+\theta(\xt))$ can be assumed to be  zero.    The function $\xi$ may therefore   be written on a given cell $E\in \C$, $E\subset C$, $\psi\circ \Phi$  with $\psi$ restricted analytic function and $\Phi$ bounded mapping of type $$\Phi(x)= (\phi(\xt),\dots \phi_k(\xt), b(\xt) \cdot |x_n|^{1/s}, c(\xt)\cdot |x_n|^{-1/s}),\;\; x=(\xt,x_n)\in E\subset\R^{n},$$
where $b,c,\phi_1,\dots, \phi_k$ are $\la$-functions and $s \in \N$ (since the $g_i$'s are reduced). For such a cell $E$, by Lemma \ref{lem_la_cell_decomposition}, we can assume  $x_n$ to be of constant sign on every cell, and,
 by Lemma \ref{lem_up_to_H}, we may assume that $s=1$ and $b\equiv 1$, so that, by Proposition \ref{pro_lionrolin_1}, the function $\xi_{|E}$ must be reducible.
\end{proof}

We are now ready to carry out the induction step of Proposition
\ref{pro_preparation_la_fonctions}:
\begin{proof}[proof of $(\hn_n)$]
By definition, any $\la$-function may be expressed as  a finite sum, product, and  composite of restricted analytic functions and power
functions. 
 Arguing by induction on the complexity of the expression of
$\xi$,
 it is enough to show that each of these operations (sum, product, power, composition with a restricted analytic function) preserves reducible functions.

 The power of a reduced function is clearly reduced.  We have seen that the product of two reduced functions is also reduced (see Remark \ref{rem_produit}).
By Proposition \ref{pro_composition_analytique}, composition with a restricted analytic function preserves reducible functions.

 It remains to show that the sum of two reducible functions $\xi_1:C \to \R$ and $\xi_2:C \to \R$  is reducible, when  $C$ is an $\la$-cell of $\R^n$.  Indeed, by Remark \ref{rem_common_refinement}, there is  an $\la$-cell decomposition such that $\xi_1$ and $\xi_2$ are reduced on every cell $E\subset C$. By Lemma \ref{lem_la_cell_decomposition}, up to a refinement, we can assume  that $\xi_1$ and $\xi_2$ have constant sign (taking values in $\{-1, 0,1\}$) on every cell. By Lemma \ref{lem_refinement_fonctions_dominantes}, we can also assume  that for every cell $E\subset C$ there is a constant $M>0$ such either $|\xi_1|\le M|\xi_2|$ or $|\xi_2|\le M|\xi_1|$ on $E$. We will suppose for simplicity that  $|\xi_1|\le M|\xi_2|$ on a fixed cell $E\subset C$.
  Then, writing $\xi=\xi_2 \cdot (1+ \frac{\xi_1}{\xi_2})$, by Remark \ref{rem_produit},  we see that it suffices to show that  $(1+ \frac{\xi_1}{\xi_2})$ is reducible (if $\xi_2$ is identically zero on $E$ we are done). But since  $G(x):=\frac{\xi_1(x)}{\xi_2(x)}$ is bounded and  reducible   (again due to Remark \ref{rem_produit}),   this follows from Proposition \ref{pro_composition_analytique}  (applied to the one variable function $f(y):=1+y$).
\end{proof}

\section{Existence of cell decompositions}\label{sect_existence_of_cell_dec}
We are now ready to establish Theorem \ref{thm_existence_cell_dec}.  The cell decomposition that we are going to construct will indeed be an $\la$-cell decomposition.
%
%
Observe first that thanks to Remark \ref{rem_common_refinement}, we can assume that $l=1$, i.e., it is enough to construct an $\la$-cell decomposition compatible with one single given set $A \in \s_n$.

\noindent {\it Reduction to the case where $A$ is globally semi-analytic.} Since  $A$ is globally subanalytic, there exists  a globally semi-analytic subset $Z \subset \R^{m}$, $m>n$, such that $\pi(Z)=A$, where $\pi:\R^{m} \to \R^n $ is the projection onto the $n$ first coordinates.  As the images under $\pi$ of the cells of a cell decomposition  of $\R^m$ constitute a cell decomposition of $\R^n$ (see Remark \ref{rem_projection_cell_decomposition}), it is enough to construct an $\la$-cell decomposition of $\R^m$ which is compatible with  the globally semi-analytic set $Z$.

\noindent {\it Proof in the case where $A$ is globally semi-analytic.} By definition of globally semi-analytic sets, for every $ z\in [-1,1]^n$, there are some analytic functions
 $f_{ij}, g_{ij}$,  $i = 1, \dots, r, j = 1, \dots , s_i$, on  a neighborhood $U_z$ of $z$   such that \begin{equation}\label{eq_image_cell_dec}
\hh_n(A)\cap U_z= \bigcup _{i=1}^r\bigcap _{j=1} ^{s_i} \{x \in U_z: g_{ij}(x) > 0 \mbox{ and } f_{ij}(x) = 0\}.                                                                                                                                                                                                  \end{equation}
Let, for each $z\in [-1,1]^n$,  $V_z$ be a cube containing $z$  and included in $U_z$.
 Since $[-1,1]^n$ is compact, it may be covered by finitely many such cubes. By Remark \ref{rem_common_refinement}, this means that it is enough to construct an $\la$-cell decomposition of $\R^n$ compatible with $A\cap W_{z} $ for every $z \in [-1,1]^n$, where $W_z:=\hh_n^{-1}(V_z)$.

Fix for this purpose  $z \in cl(\hh_n(A))$. As $V_z$ can be described by sign conditions on analytic functions and $\hh_n$ is an $\la$-mapping, by Lemma \ref{lem_la_cell_decomposition} (and Proposition \ref{pro_preparation_la_fonctions}), there is a cell decomposition $\D$ compatible with $W_z$. 
Moreover, by (\ref{eq_image_cell_dec}), we see that
\begin{equation}\label{eq_hh_de_A}
 A\cap W_z = \bigcup _{i=1}^r\bigcap _{j=1} ^{s_i} \{x \in W_z: g_{ij}(\hh_n(x)) > 0 \mbox{ and } f_{ij}(\hh_n(x)) = 0\}.
\end{equation}
As  $\xi_{ij}(x):=f_{ij}(\hh_n(x))$ and $\zeta_{ij}(x):=g_{ij}(\hh_n(x))$  are $\la$-functions,
 again thanks to Lemma \ref{lem_la_cell_decomposition}, there is a refinement $\E$ of $\D$ such that the  $\xi_{ij}$'s and the $\zeta_{ij}$'s have constant sign on every cell included in $W_z$, which, by (\ref{eq_hh_de_A}), entails that $A \cap W_z$ is a union of cells of this cell decomposition, as required.
\begin{rem}\label{rem_cell_la_cell}
The cell decomposition that we have constructed is indeed an $\la$-cell decomposition.  
\end{rem}

\section{The Preparation Theorem  and Gabrielov's Complement Theorem }\label{sect_gabrielov}
In this section, we gather some consequences of Theorem \ref{thm_existence_cell_dec} and Proposition \ref{pro_preparation_la_fonctions}.
 The first thing we establish in this section is that globally subanalytic functions are piecewise given by $\la$-functions.   This gives a very precise description of globally subanalytic functions and will lead us to the Preparation Theorem.

\begin{pro}\label{pro_la_and_globally subanalytic}
Let $\xi:E \to \R$ be a globally subanalytic function, $E \in \s_n$. There is an $\la$-cell decomposition $\C$ of $\R^n$ compatible with $E$, such that for every cell $C\subset E$ of $\C$ the function  $\xi_{|C}$ coincides with an $\la$-function.
 \end{pro}
\begin{proof}
The graph of $\xi$ being a globally subanalytic set, by Theorem \ref{thm_existence_cell_dec} (see Remark \ref{rem_cell_la_cell}), there is an $\la$-cell decomposition $\D$ of $\R^{n+1}$ compatible with $\Gamma_\xi$. This $\la$-cell decomposition induces an $\la$-cell decomposition $\C$ of $\R^n$ (see Remark \ref{rem_projection_cell_decomposition}). Let $C\in \C$ with  $C \subset E$.  There is an $\la$-cell $D\in \D$ included in $\Gamma_\xi$ which projects onto $C$.  This $\la$-cell cannot be a band since it is a subset of $\Gamma_\xi$. It is thus the graph of an $\la$-function $\zeta:C \to \R$ which coincides with $\xi_{|C}$.  
\end{proof}


\begin{thm}\label{thm_preparation}
(The Preparation Theorem) Every  globally subanalytic function is reducible.
\end{thm}
\begin{proof}
It is a consequence  of Propositions \ref{pro_preparation_la_fonctions} and \ref{pro_la_and_globally subanalytic} (see Remark \ref{rem_common_refinement}).
\end{proof}

\begin{rem}\label{rem_globally subanalytic_implies_analytic}
A reducible function induces analytic functions on the cells of some cell decomposition.  Theorem \ref{thm_preparation} thus entails that a globally subanalytic function is analytic on the cells of a suitable cell decomposition.
\end{rem}

In the case $n=1$, the Preparation Theorem (Theorem \ref{thm_preparation}) yields  the famous Puiseux Lemma for globally subanalytic functions:

\begin{pro}\label{pro_puiseux}(Puiseux Lemma)\index{Puiseux Lemma}
Let $f:(0,\eta) \to \R$ be a globally subanalytic function, with $\eta$ positive real number. There exist $\ep\in (0,\eta)$, $m \in \Z$, and  $p\in \N^*$ such that  $f$ has a convergent Puiseux expansion on $(0,\ep)$: $$f(t)=\sum_{i=m}
^{+\infty} a_i t^\frac{i}{p}, \quad a_i \in \R, \;\forall i \ge m.$$
\end{pro}
\begin{proof}
By the Preparation Theorem (Theorem \ref{thm_preparation}), there is a right-hand-side neighborhood of $0$ on which $f$  is reduced. By definition, the germs of reduced one-variable functions are germs of Puiseux series.
\end{proof}

The following two propositions may be considered as  {\it Puiseux Lemmas with parameters}. 

\begin{dfn}\label{dfn_def_partition}
 Let $A \in \s_n$. A {\bf globally subanalytic partition}\index{globally subanalytic!partition} of A is a {\it finite} partition of this set into globally subanalytic sets. A globally subanalytic partition is {\bf compatible}\index{compatible! definable partition} with a set if this set is the union of some elements of the partition.
\end{dfn}

\begin{pro}\label{pro_puiseux_avec_parametres}
 Let $A \in \St_n$ and let  $f$ be a continuous  globally subanalytic function on a neighborhood $U$ of $A \times \{0\}$ in
$A \times  \R_+$. There exist a globally subanalytic partition of $A$ into $\ccc^\infty$ manifolds and a positive integer $p$ such that for every element $C$ of this partition,  $f(x,t^p)$ is analytic on a neighborhood of $C\times \{0\}$ in $C \times  \R$.
\end{pro}
\begin{proof}
 Apply the Preparation Theorem to $f:U\to \R$. This provides a cell decomposition $\D$ of $\R^{n+1}$ compatible with $U$ such that $f(x,t)$ is reduced on
every cell $D\subset U$, that is to say, we can find
 $\la$-functions $\theta$ and $a$ on $C:=\pi(D)$ (where $\pi:\R^{n+1}\to \R^n$ is the canonical projection),
a bounded  $\la$-mapping $W$ on $D$ of type:
$$W(x,t):=(u(x), v(x)(t-\theta(x))^\frac{1}{s}, w(x)(t-\theta(x))^{-\frac{1}{s}})$$
with $u:C \to \R^d$  $\la$-mapping, $v:C \to \R$ and $w:C \to \R$  $\la$-functions, $s \in \N^*$,
as well as a function $\psi$, analytic and nowhere vanishing on a neighborhood of $cl(W(D))$ and such that for some $r \in \Q$:
\begin{equation}\label{eq_reduction_proof_puiseux_param}
f(x,t)=a(x) (t-\theta(x))^r \psi(W(x,t)).\end{equation}
 By Theorem \ref{thm_existence_cell_dec}, we may assume that our cell
decomposition is compatible with $A\times \{0\}$, and consequently, that the cell
decomposition $\pi(\D)$ of $\R^n$  (see Remark \ref{rem_projection_cell_decomposition}) is compatible with $A$. 

Let  $C\in \pi(\D)$ be a cell  included in $A$.  Refining the cell decomposition, we may assume
that either $\theta  \equiv 0$ or that $\theta$ never vanishes on $C$. We may also assume that $a$
is nonzero on $C$ (if $a \equiv 0$ the result is clear). 

Since $\D$ is  compatible with $A\times \{0\}$, there is a unique cell $D$ of $ \D$  which is a band  $(0,\xi)$, with $\xi:C \to \R$  positive  globally subanalytic function (as the integer $p$ in the statement of the proposition can be chosen even, it is enough to deal with the values of $f$ at the positive values of $t$). Let $\theta':=\min (|\theta|,\xi)$.

 If $\theta(x)$ is nonzero on $C$
then for every $x \in C$ the function $t\mapsto (t-\theta(x))^r$ induces on $[0,\frac{|\theta'(x)|}{2}]$ an analytic function.  The function $f$ therefore extends in this case to a function which is analytic with respect to $t$ on a neighborhood of $C\times \{0\}$   and the proposition is clear (in this case).

If $\theta \equiv 0$ then $W(x,t)$ is a Puiseux series in $t$ and $\frac{1}{t}$, analytic in $x$,  and hence,  by (\ref{eq_reduction_proof_puiseux_param}), so is $f$. Observe that, as $f$ is locally bounded (it is continuous), we have $r\ge 0$, and  since $ w(x)\cdot t^{-\frac{1}{s}}$ is  bounded (by definition of reduced functions, $W$ is a bounded mapping), we then see that $w\equiv 0$. As a matter of fact,
 (\ref{eq_reduction_proof_puiseux_param}) indeed gives the desired expansion in this case.
\end{proof}

In the case where $f$ does not necessarily extend continuously to $A\times \{0\}$, we have the following result. 

\begin{pro}\label{pro_puiseux_non_cont}
 Let $A \in \St_n$ and let  $f:(0,\zeta)\to \R$ be a globally subanalytic function, with $\zeta$   positive globally subanalytic function on $A$. There is a globally subanalytic partition of $A$ into $\ccc^\infty$ manifolds  such that for every element $C$ of this partition,  $f(x,t)$ coincides with a Puiseux series with analytic coefficients: $$f(x,t)=\sum_{i \ge k} a_i(x) t^\frac{i}{p},\qquad k \in \Z, \;\; p \in \ns,$$    on  $(0,\xi)$, where $\xi$  is a positive continuous globally subanalytic function on $C$ satisfying $\xi \le \zeta_{|C}$.
\end{pro}
\begin{proof}
  Apply the Preparation Theorem (Theorem \ref{thm_preparation}) to the function $f:(0,\zeta)\to \R$.  This provides a cell decomposition $\D$ of $\R^{n+1}$ compatible with $(0,\zeta)$ such that (\ref{eq_reduction_proof_puiseux_param}) holds on every cell comprised in $ (0,\zeta)$.  We may assume that our cell
decomposition is compatible with $A\times \{0\}$. Take a cell $C$ of $\pi(\D)$ (where $\pi:\R^{n+1}\to \R^n$ is the canonical projection)  included in $A$. As in the proof of the preceding theorem, there is a unique cell $D$ of $ \D$  which is a band  $(0,\xi)$, with $\xi:C \to \R_+$   globally subanalytic function.

It now follows from (\ref{eq_reduction_proof_puiseux_param}) that for any integer  $k >|r|$, the function $t^k f(x,t)$ goes to zero as $(x,t)\in D$ tends to a point of $ C\times \{0\}$, which entails that it extends continuously at every point of $C\times \{0\}$. It thus suffices to
apply  Proposition \ref{pro_puiseux_avec_parametres} to this function, for such $k$.
\end{proof}

\begin{thm}\label{thm_gabrielov}
	(Gabrielov's Complement Theorem)\index{Gabrielov's Complement Theorem} If $A\in \s_n$ then $\R^n\setminus A\in \s_n$.
\end{thm}
\begin{proof}Let $A\in \s_n$.
	By Theorem \ref{thm_existence_cell_dec}, there is a  cell decomposition of $\R^n$ compatible with $A$. The complement of $A$ being a finite union of cells of this cell decomposition, it is a globally subanalytic set, in virtue of Property \ref{pro_basic_properties_from_dfn} (\ref{item_union}).
\end{proof}

 We now give three finiteness results. 

\begin{cor} \label{cor_cc}
	Globally subanalytic sets have only finitely many connected components. They are globally subanalytic.
\end{cor}
\begin{proof}
	Cells of $\R^n$ are connected. As, by Theorem \ref{thm_existence_cell_dec}, every set $A\in \s_n$ is the union of finitely many cells, it has at most finitely many connected components, which are unions of cells.
\end{proof}

It is natural to regard a set $A \in \s_{m+n}$ as a family of subsets of $\R^n$ parametrized by $\R^m$. Let us make it more precise.

\begin{dfn}\label{dfn_familles}Let $A \in \s_{m+n}$. We define   for  $t$ in $\R^m$,  the {\bf fiber}\index{fiber} of $A$ at $t$ as:  $$A_t:=\{x \in \R^n : (t,x)\in A\}.$$\nomenclature[apv]{$A_t$}{$t$-fiber  of a set $A$\nomrefpage}
	We thus get a family  $(A_t)_{t \in \R^m}$ of globally subanalytic subsets of $\R^n$.  
	Any  family constructed in this way is said to be a {\bf globally subanalytic family of sets}\index{globally subanalytic! family of sets}.  
\end{dfn}

\noindent{\it Note.} A globally subanalytic family of sets $(A_t)_{t \in \R^m}$ is not only a collection of globally subanalytic subsets of $\R^n$. We demand that the set $$A=\bigcup_{\tim} \{t\} \times A_t$$
be itself globally subanalytic.

\begin{rem}\label{rem_fibres_cell_dec}
	Let $A \in \s_{m+n}$ and let $\C$ be a cell decomposition of $\R^{m+n}$ compatible with $A$. For $t\in \R^m$, let then $\C_t:=\{C_t:C\in \C\}$. It follows from the definition of  cell decompositions that for every $\tim$, $\C_t$ is a cell decomposition of $\R^n$ compatible with $A_t$.  
\end{rem}

We now have the following parameterized version of Corollary \ref{cor_cc}:
\begin{cor}\label{cor_cc_families} ({\it Uniform finiteness})
	Let $A \in \s_{m+n}$. The number of connected components of $A_t$ is bounded independently of $t \in \R^m$.
\end{cor}
\begin{proof}
	The same proof as for Corollary \ref{cor_cc} applies (see Remark \ref{rem_fibres_cell_dec}).
\end{proof}

We introduce in a similar way the globally subanalytic families of functions and mappings.

\begin{dfn}\label{dfn_familles_fn}A {\bf globally subanalytic family of mappings}\index{globally subanalytic!family of mappings} is a family of mappings  $f_t:A_t \to B_t$, $\tim$, with $A \in \s_{m+n}$ and $B\in \s_{m+k}$, such that the mapping $f: A  \to B $, $(t,x)\mapsto (t,f_t(x))$ is globally subanalytic.

	In the case $B_t=\R$, for all $\tim$, we call such a family a {\bf globally subanalytic family of functions}\index{globally subanalytic!family of functions}. We shall sometimes (abusively) regard a function $f:A\to \R$, $A\in \s_{m+n}$, as a family of functions $f_t:A_t\to \R$, $\tim$, setting $f_t(x):=f(t,x)$. \nomenclature[apxa]{$f_t$}{$t$-fiber of a function\nomrefpage} 
\end{dfn}

Observe that it follows from the definitions that $\Gamma_{f_t}=(\Gamma_f)_t$.
Here is an important property of globally subanalytic families of mappings:

\begin{cor}\label{cor_uniform_bound}
	Let  $f_t : A_t \to B_t, \tim$, be a globally subanalytic family of mappings,  with $A \in \s_{m+n}$, $B \in \s_{m+k}$. The number of connected components of  $ f^{-1}_t(b)$ is bounded by a constant independent of $(t,b)$ in $ B$.
\end{cor}
\begin{proof}
	Apply Corollary  \ref{cor_cc_families} to the family  $ (f^{-1}_t(b))_{(t,b)\in  B}$ (by Property \ref{pro_basic_properties_from_dfn} (\ref{item_image}) it is a globally subanalytic family).
\end{proof}

\paragraph{Historical notes.} 
The first deep insight into real semi-analytic geometry was achieved by S. \L ojasiewicz \cite{lojasiewicz59, lojasiewicz64a, lojasiewicz64b} (see also \cite{ds} for a similar content).
Subanalytic sets were introduced by A. Gabrielov \cite{gabrielov} (rewritten in \cite{gabrielov2}) who showed his Complement Theorem.
 The description of subanalytic sets in terms of convergent series with  negative rational powers that we provide in sections \ref{sect_la_functions}-\ref{sect_gabrielov} is due to several people and it is not easy to quote all the references. The first major contribution seems to be  H. Hironaka's resolution of singularities \cite{hironakares}, which lead him to establish the rectilinearization and uniformization theorems \cite{hironaka} (see also \cite{bmsemisub, bmarc}), closely related to the Preparation Theorem. As well-known, easier proofs of resolution of singularities  appeared later \cite{bmres, wlodarczyk, kollar}.   J. Denef and L. Van den Dries \cite{denefvandendries} established a quantifier elimination result, showing existence of what we call $\la$-cell decompositions (Theorem \ref{thm_existence_cell_dec}, see Remark \ref{rem_cell_la_cell}), as well as the fact that globally subanalytic functions are piecewise given by $\la$-functions (Proposition \ref{pro_la_and_globally subanalytic}).  A few years later, A. Parusi\'nski, relying on rectilinearization procedures and Hironaka's local flattening, proved the Preparation Theorem (Theorem \ref{thm_preparation}) \cite{parusinskiprep,parusinskirect}, that also yields  quantifier elimination (see also \cite[Remarks $1.7$ and $6.9 (2)$]{bmarc}).
The proofs of Proposition \ref{pro_preparation_la_fonctions} and Theorems \ref{thm_existence_cell_dec} and \ref{thm_preparation} that are presented here nevertheless follow very closely the slightly more recent proof of the Preparation Theorem given in the article of J.-M. Lion and J.-P. Rolin \cite{lionrolin}, which is inspired from the proof of the classical Puiseux Lemma and  Denef and van den Dries' article \cite{denefvandendries}, from which
 Proposition \ref{pro_de_finitude} is 
directly taken.
 Puiseux Lemma with parameters (Proposition \ref{pro_puiseux_non_cont}) is due to W. Paw\l ucki \cite{pawlucki}.

          \markboth{G. Valette}{Basic Theorems of subanalytic geometry}

  \chapter{Basic results of  subanalytic geometry}\label{chap_basic}
We describe some basic properties of globally subanalytic sets which  are consequences of the results of the previous chapter.  We start with   the very useful quantifier elimination principle (Theorem \ref{thm_s_formula}), which is a well-known logic principle that provides a convenient way to check that a set is globally subanalytic.  We then establish the famous \L ojasiewicz's inequalities as well as ``{\it  definable choice}'' (Proposition \ref{pro_globally subanalytic_choice}) and Curve Selection Lemma (Lemma \ref{curve_selection_lemma}), which will be of service many times in the next chapters.  We then shift our interest to the description of the geometric properties of globally subanalytic sets, showing existence of globally subanalytic tubular neighborhoods for globally subanalytic manifolds (sections  \ref{sect_closure} and \ref{sect_retraction}), and establishing the famous Tamm's theorem, asserting that the regular locus of a globally subanalytic set is globally subanalytic (Theorem \ref{thm_X_reg}). We also show that globally subanalytic sets and mappings can be stratified (with regularity conditions, see section \ref{sect_strat}), which will be of service in the next two chapters.

Given  $r\ge 0$ and $x\in \R^n$,   $\bou(x,r)$ (resp. $\Bb(x,r$)) will stand for the open (resp. closed) ball \nomenclature[aq]{$\bou(x,r)$}{open ball of radius $r$  centered at $x$\nomrefpage} \nomenclature[aq]{$\Bb(x,r)$}{closed ball of radius $r$  centered at $x$\nomrefpage} of radius $r$ centered at $x$ and $\sph(x,r)$ for the corresponding sphere\nomenclature[aq]{$\sph(x,r)$}{sphere of radius $r$  centered at $x$\nomrefpage}.  Balls and spheres will be taken with respect to the Euclidean norm $|.|$.
  The unit sphere of $\R^n$ centered at the origin is denoted $\sph^{n-1}$\nomenclature[aqd]{$\sph^{n-1}$}{unit sphere of $\R^n$\nomrefpage} for simplicity.

  We define the {\bf topological boundary}\index{topological boundary} of $A$, by setting $\delta A:=cl(A)\setminus int(A)$\nomenclature[ar]{$\delta A$}{topological boundary of the set $A$\nomrefpage} as well as the {\bf frontier}\index{frontier} of $A$ by setting $fr(A):=cl(A)\setminus A$\nomenclature[ar]{$fr( A)$}{frontier of the set $A$\nomrefpage}. 

   We denote by $e_1,\dots,e_n$\nomenclature[arm]{$e_1,\dots,e_n$}{canonical basis of $\R^n$\nomrefpage} the canonical basis of $\R^n$ (for all $n$, we will make it more precise when it is not obvious from the context in which $\R^n$ lies $e_i$).
  
 We denote by $d_x F$\nomenclature[at]{$d_x F$}{derivative of a mapping\nomrefpage} the derivative of a differentiable map $F$ and by $\pa_x f$ the gradient of a differentiable function $f$\nomenclature[au]{$\pa_x f$}{gradient of a function\nomrefpage}.

\section{Quantifier elimination}\label{sect_quantifier_elimination}
We give a brief introduction to quantifier elimination. This model-theoretic principle will provide an efficient tool  to check that a set is globally subanalytic.  These facts are not proper to the theory of globally subanalytic sets and play a central role in the theory of o-minimal structures \cite{costeomin, vdd_omin}, as well as in even larger frameworks.
To motivate our purpose, we start with a proposition. 
\begin{pro}\label{pro_closure}
 If $A \in \s_n$ then $cl(A)$ and $int(A)$ also belong to  $\s_n$.
\end{pro}
\begin{proof}
Observe first that the closure of $A$ may be defined as:
\begin{equation}\label{eq_formula}
\{x \in \R^n: \forall \ep >0 , \exists y \in A, |x-y|^2< \ep\}.
\end{equation}
This set coincides with the set:
$$\R^n \setminus \mu \left(\R^n \times (0,+\infty) \setminus \pi(B) \right), $$
where
$$B=\{ (x,\ep,y)\in \R^n \times \R \times A: |x-y|^2< \ep\}, $$
 and where  $\mu:\R^{n+1} \to \R^{n}$ (resp. $\pi:\R^{2n+1} \to \R^{n+1}$) is the projection omitting the last coordinate (resp. $n$ last coordinates).  It follows from Property \ref{pro_basic_properties_from_dfn} (\ref{item_image}) that  $B$ is globally subanalytic, so that Property \ref{pro_basic_properties_from_dfn} (\ref{item_projection})  and  Gabrielov's
Complement Theorem establish that $cl(A)$ and $int(A)$ are globally subanalytic. \end{proof}

The above proposition shows how stability under projections is useful to establish that a set is globally subanalytic.  It also emphasizes that  it can be tedious to prove that a set is globally subanalytic by describing it in terms of projections of globally semi-analytic sets. The following basic logic principle  makes it possible to get rid of the technical difficulties.

\begin{dfn}\label{dfn_s_formula} We define the 
	{\bf $\s$-formulas}\index{s-formula@$\s$-formula} inductively as follows.
\begin{enumerate}
 \item If $A \in \s_n$ then the formula $\Phi(x)$:=``$x\in A$'' is an $\s$-formula.
\item If $\Phi(x)$, where $x=(x_1,\dots,x_n)$, is an $\s$-formula then ``$not \; \Phi(x)$'' is an $\s$-formula.
\item If $\Phi(x)$ and $\Psi(x)$ are $\s$-formulas, where $x=(x_1,\dots,x_n)$, then ``$\Phi \;  and \;\Psi$'' and ``$\Phi \; or \; \Psi$'' are also $\s$-formulas.
\item If $\Phi(x,y)$ is an $\s$-formula, where $x=(x_1,\dots,x_n)$ and $y=(y_1,\dots,y_p)$, then  ``$\exists y\in \R^p,\; \Phi(x,y)$'' and ``$\forall y \in \R^p,\; \Phi(x,y)$'' are $\s$-formulas. 
\end{enumerate} 
\end{dfn}

Roughly speaking,  $\s$-formulas are first order mathematical sentences involving globally subanalytic sets. The point $(i)$ defines the most elementary formulas and the other axioms explain how to build new sentences from these sentences.  The minimal number of steps needed to generate a formula is called {\bf the complexity}\index{complexity! of an $\s$-formula} of the formula. The above definition of $\s$-formulas is thus by induction on the complexity of the sentence.
\begin{rems}\label{rems_s_formula}
\begin{itemize} \renewcommand{\labelitemi}{$\diamond$}
\item By  $(i)$, if $f$ is globally subanalytic then the sentences $f(x)>0$ and $f(x)=0$ are equivalent to $\s$-formulas (see Property  \ref{pro_basic_properties_from_dfn} (\ref{item_image})).  We will thus regard them as $\s$-formulas.
\item  It is important to note that the variable $y$ in $(iv)$ has to range over the whole of $\R^p$, i.e., we cannot write ``$\exists y \in \N^p\,$''.  Thanks to $(i)$, we can nevertheless write ``$\exists y \in A$'',  if $A \in \s_p$.
\item We restrict ourselves to what is called by logicians {\it first order formulas}, in the sense that the quantified variables cannot be functions or sets: they have to stand for real numbers. The sentences starting like ``$\exists$ a globally subanalytic function...''  or  ``$\exists$ a globally subanalytic set...'' are {\it not} $\s$-formulas.  
\item The formulas depend on finitely many variables $x=(x_1,\dots,x_n)$. These are called {\bf the free variables}\index{free variables}. The free variables of $\Phi(x)$ are the variables which are not quantified in the assertion $\Phi(x)$. The value of the assertion (true or false) of course depends on the chosen value for $x\in \R^n$. 
\end{itemize}
\end{rems}

\begin{thm}\label{thm_s_formula}
If $\Phi(x)$ is an $\s$-formula, $x=(x_1,\dots,x_n)$, then the set 
 \begin{equation*}\label{eq_A_phi}E_\Phi:=\{x\in \R^n:\Phi(x) \mbox{ holds true  } \}\end{equation*}
belongs to $\s_n$.
\end{thm}
\begin{proof}
We prove it by induction on the complexity of the formula.  
 If $\Phi(x)$ is the formula ``$x \in A$'', for some $A \in \s_n$, then $E_\Phi=A$ is  a globally subanalytic set.  We thus have to show that conditions  $(ii-iv)$ of Definition \ref{dfn_s_formula} also produce globally subanalytic sets.

 Indeed, if $\Phi(x)$ is an $\s$-formula then $\Phi':=$``$not \; \Phi$'' defines the complement of $E_{\Phi}$ in $\R^n$, which is a globally subanalytic set, by Theorem \ref{thm_gabrielov}      
 (and induction on the complexity). Thus, $(ii)$ provides assertions which only give rise to globally subanalytic sets. Similarly, since $\s_n$ is stable under finite union and intersection (see Property \ref{pro_basic_properties_from_dfn} (\ref{item_union})), $(iii)$ only gives rise to globally subanalytic sets. 

 For  $(iv)$, we will proceed as in the proof of Proposition \ref{pro_closure}. Let $\Phi(x,y)$  be an $\s$-formula, where $x=(x_1,\dots,x_n)$ and $y=(y_1,\dots,y_p)$, and let $\Psi(x)$ be the formula  ``$\exists y\in\R^p, \, \Phi(x,y)$''. We have: $$E_\Psi=\pi(E_\Phi),$$ where $\pi:\R^n \times \R^p \to \R^n$ is the canonical projection. By Property \ref{pro_basic_properties_from_dfn} (\ref{item_projection}), the set $E_\Psi$ is globally subanalytic since so is $E_{\Phi}$, by induction on the complexity.  
Finally, as the formula   ``$\forall y \in \R^p, \,\Phi(x,y)$'' amounts to ``$not\; (\exists y \in \R^p, \; not \;\Phi(x,y))$'', it defines a globally subanalytic set as well. 
\end{proof}

By way of conclusion, let us give the following useful facts. These are striking examples of how the above theorem is convenient to establish that a set is globally subanalytic.

\begin{pro}\label{pro_csq_qe}$ $
 \begin{enumerate}[(1)]
 \item\label{item_sup}  If $A\in \s_{m+n}$  then the set $B:=\{t \in \R^m: A_t \mbox{ is closed }\}$ belongs to $\s_m$.  Moreover, if $f:A\to \R$ is bounded and definable then $\varphi(t):=\sup_{x\in A_t} f(t,x)$ is definable.
\item If $A\in \s_n$ then the function $\R^n \ni x \mapsto d(x,A):=\inf\{|x-y|:y\in A\}$\nomenclature[axb]{$d(x,A)$}{Euclidean distance from $x$ to the set $A$\nomrefpage} is globally subanalytic.
\item \label{item_C_k}If $f:U \to \R^p$ is globally subanalytic, $U\subset \R^n$ open, and $k \in \N$ then the set of points of $U$ at which $f$ is $\ccc^k$ is globally subanalytic.
\item\label{item_der_bundle}  If   a differentiable submanifold $M$ of $\R^n$ is globally subanalytic  then so is its tangent bundle $TM:=\{(x,v)\in \R^n\times \R^n:v\in T_xM\}$. The derivative of a (differentiable) globally subanalytic mapping on $M$ is globally subanalytic.
 \end{enumerate}
\end{pro}
\begin{proof}
 The set $B$ displayed in $(\ref{item_sup})$   could be described by a formula similar as in (\ref{eq_formula}). It is left to the reader to write the corresponding formulas in either of the other cases.
\end{proof}

 It is however not  easy to prove that the set of points at which a globally subanalytic function is $\ccc^\infty$ is globally subanalytic  (see (\ref{item_C_k}) of the above proposition). We shall nevertheless establish that this is true (Corollary \ref{cor_reg_locus__definissable}).

  From now on, in order to shorten the statements,  globally subanalytic sets (resp. mappings, families, partitions) will generally be called {\bf definable}\index{definable} sets (resp. mappings, families, partitions). This terminology is usual to logicians or ``o-minimal geometers''. It is motivated by the fact that Theorem \ref{thm_s_formula} yields that
 all what can be defined by an $\s$-formula is globally subanalytic.

\section{Curve selection Lemma and \L ojasiewicz's inequalities} Curve Selection Lemma comes down from the following useful result.

\begin{pro}\label{pro_globally subanalytic_choice}(Definable choice)\index{definable choice} Let $A \in \St_{m+n}$ and let
$\pi:\R^{m}\times \R^n \to \R^m$ be the canonical projection. There exists a definable mapping $f:B\to \R^n$,
where $B:=\pi(A)$, such that $\Gamma_f\subset A$.
 \end{pro}
\begin{proof}
We prove it by induction on $n$.
Assume first that $n=1$. Taking a cell decomposition adapted to $A$ if necessary, it is enough to address the case where $A$ is a cell. 
If $A$ is a graph over a cell of $\R^m$, the result is trivial. If $A$ is a band $(\zeta,\zeta')$ (see Definition \ref{dfn_cell_decomposition}), with $\zeta$
and $\zeta'$ not infinite, then take $f:= \frac{\zeta+\zeta'}{2}$. If for instance
$\zeta'=+\infty$ and $\zeta>-\infty$  take $f=\zeta+1$. If $\zeta$ and $\zeta'$ are both infinite, we set $f\equiv 0$. This
completes the proof in the case $n=1$.

Assume the result true for $(n-1)$. Let $A \in \St_{m+n}$ and let $\mu:\R^{m+n}\to \R^{m+n-1}$ be the  projection omitting the last coordinate.   Applying the
induction hypothesis to $A':=\mu(A)$, we get a definable mapping $g:B \to \R^{n-1}$ with $\Gamma_g \subset A'$. Applying the
case $n=1$ to $A$, we get a definable mapping $h:A' \to \R$ satisfying $\Gamma_h \subset A$. It suffices to set $f(x):=h(x,g(x))$.
\end{proof}

\begin{rem}\label{rem_definable_choice}
Combining the latter proposition with Theorem \ref{thm_s_formula} provides the following version of definable choice. Let  $\Phi(x,y)$ be an $\s$-formula, with $x=(x_1,\dots,x_m)$ and $y=(y_1,\dots,y_n)$ free variables, and assume that there are $A\in  \s_m$ and $ B \in \s_n$, such that  for every $x\in A$ there is $y\in B$ for which $\Phi(x,y)$ holds. Then there is a definable mapping $f:A\to B$ such  that  $\Phi(x,f(x))$ holds for all $x\in A$.
 \end{rem}

\begin{lem}\label{curve_selection_lemma}(Curve Selection Lemma)\index{Curve Selection Lemma} Let $A \in \St_n$ and let $x_0 \in
cl(A)$. There exists an analytic arc $\gamma: [0,\ep) \to \R^n$ such that
$\gamma(0) =x_0$ and  $\gamma((0,\ep))\subset A$.
\end{lem}
\begin{proof}
 Applying Proposition
\ref{pro_globally subanalytic_choice} (with $m:=1$) to the  definable set
$$A':=\{(r,x)\in (0,+\infty) \times A:|x-x_0|<r\},$$  we get a definable map $\gamma :(0,+\infty)\to \R^n$  satisfying $\gamma(r) \in  A \cap \bou(x_0,r)$, for all $r>0$.  This yields that there is a definable arc in $A$ tending to $\xo$. Existence of an analytic parametrization then follows from Puiseux Lemma (Proposition \ref{pro_puiseux}).
\end{proof}

\begin{pro}\label{lem_connexe_par_arc}
 Any two points of a connected definable set $X$ may be joint by a continuous definable arc in $ X$. 
\end{pro}
\begin{proof}  Let $X\in \s_n$ be a connected set. By Theorem \ref{thm_existence_cell_dec}, $X$ is a finite union of cells $C_1,\dots,C_k$ of $\R^n$. Let $x \in X$ and let $E$ be the set of points of $X$ that can be joint to $x$ by means of a continuous definable curve in $X$. As any two  points of a cell  may be joint by a continuous definable arc, $E$ must be the  union of some of the $C_i$'s, which entails that it is a definable set. By Curve Selection Lemma (Lemma \ref{curve_selection_lemma}), $E$ is a closed subset of $X$. Moreover,  again due to Curve Selection Lemma,  $ X\setminus E$ is closed in $X$ as well.  As $X$ is connected and $E$ is nonempty (it contains $x$), $X=E$. 
\end{proof}


This leads us to the famous \L ojasiewicz's inequalities.

\begin{thm}\label{thm_lojasiewicz_inequality}(\L ojasiewicz's inequality)
Let $f$ and $g$ be two  definable functions on a  definable set $A$. Assume that $f$ is bounded and that  
\begin{equation}\label{eq_hyp_loj_weak}\lim_{t\to 0}f(\gamma(t))=0,\end{equation}
 for every definable arc $\gamma:(0,\ep)\to A$ such that $\lim_{t\to 0} g(\gamma(t))=0$.
 Then there exist $N \in \N^*$ and $C \in \R$ such that for any $x \in A$:
$$|f(x)|^N \leq C|g(x)|.$$
\end{thm}
\begin{proof}
 Possibly replacing $f$ and $g$ with their respective absolute values we may assume that these
 functions are nonnegative. For $t\in g(A)$, let  $$\varphi(t):= \sup_{x\in g^{-1}(t)} f(x)  .$$
By  Proposition \ref{pro_csq_qe} (\ref{item_sup}), the function $\varphi$ is definable. Thanks to  assumption (\ref{eq_hyp_loj_weak}) (and definable choice, see Remark \ref{rem_definable_choice}), we see that $\varphi$ tends to zero  as $t$ goes to zero. By  Puiseux Lemma (Proposition \ref{pro_puiseux}), there is a
real number $a$ and
a positive rational number $\alpha$ such that for $t$ positive small enough
 $$\varphi(t)=at^\alpha+\dots,$$
which  implies that there must be a constant $C$ such that 
$\varphi(t)\leq
C  t^\alpha$, for $t$ positive small enough.
Therefore,   for $g(x)$ positive small enough we can write:
 $$f(x)\leq \varphi(g(x)) \leq C
g(x)^\alpha,$$
 which means that, if we choose an integer   $N \geq \frac{1}{\alpha}$ then the desired inequality holds on
$g^{-1}([0,\ep))$, $\ep>0$ small enough.

On $g^{-1}([\ep,+\infty))$, as $g$ is bounded below away from zero and $f$ is bounded, the desired inequality will continue to hold if $C$ is chosen large enough.
\end{proof}

\begin{cor}\label{cor_lojasiewicz_inequality}
If two continuous definable functions $f$ and $g$ on a compact definable set $A$ satisfy
\begin{equation}\label{eq_hyp_loj}
g^{-1}(0)\subset f^{-1}(0),
\end{equation}
then there exist $N \in \N^*$ and $C \in \R$ such that for any $x \in A$:
$$|f(x)|^N \leq C|g(x)|.$$
\end{cor}
\begin{proof}
Since $f$ is continuous on a compact set, it is a bounded function. Moreover, if $\gamma:(0,\ep)\to A$ is a definable  arc such that $\lim_{t\to 0} g(\gamma(t))=0$ then, setting $a=\lim_{t \to 0} \gamma(t)$ (which exists since $A$ is compact and $\gamma$ is definable, see Proposition \ref{pro_puiseux}), we get    $g(a)=0$, which, via hypothesis (\ref{eq_hyp_loj}), entails that $f(a)=0$. Hence, $\lim_{t\to 0} f(\gamma(t))=f(a)=0$, and the conclusion  follows from Theorem \ref{thm_lojasiewicz_inequality}. 
\end{proof}

\begin{cor}
Let $\xi:A\to \R$ be a definable function. If $\xi$ is bounded on every bounded subset of $A$ then there are $C>0$ and  $N\in \N$ such that for all $x\in A$:
$$ |\xi(x)|\le C(1+|x|)^N.$$
\end{cor}
\begin{proof}
Apply Theorem \ref{thm_lojasiewicz_inequality} to $f(x):=\frac{1}{1+|x|}$ and $g(x):=\frac{1}{1+|\xi(x)|}$.
\end{proof}

We are going to derive from  Theorem \ref{thm_lojasiewicz_inequality} another famous estimate which is also called \L ojasiewicz's inequality (Corollary \ref{cor_lojasiewicz_gradient}). This one will compare a $\ccc^1$ definable function with its gradient. 

 \begin{lem}\label{lem_lojasiewicz_inequality_f_x_pa_x}
 Let $f:M\to \R$ be a $\ccc^1$ definable function,  with $M$ definable $\ccc^1$ submanifold of $
 \R^n$, and let $a\in cl(M)$. If $f$ extends continuously at $a$ then there is a positive constant $C$ such that for all $x\in M$ sufficiently close to $a$ $$|f(x)-f(a)|\le C |x-a|\cdot |\pa_x f|.$$ 
 \end{lem}
 \begin{proof}
  Without loss of generality, we can assume that $a=\orn$, $f(a)=0$. Thanks to Curve Selection Lemma, it suffices to show the desired inequality along a definable (non constant) arc $\gamma:(0,\ep)\to M$, with $\gamma(s)$ tending to the origin as $s$ goes to zero.
By Puiseux Lemma (Proposition \ref{pro_puiseux}), the arcs $\gamma(s)$ and  $\pa_{\gamma(s)}f$ admit  Puiseux expansions, say 
$$\gamma(s)=b s^k+\dots \quad \eto \quad \pa_{\gamma(s)}f=cs^l+\dots, \quad \mbox{ with }k,l\in \Q$$
(with $c=0$ if and only if $\pa_{\gamma(s)}f\equiv 0$, and $b\ne 0$).
It means that $\gamma'(s)=kb s^{k-1}+\dots$, so that:
 $$|f(\gamma(r))|=|\int_0 ^r\pa_{\gamma(s)} f \cdot\gamma'(s) ds| \lesssim  |c| \,r^{k+l} \lesssim |\pa_{\gamma(r)}  f|\cdot |\gamma(r) |, $$ 
 yielding the desired estimate along $\gamma$.
 \end{proof}

\begin{cor}\label{cor_lojasiewicz_gradient}
  Let $f:M\to \R$ be a $\ccc^1$ definable function  with $M$ definable $\ccc^1$ submanifold of $\R^n$ and let
$x_0\in cl(M)$.  If $f$ extends continuously at $\xo$ then there are $\rho\in (0,1)\cap \Q$ and  $C>0$ such that for all $x\in M$ sufficiently close to $x_0$ $$|f(x)-f(x_0)|^{\rho}\le  C|\pa_x f|.$$
\end{cor}
\begin{proof} 		 Lemma \ref{lem_lojasiewicz_inequality_f_x_pa_x} yields that there is $\eta>0$ such that $\pa_x f=0$ entails $f(x)=f(x_0)$, for all $x\in \bou(x_0,\eta)\cap M$. We thus will check the desired inequality on the set
	 $$ V:=\{x\in M:\pa_x f\ne 0 \mbox{ and } |x-x_0|<\eta \}.$$ 
	 Define  a function $g$ on this set by setting $g(x):=\frac{|f(x)-f(x_0)|}{|\pa_x f|}$,  and observe that, due to Lemma \ref{lem_lojasiewicz_inequality_f_x_pa_x}, this function  must be bounded in the vicinity of $x_0$. In order to apply Theorem  \ref{thm_lojasiewicz_inequality}, we first show that $g(\gamma(t))$ tends to zero for every 
	definable arc $\gamma:(0,\ep)\to V$ such that $f(\gamma(t))$ tends to $f(x_0)$. 
	
	 Such an arc $\gamma$ being bounded, it must have an endpoint $a\in cl(V)$ (as $t\to 0$). Moreover, since $\gamma(t)$ and $f(\gamma(t))$ are Puiseux arcs, we have for $t>0$ small  $$|f(\gamma(t))-f(x_0)| \le C |\gamma(t)-a|^\alpha,$$ 
	for some positive rational number $\alpha$ and some constant $C$. The arc $\gamma$ is thus either constant (in which case the needed fact is clear) or included in the manifold
	 $$M':=\{x\in V: |f(x)-f(x_0)|< 2C|x-a|^\alpha\}.$$
	Clearly,  $f(a):=f(x_0)$ extends $f_{|M'}$ continuously at $a$. Applying Lemma \ref{lem_lojasiewicz_inequality_f_x_pa_x} to  $f_{|M'}$  yields $|g(x)|\lesssim |x-a|$ for $x\in M'$ near $a$, which shows that $\lim_{t\to 0}g(\gamma(t))=0$.
	
 Hence, by Theorem \ref{thm_lojasiewicz_inequality}, there are $C>0$ and $N\in \N^*$ such that $|g(x)|^N \le C | f(x)-f(x_0)|$ for $x$ in $V$ close to $x_0$, which implies that  for such $x$: $$|f(x)-f(x_0)|^{1-\frac{1}{N}} \le C' |\pa_x f|, $$
 for some constant $C'$.
\end{proof}
 \begin{rem}
  When $f$ fails to extend continuously at $x_0$, it is possible to give an inequality which involves the so-called asymptotic critical values of $f$ \cite{kurdykaloj, annaloj}.
 \end{rem}

\section{Closure and dimension}\label{sect_closure}
We define the {\bf dimension}\nomenclature[ay]{$\dim A$}{dimension of $A$\nomrefpage} \index{dimension} of $A \in \s_n$ as $$\dim A=\max \{\dim C: C \in \C, C\subset A\},$$
where  $\C$ is a cell decomposition of $\R^n$ compatible with $A$ and $\dim C$ denotes the dimension of $C$ as a manifold (cells are analytic manifolds by definition).  By convention, the dimension of the empty set is $-1$.

\begin{pro}\label{pro_dim}
\begin{enumerate}[(1)]
 \item $\dim A$ is independent of the chosen cell decomposition.
\item If $F:A \to B$ is definable and if $E \subset A$ is also definable then $\dim F(E)\le \dim E$.
\item If $A \in \s_n$ is nowhere dense then $\dim A<n$. 
\end{enumerate}
\end{pro}
\begin{proof}
 $(1)$  comes from the fact that cell decompositions have common refinements and $(3)$ is obvious from the definitions. Since $F$ is smooth on the cells  of a suitable cell decomposition, $(2)$ is also clear. 
\end{proof}

\begin{lem}\label{lem_closure_dense}
 Given $A \in \s_{m+n}$, there is a dense definable subset  $B$ of $ \R^m$ such that  $cl(A)_t=cl(A_t)$ for any $t \in B$.
\end{lem}
\begin{proof}
 Since the set
$$E:=\{ t \in \R^m :cl (A_t) \ne  cl(A)_t \}$$ 
  may be described with an $\s$-formula, by Theorem \ref{thm_s_formula}, it is definable. We have to show that $\dim E<m$. Note that, for any $t \in \R^m$ we have  $cl(A_t ) \subset cl(A)_t$, since $cl(A)_t$ is
closed and contains $A_t$.

Suppose
that $E$ is of dimension $m$ and take a cell decomposition of $\R^{m}$ compatible with $E$. Let $C$ be a cell of dimension $m$ included in $E$.
For each $t \in C$, there are  $r_t>0$ and $a_t \in  cl(A)_t$  such that $\bou (a_t,r_t)$ does not meet $A_t$. 
 By Definable Choice (Proposition \ref{pro_globally subanalytic_choice}, see Remark \ref{rem_definable_choice}), we can assume that $r_t$ and $a_t$ are definable functions of $t$ and, by Theorem \ref{thm_preparation}, up to a refinement of the cell decomposition, we can assume that they are continuous  on $C$.
Let $$U:=\{ (t,x)\in C \times \R^n: x \in \bou (a_t,r_t)\}.$$
Since $C$ is open in $\R^m$ and because $a_t$ and $r_t$ are continuous with respect to $t$, the set $U$ is an open subset of $\R^{m+n}$. As $U$ intersects $cl(A)$ (at the points $(t,a_t)$) and is disjoint from  $A$, this is a contradiction, which yields that $\dim E<m$. 
\end{proof}

Given $A\in \s_{m+n}$ and $B\in\s_m$, we  define the {\bf restriction of $A$ to $B$}\index{restriction of $A$ to $B$}\nomenclature[ayn]{$A_B$}{restriction to $ B$ of the set $A$\nomrefpage} as:
\begin{equation}\label{eq_restriction}A_B:=A \cap  (B\times \R^n).\end{equation}
\begin{lem}\label{lem_closure_cells}
Given $A \in \s_{m+n}$, there is a definable partition $\Pa$ of $\R^m$ such that for every $B \in \Pa$ we have for any $t \in B$: $$cl(A_B)_t=cl(A_t).$$ 
\end{lem}
\begin{proof}
We prove the lemma by induction on $m$. The result being clear for $m=0$, assume it to be true for $(m-1)$, $m \ge 1$.
Let $B$ be the set provided by Lemma \ref{lem_closure_dense} (applied to $A$) and take a cell decomposition $\C$ of $\R^m$ compatible with $B$.
It is enough  to establish the lemma for the sets $A_C$,  $C \in \C$.  In fact, if $C \subset B$, this comes from Lemma \ref{lem_closure_dense}. Otherwise, as $B$ is definable and  dense in $\R^m$, $\dim C<m$. In this case, up to a definable homeomorphism, we may assume that $C \subset \R^{m'}\times \{0_{\R^{m-m'}}\}$, where $m'=\dim C<m$, and the result follows from the induction hypothesis. 
\end{proof}

%

\begin{pro}\label{pro_delta_A}
For any $A \in \s_n$, we have $\dim fr(A)<\dim A$.  
\end{pro}
\begin{proof}
Let $k:=\dim A$ and assume that $m:=\dim fr(A)$ is not smaller than  $k$. Take a cell decomposition compatible with $fr(A)$ and let $D$ be a cell  of dimension $m$ included in $fr(A)$.  Up to a homeomorphism, we may assume that  $D$ is open in $\R^m \times \{0_{\R^{n-m}} \}$.
 Taking a suitable cell decomposition if necessary,  we can assume that $A$ is a cell, which means that
 $\dim A_t=k-m\le 0$  for every $(t,0) \in D$. Hence, $A_t$ is finite or empty and consequently must be closed in $\R^{n-m}$. But, by Lemma \ref{lem_closure_dense}, we know that for almost every $(t,0) \in D$, $cl(A_t)$ contains $D_t$ which is disjoint from $A_t$. This is a contradiction.
\end{proof}
\begin{rem}\label{rem_int_du_bord}
As a matter of fact, $\delta A=cl(A)\setminus int(A)$ has  empty interior in $\R^n$ for all $A\in \s_n$.
\end{rem}

\begin{pro}\label{pro_cont_parametres}
	Let $A \in \s_{m+n}$ and let  $f:A \to \R$ be a definable  function. If $f_t$ is continuous for every $t \in \R^m$ then there is a definable partition $\Pa$ of $\R^m$ such that for every  $B\in \Pa$ the function $f$ is continuous on $A_B$. 
\end{pro}
\begin{proof}
	Possibly replacing $f$ with $\frac{f}{f^2+1}$, we can assume that $f$ is bounded.
	By Lemma \ref{lem_closure_cells} (applied to $\Gamma_f$), there is a definable partition $\Pa$ of $\R^m$ such that for every $B \in \Pa$ we have for all $t \in B$
	\begin{equation}\label{eq_closure} cl(\Gamma_{f_t})= cl(\Gamma_f \cap (B\times \R^{n+1}))_t.\end{equation}
	Since $f_t$ is continuous for all $t$,  the set $\Gamma_{f_t}$ is closed in $A_t \times \R$ which means (by (\ref{eq_closure})) that  $(\Gamma_f)_B$ is closed in $B \times \R^{n+1}$, if $B\in \Pa$. As $f$ is bounded, this implies that  it induces a continuous function on $A_B$.  
\end{proof}

\section{Orthogonal retraction onto a manifold}\label{sect_retraction}
Given  $Y \in \s_n$,  let us  define a function 
$\rho_Y:\R^n\to \R$ by $$\dis_Y(x):= d(x,Y)^2=\inf_{y\in Y} |x-y|^2 .$$ \nomenclature[az]{$\rho_Y$}{square of the distance to $Y$\nomrefpage}
%
\begin{pro}\label{pro_retraction}
 Let $Y$ be a $\ccc^p$ definable submanifold of $\R^n$, $p\ge 2$ (possibly infinite). There are a definable open  neighborhood $U$ of $Y$ and a definable $\ccc^{p-1}$ retraction $\pr_Y:U\to Y$  such that for all $x\in U$ we have
$$ \dis_Y(x)=|x-\pr_Y(x)|^2.$$
   Moreover, for $x\in U$ we have $\pa_x \dis_Y=2(x-\pr_Y(x))$, which is always orthogonal to $T_{\pr_Y(x)} Y$. The triplet $(U,\pi_Y,\rho_Y)$ will be called {\bf a tubular neighborhood} of $Y$.\index{tubular neighborhood}
\end{pro}
\begin{proof}
 For $(x,y)\in \R^n \times Y$ let $\mu(x,y):=|x-y|^2$. If $x\in \R^n$ is a point such that $d(x,Y)$ is smaller than $d(x,fr(Y))$ then there must be a point $z$ in $Y$ such that $\dis_Y(x)=\mu(x,z)$. As $z$ realizes the minimum of $\mu_{x}:Y\to \R$, we must have $\pa_z \mu_{x} =0$. Moreover, in the case where $x\in Y$, a simple computation shows that the differential of the mapping $Y\ni y\mapsto \pa_y \mu_{x}\in T_y Y$ has rank $k:=\dim Y$ at $y=x$. 
 
 By the Implicit Function Theorem,  we deduce that every point $x_0$ in $Y$ has a neighborhood $W_{x_0}$ in $\R^n$ and a $\ccc^{p-1}$ mapping $\pr_{x_0}: W_{x_0} \to Y$ for which  $\pa_y \mu_x=0$ amounts to $y =  \pr_{x_0}(x)$. Clearly, by construction, if $x$ is sufficiently close to $x_0$ then $\pr_{x_0}(x)$ is the unique point for which $\dis_Y(x)=|x-\pr_{x_0}(x)|^2$. The map-germs  $\pr_{x_0}$ thus clearly glue together into a definable retraction on a neighborhood of $Y$.  The last sentence comes down from an easy computation of derivative.
\end{proof}

\begin{pro}\label{pro_caract_reg}
 Let $Y \in \s_n$ be locally closed. The function $\dis_Y$ is $\ccc^\infty$ on a neighborhood of $Y$ if and only if $Y$ is a $\ccc^\infty$ submanifold of $\R^n$.
\end{pro}
\begin{proof}
 The if part follows from the preceding proposition. We prove the only if part by way of contradiction, assuming that $\dis_Y$ is smooth and $Y$ is singular. Observe that all the points of $Y$ are critical points of $\dis_Y$, since this function is nonnegative and vanishes identically on $Y$.  Let $y$ be a singular point of $Y$ and let $Z$ be a $\ccc^\infty$ manifold of minimal dimension containing a neighborhood of $y$ in $Y$.
 
  As $Z$ is nonsingular, it cannot coincide with $Y$ near $y$, which means that there is a sequence $z_m$ in $Z \setminus Y$ converging to $y$. For each $m$ large, let  $y_m \in Y$ be such that $\dis_Y(z_m)=|z_m-y_m|^2$, and
  set $u:=\lim \frac{z_m-y_m}{|z_m-y_m|} \in T_y Z$ (extracting a subsequence if necessary, we may assume that this limit exists), as well as, for $x\in Z$, $\lambda(x):=d_x \dis_Y(u)$.  Clearly,  $\dis_Y(x)=|x-y_m|^2$, for all $x$ on the line segment joining $z_m$ and $y_m$.  A straightforward computation of derivative in the direction $u$ thus shows  that $ d_{y} \lambda(u)=2$, which implies that  $\lambda(x)$ is a smooth submersion on $Z$ near $y$. Therefore, its zero locus is a submanifold of $Z$ which contains $Y$  (since all the points of $Y$ are critical points of $\dis_Y$) and  of lower dimension than $Z$, in contradiction with the minimality assumption on the dimension of $Z$.
\end{proof}

\section{The regular locus of a definable set}
\begin{dfn}\label{dfn_regular_locus}
Given $X\in \s_n$, 
we denote by $X_{reg}$\nomenclature[ba]{$X_{reg}$}{regular locus of $X$\nomrefpage}  the set of points at which $X$ is an analytic manifold (of dimension $\dim X$ or smaller), and by $X_{sing}$\nomenclature[bb]{$X_{sing}$}{singular locus of $X$\nomrefpage} its complement in $X$.  We call $X_{reg}$ the {\bf regular locus of $X$}\index{regular locus} and $X_{sing}$ the {\bf singular locus of $X$}\index{singular locus}.

 Given a definable mapping $f:X \to \R^k$,  let $reg(f)$ \nomenclature[bc]{$reg(f)$}{regular locus of the mapping $f$\nomrefpage}be the set constituted by the points of $X_{reg}$ at which $f$ is analytic and let $sing(f)$\nomenclature[bd]{$sing(f)$}{singular locus of the mapping $f$\nomrefpage} be its complement in $X$.

 \end{dfn}

 We shall establish that  $X_{reg}$ and $reg(f)$ are definable (Theorem \ref{thm_X_reg} and Corollary \ref{cor_reg_locus__definissable}). 
\begin{pro}\label{pro_regular_locus_function}
Let $U\in\s_{n}$ be an  open set and let $f:U\to   \R$ be a definable   function. There exists an integer $k$ such that the following assertion holds true for every $x \in U$:

$(*)$ If $f$ is $\ccc^k$ on a neighborhood of $x$ then $f$ is analytic on a neighborhood of $x$.
\end{pro}

\begin{proof}
By Theorem \ref{thm_s_formula}, the set $U'$ of points at which $f$ is continuous is definable. Possibly replacing $U$ with $U'$, we thus may assume that $f$ is continuous.

 We shall  show by downward  induction  on $d$ that for every $d \le n$, there is a definable subset $E_d\subset U$ with  $\dim E_d<d$ such that the property $(*)$ holds for the restriction of $f$ to $U\setminus cl(E_d)$.
The result clearly follows from the case $d=0$.

 By Theorem \ref{thm_preparation} (see Remark \ref{rem_globally subanalytic_implies_analytic}), there is a cell decomposition of $ \R^n$ compatible with $U$ such that $f$ is  analytic on every cell of dimension $n$ included in $U$. Hence, for $d=n$, the result is clear (for any $k$).

Choose $d<n$, assume the result to be true for $(d+1)$, and   take a cell decomposition $\E$ compatible with $U$ and $E_{d+1}$, where $E_{d+1}\subset  U$ is provided by the induction hypothesis. Let $X$ be a cell of dimension $d$ included in $E_{d+1}$ (if $\dim E_{d+1}<d$, we are done).

  For $x\in U$, let   $k_{f}(x)$ be the greatest integer $k$ such that $f$ is $\ccc^k$ on a neighborhood of $x$, with $k_{f}(x)=\infty$ if $f$ is $\ccc^\infty$ around $x$.
We are going to show that there is a definable subset   $F\subset X$  of dimension  $<d$ such that  $k_f$ takes only finitely many values (in $\N\cup \{\infty\}$) on $X\setminus F$.  As we will also prove that if $k_{f}$ is infinite at $x \in X\setminus F$ then $f$ is analytic on a neighborhood of $x$, this will complete the induction step.

Since $X$ is a cell of dimension $d$,  there is a projection onto $\R^d$ which induces an analytic diffeomorphism from $X$ onto an open subset of $\R^d$, so that we may assume that $X \subset \R^d \times \{0_{\R^{n-d}}\}$ (we will sometimes regard $X$ as a subset of $\R^d$).
 Apply Proposition  \ref{pro_puiseux_avec_parametres} to the definable  function
 $$\tilde{f}:X\times \sph^{n-d-1}\times [0,\ep] \to \R, \quad \tilde{f}(x,u,r):= f(x,ru),$$
 and let  $\C$ be a cell decomposition compatible with the partition  of $X\times \sph^{n-d-1}$  provided by this proposition. For every $W\in \C$ included in $X\times \sph^{n-d-1}$,   $\tilde{f}(x,u,r^p)$ extends to an analytic function on a neighborhood of $W\times \{0\}$ in   $W\times \R$, for suitable $p$.  We may assume that $p$ is the same for all $W \in \C$ (taking the product of all the corresponding values of $p$). Let $\D:=\pi(\C)$, where $\pi:\R^{n} \to \R^{d}$ stands for the projection onto the $d$ first coordinates (see Remark \ref{rem_projection_cell_decomposition}).

 Fix a $d$-dimensional cell  $C \in \D$ which is included in $X$. We are going to prove that $k_f$ takes only finitely many values on $C\times \{0_{\R^{n-d}}\}$ (the desired subset $F$ may thus be defined as the union of the cells of dimension $<d$ contained in $X$).

Let $W_1,\dots,W_l$ be the family of all the cells of $\C$ that are included in $C\times  \sph^{n-d-1}$. For every $j$, there exist some definable functions $(a_{i})_{i\in \N}$, analytic on $W_j$, such that for $(x,u) \in W_j \subset C \times \sph^{n-d-1}$  and $r>0$ small enough we have the convergent expansion:
\begin{equation}\label{eq_filde}\tilde{f}(x,u,r)=\sum_{i\in \N} a_{i}(x,u)\,r^{\frac{i}{p}}.\end{equation}
                                 As the $W_j$'s cover $C \times \sph^{n-d-1}$,  we shall regard each $a_i$ as a function defined on  $C \times \sph^{n-d-1}$, gluing together (discontinuously) the respective $a_{i}:W_j\to \R$ obtained on each $W_j$.
We distinguish three cases:

\medskip

\noindent \underline{{\it First case.}} $a_{i,x_0}$ is not identically zero for some $i$ not divisible by $p$ and some $x_0 \in C$.

There is $u \in \sph^{n-d-1}$ such that $a_{i,x_0}(u)\ne 0$. Let $\nu$ be such that $(x_0,u)\in W_\nu$.   Observe  now that, as $a_i$  is analytic on $W_\nu$,  it is nonzero on an open dense subset of $W_\nu$ (since $W_\nu$ is connected).
As a matter of fact, by (\ref{eq_filde}), there is $j \in \N$ such that   $\frac{\pa^j \tilde{f}}{\pa r^j}(x,u,r)$  is unbounded as $r$ goes to zero for almost all $(x,u) \in W_\nu$. Consequently, $f$ is not $\ccc^j$ at $x$, for any $x\in C=\pi(W_\nu)$, which means that $k_f(x)\le j$ for all $x \in C$.

\medskip

\noindent \underline{{\it Second case}}. There exist $i$ divisible by $p$ and  $x_0\in C$ such that $a_{i,x_0}$ is neither identically zero  nor the restriction of a homogeneous polynomial of degree $\frac{i}{p}$ .

We claim that in this case  $f$ fails to be  $\ccc^{\frac{i}{p}}$ at every point of $C$, i.e., $k_f(x)< \frac{i}{p}$ for all $x \in C$. 
 We proceed by way of contradiction: if $f$ were a $\ccc^\frac{i}{p}$ function on a neighborhood of some point of $C$ then $u\mapsto \frac{d^{j} f(x,ru)}{dr^j}|_{r=0}$ would be either identically zero or a homogeneous polynomial of degree $j$,  for all $j \le \frac{i}{p}$ and $x$ close to this point. By (\ref{eq_filde}), it means that   for such $x$, $a_{i,x}$ would be either $0$ or the restriction of a homogeneous polynomial of degree $\frac{i}{p}$.  As $a_{i}$ is analytic on the $W_j$'s which are connected, this would imply that $a_{i,x_0}$ coincides with this polynomial, leading to a contradiction. Hence, $k_f(x)< \frac{i}{p}$ for  every  $x\in C$.

\medskip

 \noindent \underline{{\it Third case.}} Negation of the two above cases: for all $x\in C$ we assume  that for all $i$ not divisible by $p$, $a_{i,x}\equiv 0$,   and that for each $i$  divisible by $p$,   $a_{i,x}$ is either zero or a homogeneous polynomial of homogeneous degree $\frac{i}{p}$.

 In this case, by (\ref{eq_filde}) $$f(x,y)=\tilde{f}(x,\frac{y}{|y|},|y| )=\sum_{l \in \N} a_{lp}(x,\frac{y}{|y|})|y|^l=\sum_{l \in \N} a_{lp}(x,y),$$
 where each $a_{lp,x}$ is a homogeneous polynomial of degree $l$ for all $x\in C$. Using some estimates for germs of homogeneous polynomials \cite{bochnaksiciak, ds}, it is then not difficult to show that the series $\sum_{l \in \N} a_{lp}(x,y)$ converges locally uniformly and therefore defines an analytic function. It means  that  $f$ is analytic on a neighborhood of $C$ in $\R^n$, and $k_f(x)\equiv \infty$ on $C$.
\end{proof}

\begin{rem}\label{rem_familles_fonction_lieu_reg}
   If $f:U \to \R$ is a definable function with $U \in \s_{m+n}$ open and if, for each $\tim$, $k_t$ is the integer satisfying the property $(*)$ of Proposition \ref{pro_regular_locus_function} (applied to $f_t$) then one can see (examining  the proof of  Proposition \ref{pro_regular_locus_function}) that $k_t$ may be bounded away from infinity independently of $t$.  
\end{rem}

For $k\in \N^*$, let $X^k_{reg}$\nomenclature[be]{$X^k_{reg}$}{$\ccc^k$-regular locus\nomrefpage} denote the set of points of $X$ at which $X$ is a $\ccc^k$ manifold.

\begin{thm}\label{thm_X_reg}
If $X \in \St_n$ then 
 $X_{reg}$ is  definable and dense in $X$.  Indeed, $X^k_{reg}=X_{reg}$ for all $k$ sufficiently large.
\end{thm}
\begin{proof}It is easy to derive from existence of cell decompositions that $X_{reg}$ is dense (cells are analytic manifolds by definition). Let us show that it is definable. The set of points at which a set is locally closed being definable, we can assume $X$ to be locally closed.  By Proposition \ref{pro_caract_reg}, $X$ is a $\ccc^\infty$ manifold at $x$ if and only if $\rho_X$ is $\ccc^\infty$ in the vicinity of $x$. Let $k$ be an integer for which $\rho_X$ fulfills the property $(*)$  of Proposition \ref{pro_regular_locus_function}.  The set of points at which $\rho_X$ is $\ccc^k$ is definable (see Proposition \ref{pro_csq_qe} (\ref{item_C_k})) and, in virtue of the property $(*)$, coincides with $X_{reg}$.
\end{proof}

\begin{rem}
 It is worthy of notice that the same argument could be used to prove the following parametrized version of the above theorem: if  $X \in \s_{m+n}$ then the family $((X_t)_{reg})_{t\in \R^m}$ is a definable family of sets (see Remark \ref{rem_familles_fonction_lieu_reg}).
\end{rem}
 
\begin{cor}\label{cor_reg_locus__definissable}
 If $f:X\to \R$ is a definable  function then $reg(f)$ is dense in $X$ and definable. Consequently, $sing(f)$ is definable and has lower dimension than $X$.
\end{cor}
\begin{proof}
Let $g$ be the restriction of $f$ to the set of points of $X_{reg}$ at which $f$ is $\ccc^1$ (this set is dense in $X_{reg}$, see Remark \ref{rem_globally subanalytic_implies_analytic}). Clearly $reg(f)=\pi((\Gamma_g)_{reg})$, where $\pi:X\times \R\to X$ is the canonical projection. The result thus follows from Theorem \ref{thm_X_reg}.
\end{proof}

\section{Stratifications}\label{sect_strat}Stratifications satisfying regularity conditions constitute very useful tools to perform differential calculus on singular sets.  We introduce some of the famous regularity conditions for stratifications, such as Whitney's $(b)$ or Kuo-Verdier's $(w)$ condition, and show how to construct stratifications satisfying regularity conditions.
\begin{dfn}\label{dfn_stratifications}
 A {\bf 
stratification of}\index{stratification} a definable set $X$ is a finite partition $\Sigma$ of it into
definable $\ccc^\infty$ submanifolds of $\R^n$, called {\bf strata}\index{stratum}. We then say that $(X,\Sigma)$ is a {\bf stratified set}\index{stratified set}.
 A stratification is {\bf compatible} with a set  if this set is the union of some strata. \index{compatible!  stratification}  
   A {\bf refinement of a stratification}\index{refinement! of a stratification} $\Sigma$ of $X$ is a stratification of $X$ compatible with every stratum of $\Sigma$.
\end{dfn}

 The definition that we have chosen is very general. Some authors require in addition that the strata are connected or that the frontier condition holds (see Definition \ref{dfn_frontier_condition}). Connectedness of the strata can always be obtained by refining the stratification, and  the issue of the frontier condition will be discussed later on (see Remark \ref{rem_frontier_condition}).

\subsection{Whitney's and Kuo-Verdier's conditions.}   The notion of stratification is however generally too weak to perform differential geometry on singular sets, and it is most of the time needed to consider stratifications satisfying extra regularity conditions that describe the way the different pieces glue together. The most famous ones are Whitney's conditions.

We denote by $\G^{n}_k$ \nomenclature[ben]{$\G^{n}_k$}{Grassmannian of $k$-dimensional vector subspaces of $\R^n$\nomrefpage} the Grassmannian manifold of $k$-dimensional linear vector subspaces of $\R^n$. 
  The {\bf angle}\index{angle} \nomenclature[bep]{$\angle (P,Q)$}{angle  between two vector subspaces $P$ and $Q$ of $\R^n$\nomrefpage} between two given vector subspaces $E$ and $F$ of $\R^n$ will be
estimated as:
$$\angle(E,F):=\underset{u\in E,|u|
\le 1}{\sup} d(u,F).$$
The angle between a vector of $\R^n$ and a vector subspace of $\R^n$ is  defined as the angle between the space generated by this vector and this vector subspace.

\begin{dfn}\label{dfn_conditon_b_et_w}
Let $X$ and $Y$ be a couple of disjoint  submanifolds of $\R^n$  and let $z \in Y\cap cl(X)$. We say that  $(X,Y)$ satisfies {\bf Whitney's $(b)$ condition
at
 $z \in Y\cap cl(X)$}\index{Whitney's $(b)$ condition} if for any sequences $(x_k)_{k \in \N}$ and $(y_k)_{k\in \N}$, of points of
 $X$ and $Y$ respectively, converging to $z$ and such that $$\tau :=\lim T_{x_k} X\in\G^n_p\quad\eto\quad v:=\lim
 \frac{x_k-y_k}{|x_k-y_k|}\in \sph^{n-1}$$ exist (where
 $p=\dim X$), we have $v \in \tau$.

 We will say that  $(X,Y)$ satisfies  {\bf Whitney's $(a)$ condition}\index{Whitney's $(a)$ condition} at $z $ if
  $  \angle (T_{z} Y,T_x X)$ tends to zero as $x \in X$ tends to $z$.

 We will say that $(X,Y)$ satisfies the {\bf $(w)$ condition at $z$}\index{w@$(w)$ condition} (of Kuo-Verdier) if there exists a constant $C$ such that for $x\in X$ and $y
\in Y$ in a neighborhood of $z$:
\begin{equation}\label{eq_w}
 \angle (T_y Y,T_x X) \leq C
|y-x|.\end{equation}

\medskip

Finally, let $\pi$ be a $\ccc^\infty$ local retraction onto $Y$ near $z$ (it can easily be checked that the condition below is independent of $\pi$). We will say that $(X,Y)$ satisfies the {\bf $(r)$ condition}\index{r@$(r)$ condition} (of Kuo) at $z$ if:
\begin{equation}\label{eq_r}
 \lim_{x\to z,\, x\in X}   \frac{\angle (T_{\pi(x)} Y,T_x X)\cdot |x-z|}{|x-\pi(x)|} =0.\end{equation}
\end{dfn}

A stratification $\Sigma$  is {\bf $(w)$ regular}\index{stratification!$(w)$, $(b)$, $(a)$, $(r)$ regular} (resp. {\bf $(b)$, $(a)$, $(r)$ regular}) if  every couple $(X,Y)$ of strata of $\Sigma$ satisfies the $(w)$ condition (resp. $(b)$, $(a)$, $(r)$ condition) at every point of $Y$.

Whitney's $(b)$ condition is famous because it is involved in Thom-Mather's Isotopy Theorems \cite{mather}, which are of foremost importance in singularity theory. Kuo-Verdier's conditions ($(r)$ and $(w)$) are useful in subanalytic geometry because they provide analytic criteria for Whitney's $(b)$ condition, as established by the proposition below. An explicit example, where  $(r)$ is checked and  $(w)$ fails, is provided in Chapter \ref{chap_gmt} (Example \ref{exa_densite}), where these conditions will be of service.

\begin{pro}\label{pro_cond_reg_implications}
	For couples of definable manifolds, 
$\;(w) \Rightarrow (r) \Rightarrow (b) \Rightarrow (a)$.
\end{pro}
\begin{proof}
$(w) \Rightarrow (r)$ is a straightforward consequence of the definitions. To show that $(r)\Rightarrow (b)$,
fix a couple of strata $(X,Y)$ satisfying $(r)$.  Up to a choice of coordinate system of $Y$, we may assume\footnote{$(a)$ and $(b)$ are preserved by $\ccc^1$ diffeomorphisms, $(w)$ and $(r)$ by $\ccc^2$ diffeomorphisms.} that $Y$ is an open subset of $ \R^l \times \{0\}$, $l:=\dim Y$, and work near $z=0$.  Denote by $\pi$ the orthogonal projection onto $\R^l \times \{0\}$.

 Thanks to Curve Selection Lemma, we simply have to check Whitney's $(b)$ condition along two definable arcs $x:(0,\ep)\to X$ and  $y:(0,\ep)\to Y$ tending to the origin at $0$.
 Since $x$ is a Puiseux arc, we may suppose it to be parametrized by its distance to the origin, i.e., $|x(t)|=t$.
We have: 
\begin{equation}\label{eq_gamma}|x(t)-\pi(x(t))|\sim t
|x'(t)-\pi(x'(t))|.\end{equation} 
Denote by $P_t$ the orthogonal projection onto the orthogonal complement of $T_{x(t)} X$.  As $P_t(x'(t))\equiv 0$, we may write using (\ref{eq_gamma}):
$$P_t\left(\frac{x'(t)-\pi(x'(t))}{|x'(t)-\pi(x'(t))|}\right)\sim\frac{t|P_t(\pi(x'(t)))|
}{|x(t)-\pi(x(t))|},$$
which tends to zero in virtue of the $(r)$ condition.  This shows that the angle between
$(x'-\pi(x'))(t)$ and $T_{x(t)} X$ tends to zero as $t$ goes to zero.   
Since $(x-\pi(x))$ has a Puiseux parametrization we have:
\begin{equation}\label{eq_gamma_direction}\lim_{t\to 0}
\frac{x(t)-\pi(x(t))}{|x(t)-\pi(x(t))|}=\lim_{t\to 0}
\frac{x'(t)-\pi(x'(t))}{|x'(t)-\pi(x'(t))|},\end{equation}
which means that the angle between $(x(t)-\pi(x(t)))$ and $T_{x(t)} X$ must tend
to zero as well. Moreover, as $(y(t)-\pi(x(t)))$ belongs to $Y=T_{\pi(x(t))}Y$ for every $t$,  $(r)$ entails that the angle between $(y(t)-\pi(x(t)))\in T_{\pi(x(t))}Y$ and $T_{x(t)} X$ tends to zero (by definition of the angle $\angle (T_{\pi(x(t))}Y, T_{x(t)}X )$).  Together with the preceding sentence, this establishes that the angle between $$(x(t)-y(t))=(x(t)-\pi(x(t)))+(\pi(x(t))-y(t))$$ (this sum is orthogonal) and $T_{x(t)} X$ tends to zero, yielding  $(b)$  for  $(X,Y)$.

It remains to show that $(b)$ implies $(a)$. Again, we will assume $Y=\R^l \times \{0\}$.  Suppose that $(b)$  holds and  $(a)$ fails at $z\in Y$, i.e.,  assume that there is a sequence $(x_k)_{k\in \N}$ in $X$ tending to $z$ such that  $\tau:=\lim T_{x_k}X$ (exists and) does not contain $T_z Y=\R^l \times \{0\}$. 
 Let  $u\in T_z Y \setminus \tau$ and set $y_k:=z+\frac{1}{k}u$. Extracting a subsequence if necessary, we may assume that $(x_k)_{k\in \N}$ tends to $z$ as fast as we wish. In particular, we may assume that $\frac{y_k-x_k}{|y_k-x_k|}$ tends to $u$. By Whitney's $(b)$ condition, this implies that $u\in \tau$, a contradiction.
\end{proof}

\begin{rem}The  proof of $(r)\Rightarrow (b)$ relies on Curve Selection Lemma. This implication, unlike the other implications of the above proposition, is no longer true if $X$ and $Y$ are not definable,  as shown by the example $X:=\{(x,y):y=\sin \frac{1}{x}, x>0\}$ and $Y:= \{0\} \times (-1,1)$, which satisfies $(w)$ and not $(b)$.
\end{rem}

\subsection{Existence of regular stratifications of definable sets} To prove existence of Whitney or Kuo-Verdier regular stratifications, we are actually going to establish that we can always construct a stratification satisfying any sufficiently reasonable given regularity condition,   which leads us to the following definitions.

\begin{dfn}
A {\bf regularity condition on stratifications}\index{regularity condition on stratifications} is the data for every stratified set $(A,\Sigma)$ of a mathematical formula $\mathbf{G}(x,A,\Sigma)$, where $x\in A$.  We say that $\Sigma$ {\bf satisfies  $(\mathbf{G})$} if  $\mathbf{G}(x,A,\Sigma)$ holds true for every $x\in A$.

Such a condition is said to be {\bf local}\index{local regularity condition}, if, given any $A$ and $\Sigma$,  the value of  $\mathbf{G}(x,A,\Sigma)$ (true or false) 
just depends on the germ of $(A,\Sigma)$ at $x$ (for each $x\in A$), i.e., if for every definable open neighborhood $U$ of $x$ in $A$,  $\mathbf{G}(x,A,\Sigma)$ is equivalent to   $\mathbf{G}(x,U,\Sigma\cap U)$, where $\Sigma\cap U$ is the stratification induced by $\Sigma$ on $U$. 

A regularity condition is {\bf stratifying}\index{stratifying regularity condition} if, given $A$ and $\Sigma$, as well as $S\in \Sigma$, there is a definable open dense subset $W$ of $S$ such that  $\mathbf{G}(x,A,\Sigma)$ holds true $\forall x\in W$.
\end{dfn}

 Whitney's and Kuo-Verdier's conditions provide examples of local regularity conditions on stratifications. We shall show that they are stratifying. We first show  that we can always construct a stratification satisfying a local stratifying condition.

\begin{pro}\label{pro_existence}
 Given a stratifying local condition and a definable set $A$, there is a stratification of $A$ satisfying this condition. We can require this stratification to be compatible with finitely many given definable subsets of $A$.
\end{pro}

\begin{proof}Let $A$ be a definable subset of $\R^n$, $X_1,\dots,X_l$  definable subsets of $A$,  and  let $(\mathbf{G})$ be a local stratifying condition.
 We construct by decreasing induction on $0\le k\le d+1$, $d:=\dim A$,    a closed definable subset  $E_k$ of $A$ of dimension less than $k$, such that there is a stratification $\Sigma^k$ of $A\setminus E_k$ which is compatible with  $X_1\setminus E_k,\dots, X_l\setminus E_k$  and satisfies the condition $(\mathbf{G})$.  We can start with $E_{d+1}:=A$.

 Take then $k\le d$ for which $\Sigma^{k+1}$ and $E_{k+1}$ have already been constructed, and let $\C$ be a cell decomposition  of $\R^n$ compatible with $E_{k+1,reg}$ and the $X_i$'s. Denote by $S_1,\dots, S_m$ the cells of $\C$ of dimension $k$ that are included in  $E_{k+1}$.
%
 Observe that if we set $Z:=E_{k+1}\setminus \cup_{i=1}^m S_i $ then $\dim Z<k$ and 
 $$\Sigma'^{k+1}:=\Sigma^{k+1} \cup \{S_1,\dots,S_m\}$$
  is a stratification of $A\setminus Z$.  It remains to take off the points of the $S_i$'s at which the condition  $(\mathbf{G})$ fails.
  Since this condition is stratifying, there is for each $i\le m$ an open dense definable subset $W_i\subset S_i$ such that $\mathbf{G}(x,A\setminus Z, \Sigma'^{k+1})$ holds for every $x\in W_i$. Clearly, if we set $E_k:=E_{k+1}\setminus \cup_{i=1}^m W_i$ then
 $\Sigma^k:= \Sigma^{k+1}\cup \{W_1,\dots,W_m \}  $
 is a stratification of $A\setminus E_k$.
 Moreover, since the condition $(\mathbf{G})$ is local, and because the germs at the points of $\cup_{i=1}^m W_i$ of the respective strata of $\Sigma'^{k+1}$ and $\Sigma^k$ coincide, we see that $\Sigma^k$ fulfills the condition $(\mathbf{G})$. 
\end{proof}
 
\begin{rem}\label{rem_reg_ref}$ $\begin{enumerate}[(1)]      \item The conjunction of several local stratifying conditions being local and stratifying, we have proved that we can construct a stratification satisfying several local stratifying conditions simultaneously.  
                               \item The fact that we can assume the constructed stratification to be compatible with finitely many definable subsets ensures that we can refine any given stratification into a stratification satisfying a given local stratifying condition.
                               \item  The algorithm of construction that we gave ensures that, when we wish to refine a stratification $\Sigma$ that already satisfies $(\mathbf{G})$  on a definable open set $U$, we do not need to modify $\Sigma$ on $U$.
                              \end{enumerate}

\end{rem}

\begin{pro}\label{pro_w_stratifying}
The conditions $(a)$, $(b)$, $(r)$, and $(w)$ are stratifying.
\end{pro}
\begin{proof} Thanks to Proposition \ref{pro_cond_reg_implications}, it suffices to show the result for the condition $(w)$. Let $\Sigma$ be a  stratification and let $X$ and $Y$ be two strata.
  Up to a choice of coordinate system of $Y$, we may identify  $Y$ with a neighborhood of $0$ in $\R^{k}\times
\{0_{\R^{n-k}}\}$  (where $k=\dim Y$). Below, we will sometimes abusively consider $Y$ as a subset of $\R^k$. 

Note that, by Theorem \ref{thm_s_formula}, the set of points at which the $(w)$ condition fails is definable. We shall proceed by way of contradiction:  assume that there is a definable open subset
$U$ of $ Y$  such that  $(w)$  fails for $(X,Y)$  at
every point of $U$. It means that for any $y \in U$ and any $r>0$ there is $\omega(y,r)\in X\cap \bou(y,r)$ such that:
\begin{equation}\label{eq_w_fail}d(\omega(y,r),Y) \leq r \cdot \angle(Y,T_{\omega(y,r)} X) .\end{equation}
By Definable Choice (Proposition \ref{pro_globally subanalytic_choice}), $\omega(y,r)$ can be chosen definable.   
 Let
 $$W:=\omega(U \times (0,\ep))\subset X.$$
 The mapping $\omega$ can be seen as a parametrization of the set $W$. We are going to find another parametrization of this set of type $(y,\xi(y,t))$ in order to make our computations easier. By Lemma \ref{lem_closure_dense}, there is a definable open dense subset $U'$ of $U$ such that for any $y \in U'\subset \R^k$, we have $cl(W_y)=cl(W)_y$. Hence, for all $y \in U'$, $0_{\R^{n-k}}$ belongs to $cl(W_y)$. As a matter of fact, for all $y \in U'$  there exists $\xi(y,t)$ in  $\sph(0_{\R^{n-k}},t) \cap W_y$, for each $t>0$ small enough. By Definable Choice again  we may assume that $\xi$ is definable.

  Write $\xi=(\xi_1,\dots,\xi_{n-k})$ and apply Proposition \ref{pro_puiseux_non_cont} to each of the $\xi_i$'s. This provides a definable partition $\Pa$ of $U'$ into $\ccc^\infty$ manifolds such that for every $C\in\Pa$,  $\xi(y,t)$ coincides with a Puiseux series with analytic coefficients   on a neighborhood $N$ of $C\times \{0\}$ in $C \times  \R_+$.
  Fix $C \in \Pa$  of dimension $k$ and define  a  mapping  by setting for $y\in C$ and $t>0$ small $$g(y,t):=(y,\xi(y,t)) .$$

  We are going to show that for each $ y\in C$ and $x\in g(N)\cap W$ close to $y$ we have
\begin{equation}
\label{eq_dist_T_xX} \angle(Y,T_{x} X)\lesssim d(x,Y),\end{equation}
which will contradict (\ref{eq_w_fail}) and establish the proposition.
      To prove this,  
notice that since $g(y,t)\in W$ for all $y\in C$ and  $t>0$ small, the vector
 $$v_i(y,t):= d_{(y,t)} g (e_i)$$
  belongs to $T_{g(y,t)} X$ for any $y\in C$, $t>0$ small, and $i\le k$. Furthermore, if the Puiseux expansion of $\xi(y,t)$
starts like
$ \xi(y,t)=a(y)\cdot t+\dots,$
with $a:C \to \sph^{n-k-1}$ analytic function, then for all $i\le k$:
\begin{equation*}\label{eq_der1_xi_puiseux}
 \frac{\partial \xi}{\pa y_i}(y,t)=\frac{\partial a}{\pa y_i}(y)\cdot t
+\dots.
\end{equation*}
As a matter of fact, in the vicinity of any given point of $C \times \{0\}$ we have for  $i\le k$:
$$|v_i(y,t)-e_i| \lesssim t= |\xi(y,t)|=d(g(y,t),Y), $$
showing (\ref{eq_dist_T_xX}) (since $v_i(y,t)\in T_{g(y,t)} X$). 
\end{proof}

\subsection{Stratifications of mappings}

\begin{dfn}\label{dfn_stratified_mapping}
  Let $F:A\to B$ be a definable mapping. We say that $F:(A,\Sigma)\to (B,\Sigma')$  is a {\bf stratified mapping}\index{stratified mapping} if $\Sigma$ and $\Sigma'$ are stratifications of $A$ and $B$ respectively such that for every stratum  $S\in \Sigma$, $F(S)$ is included in an element  $S'$ of $ \Sigma'$ and the restricted mapping $F_{|S}:S\to S'$ is a $\ccc^\infty$ submersion.
\end{dfn}

\begin{pro}\label{pro_strat_mapping}
 Given a local stratifying regularity condition for stratifications $(\mathbf{G})$ and a definable mapping $F:A\to B$, we can find two stratifications  $\Sigma$ of $A$ and $\Sigma'$ of $B$ satisfying  $(\mathbf{G})$ and making of $F:(A,\Sigma)\to (B,\Sigma')$ a  stratified mapping. Moreover, these stratifications may be required to be compatible with finitely many given definable subsets of $A$ and $B$ respectively.
\end{pro}
\begin{proof}Let $A_1,\dots, A_k$ (resp. $B_1,\dots, B_\kappa$) be definable subsets of $A$ (resp. $B$).
  We are going to prove by decreasing induction on $p\in \{0,\dots,l+1\}$, $l=\dim B$, that for every such $p$ there is a definable closed subset $E_p\subset B$ of dimension less than $p$ such that we can find stratifications $\Sigma^p$ and $\Sigma'^p$ of $F^{-1}(B\setminus E_p)$ and $B\setminus E_p$ respectively, satisfying  $(\mathbf{G})$  and such that $F$ maps submersively  the strata of $\Sigma_p$ into the strata of $\Sigma'_p$. These stratifications will respectively be compatible with the sets $A_i\setminus F^{-1}(E_p)$ and $B_j\setminus E_p$, $i\le k$, $j\le \kappa$.
  
   In the case $p=l+1$, we set $E_{l+1}:=B$ and we are done.
 Assume thus that $E_p$, $\Sigma^p$, and $\Sigma'^p$ have been constructed for some $p \le l+1$. If $\dim E_p<p-1$, we are done. Otherwise, 
  let $N$ denote the set of points of $E_{p}$ at which this set is a smooth manifold of dimension $(p-1)$, and let $\mathscr{S}$ be a stratification of $F^{-1}(N)$  such that $F$ is smooth on each stratum (see Remark \ref{rem_globally subanalytic_implies_analytic}).
  
    As  $(\mathbf{G})$  is local and stratifying, we can refine $\mathscr{S}$  into a stratification (still denoted $\mathscr{S}$) such that  $\mathscr{S}\cup \Sigma^p$ satisfies this condition (see Remark \ref{rem_reg_ref} $(3)$) and is compatible with  the $F^{-1}(N)\cap A_i$'s.
  Given a point $x\in F^{-1}(N)$, we will denote by $S^x$ the element of $ \mathscr{S}$ that contains $x$.   
  Let $D$ be the closure of the set of points $x\in F^{-1}(N)$ at which the rank of the  derivative of the restricted mapping  $F_{|S^x}:S^x\to N$ is less than $(p-1)$.

   Let now $\mathscr{S}'$ be any stratification of $N$ compatible with  $F(D),B_1\cap N,\dots, B_\kappa\cap N$, as well as with the sets  $N\cap F(S)$, $S\in \mathscr{S}$. Let $S_1,\dots, S_m$ be the strata of $\mathscr{S}'$ that have dimension $(p-1)$, and set $\Sigma''^p:=\Sigma'^p\cup\{S_1,\dots,S_m\}$.
    As  $(\mathbf{G})$ is stratifying, there is for every $i$ an open definable dense subset $Y_i$ of $S_i$ on which $\Sigma''^p$ satisfies this condition. 
    

We set $\Sigma'^{p-1}:=\Sigma'^p\cup\{Y_1,\dots,Y_m\}$ as well as $\Sigma^{p-1}:=\Sigma^p\cup \{ F^{-1}(Y_i)\cap S, S\in \mathscr{S}, i\le m \}.$ Because  $Y_1,\dots,Y_m$ are open in $N$, which is open in $E_p$, the set $E_{p-1}:=E_p\setminus \cup_{i=1}^m Y_i$ is closed in $B$.  Moreover, by construction,  the local condition $(\mathbf{G})$ must hold for the stratifications $\Sigma^{p-1}$ and $\Sigma'^{p-1}$.

To check that $F$ induces a stratified mapping on $A\setminus F^{-1}(E_{p-1})$, take  $Z\in \Sigma^{p-1}\setminus \Sigma^p$ (the desired property already holds for the other strata by induction). As $Z\notin  \Sigma^p$, $Z=  F^{-1}(Y_i)\cap S$, for some $S\in \mathscr{S}$ and  some $i \le m$. Since
  (by Sard's Theorem)  $\dim F(D)<p-1$,  $Y_i$ must be disjoint from $F(D)$ (recall that $Y_i\subset S_i$, $\dim S_i=p-1$, and $S_i\in \mathscr{S}'$ which is compatible with $F(D)$). Hence,
  the rank of the restricted mapping
   $F_{|Z}:Z\to Y_i$ is $(p-1)$,  which  means that  $F$ maps submersively the strata of $\Sigma_p$ into the strata of $\Sigma'_p$.
      \end{proof}

\paragraph{Horizontally $\ccc^1$ maps.}We now are going to construct nicer stratifications for definable continuous  mappings with bounded first derivative (recall that definable mappings are smooth almost everywhere). In this case, we can have a continuity property of the derivative when passing from one stratum to one another. This will be useful to study the
pull-back of a differential form by a definable
bi-Lipschitz mapping (not necessarily smooth) in Chapter \ref{chap_Geometric measure theory}.

\begin{dfn}
A stratified mapping  $h:(X,\Sigma)
\to (Y,\Sigma')$  is said to be {\bf horizontally
$\ccc^1$}  \index{horizontally $\ccc^1$} if for any sequence
$(x_l)_{l\in \N}$ in a stratum $S$ of $\Sigma$ tending to some
point $x$ in a stratum $S'\in \Sigma$ and for any  sequence $u_l \in
T_{x_l} S$ tending to a vector $u$ in $T_x S'$, we have
$$\lim d_{x_l} h_{|S}(u_l)=d_x h_{|S'} (u).$$
\end{dfn}

If $h$ is horizontally $\ccc^1$ then the norm of $d_x h_{|S}$ (as a linear map) is bounded away from infinity on every subset of $S$ that is relatively compact in $X$ for every $S \in \Sigma$. The converse is not true, even if the mapping is continuous, as shown by the following example.
\begin{exa}
 Consider the function $f(x,y):=\frac{x^2y}{x^2+y^2}$,  with $f(0,0):=0$, defined on the set $X:=\{ (x,y)\in \R^2:y\ge 0\}$, that we will stratify by $S:=\{y>0\}$ and $S':=\{y=0\}$ (stratifying the target space by $\{0\}$ and $\R\setminus \{0\}$). Note that $$\pa_{(x,y)}f=(\frac{2xy^{3} }{(x^2+y^2)^{2}},\frac{x^4-x^2y^{2} }{(x^2+y^2)^{2}}),$$ which is bounded. The derivative of $f_{|S'}$ is zero but $\frac{\pa f}{\pa x}(x,x)$ does not tend to zero as $x\to 0$. The function $f$ is therefore not horizontally $\ccc^1$ with respect to this stratification.\end{exa}


 In the above example,  if we refine the stratification by adding a stratum at the origin, the function $f$ is then horizontally $\ccc^1$ with respect to the obtained stratification.
  The proposition below shows that it is always possible to do so.

\begin{pro}\label{pro_h_hor_C1}
Let $h:X\to Y$ be a definable continuous mapping. If $|d_x h|$ is bounded on every subset of $reg(h)$ which is relatively compact in $X$ then there
 exist stratifications $\Sigma$ of $X$ and $\Sigma'$ of $Y$ making of $h:(X,\Sigma)\to (Y,\Sigma')$  a horizontally $\ccc^1$ stratified mapping. 
  Moreover, we may require these stratifications to be compatible with finitely many definable subsets of $X$ and $Y$   respectively and to satisfy a given local stratifying condition.
\end{pro}
\begin{proof}Let $\pi_1:\Gamma_h \to X$ (resp. $\pi_2:\Gamma_h\to Y$) be the projection onto
the source (resp. target) space of $h$.  By Propositions \ref{pro_w_stratifying} and \ref{pro_strat_mapping}, there exist two Whitney $(a)$ regular stratifications $\Sigma_h$ and $\Sigma'$ of $\Gamma_h$ and $Y$ respectively such that  $\pi_2:(\Gamma_h,\Sigma_h) \to (Y, \Sigma')$ is a stratified mapping.  Since $\Sigma_h$ can be required to refine any given stratification, we may assume that $h$ is smooth on the images of the strata of $\Sigma_h$ under $\pi_1$ (see Remark \ref{rem_globally subanalytic_implies_analytic}).
Let $\Sigma:=\{\pi_1(S):S\in \Sigma_h\}$.

 Notice that  $h:(X,\Sigma)\to (Y,\Sigma')$ is a stratified mapping. 
 In order to show that it is horizontally $\ccc^1$, fix  $S\in\Sigma$, a
sequence $x_l \in S$ tending to a point $x\in S'\in \Sigma$, as well as
a sequence $u_l \in T_{x_l} S$ tending to some $u \in T_x S'$.
Let $Z$ be the stratum of $\Sigma_h$ that projects onto $S$ via $\pi_1$,   set   $\tau:=\lim T_{(x_l,h(x_l))}  Z$ (extracting a subsequence, we may assume that this sequence is convergent),  and let us show the following

\noindent {\bf Claim.} The restriction of $\pi_1$ to $\tau$ is one-to-one. 

Observe for this purpose that, as $(\Gamma_h)_{reg}$ is dense in $\Gamma_h$, for every $l$ we can find  $y_l\in (\Gamma_h)_{reg}$ arbitrarily close to $(x_l,h(x_l))$.  For each $l$, let $Z^{l}$ be the stratum of $\Sigma_h$ containing $y_l$.  Choosing $y_l$ sufficiently generic, we may assume that $Z^l$ is open in $\Gamma_h$, and, extracting a subsequence if necessary, we can assume that $T_{y_l} Z^{l}$ tends to a limit $\tau'$.  Moreover, by Whitney's $(a)$ condition, if $y_l$ is sufficiently close to $(x_l,h(x_l))$  then  $\tau'\supset \tau$. As $|d_{x_l}h|$  is bounded independently of $l$,   $\pi_{1|\tau'}$ must be one-to-one, which yields the claim.

For every $l$, there 
is a unique vector $v_l \in T_{(x_l,h(x_l))} Z$ which projects
onto $u_l$ via $\pi_1$. As  $|d_{x_l} h|$ is bounded independently of $l$, the norm of $v_l$ must be bounded above, which means that
we may assume that $v_l$ is converging to a vector $v$. We clearly have $\pi_1(v)=u$.
 Let $Z'$  be the stratum of $\Sigma_h$
that projects onto $S'$ via $\pi_1$ and let $v'\in T_{(x,h(x))} Z'$ be such that $\pi_1(v')=u$. By Whitney's  $(a)$ condition $(v-v')\in\tau$, which,  since  $\pi_1(v)=u=\pi_1(v')$, yields $v=v'$ (by the above claim). Hence, $v$ is tangent to  $Z'$, which entails that $\pi_2(v)= d_{x} h_{|S'}(u)$ and therefore:
$$\lim d_{x_l} h_{|S} (u_l)=\lim \pi_2(v_l)=\pi_2(v)= d_{x} h_{|S'}(u),$$
which shows that $h$ is horizontally $\ccc^1$.

Since the stratifications  $\Sigma$ and $\Sigma'$ that we just constructed are merely provided by a Whitney $(a)$ regular stratification of $\pi_2$, we see that we can require  them to be compatible with any given finite families of definable subsets of $X$ and $Y$ respectively.

 By Proposition \ref{pro_strat_mapping},  there exist refinements of $\Sigma$ and $\Sigma'$  satisfying any prescribed stratifying local condition.  Since stratified mappings are smooth on strata,  the property of being horizontally $\ccc^1$ is preserved under refinements.
\end{proof}

\subsection{Some more properties of Whitney and  Kuo-Verdier stratifications} 
\begin{pro}\label{pro_whitney_pi_rho_subm} Let $(X,Y)$ be a couple of strata and let $y\in Y$. If $(X,Y)$ satisfies 
	Whitney's $(b)$ condition at $y$ then there is a neighborhood $U$ of $y$ such that the restricted mapping $(\pr_Y,\dis_Y)_{|U\cap X}:U\cap X \to Y\times \R$ (see section \ref{sect_retraction})   is a submersion.
\end{pro}
\begin{proof}
	By Whitney's $(b)$ condition, the angle between $\pa_x \rho_Y=2(x-\pr_Y(x))$ and $T_x X$ tends to zero as $x\in X$ tends to $y$, which implies that $\rho_{Y|X}$ is a submersion near $y$. Hence, to complete the proof,
	it suffices to show that for $x\in X$ sufficiently close to $y$ the restriction of $d_x \pr_{Y}$ to $\ker d_x \dis_{Y|X}$ is onto.  Assume that this fails along a sequence $(x_k)_{k\in \N}$ in $X$ tending to $y$. Extracting a sequence if necessary, we can assume that $\ker d_{x_k} \dis_{Y}$ and $T_{x_k} X$ respectively converge to vector spaces $L$ and $\tau$. As Whitney's $(b)$ condition implies $(a)$ (Proposition \ref{pro_cond_reg_implications}), we have $T_y Y \subset \tau$.
	
	By Proposition \ref{pro_retraction}, we know that  $\pa_{x_k} \dis_Y$ is orthogonal to  $T_{\pr_Y(x_k)} Y$, which implies  $T_y Y\subset L$.   As the angle between   $\pa_{x_k} \rho_Y$ and  $T_{x_k} X$ tends to zero, we  see that 
	$$ \lim \ker d_{x_k} \dis_{Y|X} =(\lim \ker d_{x_k} \dis_{Y}) \cap (\lim T_{x_k} X) =L\cap \tau.$$
	 As $d_y \pr_Y$ induces the identity map on $T_y Y\subset L\cap \tau$,  this equality yields that the restriction of $d_{x_k} \pr_Y$ to $\ker d_{x_k} \dis_{Y|X}$ is surjective for every $k$ large,  contradicting our assumption on $(x_k)_{k\in \N}$.  
\end{proof}
\begin{rem}
This proposition yields that if a couple of strata $(X,Y)$ satisfies  $cl(X)\cap Y\neq \emptyset$ and Whitney's $(b)$ condition then $\dim X\ge \dim Y+1.$\end{rem}

\begin{dfn}\label{dfn_frontier_condition}Let $\Sigma$ be a stratification of a set $A$. We say that $\Sigma$ {\bf satisfies the frontier condition}\index{frontier condition} if for every $S\in \Sigma$ the set $fr(S)\cap A$ is the union of some elements of $\Sigma$. 
\end{dfn}

\begin{pro}\label{pro_frontier_condition_whitney}
	Let $\Sigma$ be a Whitney $(b)$ regular stratification  of a locally closed set. If all the strata of $\Sigma$ are connected then $\Sigma$ satisfies the frontier condition.
\end{pro}
\begin{proof} Take a couple of strata 
	$(X, Y)$ that satisfies $fr(X)\cap Y\neq \emptyset$. Arguing by downward induction on the dimension of $Y$, we can assume that the desired property holds for the strata of dimension bigger than $\dim Y$. Observe that if there is a stratum $Z\subset fr(X)$ such that $cl(Z)\supset Y$ then clearly $cl(X) \supset Y$.  As a matter of fact, since we can argue inductively on the dimension of $X$, we can also assume $X\cup Y$ to be locally closed at every point of $Y$.
	
	As the strata of $\Sigma$ are connected, it suffices to show that if $y\in cl(X)\cap Y$ then $\bou(y,\alpha)\cap Y\subset cl(X)\cap Y$, for $\alpha>0$ small.  Up to a choice of coordinate system, we can assume that $Y$ is an open subset of $\R^k \times \{0_{\R^{n-k}}\}$. Let $(U,\pi_Y,\rho_Y)$ be a tubular neighborhood of $Y$ (see Proposition \ref{pro_retraction}).
	
	Take $y \in cl(X)\cap Y$ and  $\alpha>0$ which is sufficiently small  for $\bou(y,2\alpha)\cap (X\cup Y)$ to be closed in $\bou(y,2\alpha)$ and  for $\bou(y,2\alpha)$ to be contained in $U$. Set for $\ep>0$,  $A_\ep:=\{x\in U\cap X:\dis_{Y}(x)=\ep \}$. Since $\pr_Y$ is the identity on $Y$, it suffices to check that $ \bou(y,\alpha)\cap Y\subset \pi_Y(A_\ep)$, for all $\ep>0$ small.
	
	By Proposition \ref{pro_whitney_pi_rho_subm}, every point of $Y$ has  a neighborhood $W$  such that the restriction of $(\pr_Y,\dis_Y)$  to $W \cap X$  is a submersion. It means that for $\ep>0$ small the restriction of $\pr_Y$ to  $A_\ep$ is a submersion near $\bou(y,\alpha)\cap Y$, which implies that $\bou(y,\alpha)\cap\pr_Y(A_\ep)$ is an open subset of $\bou(y,\alpha)\cap Y$.
We claim that $\bou(y,\alpha)\cap \pr_Y(A_\ep)$ is also closed in  $\bou(y,\alpha)\cap Y$ for all $\ep>0$ small. As the latter subset is connected, this will imply that $ \bou(y,\alpha)\cap Y\subset \pr_Y(A_\ep)$, for all $\ep>0$ small, as needed. Indeed, if $z_i\in \bou(y,\alpha)\cap \pr_Y(A_\ep)$ is a sequence  tending to some  $z\in \bou(y,\alpha)\cap Y$, then  $z_i=\pr_Y(x_i)$ for some $x_i\in A_\ep$, and, extracting a sequence if necessary, we can assume that $x_i$ tends to some $x$. For $\ep>0$ small, $x\in \bou(y,2\alpha)\subset U$, which, since   $\bou(y,2\alpha)\cap (X\cup Y)$ is  closed in $\bou(y,2\alpha)$, entails $x\in A_\ep$ and hence $z=\pi_Y(x)\in \bou(y,\alpha)\cap \pr_Y(A_\ep)$.
\end{proof}

\begin{rem}\label{rem_frontier_condition}
	This proposition, together with Propositions \ref{pro_existence} and \ref{pro_w_stratifying} (see Remark \ref{rem_reg_ref} $(1)$), implies that we can always construct a stratification of a locally closed definable set satisfying both a given stratifying local condition and the frontier condition.
\end{rem}

The following proposition will be needed to study the continuity of the density along the strata of a Whitney stratification  in Chapter \ref{chap_Geometric measure theory}.
 \begin{pro}\label{pro_cond_reg_dim_1} Let $(X,Y)$ be a couple of strata with $\dim Y=1$. If $(X,Y)$ satisfies Whitney's $(b)$ condition at $z \in Y\cap cl(X)$  then it satisfies the condition $(r)$ at this point. \end{pro}
\begin{proof}
	Fix such a couple of strata $(X,Y)$ and assume that it satisfies   Whitney's $(b)$ condition.  We can suppose that $Y$ is an open neighborhood of the origin in $\R \times \{0\}$.  Thanks to Curve Selection Lemma (Lemma \ref{curve_selection_lemma}), it suffices to show that  (\ref{eq_r}) holds along any
	definable arc $x:(0,\ep)\to X$ tending to the origin at $0$. We may assume that $x$ is parametrized
	by its distance to the origin, which means that $x'(t)$ does not tend to zero as $t$ goes to zero.  Let $\pi$ be the orthogonal projection onto  $\R \times \{0\}$ and $P_t$  the orthogonal projection onto the orthogonal complement of $T_{x(t)} X$. As $(X,Y)$ is  $(b)$ regular, we have
	$
	\lim_{t\to 0} P_t\lepa\frac{x(t)-\pi(x(t))}{|x(t)-\pi(x(t))|} \ripa=0,$
	which entails (see (\ref{eq_gamma_direction}))
	\begin{equation}\label{eq_lim_preuve_w_implique r}\lim_{t\to 0}
		P_t\lepa\frac{x'(t)-\pi(x'(t))}{|x'(t)-\pi(x'(t))|} \ripa=0. \end{equation}
	As $P_t(x'(t))\equiv 0$,  we also have:
	\begin{equation}\label{eq_lim_preuve_w_implique r2}
	\frac{|P_t(\pi(x'(t)))|\cdot |x(t)|}{|x(t)-\pi(x(t))|}=\frac{
		|P_t(x'(t)-\pi(x'(t)))|t}{|x(t)-\pi(x(t))|}\overset{(\ref{eq_gamma})}{\lesssim}\frac{|P_t(x'(t)-\pi(x'(t)))|}{|x'(t)-\pi(x'(t))|},\end{equation}
	which tends to zero by (\ref{eq_lim_preuve_w_implique r}).
Now, if $\pi(x'(t))$ does not tend to zero then, as $\angle(Y,T_{x(t)}X)\sim |P_t(\pi(x'(t)))|$ (since $Y$ is one-dimensional),  we see that (\ref{eq_lim_preuve_w_implique r2}) yields (\ref{eq_r}) holds along the arc $x$. Moreover, in the case where $\pi(x'(t))$ tends to zero, then $|x(t)|\sim |x-\pi(x(t))| $ (since $x'(t)$ does not tend to zero). But,  since $(b)\Rightarrow (a)$ (Proposition \ref{pro_cond_reg_implications}), we already know that $\angle(Y,T_{x(t)}X)$ tends to zero as $t$ goes to zero, which yields
  (\ref{eq_r}) along the arc $x$  (since the ratio $\frac{|x|}{|x-\pi(x)|}$ is bounded).
\end{proof}

\section{Approximations and partitions of unity}
Given  $Y\in\s_n$, we denote by $\s^+(Y)$ the set of definable positive continuous functions on $Y$.\nomenclature[beq]{$\s^+(Y)$}{definable positive continuous functions on $Y$\nomrefpage}

\begin{pro}\label{pro_approx} Let $M$ be a definable $\ccc^k$ submanifold of $\R^n$, $k\in \N^*$.
\begin{enumerate}
 \item (Definable $\ccc^k$ partitions of unity)  Given a finite covering of $M$ by definable open subsets $(U_i)_{i\in I}$ there are finitely many definable $\ccc^k$ functions $\varphi_j:M\to [0,1]$, $j\in J$, such that $\sum_{j\in J} \varphi_j =1$ and such that  every $\varphi_j$ is supported in $U_{i(j)} $ for some $i(j)\in I$.
 \item (Definable $\ccc^k$ approximations) Given a continuous definable function $f:M\to \R$ and $\ep \in \s^+(M)$, there is a $\ccc^k$ definable function $g$ on $M$ such that $|f-g|<\ep$ on $M$.
\end{enumerate}

\end{pro}
\begin{proof} We prove these two assertions by induction on  $m:=\dim M$. Both statement being obvious for $m=0$, take $m\ge 1$ and denote by $(i)_{<m}$ and $(ii)_{<m}$ the respective induction hypotheses.

We first perform the induction step of $(i)$. Let $\Sigma$ be a stratification of $M$ compatible with all the $U_i$'s and satisfying the frontier condition (see Remark \ref{rem_frontier_condition}).
We denote by $\Sigma_{<m}$ (resp. $\Sigma_{=m}$) the collection of the strata of $\Sigma$ that have dimension less than (resp. equal to) $m$, and by $\Sigma_S$ the set of  strata of $\Sigma$ that are included in $fr(S)$.  We fix a tubular neighborhood $(U_S,\pi_S,\rho_S)$ of every $S\in \Sigma$ (see Proposition \ref{pro_retraction}), and, given  $S\in \Sigma_{<m}$ as well as $\delta \in \s^+(S)$, we let
\begin{equation}\label{eq_udelta}
U_S^\delta:=\{x\in M\cap U_S:\rho_S(x)<\delta (\pi_S(x))\}.
\end{equation}

We first point out a useful  consequence of $(ii)_{<m}$:

 \noindent{\bf Observation.} If $S\in \Sigma_{<m}$ and  $\delta$ as well as $\ep$ belong to $\s^+(S) $ then, by $(ii)_{<m}$, there is a $\ccc^k$ definable function $\delta'$ on $S$ such that $|\delta'-\delta|<\min(\frac{\ep}{3},\frac{\delta}{3})$. We also can find $\ccc^k$ definable approximations of $2|\delta-\delta'|$ (as in $(ii)$ of the proposition with $\ep$ as small as we wish). Subtracting from $\delta'$ such a function if necessary, we see that we actually can find a $\ccc^k$ definable function $\delta'$ on $S$ satisfying both $|\delta-\delta'|<\ep$ and $\delta' <\delta $. Note that then $U_S^{\delta'}\subset U^\delta _S$, which means that every $U^\delta _S$ contains a neighborhood defined by a $\ccc^k$ definable function.

 In order to construct our partition of unity let us take a $\ccc^k$ piecewise polynomial function $\psi$ that satisfies $\psi(x)= 1$ if $x\le \frac{1}{2}$ and $\psi(x) =0$ if $|x|\ge \frac{3}{4}$ (one can take for instance $\psi(x):=\frac{1}{a}\int_\frac{3}{4}^x (t-\frac{1}{2})^k (t-\frac{3}{4})^k dt$ for $x\in [\frac{1}{2},\frac{3}{4}]$, with $a=\int_\frac{3}{4}^{\frac{1}{2}} (t-\frac{1}{2})^k (t-\frac{3}{4})^k dt$). Given  $S\in \Sigma_{<m}$ and a $\ccc^k$ positive definable function $\delta$ on $S$, we then define a $\ccc^k$ function on $U_S$ (recall that strata are $\ccc^\infty$ submanifolds of $\R^n$) by setting
$$\psi_S^\delta (x):=\psi\left(\frac{\rho_S(x)}{\delta(\pi_S(x))}\right).$$
By construction,  $\psi_S^\delta $ is  supported  in $U^\delta_S$ and is equal to $1$ on $U^\frac{\delta}{2}_S$. Hence, for  $\delta< d(\cdot , M\setminus U_S)$, we  can extend $\psi_S^\delta $ (by $0$) to a $\ccc^k$ function on $M\setminus fr(S)$.

Given $S\in \Sigma_{=m}$ and a $\ccc^k$ positive definable function $\delta $ on $S$, we set $U^\delta _S:=S$ 
as well as $\psi_S ^\delta :\equiv 1$. Given a collection of function $\delta=(\delta_Y)_{Y\in \Sigma}$, with $\delta_Y$ positive $\ccc^k$ definable function on $Y$ for every $Y$, and a stratum $S\in \Sigma$, we then can define a nonnegative $\ccc^k$ function on $M\setminus fr(S)$ by setting
$$\psit^\delta_S(x):=\psi_S^{\delta_S} (x)\cdot \prod_{Y\in \Sigma_S} (1-\psi_Y^\frac{\delta_Y}{2}(x)).$$
As this function  vanishes on $ U^\frac{\delta_Y}{4} _Y$ if $Y\in \Sigma_S$, it can be extended (by $0$) to a $\ccc^k$ function on $M$.
Note that  $\psit^\delta_S(x)=1$ on $W_S:=U^{\frac{\delta_S}{2}}_S\setminus \bigcup_{Y\in \Sigma _S} U^{\frac{\delta_Y}{2}}_Y$
 for every $S\in \Sigma$. 
 Since $\cup_{S\in \Sigma} W_S=M$, the function $\sum_{S\in \Sigma} \psit^\delta_S(x) $ is bounded below away from zero on $M$, which makes it possible to set
 \begin{equation}\label{eq_phidelta}
 \varphi_S^\delta:=\frac{\psit_S^{\delta_S}}{\sum_{Y\in \Sigma} \psit_Y^{\delta_Y}}\, ,
 \end{equation}
so that $\sum_{S\in \Sigma}\varphi_S^\delta\equiv 1$. Since the $U_i$'s are open and because $\Sigma$ is compatible with them, every $S\in \Sigma_{<m}$ has a neighborhood $U^\mu_S$, $\mu \in \s^+(S)$,  which fits in $U_{i(S)}$ for some $i(S)\in I$. Hence, since we can choose $\delta_S$ smaller than $\mu$ (and yet smooth, by the above observation), we see that we can assume each $\varphi_S^\delta$, $S\in \Sigma$, to be supported in some $U_{i(S)}$. The family $(\varphi_S^\delta)_{S\in \Sigma}$ thus constitutes the desired partition of unity.

 We now perform the induction step of $(ii)$. Fix $\ep \in \s^+(M)$.  Let $\Sigma$ be a stratification of $M$ such that $f$ is $\ccc^\infty$ on every stratum (see Remark \ref{rem_globally subanalytic_implies_analytic}), and let $\Sigma_{<m}$ as well as $\Sigma_{=m}$ be as in the proof of $(i)$. Let for each $S\in \Sigma$ and each collection of functions $\delta=(\delta_S)_{S\in \Sigma}$ (with $\delta_S$ definable $\ccc^k$ positive function on $S$, as above), $U^{\delta_S}_S$ and $\varphi_S^\delta$ be as in (\ref{eq_udelta}) and  (\ref{eq_phidelta}).  Given $S\in \Sigma_{<m}$ and $x\in U^{\delta_S}_S$, we set $g_S(x):=f( \pi_S(x))$. We also set $g_S:=f$ if $S\in \Sigma_{=m}$.
 Note that, as $f$ is continuous, it follows from definable choice (Proposition \ref{pro_globally subanalytic_choice}) that if  $S\in \Sigma_{<m}$ and $\delta_S$ is sufficiently small (and yet positive definable $\ccc^k$, see the above observation) then $|f-g_S|<\ep$ on $U^{\delta_S}_S$.  As a matter of fact, if we set $g:=\sum_{S\in \Sigma} \varphi_S^{\delta} g_S$ then
 $|f-g|=|\sum_{S\in \Sigma}\varphi_S^{\delta}(f-g_S)|<\ep. $
 \end{proof}

\paragraph{Historical notes.}Section \ref{sect_quantifier_elimination} is constituted by basic model theoretic principles that are applied to our framework. Most of the properties of subanalytic sets were already present  in S. \L ojasiewicz's fundamental work  \cite{lojasiewicz59, lojasiewicz64a, lojasiewicz64b}. Let us however mention that \L ojasiewicz's inequality (Theorem \ref{thm_lojasiewicz_inequality}), which was the essential ingredient of his solution of Schwartz's division problem (see the introduction), was independently established by L. H\"ormander in the semi-algebraic category, who provided a solution of Schwartz's problem in this framework \cite{hormander}.  The material of this chapter was then rewritten and generalized independently by  many people (see in particular \cite{vdd_omin, shiota, costeomin} for a complete expository).  The presentation which is provided here is fairly close to the introductory book \cite{costeomin}, whose content is partially inspired by the book of L. van den Dries' book \cite{vdd_omin}. The subanalycity of the regular locus is due to M. Tamm \cite{tamm}, although the proof we have presented here was  taken from \cite{kurdykapointsreguliers}. It is difficult to quote an original reference for tubular neighborhoods provided by closest point retractions; a clear proof, close to our content, can be found in \cite{polyraby}.
Whitney's $(b)$ condition was introduced in \cite{whitney tangent}. Constructions of Whitney regular stratifications  in the semi-analytic category seem to go back to \L ojasiewicz. Kuo-Verdier's stratifications are generalizations of Whitney's stratifications that appeared later \cite{kuo,verdier}. A proof of existence was given in \cite{lojwacsta}. One of the first results of approximations of continuous definable functions by smooth definable functions with the topology considered in Proposition \ref{pro_approx} (i.e., $|f-g|<\ep$ with $\ep \in \s^+(M)$) is Efroymson's Approximation Theorem \cite{efroymson}, which is devoted to the semialgebraic category. This theorem is however much more difficult to prove than the latter proposition since it provides a $\ccc^\infty$ approximation, which prevents from using partitions of unity. Proposition \ref{pro_approx} is closer  to a theorem proved by M. Shiota \cite{shiota} (see also \cite{escribano}, these works however give in addition approximation of $p$ derivatives if $f$ is $\ccc^p$). M. Shiota \cite{shiotapprox} also proved (in the semialgebraic category) an approximation theorem which provides $\ccc^\infty$ definable approximations of $\ccc^p$ definable functions with approximation of the derivatives. Efroymson-Shiota's theorem was extended to the subanalytic category (and even o-minimal) in \cite{ef} (with the approximation of the first derivative only).
%
%
%


          \chapter{Lipschitz Geometry}\label{chap_lipschitz_geometry}
In this chapter, we undertake the description of  singularities of globally subanalytic sets from the Lipschitz point of view. We first study the interplay between Lipschitz functions, regular vectors (Definition \ref{boule_reguliere}), and the inner metric, which is the metric provided by the length of the shortest arc connecting two given points. We prove in particular that every definable set can be decomposed into Lipschitz cells (Theorem \ref{thm_lipschitz_cells}), which enables to compare the inner and outer metrics of a definable set. We then enter the explicit description of the Lipschitz geometry of singularities, introducing and constructing some triangulations, called {\it  metric triangulations}, that completely capture the Lipschitz geometry of definable sets.  We derive some consequences about the Lipschitz conic structure of definable singularities. These results recently turned out to be useful to study Sobolev spaces and geometric integration theory on definable sets \cite{lebeau, poincwirt, trace, lprime, laplace, depauw, Oud}.

We recall that the word ``{\bf definable}'' is used as a shortcut of ``globally subanalytic'' (see right after Proposition \ref{pro_csq_qe}).

 To give an intuitive idea of our concept of metric triangulation on a simple example, let us consider the cusp of equation $y^2=x^3$
in  $\mathbb{R}^2$. It is impossible  to find  a
triangulation of this set which is a  bi-Lipschitz homeomorphism.  The best that we can do is
to construct a homeomorphism which
``contracts" the vertical distances by  multiplying  them
by a power of the distance to the origin (see Figure \ref{fig1}).
  \begin{figure}[h]
\includegraphics[ scale=0.9]{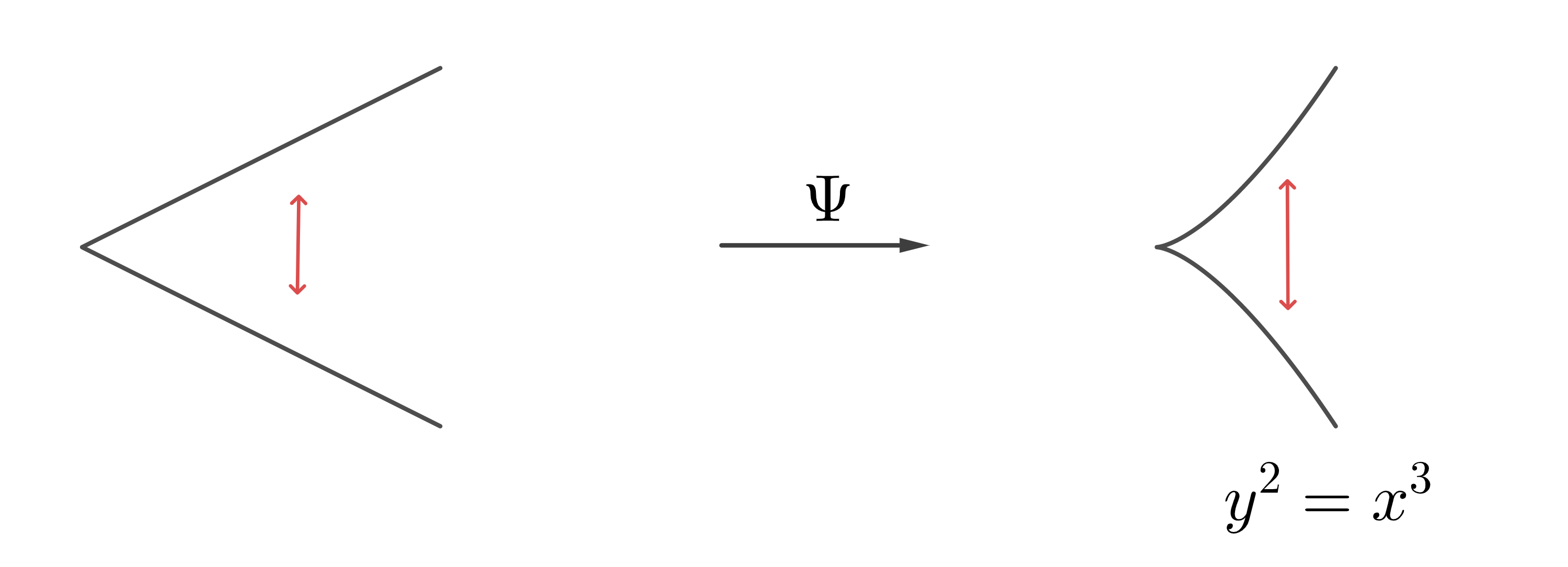}
\centering
  \caption{Triangulation of a germ of cusp. }
\label{fig1}\end{figure}

  \begin{figure}[h]\label{fig2}
\includegraphics[ scale=0.7]{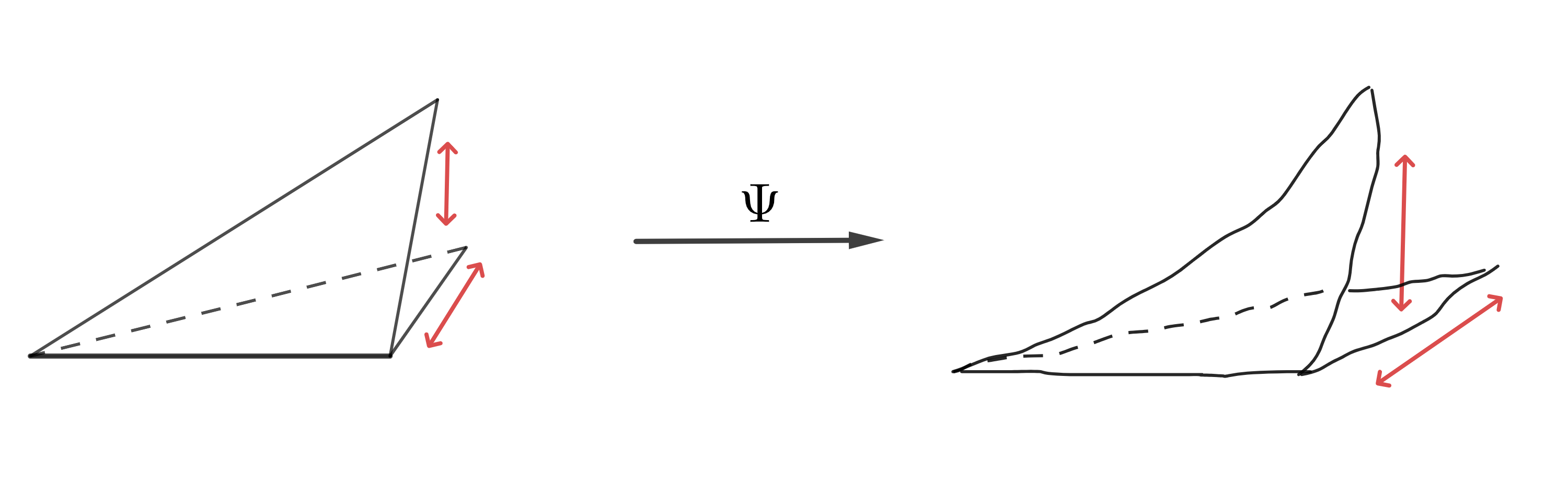}
\centering
  \caption{Triangulation of a germ of surface. }\label{fig2}
\end{figure}

It is getting more complicated as the dimension is increasing. Nevertheless,  in higher dimensions, the idea is to find a homeomorphism $\Psi$,   from a simplicial complex onto a given definable set, that contracts distances in a way that we explicitly describe in terms of distances to the faces of the simplices (see Figure \ref{fig2}).

\bigskip


\noindent {\bf Some notations.} Given $P\in \G_k^n$, we denote by $\pi_P$\nomenclature[bi]{$\pi_P$}{orthogonal projection onto $P\in \G_k^n$\nomrefpage}  orthogonal projection onto $P$. 
   Given  $\lambda \in \sph^{n-1} $, we denote by $N_\lambda$\nomenclature[bj]{$N_\lambda$}{vector space normal to $\lambda\in \sph^{n-1}$\nomrefpage} the hyperplane of $\R^n$ normal to the vector $\lambda$, and by $q_\lambda$ the coordinate
of $q\in \R^n$ along $\lambda$, i.e., the number given by the euclidean inner product of $q$ by $\lambda$\nomenclature[bjb]{$q_\lambda$}{coordinate
of $q$ along $\lambda$ \nomrefpage}.


 Given $B\in \s _{n}$ and  $\lambda \in \sph^{n-1} $, with $B \subset
 N_\lambda$, as well as
 a function $\xi:B \to \R$, we\nomenclature[bjg]{$\Gamma_\xi ^\lambda$}{graph of $\xi$ for $\lambda$\nomrefpage} set 
\begin{equation}\label{eq_graph_for_lambda}
\Gamma_\xi ^\lambda:=\{ q \in \R^n : \pi_{_{N_\lambda}}(q)  \in B
\quad  \mbox{and} \quad q_\lambda = \xi(\pi_{_{N_\lambda}}(q))\},
\end{equation} and call this set {\bf the graph of $\xi$ for $\lambda$}\index{graph!for $\lambda$}.

Define finally the {\bf $m$-support}\index{m-support@$m$-support of a set}\nomenclature[bjh]{$\supp_m(A)$}{$m$-support of the set $A$\nomrefpage} of a set $A\in \s_{m+n}$ by $$\supp_m A:=\{t \in \R^m :A_t \ne \emptyset\}.$$

\section[Regular vectors and Lipschitz functions]{Regular vectors and Lipschitz functions}\label{sect_reg_lips}
A mapping $\xi:A \to B$, $A\subset\R^n$, $B\subset \R^k$, is said to be {\bf Lipschitz}\index{Lipschitz} if there is a constant $L$ such that for all $x$ and $x'$ in $A$: 
$$|\xi(x)-\xi(x')|\le L|x-x'|.$$
 We say that $\xi$ is {\bf $L$-Lipschitz}\index{l-Lipschitz@$L$-Lipschitz} if we wish to specify the constant. 
The smallest nonnegative number $L$ having this property is called the {\bf Lipschitz constant of $\xi$}\index{Lipschitz! constant} and is denoted $L_\xi$\nomenclature[bk]{$L_\xi$}{Lipschitz constant of the function $\xi$\nomrefpage}.
By convention, if $A$ is empty then $\xi$ is Lipschitz and $L_\xi=0$.

A mapping $\xi:A \to B$  is {\bf bi-Lipschitz}\index{bi-Lipschitz} if it is a homeomorphism onto its image such that $\xi$ and $\xi^{-1}$ are Lipschitz.

A definable family $(f_t)_{t \in \R^m}$ is  {\bf uniformly Lipschitz}\index{uniformly Lipschitz} if $f_t$ is $L$-Lipschitz for all $t\in \R^m$, {\it with  $L $  independent of  $t$}.  The {\bf uniformly bi-Lipschitz}\index{uniformly bi-Lipschitz} families  are then defined analogously.

\begin{pro}\label{pro_extension_fonction_lipschitz} Every definable Lipschitz function $\xi:A \to \R$,  $A \in \s_n$, can be extended to an $L_\xi$-Lipschitz definable function  $\overline{\xi}:\R^n \to \R$. 
\end{pro}
\begin{proof}
	Set $\overline{\xi}(q):=\inf \{\xi(p)+L_\xi |q-p|:p\in A\}$ (for $A \ne \emptyset$). By the quantifier elimination principle (Theorem \ref{thm_s_formula}), it is a definable function. An easy computation shows that $\overline{\xi}$ is $L_\xi$-Lipschitz.
\end{proof}

\begin{rem}\label{rem_extension_familles_fonctions_lipschitz}
	Let $A \in \St_{m+n}$  and let a definable function $\xi:A\to  \R$ be such that $\xi_t:A_t \to \R$ is a Lipschitz function for every $t\in \R^m$. 
	The respective extensions $\overline{\xi}_t$  of $\xi_t$, $t \in \R^m$ (with for instance $\overline{\xi}_t\equiv 0$ if $t \notin \supp_m A$), provided by the proof of the  above proposition constitute a definable family of functions. We thus can extend definable families of Lipschitz functions to definable  families of Lipschitz functions. 
\end{rem}

 \subsection{Regular vectors}\label{sect_regular_vector}
 Given a
definable set $A \subset \R^n$,  let\nomenclature[bm]{$\tau(A)$}{set of all the tangent spaces at regular points of $A$\nomrefpage} 
$$\tau(A):=cl ( \{ T_x A  \in \cup_{k=1}^{n}\G^{n}_k : x \in A_{reg} \}).$$
For $\lambda$ of $\sph^{n-1}$ and  $Z \subset \cup_{k=1}^{n}\G^{n}_k$\nomenclature[bn]{$d(\lambda,Z)$, $Z\subset \cup_{k=1}^{n}\G^{n}_k$,}{ euclidean distance to a subset of $\cup_{k=1}^{n}\G^{n}_k$\nomrefpage}, we set (caution, here $Z$ is not a subset of $\R^n$):$$d(\lambda,Z) := \inf \{d(\lambda,T):T \in Z\},$$ with by convention $d(\lambda, \emptyset):=+\infty$.
\begin{dfn}\label{boule_reguliere}
Let $A\in \s_n$. An element $\lambda$ of
$\sph^{n-1} $ is said to be {\bf regular  for the set $A$}\index{regular! for a set}  if there is $\alpha >0$ such that:
$$
d(\lambda, \tau(A)) \geq \alpha.
$$
More generally, we say that $\lambda \in \sph^{n-1} $ is {\bf regular for  $A\in \s_{m+n}$}  if there exists $\alpha >0$ such that for all  $t\in \R^m$:
\begin{equation}\label{eq_reg_familles}
  d(\lambda, \tau(A_t)) \geq \alpha.
 \end{equation}
We then also say that $\lambda$ is {\bf regular for the family $(A_t)_{t \in \R^m}$}\index{regular! for a family}.
\end{dfn}

 If $\lambda\in \sph^{n-1}$ is regular for $A\in \smn$, it is regular for $A_t\in \s_n$ for all $t \in \R^m$. It is however much stronger since 
in (\ref{eq_reg_familles}), the  angle between  $\lambda$ and the tangent spaces to the fibers is required to be bounded below away from zero by a positive constant {\it independent of the parameter $t$}.

\paragraph{The regular vector theorems.} The theorems below are two essential ingredients of the construction of metric triangulations  (Theorem \ref{thm_proj_reg_hom_pres_familles} will be needed to prove Theorem \ref{thm existence des triangulations} and Theorem \ref{thm_proj_loc} will be needed to prove the local version presented in Theorem \ref{thm_triangulations_locales}). The proof of these results being too long to be included in these notes, the reader is referred to \cite{vlt, gvhandbook}.

Regular vectors do not always exist, even if the considered sets have empty interior, as it is shown by the simple example of a circle. Nevertheless, when the considered sets have empty interior, we can find such a vector, up to a globally subanalytic family of bi-Lipschitz mappings \cite[Theorem $1.3.2$]{gvhandbook}:
\begin{thm}\label{thm_proj_reg_hom_pres_familles}
	Let $A\in \smn$ be such that $A_t$ has empty interior for every $t \in \R^m$. There exists a uniformly bi-Lipschitz 
	definable family of homeomorphisms  $h_t :  \R^n \rightarrow  \R^n$, $\tim$,  such that  $e_n$   is regular for the family $(h_t(A_t))_{t\in \R^m}$.
\end{thm}

As we just said, the construction of metric triangulations of germs will require another version of the regular vector theorem (see the introduction of section \ref{sect_metric_triang_local} for more details).
For $R>0$  and $n\in \N\setminus \{0,1\}$,  we first set
\begin{equation}\label{eq_ccn}
	\ccn :=\{(t,x) \in [0,+\infty)\times \R^{n-1}: |x|\leq Rt\}. \end{equation} 
We also set $\mathcal{C}_1(R):=[0,+\infty)$\nomenclature[bu]{$\ccn$}{\nomrefpage}.

	Let $A, B \subset \R^n$. A definable map $h:A \to B$  is {\bf vertical}\index{vertical} if it
	preserves the first coordinate in the canonical basis of $\R^n$, i.e. if for any $t \in \R$, $\pi(h(t,x))=t$, for all $x \in A_t$,  where $\pi:\R^n \to \R$ is the  projection onto the first coordinate.

In the case of germs of subsets of $\mathcal{C}_{n}(R)$, the homeomorphism  provided by Theorem \ref{thm_proj_reg_hom_pres_familles} may be required to  be a vertical map (note however that in the theorem below, there is no parameter $t$, unlike in the preceding theorem):

\begin{thm}\label{thm_proj_loc} \cite[Theorem $1.5.14$]{gvhandbook}
	Let $X$ be the germ at $0$ of a definable  subset of $\mathcal{C}_{n}(R)$ (for some $R$) of
	empty interior. There exists a germ of
	vertical bi-Lipschitz definable homeomorphism (onto its image) $H:(\C_{n}(R),0)\to
	(\C_n(R),0)$ such that $e_{n}$ is regular for $H(X)$.
\end{thm}

\subsection{The inner metric}\label{sect_inner_metric}
Any definable arc $\gamma:(\eta,\ep)\to \R^n$ (not necessarily continuous), $\eta<\ep$, is piecewise analytic. Its length (possibly infinite) is therefore well-defined as:\nomenclature[bum]{$lg(\gamma)$}{length of the arc $\gamma$\nomrefpage}
$$lg(\gamma):=\int_\eta ^\ep |\gamma'(t)|\, dt.$$
It follows from Puiseux Lemma  that this integral is always finite if $\gamma$ is bounded. We thus may define {\bf the inner metric of a set}\index{inner metric} $X \in \St_n$ by setting for $a$ and $b$ in the same connected component of $X$ (see Proposition \ref{lem_connexe_par_arc})\nomenclature[bup]{$d_{X}(a,b)$}{inner metric of $X$\nomrefpage}:$$d_{X}(a,b):=\inf \{ lg(\gamma):\gamma:[\eta,\ep]\to X, \mbox{ $\ccc^0$ definable arc joining } a \mbox{ and } b\}.$$
We rather have defined a metric on every connected component of $X$. By convention, $d_{X}(a,b)=+\infty$, when $a$ and $b$ do not lie in the same connected  component of $X$. Note that for all $a$ and $b$ in $X$:
\begin{equation}\label{d_inner_plus_grande}
	|a-b|\le d_{X}(a,b).
\end{equation}
A definable mapping $f:X \to Y$  is {\bf Lipschitz with respect to the inner metric}\index{Lipschitz! with respect to the inner metric} if  for some constant $L$ we have for all $a$ and $b$ in $X$:
$$d_Y(f(a),f(b)) \le L \cdot d_{X}(a,b).$$ 
We say {\bf $L$-Lipschitz with respect to the inner metric}\index{l-Lipschitz@$L$-Lipschitz with respect to the inner metric} if we wish to specify the constant. 
 By analogy, the restriction of the euclidean metric to $X$ is called the {\bf outer metric of $X$}\index{outer metric}. 
In general, the outer metric is not equivalent to  the inner metric as it is shown by the simple example of the cusp $y^2=x^3$ in $\R^2$. 

\begin{pro}\label{pro_bounded_der_lipschitz}
	Let $f:X \to Y$ be a definable continuous mapping and let $L \in \R$.  The mapping $f$ is $L$-Lipschitz with respect to the inner metric of $X$ if and only if $|d_x f|\le L$  for all  $x \in reg(f)$.   
\end{pro}
\begin{proof}The only if part is obvious. Assume that $|d_x f|\le L$ on $reg(f)$ and let us show that $f$ is $L$-Lipschitz with respect to the inner metric.
	By Proposition \ref{pro_h_hor_C1}, there are Whitney $(a)$ regular stratifications  of $X$ and $Y$  with respect to which $f$ is horizontally $\ccc^1$. By continuity, we see that $\sup_{x \in S} |d_x f_{|S}| \le L$ for every stratum $S$.   Let $a$ and $b$ be two points in the same connected component of $X$ and let $\gamma:[0,1]\to X$ be a $\ccc^0$ definable curve joining $a$ and $b$.  Let $t_0=0\le t_1<\dots<t_k=1$ be such that  $\gamma$ stays in the same stratum on $(t_i,t_{i+1})$ for all $i<k$.  By the Mean Value Theorem,  for every $i<k$, we have: $$d_Y(f(\gamma(t_i)),f(\gamma(t_{i+1})))\le lg(f(\gamma_{|[t_i,t_{i+1}]}))\le L\cdot  lg(\gamma_{|[t_i,t_{i+1}]}), $$
	which shows that $d_Y(f(a),f(b))$ is not bigger than $L\cdot lg(\gamma)$.
\end{proof}
\begin{rem}\label{rem_bounded_der_lipschitz}
	In particular, if the derivative of a continuous definable mapping $f:\R^n \to Y$ is  bounded  on an open dense set by a real number $L$ then $f$ is $L$-Lipschitz  (with respect to the {\it outer} metric). This fact is of course not always true for functions that are not definable (like the so-called Cantor functions).
\end{rem}

\subsection{Lipschitz cell decompositions}\label{sect_lipschitz_cell_decomposition}
We are going to show that every definable set can be decomposed into Lipschitz cells (Definition \ref{dfn_la_regular_cell} and Theorem \ref{thm_lipschitz_cells}), which will entail that any continuous definable function with bounded derivative is piecewise Lipschitz (Corollary \ref{cor_bounded_der_lips}). This requires some lemmas that we present first.

\begin{lem}\label{lem_reg_inclusion}
	Let $A$ and $B$ in $\s_{n+m}$ with $B\subset A$.  If $\lambda \in \sph^{n-1}$ is regular for $A$, then it is regular for $B$.
\end{lem}
\begin{proof}
	Assume that $\lambda \in \sph^{n-1}$ is not regular for $B$. It means that there is a sequence $((t_i,b_i))_{i \in \N}$, with $b_i \in B_{t_i,reg}$ such that $\tau:= \lim T_{b_i}B_{t_i,reg}$ exists and contains $\lambda$.  Choose for every $i$ a Whitney $(a)$ regular stratification of $A_{t_i}$ compatible with $B_{t_i}$ and $B_{t_i,reg}$  and denote by $S_i$ the stratum containing $b_i$.  Moving slightly $b_i$ if necessary, we may assume that $S_i$ is open in $B_{t_i,reg}$ (since $B_{t_i,reg}$ is open and dense in $B_{t_i}$), which entails that  $T_{b_i} S_i=T_{b_i}B_{t_i,reg}$.  As $A_{t_i,reg}$ is dense in $A_{t_i}$, for every $i\in \N$,  we can find $a_i$ in $A_{t_i,reg}$ which is close to $b_i$.    Moreover, possibly extracting a sequence, we may assume that $\tau':= \lim T_{a_i}A_{t_i,reg}$ exists.
	If $a_i$ is sufficiently close to $b_i$, by Whitney's $(a)$ condition, we deduce that $\tau' \supset \tau$, which contains $\lambda$. This yields that $\lambda$ is not regular for $A$.
\end{proof}

\begin{rem} It is worthy of notice that the proof of the above lemma  shows that the corresponding number $\alpha$ (see (\ref{eq_reg_familles})) can remain the same for $B$.
\end{rem}

\begin{lem}\label{lem  kurdyka proj fixees}
	Given $\nu \in \N$, there exist $\lambda_1, \dots, \lambda_N$ in $\sph^{n-1}$
	and $\alpha_\nu >0$ such that for any $P_1 , \dots, P_\nu$ in $\bigcup_{l=1} ^{n-1} \G
	^n_l$ we can find $i \leq N$ such that
	$ d(\lambda_i  , P_j)> \alpha_\nu,$  for all $j \le \nu$.
\end{lem}
\begin{proof}
	Given $P_1,\dots, P_\nu$ in $\bigcup_{l=1} ^{n-1} \G
	^n_l$, let $\varphi(P_1,\dots,P_\nu):=\sup_{\lambda\in \sph^{n-1}}\min_{j\le \nu} d(\lambda,P_j)$. Since the $P_j$'s have positive codimension, $\varphi$ is a positive function. As the Grassmannian is compact, $\varphi$  must be bounded below  by some positive number $t_\nu$.
	
	Let 
	$\lambda_1, \dots, \lambda_N$ in $\sph^{n-1}$ be such that
	$\bigcup_{i=1}^N \bou(\lambda_i,\frac{t_\nu}{2})\supset \sph^{n-1}$.  Suppose that there are  $P_1,
	\dots, P_\nu$ in $\bigcup_{l=1} ^{n-1} \G
	^n_l$ such that for any $i \in \{1,\dots, N\}$
	we have $d(\lambda_i
	,\bigcup_{j=1} ^\nu P_j)\le \frac{t_\nu}{2}$. Then
	any $\lambda$ in $\sph^{n-1}$ satisfies  $d(\lambda
	,\bigcup_{j=1} ^\nu P_j)<  t_\nu,$ in contradiction with our choice of $t_\nu$. It is thus enough to set $\alpha_\nu:=\frac{t_\nu}{2}$.
\end{proof}

We recall that we estimate the   angle between two vector subspaces $P$ and $Q$ of $\R^n$ in the following
way:
$$\angle (P,Q)=\sup \{ d(\lambda,  Q): \lambda\: \, \mbox{is a unit vector of}\:\, P\}.$$

\begin{dfn}\label{df alpha flat}
	Let $\alpha >0$ and let $Z\in \s_{m+n}$. We say that $Z$ is {\bf $(m,\alpha)$-flat}\index{flat@$(m,\alpha)$-flat} if: $$\sup \{\angle (P,Q):P, Q \in \bigcup_\tim \tau(Z_{t,reg})\} \le \alpha.$$
 	 We then also say that $(Z_t)_{\tim}$ is {\bf $\alpha$-flat}\index{flat@$\alpha$-flat}.	When $m=0$, we say that $Z$ is {\bf $\alpha$-flat}.
\end{dfn}

\begin{rem}\label{rem_alpha_flat}
	It follows from Lemma \ref{lem  kurdyka proj fixees} that if $Z_{1,t},\dots, Z_{\nu,t}$, $\tim$, are $\alpha_{\nu}$-flat definable families (where $\alpha_\nu$ is the constant provided by the latter lemma) of subsets of $\R^n$ of empty interiors then one of the $\lambda_i$'s (that are also provided by the latter lemma) is regular for all these families. 
\end{rem}
\begin{lem}\label{lem_partition_tangents}
	Given $Z\in \s_{m+n}$ and $\alpha>0$, we can find a finite covering of $Z$ by $(m,\alpha)$-flat definable subsets of $Z$. 
\end{lem}
\begin{proof}Dividing $Z$ into cells, we may assume that $Z$ is a cell.
	 If $\dim Z=l+m$, we can cover $\G_l^n$ by finitely many balls of radius $\frac{\alpha}{2}$, which gives rise to a covering $U_1,\dots,U_k$ of $Z$  (via  the family of  mappings $Z_{t}\ni x \mapsto T_x Z_{t}$) by  $(m,\alpha)$-flat sets. 
\end{proof}

\begin{pro}\label{pro_dec_L_regulier_avec_proj_fixees}
	There exist $ \lambda_1, \dots, \lambda_N$ in  $\sph^{n-1}$
	such that for any $A_1,\dots, A_p$   in  $\smn$,
	there is a cell decomposition $\C$ of
	$\R^{m+n}$ compatible with all the  $A_i$'s and such
	that for each cell $C \in \C$ satisfying $\dim C_t=n$ (for all $t\in \supp_m C$), we may find $\lambda _{j(C)}$, $ 1
	\leq j(C) \leq N$,
	regular for the family $ (\delta C_t)_\tim$.
\end{pro}
\begin{proof}
	According to Lemma \ref{lem  kurdyka proj fixees} (see Remark \ref{rem_alpha_flat} and Lemma \ref{lem_reg_inclusion}) it is sufficient
	to prove by induction on $n$ the following assertions: given
	$\alpha >0$ and $A_1, \dots, A_p$ in $\smn$,
	there exists a cell decomposition of $\R^{m+n}$
	compatible with $A_1, \dots, A_p$ and
	such that for every  cell $C\subset \R^{m+n}$ of this cell decomposition satisfying $\dim C_t=n$ (for all $t\in \supp_m C$), there are  $\alpha$-flat definable families of sets of empty interior $V_{1,t},\dots , V_{l,t}$, with $l\le 2n$, such that $\delta C_t\subset \cup_{i=1}^l V_{i,t}$ for all $\tim$.

	For $n=0$ this is clear.
	Fix $n \ge 1$, $\alpha >0$, as well as $A_1 ,\dots, A_p$
	in $ \s _{m+n}$. Apply Lemma \ref{lem_partition_tangents}  to all the cells of a cell decomposition $\E$ compatible with the $A_i$'s, and take a cell decomposition $\D$ of $\R^{m+n}$ compatible with all the elements of the obtained coverings.
	Applying then the
	induction hypothesis to the elements of  $\pi(\D)$ ($\pi$ being the canonical projection onto $\R^{m+n-1}$), we get a refinement $\D'$ of $\pi(\D)$.
	
	Given a cell $D$ of $\D'$, each cell of $\E$ is above $D$, either the graph of a definable function, say $\zeta_{i,D}$, or a band, say $(\zeta_{i,D},\zeta_{i+1,D})$, with $\zeta_{i,D}<\zeta_{i+1,D}$ definable functions on $D$ (or $\pm \infty$, see Definition \ref{dfn_cell_decomposition}). Let $\C$ be the cell decomposition given by all the graphs $\Gamma_{\zeta_{i,D}}$, $D\in \D'$. To check that it has the required property, fix a cell  $C=(\zeta_{i,D},\zeta_{i+1,D})$, with $\zeta_{i,D}<\zeta_{i+1,D}$ definable functions (or $\pm \infty$) on a cell $D$ of $\D'$ satisfying $\dim D_t=n-1$ for all $t\in \supp_m D$.   Since $\D'$ is compatible with the images under  $\pi$ of the $(m,\alpha)$-flat sets that cover the cells of $\E$, the sets  $\Gamma_{\zeta_{i,D}}$ and  $\Gamma_{\zeta_{i+1,D}}$ must be
	$(m,\alpha)$-flat, and since
	$$\delta C_t \subset \left(\Gamma_{\zeta_{i,D}}\right)_t \cup \left(\Gamma_{\zeta_{i+1,D}}\right)_t \cup \pi^{-1}(\delta D_t), $$
	we see that the needed fact follows from the induction hypothesis.
\end{proof}

\begin{rem}\label{rem_borne_ligne_reguliere_L_reg_dec}
	We have proved a stronger statement: the distance between the regular vector $\lambda_{j(C)}$ and the tangent spaces to $\delta C_t$ can be bounded below away from zero by a positive number depending only on $n$, and not on the $A_i$'s. This is due to the fact that in the above proof we apply Lemma \ref{lem  kurdyka proj fixees} with $\nu=l \le 2n$.
\end{rem}

This proposition will be useful to compare the  inner and outer metrics of globally subanalytic sets.

\begin{dfn}\label{dfn_la_regular_cell}
	We define the {\bf Lipschitz cells of $\R^n$}\index{Lipschitz! cell} by induction on $n$. Let $E$ be a cell of $\R^n$ and denote by  $D$ its basis. 
	If $n=0$ then $E$ is always a Lipschitz cell.  If $n\ge 1$, we say that $E$ is a Lipschitz cell if so is $D$  and if in addition one of the following properties holds:
	\begin{enumerate}[(i)]
		\item  $E$ is the graph of some Lipschitz  definable function $\xi :D \to \R$. 
		\item  $E$ is a band $(\xi_1,\xi_2)$, $\xi_1<\xi_2$, where  $\xi_1$ is either   $ -\infty$  or a Lipschitz definable function on $D$, and $\xi_2$ is either $ +\infty$ or  a Lipschitz definable function on $D$.
	\end{enumerate}
	
\end{dfn}

\begin{lem}\label{lem_Lreg}
	If $E$ is a Lipschitz cell of $\R^n$ then the outer and inner metrics of $E$ are equivalent,  i.e., there is a constant $C$ such that for all $x$ and $y$ in $E$:
	\begin{equation*}\label{eq_equivalence_metrics}
		|x-y|\le d_E(x,y)\le C|x-y|.
	\end{equation*}
	Moreover, the constant $C$ just depends on $n$ and on the Lipschitz constants of the functions defining $E$. The same holds true for $cl(E)$.
\end{lem}
\begin{proof}
The definition of a Lipschitz cell being inductive, this lemma can be proved by induction on $n$. It is an easy exercise to construct an arc joining two given points using the functions defining the cell, and to estimate its length in terms of the Lipschitz constants of these functions. 
\end{proof}

This leads us to the following ``Lipschitz cell decomposition Theorem''.

\begin{thm}\label{thm_lipschitz_cells}
	Given $A_1,\dots ,A_l $ in $\smn$, there is a definable partition $\Pa$  of $\R^{m+n}$ compatible with $ A_1,\dots ,A_l$ and such that for each $V\in \Pa$ there is a linear isometry $\Lambda$ of $\R^n$ for which $\bigcup_{\tim}\{t\}\times \Lambda(V_t)$ is a cell and $\Lambda(V_t)$ is a Lipschitz cell for every $\tim$. Moreover:
	\begin{enumerate}[(i)]
		\item\label{item_csts}  The Lipschitz constants of the functions defining the cells $\Lambda(V_t)$ can be bounded by a function of $n$ (thus independent of $t$ and of the $A_i$'s).
		\item\label{item_system} The corresponding linear isometry $\Lambda$ can be chosen among a finite family that only depends on $n$ (and not on the $A_i$'s).
	\end{enumerate}

\end{thm}
\begin{proof}
	We prove this proposition by induction on $n$ (for $n=0$ the result is trivial). Taking a cell decomposition of $\R^{m+n}$ if necessary, we may assume that the family  $A_1,\dots ,A_l $ is reduced to one single set $A$ which is a cell. 
	
	We start with the case where $A_t$ has empty interior for all $t\in \supp_m A$.  Fix $\alpha>0$ sufficiently small for Lemma \ref{lem  kurdyka proj fixees} to  hold with $\nu=1$. By Lemma \ref{lem_partition_tangents},  $A$ can be covered by finitely many $(m,\alpha)$-flat sets.  Take a definable partition of $A$ compatible with the elements of this covering, and let $E$ be an element of this partition.  By Lemmas \ref{lem  kurdyka proj fixees} (with $\nu=1$) and \ref{lem_reg_inclusion},  the family $(E_t)_\tim$ has a regular vector (which, up to a linear isometry, can be assumed to be $e_n$). Take a cell decomposition compatible with $E$ and let $C$ be a cell included in $E$.
	The set  $C$ is thus the   graph (for $e_n$)  of some  function, say  $\xi:B \to \R$, and the derivative of $\xi_{t}$ must be bounded independently of $\tim$.  By induction, we know that we can cover $B$ with some sets $W_1,\dots,W_l$ such that for all $\tim$ and each $i$, $W_{i,t}$ is a Lipschitz cell after a possible linear isometry. By Proposition \ref{pro_bounded_der_lipschitz} and  Lemma \ref{lem_Lreg},  $\xi_t$ induces a Lipschitz function on each $W_{i,t}$.

	We now carry out the induction step in the case where $\dim A_t=n$ for all $t\in \supp_m A$. By Proposition \ref{pro_dec_L_regulier_avec_proj_fixees}, splitting our set $A$ into several sets if necessary, we may assume that  $ (\delta A_t)_\tim, $ has a regular vector (which again, up to a linear isometry, can be supposed to be $e_n$).  Take a cell decomposition $\C$ of $\R^{m+n}$ compatible with $\bigcup_\tim \{t\}\times \delta A_t$. If $C\in\C$ is  such that $C_t\subset \delta A_t $  for all $\tim$ then $C_t$ is for each $t\in \supp_m C$ the graph of some  function $\xi_t$, $t\in \supp_m C$, that has bounded gradient (independently of $t$).
	By the same argument as in the case $\dim A_t<n$, the family $\xi_t$ induces a  family of Lipschitz functions on the elements of a partition of the basis of $C$ (into sets which are cells up to a linear isometry) obtained from the  induction hypothesis. As cells are connected, a cell decomposition of $\R^n$ which is compatible with $\delta A_t$  is compatible with $A_t$.
	
	That the  Lipschitz constants of the functions defining these Lipschitz cells can be bounded by a function of $n$ comes down from Remark \ref{rem_borne_ligne_reguliere_L_reg_dec} (and Lemma \ref{lem  kurdyka proj fixees} in the case $\dim A_t<n$) together with Lemma \ref{lem_Lreg}. Assertion $(ii)$ follows from the fact that in Proposition \ref{pro_dec_L_regulier_avec_proj_fixees}, the family $\lambda_1,\dots, \lambda_N$ is independent of the $A_i$'s.
\end{proof}

Thanks to Lemma \ref{lem_Lreg}, the above theorem has the following immediate consequence.

\begin{cor}\label{cor_inner_outer}
Every set $A\in \s_{m+n}$ admits a definable partition $\Pa$ such that for every $V\in \Pa$ and $\tim$ the inner and outer metrics of $V_t$ are equivalent. The constants of these equivalences can be bounded by a function of $n$.
\end{cor}

Since every definable function is piecewise continuous, thanks to Proposition \ref{pro_bounded_der_lipschitz}, we then derive: 
\begin{cor}\label{cor_bounded_der_lips}
	Let $\xi_t:A_t\to \R$, $\tim$, $A\in \s_{m+n}$, be a  definable family of  functions, and set 
	$K_{\xi_t}:=\sup_{x\in reg(\xi_t)} |\pa \xi_t(x)|,  $ for each $\tim$.
	 There are a constant $C$ (depending only on $n$) and a definable partition $\Pa$ of $A$ such that $ \xi_t$ is $C \cdot K_{\xi_t}$-Lipschitz  on  $V_t$,  if $V$ is an element of $\Pa$ and $\tim$ is such that	$K_{\xi_t}<+\infty$.    
\end{cor}

The following corollary unravels the close interplay between Lipschitz functions and regular vectors.

\begin{cor}\label{cor_proj_reg_decomposition_en_graphes}
	The vector $\lambda \in \sph^{n-1}$ is regular for the set $A \in \smn$  if and only if there are finitely many uniformly Lipschitz definable families of functions $\xi_{i,t}: B_{i,t} \to  \R$, $\tim$, $i=1,\dots,p$, with $B_i \subset \R^m \times N_\lambda$,  such that for all $\tim$:
	$$A_t = \cup_{i=1}^p\;\Gamma_{\xi_{i,t}}^\lambda.$$
\end{cor}
\begin{proof}
	As the ``if'' part is clear, we will focus on the converse. Up to a linear isometry we can assume that $\lambda=e_n$. Take a cell decomposition compatible with $A$ and let $C$ be a cell included in $A$. This cell cannot be a band since $e_n$ is regular for $A$ (see Lemma \ref{lem_reg_inclusion}). It is thus the graph of a function $\xi:D \to \R$, with $D \in \St_{m+n-1}$. For every $t \in \R^m$, the function $\xi_t:D_t \to \R$  has bounded first  derivative (with a bound independent of $t$).   By Corollary \ref{cor_bounded_der_lips}, there is a definable partition $\Pa$ of $D$, such that the family $\xi_{t|V_t}$ is uniformly Lipschitz for every $V\in \Pa$.
\end{proof}

\begin{rem} \label{rem_famille_lipschitz_ordonnees}
	By Proposition \ref{pro_extension_fonction_lipschitz} (see Remark \ref{rem_extension_familles_fonctions_lipschitz}), we may extend the $\xi_{i,t}$'s to  $N_\lambda$. Using the $\min$  operator, it is then 
	not difficult to show (see \cite[Proposition $1.4.7$]{gvhandbook}) that we can assume that these extensions satisfy $\xi_{1,t}\le \dots \le \xi_{p,t}$ on $N_\lambda$ (and $A_t \subset \cup_{i=1}^p\;\Gamma_{\xi_{i,t}}^\lambda$).
\end{rem}

\subsection{Comparing the inner and outer metrics}\label{sect_comparing} 
 We define  the {\bf diameter}\index{diameter} of a set $X\subset \R^n$ by 
\nomenclature[bupar]{$\mbox{diam}(X)$}{diameter of $X$ \nomrefpage}
\begin{equation}\label{eq_diameter_def}\mbox{diam}(X):=\sup \{ |x-y|:x\in X,y\in X\}.\end{equation}

\begin{pro}\label{pro_diametre_inner}
	Given $A\in \smn$, there is a constant $C$ such that for all $\tim$, we have for each $x$ and $y$ in the same connected component of $A_t$:
	\begin{equation}\label{eq_diameter}d_{A_t}(x,y)\le C \mbox{diam}(A_t).\end{equation}
\end{pro}
\begin{proof}
	Let $U_1,\dots, U_k$ be the elements of the partition  of $A$ provided by
	Corollary \ref{cor_inner_outer}, and set $V_{i,t}:=cl(U_{i,t})\cap A_t$ for $i\le k$ and $\tim$. We claim that the inner and outer metrics of the $V_{i,t}$'s are equivalent, and that the constants of these equivalences are bounded independently of $t$. Indeed, given $x_1$ and $x_2$ in $V_{i,t}$, with $\tim$ and $i\le k$, by Curve Selection Lemma (Lemma \ref{curve_selection_lemma}) we can find two $\ccc^0$ definable arcs $\gamma_1:(0,\ep)\to U_{i,t}$ and $\gamma_2:(0,\ep)\to U_{i,t}$ tending at $0$ to $x_1$ and $x_2$ respectively. Observe that Corollary \ref{cor_inner_outer} entails that for $s>0$ small
	$$d_{V_{i,t}}(\gamma_1(s),\gamma_2(s))\le d_{U_{i,t}}(\gamma_1(s),\gamma_2(s))\lesssim  |\gamma_1(s)-\gamma_2(s)| ,$$ with a constant  independent of $s$ and $t$ (since this constant is given by the latter corollary). Hence, for such $s$
	$$d_{V_{i,t}}(x_1,x_2) \lesssim d_{V_{i,t}}(x_1,\gamma_1(s))+ |\gamma_1(s)-\gamma_2(s)|+d_{V_{i,t}}(\gamma_2(s),x_2),$$
 Making $s\to 0$ then yields $d_{V_{i,t}}(x_1,x_2) \lesssim |x_1-x_2|$, establishing the claimed fact.

	We proceed by induction on $k$.
	 If $k=1$ then the result immediately follows from  the fact that the inner and outer metrics of $V_{1,t}$ are equivalent. Set $B^j_t:=\cup_{i=1}^j V_{i,t}$, for $j\le k$, and assume that the desired fact holds for  $B^{k-1}_t$.
	
	 For $\tim$, let $E_t$ denote the set of all the couples $(x,y)\in B_t^{k-1}\times V_{k,t}$ of elements that are in the same connected component of $B^k_t$.
	Thanks to our induction hypothesis, it suffices to show that 
	$$ \sup \{d_{B^k_t}(x,y):(x,y)\in E_t,\tim\}<+\infty.$$
	Let $t\in \supp_m A$ and $(x,y)\in E_t$, with $x\ne y$, as well as $\ep>0$. Take a continuous definable arc $\gamma:[0,1]\to B^k_t$ joining $x$ and $y$ such that $lg(\gamma)\le d_{B^k_t}(x,y)+\ep$,  and set $s_0:=\sup \{s:\gamma([0,s])\subset B_t^{k-1}\}$ as well as $z:=\gamma(s_0)$.  The arc $\gamma_1:=\gamma_{|[0,s_0]}$ joins $x$ and $z$, while the arc  $\gamma_2:=\gamma_{|[s_0,1]}$ joins $z$ and $y$. But since we obviously have $$lg(\gamma_1)\le  d_{B_t^{k-1}}(x,z)+\ep\;\; \eto\;\; lg(\gamma_2)\le  d_{V_{k,t}}(z,y)+\ep ,$$
	the required fact follows from the induction hypothesis together with the fact that the inner and outer metrics of $V_{k,t}$ are equivalent (since $\ep$ can be taken arbitrarily small).
\end{proof}

The inner metric $A\times A\ni (x,y) \mapsto d_A(x,y)$   is not necessarily a  definable function. We however have:
\begin{lem}\label{lem_equiv_inner_defin}
	Given a definable connected set $A$, there is a continuous definable function $\rho:A\times A\to[0,+\infty)$ such that $d_A(x,y)\sim \rho(x,y)$ on $A\times A$.
\end{lem}
\begin{proof} By Corollary \ref{cor_inner_outer}, 
	we can cover $A$ with finitely many definable connected sets, say $V_1,\dots, V_k$, such that the inner and outer metrics of each $V_j$ are equivalent. Possibly replacing these sets with their closure in $A$, we can assume that they are closed in $A$.  It is then easy to check that 
	for $x$ and $y$ in $A$  we have
		$d_A(x,y)\sim \rho(x,y):= \inf \;\sum_{i=1}^{l-1} |x_i-x_{i+1}|$,
	where the infimum is taken over all the families $x_1,\dots,x_l$, $l\le k$, of points of $A$ such that $x_1=x$, $x_l=y$, and such that for each $i<l$, the points $x_i$ and $x_{i+1}$ both lie in $V_{j_i}$, for some $j_i\le k$.
	%
\end{proof}

  \begin{dfn}\label{dfn_normal}
	We say that $A\in \s_n$ is {\bf connected at $x\in cl( A)$}\index{connected at $x$} if  $\bou(x,\ep)\cap A$ is connected for all  $\ep>0$ small enough. 	We will say that it is {\bf connected along $Z\subset cl(A)$} if it is connected at each point of $Z$.
	We say that $A$ is {\bf normal}\index{normal} if it is connected at each $x\in cl( A)$.
\end{dfn}

\begin{pro}\label{pro_holder_normal}
	If  $A\in \s_n$ is normal, connected, and bounded then  there are a constant $C$ and a positive rational number $\theta$ such that
	\begin{equation*}
		|x-y| \le d_A (x,y)\le C|x-y|^\theta.
	\end{equation*}
\end{pro}
\begin{proof} 	Let $\rho$ be provided by applying Lemma \ref{lem_equiv_inner_defin} to $A$.  By (\ref{eq_diameter}), $ d_A(x,y)$ is bounded on $A\times A$, and hence,  so is $\rho$.
	The desired fact follows by applying \L ojasiewicz's inequality (Theorem \ref{thm_lojasiewicz_inequality}) to $\rho(x,y)$ and $|x-y|$, which nevertheless requires to  check that $\rho(x(t),y(t))$ tends to zero as $t\to 0^+$, for each couple of definable arcs $x:(0,\ep)\to A$ and $y:(0,\ep)\to A$ tending to the same point $z\in cl(A)$ (as $A$ is bounded, definable arcs in $A$ must have  endpoints in $cl(A)$).
		We can assume   $x(t)$ and $y(t)$ to be parameterized by their distance to $z$.   If $B_t:=\sph(z,t)\cap A$ is not connected (for $t>0$ small) then,  as  $A$ is connected at $z$ (since it is normal), $z\in A$ (this fact  can be deduced from the local conic structure theorem which is proved in section \ref{sect_lcs} independently), in which case the needed fact is clear. Otherwise, if $B_t$ is connected, we can write  $$\rho(x(t),y(t))\lesssim d_A(x(t),y(t)) \le d_{B_t}(x(t),y(t)) \overset{(\ref{eq_diameter})}\le Ct,$$
	 which yields that $\rho(x(t),y(t))$ tends to zero as $t$ goes to zero.   
\end{proof}

	\begin{rem} The assumption ``$A$ normal'' is clearly necessary and
	it is easy to see that the proposition is no longer true if one drops the boundedness assumption.
\end{rem}

In \cite{birmos}, a related theorem ensures that every semialgebraic set can be {\it normally embedded} (their proof
goes over to the globally subanalytic category), which means  that, given a semialgebraic set $A$, we can construct a definable homeomorphism (onto its image) $h:A \to \R^k$,  bi-Lipschitz with respect to the inner metric and such that the inner and outer metrics of $h(A)$ are equivalent.

 \section{Metric triangulations}\label{sect_metric triangulations}
\paragraph{Some preliminary definitions.} A {\bf simplex}\index{simplex} $\sigma$ of $\R^n$ of dimension $k$ is the convex hull of  $(k+1)$ affine independent (i.e., not contained in an affine space of dimension $(k-1)$) elements $u_0,\dots,u_k$ of $\R^n$.
We say that $\{u_0,\dots,u_k\}$ {\bf generates} $\sigma$.

A {\bf face}\index{face} of $\sigma$ is then a simplex $\sigma'$ generated by a subset of $\{u_0,\dots,u_k\}$. A face $\tau$ of $\sigma$  satisfying $\dim \tau<\dim \sigma$ is called a {\bf proper face}\index{face!proper}.
An {\bf open simplex}\index{simplex!open} is a simplex from which the proper faces have been deleted.

A {\bf simplicial complex}\index{simplicial complex} $K$ of $\R^n$ is a {\it finite} set of open simplices of $\R^n$ such that for any $\sigma_1$ and $\sigma_2$ in $K$, the set $cl(\sigma_1) \cap cl(\sigma_2)$ is a common face of $cl(\sigma_1)$ and $cl(\sigma_2)$\footnote{It should be emphasized that our notion of simplicial complex is slightly different from the usual one since we take {\it open} simplices.}. We denote by $|K|$ the {\bf polyhedron}\index{polyhedron}
\nomenclature[bupaz]{$\vert K \vert$}{polyhedron of the simplicial complex $K$\nomrefpage} of $K$ which is the set constituted by the union of all the elements of $K$.

A {\bf triangulation of a set $X\in \s_n$}\index{triangulation} is a globally subanalytic homeomorphism $\Psi:|K|\to X$, with $K$ simplicial complex of $\R^n$. 

\paragraph{Basic idea.}We are going to
define a notion of triangulation adapted to the study of Lipschitz
geometry. A {\it metric triangulation} will be a homeomorphism from a simplicial complex onto the given set, which means that
it will be a triangulation in the usual meaning. Of course, the distances are modified by such a  homeomorphism.  We shall require that over
each simplex the distances are preserved up to ``some contractions"
which will be  characterized by finitely many iterations of sums, products, and powers of  distance functions  to the faces. Such 
functions will be called  {\it standard simplicial functions}.
Indeed, Definition \ref{triangulation} will involve standard
simplicial functions on $\sigma \times \sigma$ depending on two variables $q$ and $q'$ in $\sigma$. These functions are such combinations
of distance functions  $d(q,\tau)$ and  $d(q',\tau)$, $\tau$ proper face of $\sigma$.

What will matter is that two sets having the same metric triangulation (with the same coordinate systems and
equivalent contraction functions) will be definably
bi-Lip\-schitz homeomorphic (Proposition \ref{pro_meme_triangulation}).

\paragraph{Definition of metric triangulations and the main theorem.}
 It will not be possible to express the contractions along the directions of the canonical basis. The reason is that this kind of mappings alter not only distances, but also angles.  Hence, we first introduce the concept of tame system of coordinates along which the contractions will apply.


\begin{dfn}\label{pl coordinates} Let $\sigma $ be any open simplex of $\R^n$ and denote by $\pi_i : \R^n \rightarrow \R^i$  the
	projection onto the $i$ first coordinates. 
		A {\bf tame system of coordinates of $\sigma$}\index{tame system of coordinates} is a mapping
	$\mu:
	\sigma \to [0,1]^n$, $\mu=(\mu_{1}, \dots, \mu_{n})$,  which is a homeomorphism onto its image  of the following form for $q=(q_1,\dots,q_n)\in \sigma$ and $i\in \{1,\dots,n\}$:
	\begin{equation}\label{eq def pl coord}\mu_{i}(q)=\begin{cases}
			\frac{q_{i}-
				\zeta_i(\pi_{i-1}(q))}{\zeta_i'(\pi_{i-1}(q))-\zeta_i(\pi_{i-1}(q))} &\mbox{if $\zeta_{i} < \zeta_i'$}\\
			0 & \mbox{whenever } \zeta_{i}=\zeta'_i \mbox{ on } \pi_{i-1}(\sigma),
		\end{cases}
	\end{equation}
	where $\zeta_i$ and $\zeta'_i$ are
	linear functions on $\pi_{i-1}(\sigma)$ satisfying for all $q\in \sigma$ $$\zeta_{i}( \pi_{i-1}(q))\le q_i\le \zeta_i'(\pi_{i-1}(q)).$$
	
	A  \textbf{standard simplicial function on $\sigma$}\index{standard simplicial function}  is a function $\varphi(q)$ given by the quotient of  finite sums of monomials of type
	\begin{equation}\label{eq_monomials}d(q,\tau_1)^{\alpha_1} \cdots d(q,\tau_k)^{\alpha_k} ,\end{equation}
	with $\alpha_1,\dots, \alpha_k$ positive rational numbers and $\tau_1,\dots, \tau_k$ faces of $\sigma$. A {\bf standard simplicial  function on $\sigma\times \sigma$} is a function $\varphi(q,q')$ given by the quotient  of  finite sums   of monomials as in (\ref{eq_monomials}) but involving both functions of type $\sigma \ni q \mapsto d(q,\tau)$ and of type $\sigma \ni q' \mapsto  d(q',\tau)$, $\tau$ proper face of $\sigma$ (standard simplicial  functions on $\sigma\times \sigma$ are actually just needed if we triangulate unbounded sets, see Remark \ref{rem sur les triangulations}).
	\end{dfn}

\begin{dfn}\label{triangulation}
	Let $X\in \s_n$. A \textbf{metric triangulation of $X$}\index{metric triangulation! of a set} is  a
	triangulation $\Psi
	:|K| \rightarrow X$ of $X$ such that
	for every $\sigma \in K$ there exist
	$\varphi_{\sigma,1} \: ,  \dots , \varphi_{\sigma,n}$, standard
	simplicial functions on $\sigma \times \sigma$ satisfying for
	 $q$ and $q'$ in $\sigma$:
	\begin{equation}\label{eq h dist ds triang}
		|\Psi(q)-\Psi(q')| \sim \sum_{i=1}^{n}\varphi_{\sigma,i} (q,q')
		\cdot |\mu_{\sigma,i}(q) - \mu_{\sigma,i}(q')|,
	\end{equation}
	with $\mu_\sigma:=(\mu_{\sigma,1}, \dots, \mu_{\sigma,n})$  tame
	system of coordinates of $\sigma$. 
		The functions  $\varphi_{i,\sigma}$ are then called {\bf the contraction functions}\index{contraction function} of the metric triangulation $\Psi$.
	
	Given some subsets $A_1, \dots ,A_\kappa$ of $X$, we say that a metric triangulation $\Psi:|K|\to X$ is {\bf compatible with} $A_1, \dots ,A_\kappa$\index{metric triangulation! compatible with}\index{compatible!metric triangulation}  if each $\Psi^{-1} (A_i)$ is the  union of some elements of $K$.
\end{dfn}

Let us make more precise  the extent to which the metric structure of the set is captured by the triangulation by stating the following immediate proposition which, roughly speaking, yields that two subsets having the same metric triangulation are definably bi-Lipschitz homeomorphic:

\begin{pro}\label{pro_meme_triangulation}
	Let $X_1$ and $X_2$ be elements of $\s_n$ and let $\Psi_1:|K| \to X_1$ and $\Psi_2:|K|\to X_2$ respectively be metric triangulations of $X_1$ and $X_2$. Assume that these metric triangulations involve the same contraction functions   $\varphi_{i,\sigma}$ and the same tame systems of coordinates $(\mu_{\sigma,1},\dots,\mu_{\sigma,n})$ on every simplex $\sigma \in K$. Then $\Psi_2\Psi^{-1}_1$ is a definable bi-Lipschitz homeomorphism from $X_1$ onto  $X_2$.
\end{pro}
\begin{proof}
	This follows from (\ref{eq h dist ds triang}) for both $\Psi_1$ and $\Psi_2$.
\end{proof}

We are going to prove that every globally subanalytic set  admits a  metric triangulation.  We shall establish it not only  for a definable set, but for a definable family, up to a partition of the parameter space. This will yield  that definable bi-Lipschitz triviality holds almost everywhere (Corollary \ref{cor_hardt}).

\begin{thm}\label{thm existence des triangulations}
Given $A_1,\dots,A_\kappa$ in $\s_{m+n}$, there is a definable partition $\Pa$ of $\R^m$ such that for each $B\in \Pa$ there are a simplicial complex $K$ of $\R^n$ and a definable family of metric triangulations $\Psi_t: |K| \to \R^n$ (of $\R^n$) compatible with $A_{1,t},\dots, A_{\kappa,t}$ for each $t\in B$. The contraction functions  and the tame systems of coordinates of $\Psi_t$  are independent of $t\in B$.
\end{thm}

The proof of this theorem requires two preliminary lemmas that will also be used in the proof of the local version displayed in Theorem \ref{thm_triangulations_locales}.



\begin{lem}\label{lem_eq_dist_cor_th_prep}
	Let $\xi :  A \rightarrow \R$ be a definable nonnegative function, $A\in \s_{m+n}$.
	There exist some definable subsets $W_1,\dots,W_k$ of $cl(A)$  and a  definable  partition $\Pa$ of $A$ such that for any $V \in \Pa$ there are some rational numbers  $\alpha_1,\dots,\alpha_k$ such that for each $\tim$ we have on $V_t \subset \R^n$:  \begin{equation}\label{eq_equiv_dist}\xi_t(x) \sim d(x,W_{1,t})^{\alpha_1} \cdots d(x,W_{k,t})^{\alpha_k}   .
	\end{equation}
	\end{lem}

\begin{proof}
	We prove it by induction on $n$. For $n=0$, the result is clear since $\xi$ is a function of $t$ (the constants of the needed equivalence may depend of $t$).
		Take $n\geq 1 $ and assume that the proposition is true for $(n-1)$.  Denote by $\Lambda_1,\dots, \Lambda_\kappa$ the linear isometries involved in Theorem \ref{thm_lipschitz_cells} (see (\ref{item_system}) of this theorem). We will sometimes regard them as changes of coordinates of $\R^{n+m}$, preserving the $m$ first variables.
	
	For each $i\le \kappa$, applying the Preparation Theorem  to $\xi \circ \Lambda_i^{-1}:\Lambda_i(A) \to \mathbb{R}$,   we obtain a definable partition of   $\Lambda_i(A)$. The images of all the elements of this partition under the map $\Lambda_i^{-1}$ constitute (for each $i\le \kappa$) a partition of $A$, denoted by $\Pa_i$.
	Apply  Theorem \ref{thm_lipschitz_cells} to  the elements of a common refinement  of all the  $\Pa_i$'s and denote by $\Pa$ the resulting partition of $A$.
	
	Since it  suffices to show the result for the restriction of $\xi$ to each $C \in \Pa$, let us	fix   such $C$.  By  Theorem \ref{thm_lipschitz_cells}, there is $i\le \kappa$ such that  $\Lambda_i(C)$ is a cell and $\Lambda_i(C)_t$ is a Lipschitz cell for each $\tim$.
	 If $\Lambda_i(C)_t$ is the graph of some uniformly Lipschitz family of functions,  the needed estimate easily follows by induction on $n$. Otherwise, $\Lambda_i(C)=(\eta_1,\eta_2)$ with
	$\eta_1< \eta_2$   definable functions on the basis  of $\Lambda_i(C)$ such that $(\eta_{1,t})_{\tim}$ and  $(\eta_{2,t})_{\tim}$ are uniformly Lipschitz (or $\pm \infty$).

  As $\Pa$ is compatible with the $\Pa_j$'s, the function  $\xi\circ \Lambda_i^{-1}$ is reduced on $C$. Since we can argue up to a  linear isometry, we will from now identify $\Lambda_i$  
  with the identity.  
   Hence,
	there are $r \in \Q$ as well as some functions $a$ and $\theta$ on the basis $B$ of $C$ such that for $x=(\xt,x_n)\in C_t$, $\tim$, we have:$$\xi_t (x)\sim a_t(\xt)|x_n-\theta_t(\xt)|^r.$$    
	Thanks to the induction hypothesis, we thus only have to check the result for the function $|x_n-\theta_t(\xt)|$.  
	As $\Gamma_\theta\cap C=\emptyset$,
	we  can assume for every $(t,\xt)\in B$,   either
	$\theta_t(\xt) \leq \eta_{1,t}(\xt)$ or  $ \theta_t(\xt) \geq \eta_{2,t}(\xt)$.
	Suppose for instance that $\theta_t(\xt) \leq \eta_{1,t}(\xt)$, and 
	write for $t \in\supp_mC$ and $x=(\xt,x_n) \in C_{t}$: 
	$$x_n -
	\theta_t(\xt)=(x_n-\eta_{1,t}(\xt))+(\eta_{1,t}(\xt)-\theta_t(\xt)).$$
    There is a definable partition of $C$ such that the two terms of the sum appearing in the right-hand-side of this equality are comparable on every element $E$ of this partition.  As they are both nonnegative, we then see that
either	$|x_n - \theta_t(\xt)|\sim (x_n-\eta_{1,t}(\xt))$ or
	$|x_n - \theta_t(\xt)|\sim (\eta_{1,t}(\xt)-\theta_t(\xt))$ on $E_t$ for each $\tim$.
		In the second case, the desired result comes down from the induction
	hypothesis (since $(\eta_{1,t}(\xt)-\theta_t(\xt))$ is an $(n-1)$-variable function). Moreover, since
	$\eta_{1,t}$ is Lipschitz,  $|x_n -
	\eta_{1,t}(\xt)|\sim d(x,\Gamma_{\eta_{1,t}})$ on $E_t$. 
\end{proof}

\begin{rem}\label{rem_constant_t} The constants of the equivalence in the above lemma may depend on $t$. The  exponents  $\alpha_1,\dots,\alpha_k$ however just depend on $V\in \Pa$.
	\end{rem}

The second result that we shall need in the proof of Theorem \ref{thm existence des triangulations} is an elementary fact about families of functions that will help us to refine partitions. 

\begin{lem}\label{lem_graphes_en_plus}
	Given definable families of Lipschitz functions $c_{1,t},\dots,c_{k,t}$ on $\R^{n-1}$, $\tim$,
	we can find definable families of  Lipschitz functions
	$\xi_{1,t}\le \dots \le \xi_{l,t}$ on $\R^{n-1}$ and a cell decomposition $\D$ of
	$\R^{m+n-1}$ such that for every $D \in \D$,
	the collection of functions $$c_{1,t}(\xt),\dots, c_{k,t}(\xt),|x_{n}-c_{1,t}(\xt)|,\dots ,|x_n-c_{k,t}(\xt)|$$  is totally ordered
	on  $(\xi_{i,t|D_t},\xi_{i+1,t|D_t})$ for each $\tim$ and $i\in \{0,\dots,l\}$   (with $\xi_{0,t}\equiv -\infty$ and $\xi_{l+1,t}\equiv \infty$).
	%
\end{lem}
\begin{proof}
	Take a cell decomposition $\D$ of $\R^{m+n-1}$ compatible with the sets $Z_{ij}:=\{(t,z) \in \R^{m+n-1}: c_{i,t}(z)\le c_{j,t}(z)\}$. Let us complete the finite collection constituted by the $c_{i,t}$'s with the functions $(c_{i,t}+c_{j,t})$, $(c_{i,t}-c_{j,t})$,  and  $\frac{c_{i,t}+c_{j,t}}{2}$, $i,j \in \{1,\dots,k\}$ (which are also Lipschitz), and denote by $\xi_{1,t}, \dots , \xi_{l,t}$  the completed family. Using the $\min$  operator if necessary, it is then not difficult to see that we can assume  $\xi_{1,t}\le \dots \le \xi_{l,t}$ (see \cite[Proposition 1.4.7]{gvhandbook}).
	
	To check that it has the required property fix a cell $D$ of $\D$ and observe that the compatibility of $\D$ with the  $Z_{ij}$'s entails that the $c_{i,t}$'s  are comparable with each other on $E_{p,t}:=(\xi_{p,t|D_t},\xi_{p+1,t|D_t})$, for all $p$. Moreover, since the graphs of the $c_{i,t}$'s are included in the union of the graphs of the $\xi_{i,t}$'s, the function $(x_n-c_{i,t}(\xt))$ has constant sign on  $E_{p,t}$, for each $i$ and each $p$. If, for instance,
	$(x_n-c_{i,t}(\xt))>0$ and $(x_n-c_{j,t}(\xt))<0$ on $E_{p,t}$, then since
	$$ (x_n-c_{i,t}(\xt)) -(c_{j,t}(\xt)-x_n)=2\left(x_n-\frac{c_{i,t}+c_{j,t}}{2}\right),  $$
	we see that the inclusion $\Gamma_{\frac{c_{i,t}+c_{j,t}}{2}}\subset \bigcup_{\iota=1}^l \Gamma_{\xi_{\iota,t}}$ entails that $|x_{n}-c_{i,t}(\xt)|$ and $|x_{n}-c_{j,t}(\xt)|$  are comparable with each other on $E_{p,t}$, for all $p$. The  inclusion of the graphs of the functions $(c_{i,t}+c_{j,t})$ and $(c_{i,t}-c_{j,t})$ in $\bigcup_{\iota=1}^l \Gamma_{\xi_{\iota,t}}$ can  be used analogously to prove that $c_{i,t}$ is comparable with  $(x_n-c_{j,t}(\xt))$ on  $E_{p,t}$ for all $p,i,$ and $j$.
\end{proof}
\begin{rem}
	The proof has established that if the $c_{i,t}$'s are $L$-Lipschitz for some $\tim$ then $\xi_{p,t}$ is $2L$-Lipschitz (for this value of $t$) for all $p\le l$ (since we only complete the collection with the families $(c_{i,t}+c_{j,t})$, $(c_{i,t}-c_{j,t})$,  and  $\frac{c_{i,t}+c_{j,t}}{2}$).
\end{rem}


\begin{proof}[proof of Theorem \ref{thm existence des triangulations}]
	We shall actually prove by induction on $n$ the following stronger statements:

	\noindent{$\mathbf{(H_n).}\;$} Let $A_1,\dots, A_\kappa$ in $\s_{m+n}$ and let $\eta_1,\dots,\eta_l$ be some  definable nonnegative functions on  $ \R^{m+n}$. There is a definable partition $\Pa$ of $\R^m$ such that for each $B\in \Pa$, there is a definable family  of
 metric triangulations $\Psi_t: |K| \to \R^n$, $t\in B$, of $\R^n$ compatible with   $A_{1,t},\dots, A_{\kappa,t}$ and
	such  that for each $\sigma\in K$, $t\in B$,  and   $i\le l$,
	the function  $\eta_{i,t} \circ \Psi_{t|\sigma}$ is $\sim$ to a  standard simplicial  function $v_{i,\sigma}:\sigma \to \R$. Moreover, the contraction functions and the tame systems of coordinates  of $\Psi_t$, as well as the functions $v_{i,\sigma}$,  are independent of $t\in B$.
	
	For $n=0$ the result is
	clear.  Take $n\in \N^*$, assume $\mathbf{(H_{n-1})}$, and  fix $A_1,\dots,A_\kappa$ in $\s_{m+n}$ as well as some  definable nonnegative functions $\eta_1,\dots,\eta_l$ on $\R^{m+n}$.
	
	 We denote by $\pi:  \R^{n}
	\rightarrow  \R^{n-1}$ the projection omitting the last coordinate.  Below, we sometimes regard $\pi$ as a projection $\pi:\R^{m+n} \to \R^{m+n-1}$ (still omitting the last coordinate). We also sometimes take for granted that a  family of functions $\xi_t:\R^{n-1} \to \R$, $\tim$, gives rise to a map $\xi:\R^{m+n-1} \to \R^m \times \R$, $(t,x)\mapsto (t,\xi_t(x))$.

	By Lemma \ref{lem_eq_dist_cor_th_prep},  there is a definable partition $\Pa$ of $ \R^{m+n}$ and some definable subsets  $W_1,\dots, W_k$ of $ \R^{m+n}$ such that for every $C\in \Pa$ and each $i\le l$, we can find   some rational numbers $r_1,\dots,r_k$ such that for  $x\in C_t$ (for each $\tim$):
	\begin{equation}\label{eq_proof_trian_xi}
		\eta_{i,t}(x)\sim  d(x, W_{1,t})^{r_1}\cdots d(x,W_{k,t})^{r_k} .
	\end{equation}
		By Theorem \ref{thm_proj_reg_hom_pres_familles},  Corollary \ref{cor_proj_reg_decomposition_en_graphes}, and Remark \ref{rem_extension_familles_fonctions_lipschitz}, there is a uniformly  bi-Lipschitz definable family of  homeomorphisms and  definable families of Lipschitz functions $\theta_{1,t} , \dots , \theta_{\lambda,t}$  defined on $\R^{n-1}$ such that:  the respective images  of the $\delta A_{i,t}$'s, the
	$\delta W_{i,t}$'s,   and the  $\delta C_t$ for $C \in \Pa$, are contained in the  union  of the graphs of   $\theta_{1,t} , \dots , \theta_{\lambda,t}$. For simplicity,   we will identify this family of homeomorphisms with the identity.

	By Lemma \ref{lem_graphes_en_plus} (applied to the $\theta_{j,t}$'s and to
	the families of $(n-1)$-variable functions $\R^{n-1} \ni\xt \mapsto d(\xt,\pi(\delta W_{i,t}\cap \Gamma_{\theta_{j,t}}))$),  there exist a cell decomposition $\D$ of $\R^{m+n-1}$
	and finitely many Lipschitz functions $\xi_{1,t}\leq \dots \leq \xi_{N,t}$ such that for every $D \in
	\D$ and  $\tim$,   all the functions $|x_{n}-\theta_{j,t}(\xt)|$, $x=(\xt,x_n)\in \R^{n-1}\times \R$,
	are comparable with each other and comparable with the functions
	$\R^{n-1} \ni \xt \mapsto d(\xt,\pi(\delta W_{i,t}\cap \Gamma_{\theta_{j,t}}))$, $i\le k$, $j\le \lambda$, on the set $(\xi_{\iota,t|D_t},\xi_{\iota+1,t|D_t})$ for all $\iota\le N$. Adding some graphs if necessary, we can assume that $\bigcup_{i=1}^N \Gamma_{\xi_i} \supset \bigcup_{i=1}^\lambda \Gamma_{\theta_i} $.

	Consider a cell decomposition $\C$ of $\R^{m+n}$ compatible with the $\Gamma_{\xi_i}$'s, the $A_{i,t}$'s, and the $W_{i,t}$'s. 
	Take then a common refinement $\E$ of $\pi(\C)$ and $\D$ 
	such that all the functions     $(\xi_{i}-\theta_{j})$ have constant sign (in $\{-1,0,1\}$, see Definition \ref{dfn_comparable}) on every cell and 
	apply $\mathbf{(H_{n-1})}$ to the cells of $\E\subset \s_{m+(n-1)}$ to get a partition $\Pa$. Fix $B$ in $\Pa$ and $t\in B$ (we argue for one single fixed $t$, the proof of the definability of the  family of triangulations   is easy and will be left to the reader); there is (for this fixed value of $t$) a  metric
	triangulation $(K,\Psi_t)$ of $\R^{n-1}$ compatible  with all the sets $ E_t$, $E\in \E$.

	Let $\zeta_1 \leq \dots
	\leq \zeta_N$ be piecewise linear functions over $|K|$ such that for every $i$,
	$\zeta_i \equiv \zeta_{i+1}$ on the set  $\{\xi_{i,t}\circ \Psi_t=\xi_{i+1,t} \circ
	\Psi_t\}$ (this set is a subcomplex of $K$ since $\C$ is compatible with the $\Gamma_{\xi_i}$'s). Let also $\zeta_0:= \zeta_1-1$ and $\zeta_{N+1}:=\zeta_N+1$, as well as
	$$ Z:= \{ (p,y) \in |K| \times \R \; : \;
	\zeta_{0}(p) < y < \zeta_{N+1}(p) \}.$$
		We obtain a polyhedral decomposition of $Z$ by taking the respective inverse
	images by $\pi_{|Z}$ of the simplices of $K$  as well as  all the images of the simplices of $|K|$ by the
	mappings $p \mapsto (p,\zeta_i(p))$, $1\le i \le N$. After a
	barycentric subdivision of this polyhedron, we get a simplicial
	complex $\widehat{K}$ satisfying $|\widehat{K}|=Z$.

	Define now  the desired  homeomorphism
	$\widehat{\Psi}_t$ in the following way:
	$$\widehat{\Psi}_t(p,\nu \, \zeta_{i}(p)+(1-\nu)\zeta_{i+1}(p)):=
	(\Psi_t(p),\nu\, \xi_{i,t}\circ \Psi_t(p)+(1-\nu)\xi_{i+1,t}\circ \Psi_t(p))$$
	for $ 1 \leq i \leq N-1$, $p \in |K|$ and $\nu \in [0,1]$.
	For $p\in |K|$ and $\nu \in [0,1)$, set:
	$$\widehat{\Psi}_t(p,\nu \, \zeta_{0}(p)+(1-\nu)\, \zeta_{1}(p)):=
	(\Psi_t(p),\xi_{1,t}(\Psi_t(p))-\frac{\nu}{1-\nu})$$ as well as
	$$\widehat{\Psi}_t(p,\nu \, \zeta_{N+1}(p)+(1-\nu)\, \zeta_{N}(p)):=
	(\Psi_t(p),\xi_{N,t}(\Psi_t(p))+\frac{\nu}{1-\nu}).$$
	This clearly defines a  homeomorphism $\widehat{\Psi}_t:
	|\widehat{K}|\rightarrow  \R^{n}$.
		By construction, the $A_{j,t}$'s and the $W_{i,t}$'s are images of open simplices.

		We shall check that, over each simplex $\sigma \in \widehat{K}$, the mapping
	$\widehat{\Psi}_t$ fulfills (\ref{eq h dist ds triang}) (for some functions $\varphi_{\sigma,i}$ and some tame systems of coordinates that we shall introduce). 
		Let $\sigma\in \widehat{K}$ and let $\tau$ be the simplex of $K$ containing $\pi(\sigma)$.

	Thanks to the induction hypothesis, we can find some functions
	$\varphi_{\tau,1},\dots,\varphi_{\tau,n-1} $ and a tame system
	of coordinates $(\mu_{\tau,1},\dots,\mu_{\tau,n-1})$ such that for
	any $p$ and $p'$ in $\tau$:
	\begin{equation}\label{eq pr sigma prime}
		|\Psi_t(p)-\Psi_t(p')| \eqr \sum_{j=1} ^{n-1} \varphi_{\tau,j}(p,p')
		|\mu_{\tau,j}(p)-\mu_{\tau,j}(p')|.
	\end{equation}
	There is $0 \leq i \leq N$ such that $\sigma \subset [\zeta_i,\zeta_{i+1}]$.		If $\sigma \subset \Gamma_{\zeta_i}$ or $\sigma \subset \Gamma_{\zeta_{i+1}}$ then the result follows from (\ref{eq pr sigma prime}) and the Lipschitzness of the $\xi_{i,t}$'s. Otherwise,	let $q$ and
	$q'$ be two points of $\sigma$.  These two points  may be expressed  
	$$q=(p,\nu\zeta_i (p)
	+(1-\nu)\zeta_{i+1}(p))\quad\eto\quad q'=(p',\nu'\zeta_i(p')
	+(1-\nu')\zeta_{i+1}(p')),$$
	 for some  $(p,p')$ in $\tau\times \tau$ and $(\nu,\nu')$ in
	$(0,1)^2$. Define then (see Figure \ref{fig3}) $$q'':=
	(p,\nu'\zeta_i(p)+(1-\nu')\zeta_{i+1}(p)).$$
	  \begin{figure}[h]
\includegraphics[ scale=4]{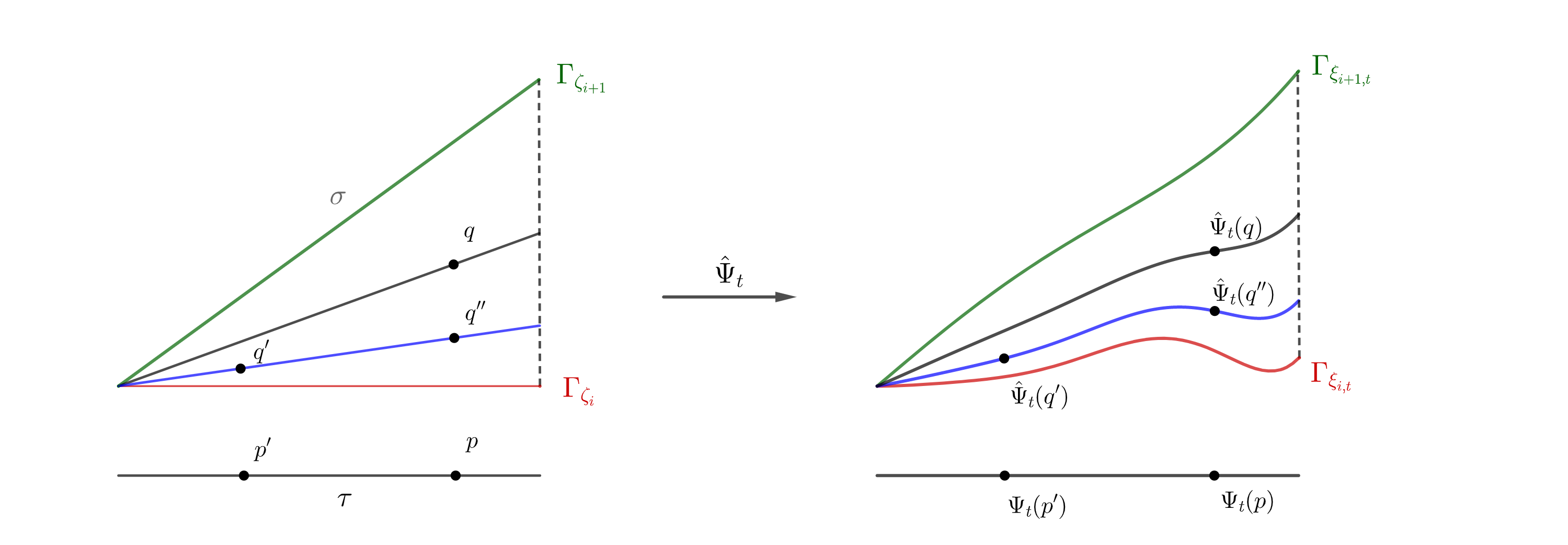}
\centering
  \caption{Main notations  of the proof (here $n=2$). }\label{fig3}
\end{figure}

	We begin with the case where $1 \leq i \leq N-1$.
		We may consider $\nu,\nu',p$, $p'$, and
	$q''$ as functions of $q$ and $q'$.  Due to the definition of $\widehat{\Psi}_t$, since $\xi_{i,t}$ and $\xi_{i+1,t}$
	are Lipschitz functions, we must have over $\sigma \times \sigma$:
	\begin{equation}\label{htilde dist ds preuve}
		|\hat{\Psi}_t(q)-\hat{\Psi}_t(q')| \eqr |\hat{\Psi}_t(q)-\hat{\Psi}_t(q'')|+|\Psi_t(p)-\Psi_t(p')|.
	\end{equation}
 As $\pi(q)=\pi(q'')$, by definition of $\widehat{\Psi}_t$, we have:
	$$|\hat{\Psi}_t(q)-\hat{\Psi}_t(q'')|
	\eqr (\xi_{i+1,t}(\Psi_t(p))-\xi_{i,t}(\Psi_t(p)))
	\cdot |\nu-\nu'|.$$
		Thanks to the induction hypothesis,  we can assume that  the triangulation $(K,\Psi_t)$ is such that  $(\xi_{i+1,t}-\xi_{i,t}) \circ \Psi_t$  is  $\eqr$  to a
	standard simplicial function on $\tau$.  The composite
	$(\xi_{i+1,t}-\xi_{i,t}) \circ \Psi_t \circ \pi$ is thus $\sim$ to a standard simplicial function on
	$\sigma$ that we will denote by $\varphi_{n,\sigma}$. The functions $\zeta_i$ and $\zeta_{i+1}$ define a
	tame coordinate of $\R^{n}$ (as in  (\ref{eq def pl coord})) that we will denote by $\mu_{\sigma,n}$. Observe that $\nu=\mu_{\sigma,n}(q)$ and $\nu'=\mu_{\sigma,n}(q')$.  The preceding estimate may thus be rewritten as:
	\begin{equation}\label{preuve triang n+1 eme coord}
		|\hat{\Psi}_t(q)-\hat{\Psi}_t(q'')| \eqr \varphi_{\sigma,n}(q) \, \cdot
		\,| \mu_{\sigma,n}(q)-\mu_{\sigma,n}(q') \: |.
	\end{equation}
	Define for $j<n$:
	$$\varphi_{\sigma,j}(q,q')=\varphi_{\tau,j}(\pi(q),\pi(q'))\;\; \mbox{ and }\;\; \mu_{\sigma,j}(q):=\mu_{\tau,j}(\pi(q)).$$
	Then by  (\ref{eq pr sigma
		prime}), (\ref{htilde dist ds preuve}), and (\ref{preuve triang n+1 eme coord}), we get the desired
	equivalence (in the case $1\leq i \leq N-1$). Observe that the introduced function $\varphi_{\sigma,n}$ just depends on $q$. In the case $i=0$ or $N$, it will depend on both $q$ and $q'$.

	Let us now focus on the case $i=0$ (the case $i=N$ is completely analogous and will be left to the reader). Note that  we have  by construction
	over $\sigma \times \sigma$:
	\begin{equation}\label{eq p - p'' a l infini}
		|\hat{\Psi}_t(q) - \hat{\Psi}_t(q'')|\eqr \frac{1}
		{(q_{n}-\zeta_0(p))(q'_{n}-\zeta_0(p'))} \cdot |\nu -
		\nu'|,
	\end{equation}
	where $q_n$ and $q'_n$ respectively denote the last coordinate of $q$ and $q'$. 	
	Remark also that  $| \nu- \nu'| $ is $\eqr$ to the difference of the
	tame coordinates of $q$ and $q''$ defined by $\zeta_0$ and $\zeta_1$ (as in (\ref{eq def pl coord})).
	As $|q_{n}-\zeta_0(p)|\eqr d(q,\Gamma_{\zeta_0})$, which is $\sim$ to a standard simplicial function on $\sigma$, we may apply the same argument
	as in the case $1\le i\le N-1$, which completes the case $i=0$. Since we can  do the same job  in the case  $i=N$,  this
	yields (\ref{eq h dist ds triang}) for  $\widehat{\Psi}_t$.
	
	It remains to check that we can assume that the  functions $\eta_{j,t}\circ \widehat{\Psi}_t$ are $\eqr$
	to standard  simplicial  functions over any  $\sigma\in \widehat{K}$. For this purpose, let us fix
	$\sigma \in \widehat{K}$. If  $\widehat{\Psi}_t(\sigma)\subset \Gamma_{\xi_{i,t}}$ for some $i$, the result follows by induction. So, assume that  $\widehat{\Psi}_t(\sigma)\subset(\xi_{i,t},\xi_{i+1,t})$, for some $0 \leq
	i \leq N$ (setting $\xi_{0,t}\equiv-\infty$ and $\xi_{N+1,t}\equiv+\infty$).

	The set  $\widehat{\Psi}_t(\sigma)$ thus must fit in some $C_t$, with $C\in \Pa$ (we recall that  for each $C\in \Pa$ we have $\delta C_t\subset \cup_{\iota=1} ^\lambda  \Gamma_{\theta_{\iota,t}}\subset \cup_{\iota=1} ^N \Gamma_{\xi_{\iota,t}}$,  which is therefore disjoint from $\widehat{\Psi}_t(\sigma)\subset(\xi_{i,t},\xi_{i+1,t})$), and hence the  $\eta_{j,t}$'s are $\sim$ on $\widehat{\Psi}_t(\sigma)$ to products of powers of
	distances
	to the  $W_{j,t}$'s (see (\ref{eq_proof_trian_xi})).
	It thus suffices to show the result for the functions  $q
	\mapsto d(\widehat{\Psi}_t(q),W_{j,t})$, $j=1,\dots,k$. Fix $j\le k$.

	As $(\widehat{\Psi}_t,\widehat{K})$ is compatible with  $W_{j,t}$,   either $\widehat{\Psi}_t(\sigma)\subset W_{j,t}$ or  $d(x,W_{j,t})=d(x,\delta W_{j,t})$ for all $x\in \widehat{\Psi}_t(\sigma)$. In the former case, the result is obvious since   $q
	\mapsto d(\widehat{\Psi}_t(q),W_{j,t})$ is zero over $\sigma$.  We thus can suppose that $d(x,W_{j,t})=d(x,\delta W_{j,t})$ on $\widehat{\Psi}_t(\sigma)$.

	By definition of the $\theta_{\iota,t}$'s,
	$\delta W_{j,t}\subset \cup_{\iota=1}^\lambda\Gamma_{\theta_{\iota,t}}$.
	Moreover,  for each  $\iota \in \{1,\dots,\lambda\}$, since $\theta_{\iota,t}$ is Lipschitz, we have for $x=(\xt,x_{n})\in \widehat{\Psi}_t(\sigma) \subset   \R^{n-1} \times \R$:
	\begin{equation}\label{eq_triang_W_j_et_graphes}
		d(x,\delta W_{j,t} \cap \Gamma_{\theta_{\iota,t}} ) \eqr |x_{n} -\theta_{\iota,t}(\xt)|+d(\xt,\pi(\delta W_{j,t} \cap
		\Gamma_{\theta_{\iota,t}} )).
	\end{equation}
		As both terms of the  right-hand-side are positive, the sum
	is $\eqr$  to the $\max$ of these two terms, that is to say, is $\eqr$ to one
	of them since they are comparable over $\widehat{\Psi}_t(\sigma)$ (thanks to the definition of the $\xi_i$'s). Note that
	clearly  $ d(x,\delta W_{j,t})
	=\underset{1 \leq\iota \leq \lambda}{\min}
	d(x,\delta W_{j,t} \cap \Gamma_{\theta_{\iota,t}} ).$ 
	 But as by construction the functions 
	 $$g_{\iota,t}(\xt):=
	d(\xt,\pi(\delta W_{j,t} \cap \Gamma_{\theta_{\iota,t}} )) $$  are
	comparable with each other  and  comparable with all the functions  $x=(\xt,x_n)\mapsto
	|x_{n} -\theta_{\iota,t}(\xt)|$ on $\widehat{\Psi}_t(\sigma)$, by (\ref{eq_triang_W_j_et_graphes}), the function $d(x,\delta W_{j,t})$ is
	equivalent over $\widehat{\Psi}_t(\sigma)$ to one of the functions $g_{\iota,t}(\xt)$ or to
	some  function $ |x_{n} -\theta_{\iota,t}(\xt)|$.

	Thanks to the induction hypothesis, we can assume that the triangulation $(\Psi_t,K)$ is such that for each $\iota\le \lambda$ the function
	$g_{\iota,t}\circ \Psi_t$ is $\eqr$ to a
	standard simplicial function on $\tau$.  
	Hence, by the preceding paragraph, it only remains to prove that the functions $q=(p,q_n) \mapsto |\widehat{\Psi}_{t,n}(q)-\theta_{\iota,t}(\Psi_t(p))| $ (setting $\widehat{\Psi}_t:=(\Psi_t,\widehat{\Psi}_{t,n})$) are
	$\eqr$ over $\sigma$ to a standard simplicial function (independent of $t$).

	For this purpose, fix a positive integer $\iota \le \lambda$. As $\Gamma_{\theta_{\iota,t}} \subset \bigcup_{j=1}^N  \Gamma_{\xi_{j,t}}$, we have on
	$\pi(\widehat{\Psi}_t(\sigma))$  either  $\theta_{\iota,t} \geq \xi_{i+1,t}$ or $\theta_{\iota,t} \leq \xi_{i,t}$. For simplicity, we will assume that the latter inequality holds.
	On $ \sigma $ we have: \begin{equation}\label{eq_fin_preuve_triangulation}|\widehat{\Psi}_{t,n}-\theta_{\iota,t} \circ \Psi_t|=(\widehat{\Psi}_{t,n}-\xi_{i,t} \circ
		\Psi_t)+(\xi_{i,t} \circ \Psi_t-\theta_{\iota,t} \circ \Psi_t).\end{equation}
		In the case $0<i<N$, by (\ref{preuve triang n+1 eme coord}), we have over $\sigma$ for $q=(p,q_{n})$:
	\begin{equation}\label{eq_psi_n_moins_xi}
		\hat{\Psi}_{t,n}(q)-\xi_{i,t}(\Psi_t(p)) \eqr \mu_{\sigma,n}(q) \, . \, \varphi_{\sigma,n}(q).
	\end{equation}
	Moreover, by (\ref{eq p - p'' a l infini}) (with $\nu'=0$), a similar estimate holds in the case $i=0$ (or $N$).
	Observe also that  $\mu_{\sigma,n}$ is obviously $\eqr$ to a
	standard  simplicial  function. Equivalence (\ref{eq_psi_n_moins_xi}) therefore yields that the first summand of the right-hand-side of (\ref{eq_fin_preuve_triangulation}) is $\sim$ to a standard  simplicial  function.	
	Since by  induction we can assume that so is $|\xi_{i,t}\circ \Psi_t-\theta_{\iota,t}
	\circ \Psi_t|$, the conclusion thus comes down from (\ref{eq_fin_preuve_triangulation}).
\end{proof}

\begin{rem}\label{rem sur les triangulations}
	It is worthy of notice that the above proof has explicitly shown that if $\sigma$ is an open simplex such that $\Psi(\sigma)$ is
	bounded then it is not necessary to involve in the definition of metric
	triangulations standard simplicial
	functions depending on both $q$ and $q'$. More precisely, instead of
	(\ref{eq h dist ds triang}),  it is enough
	to require on such $\sigma$:
	\begin{equation}\label{eq_h_contrations_local}
		|\Psi(q)-\Psi(q')| \sim \sum_{i=1}^{n}\varphi_{\sigma,i} (q)
		\cdot |\mu_{\sigma,i}(q) - \mu_{\sigma,i}(q')|,
	\end{equation}
	with $\varphi_{\sigma,i}
	(q)$ standard simplicial functions on $\sigma$ for each $i$. Moreover, for such a simplex $\sigma$,
 the simplices involved in the expression of  $\varphi_{\sigma,i}$ are	by construction of dimension $(i-2)$ for all $i \ge 2$, and  $\varphi_{\sigma,1}
	\equiv 1$. 
	In addition to this, for such a simplex $\sigma$,  the contraction functions $\varphi_{\sigma,i}$ can be required to be bounded away from infinity. 
\end{rem}

\begin{rem}\label{rem_triang_tame}
	The tame systems of coordinates of the Lipschitz triangulation constructed in the proof of the above theorem have the following property. For every  $\sigma\in K$ (where $K$ is as in the theorem) and every  $i\le n$, if $\zeta_{i,\sigma}$ and  $\zeta_{i,\sigma}'$ denote the two linear functions defining  $\mu_{\sigma,i}$ (see (\ref{eq def pl coord}) and (\ref{eq h dist ds triang})), their graphs $\Gamma_{\zeta_{i,\sigma}}$ and $\Gamma_{\zeta_{i,\sigma}'}$ can respectively be written    $\pi(\tau)$ and $\pi(\tau')$, where $\pi:\R^n \to \R^i$ is the canonical projection and $\tau$ as well as $\tau'$  are faces of  elements
	 of $K$ (that depend on $i$ and $\sigma$).
\end{rem}

A consequence of Theorem \ref{thm existence des triangulations} is the following corollary that tells us how many classes subanalytic bi-Lipschitz equivalence admits.

\begin{cor}\label{cor_sullivan}
	Up to globally subanalytic bi-Lipschitz mappings, globally subanalytic sets are countable. 
\end{cor}
\begin{proof}
	Up to globally subanalytic bi-Lipschitz mappings, globally subanalytic sets are clearly at least countable. Let us show that they are at most countable.
	
 By  Proposition \ref{pro_meme_triangulation}, two sets that can be triangulated by the same simplicial complex $K$, the same contraction functions $\varphi_{\sigma,i}$, and the same tame systems of coordinates are definably bi-Lipschitz homeomorphic. 
	
	The vertices of the simplicial complex provided by Theorem \ref{thm existence des triangulations} can be chosen in
	$\Q^n$.   The class of such finite simplicial complexes is countable. Moreover, given such a  simplicial complex,
	the tame systems of coordinates of the triangulation constructed  in the proof of the latter theorem can be chosen among a finite family (see Remark \ref{rem_triang_tame}).  The contractions are given by finitely many combinations of sums, products, and powers of  distances to the faces. As these powers belong to $\Q$, the standard simplicial functions of such simplicial complexes are countable.
\end{proof}

We are going to derive other consequences of Theorem \ref{thm existence des triangulations}.
	We will say
	that $A\in \s_{m+n}$ is \textbf{definably bi-Lipschitz trivial along}\index{definably bi-Lipschitz trivial family} $U
	\subset \R^m$ if there exist $t_0 \in U$ and a definable family of bi-Lipschitz
	homeomorphisms $h_t :  A_{t_0}  \rightarrow A_t$, $t\in U$.  As a byproduct of Theorem \ref{thm existence des triangulations} and Proposition \ref{pro_meme_triangulation}, we have:
\begin{cor}\label{cor_hardt}
	Given $A\in \s_{m+n}$, there
	exists a definable  partition of $\R^m$ such that  $A$
	is definably bi-Lipschitz trivial along each  element of this
	partition.  
\end{cor}
\begin{rem}\label{rem_triv_subsets}
Since Theorem \ref{thm existence des triangulations} ensures that we can triangulate several families simultaneously, we can trivialize several definable sets simultaneously. Namely,	given some definable subsets $B_1,\dots,B_k$ of $A$, possibly refining the partition of the above corollary, we can require  the trivialization of $A$ that we have along every element of this partition to induce on the $B_i$'s a trivialization of these sets.
\end{rem}

\begin{exa}\label{exa_xt}
 If for $t>0$ we set $A_t:=\{(x,y)\in (0,1) \times (-1,1):|y|=x^t\}$ then $A_t$ is not bi-Lipschitz homeomorphic to $A_{t'}$ for $t\ne t'$. This shows that Corollary \ref{cor_hardt} is not true on non polynomially bounded o-minimal structures \cite{vdd_omin, costeomin}.
\end{exa}

	We are going to see that, refining the partition given by this corollary if necessary, we can require the trivialization to be bi-Lipschitz with respect to parameters on compact sets. This will yield that definable local bi-Lipschitz triviality is a stratifying condition for stratifications (Corollary \ref{cor_stra_bili_triv}).  This requires the following proposition which can be regarded as a Lipschitz version of Proposition \ref{pro_cont_parametres}.

\begin{pro}\label{pro_lipschitz_parametres}
	Let $A \in \s_{m+n}$ and let $f_t:A_t \to \R$ be a definable family of functions. If $f_t$ is Lipschitz for all $\tim$ then there exists a definable partition $\Pa$ of $\R^m$ such that for every $B \in \Pa$,   $f:A\to  \R$, $(t,x)\mapsto f_t(x)$ induces a Lipschitz function on  $A\cap K$, for every compact subset $K$ of  $ B\times \R^n$.
\end{pro}
\begin{proof} We prove the result by induction on $m$. The case $m=0$ being vacuous, assume the result to be true for $(m-1)$, $m\ge 1$. By Proposition \ref{pro_extension_fonction_lipschitz} (see Remark \ref{rem_extension_familles_fonctions_lipschitz}), we may assume that $A= \R^{m+n}$.  By Proposition \ref{pro_cont_parametres}, there is  a definable partition $\Pa$ of the parameter space such that  $f$ is continuous on every $B\times \R^n$, $B\in \Pa$. Fix an element $B\in \Pa$ (we shall refine several times the partition $\Pa$). 
	
	We start with the (easier) case where $\dim B <m$. In this case, by Lemmas  \ref{lem_partition_tangents} and \ref{lem  kurdyka proj fixees}, there is a partition of $B$ such that every element of this partition has a regular vector that, without loss of generality, we can assume to be $e_m\in \sph^{m-1}$.   Thanks to Corollary \ref{cor_proj_reg_decomposition_en_graphes}, it is therefore enough to deal with the case where $B$ is the graph of a Lipschitz function, say $\xi:D \to \R$, $D \in \s_{m-1}$. The result in this case now follows from the induction hypothesis applied to the function $D \times \R^n \ni (t,x) \mapsto f(t,\xi(t),x)$.

	We now address the case  $\dim B=m$. The function $B\ni t \mapsto L_{f_t}$ being definable, partitioning $B$ if necessary, we  can assume this function to be continuous on this set. In particular, it is bounded on compact subsets of $B$.
	Let $Z$ be the set of points  $q \in \Gamma_f$ for which there exists a sequence $q_k\in (\Gamma_f)_{reg}$ tending to $q$ such that   $$(0_{\R^m},e_{n+1}) \in\lim T_{q_k} (\Gamma_f)_{reg},$$ where $e_{n+1}$ is the last vector of the canonical basis of $\R^{n+1}$. Let $\pi :\R^m \times \R^{n+1} \to \R^m$ denote the projection omitting the last $(n+1)$ coordinates. We claim that $\pi(Z)$ has dimension less than $m$. 
	
	Assume otherwise. Take a Whitney $(a)$ regular definable stratification of $\Gamma_f$ compatible with $Z$   and  let $S\subset Z$  be a stratum such that $\pi(S)$ has dimension $m$. Let $S'$ be the set of points of $S$ at which $\pi_{|S}$ is a submersion. Since $\pi(S)$ is of dimension $m$, by Sard's Theorem,  the set $S'$ cannot be empty. Moreover, by definition of $S'$, $T_q S'$ is transverse to $\{0_{\R^m}\}\times \R^{n+1}$   at any point $q$ of $S'$.
		Let $q \in S'\subset Z$. By definition of $Z$,  there is a sequence $q_k$ tending to $q$ such that $(0_{\R^m},e_{n+1}) \in \tau_q:=\lim T_{q_k} (\Gamma_f)_{reg}$, and  Whitney $(a)$ condition ensures that $\tau_q\supset T_q S'$. Consequently, $\tau_q$ is transverse to $\{0_{\R^m}\}\times \R^{n+1}$ as well.
	But since $L_{f_t}$ is locally bounded (it was assumed to be continuous),   $e_{n+1}$ does not belong to $\lim T_{x_k}\Gamma_{f_{t_k}}$, if $q_k=(t_k,x_k)$ in $\R^m \times \R^{n+1}$ (extracting a sequence if necessary, we can assume that this limit exits), which means that
	$$(\lim T_{q_k}\Gamma_f) \cap\{0_{\R^m}\}\times \R^{n+1}\neq \lim  \big{(}T_{q_k}\Gamma_f \cap\{0_{\R^m}\}\times \R^{n+1}\big{)} $$
	(since the latter  does not contain the vector  $(0_{\R^m},e_{n+1})$ while the former does), and hence, that $\tau_q$ cannot be transverse to $\{0_{\R^m}\}\times \R^{n+1}$ (since otherwise the intersection with the limit would be the limit of the intersection). A contradiction.
	
	This establishes that $\dim \pi(Z)<m$. Since we can refine $\Pa$ into a partition which is compatible with $\pi(Z)$, we thus see that we can suppose $B \subset\R^m \setminus \pi(Z)$  (we are dealing with the case $\dim B=m$).
		For $(t,R)\in (B \setminus \pi(Z)) \times [0,+\infty)$ set:   $$\varphi(t,R):=\sup\{|\frac{\pa f}{\pa t}(t,x)|:x\in \bou(0_{\R^n},R), \mbox{ $f$ is $\ccc^1$ at $(t,x)$}   \}$$  (which is finite, by definition of $Z$).
	As $\varphi$ is definable,  up to a partition of $B$, this function  may be assumed to be continuous (and thus bounded on compact sets)  for $R \ge \zeta(t)$, with $\zeta:B\to \R$ definable function.   The function $f$ therefore induces a function which is Lipschitz with respect to the inner metric on every compact set of $B\times \R^n$.
	By Corollary \ref{cor_inner_outer}, up to an extra refinement of the partition,  we can suppose the inner   and  outer metrics of $B$ to be equivalent, which means that so are the inner and outer metrics of $B\times \R^n$, establishing that $f$ is Lipschitz on every compact set of $B\times \R^n$.
\end{proof}
\begin{dfn}\label{dfn_stra_bil_triv} A  stratification $\Sigma$ of a set $X$ is {\bf locally definably bi-Lipschitz trivial at $x_0\in S\in \Sigma$}\index{locally definably bi-Lipschitz trivial stratification} if there are  a tubular neighborhood  $(V_S, \pi_S)$ of $S$ (see Proposition \ref{pro_retraction}) and an open neighborhood $W$ of $x_0$ in $S$  for which there is a definable bi-Lipschitz homeomorphism  $\Lambda:\pi_S^{-1}(W)\to \pi_S^{-1}(x_0) \times W $  satisfying:
	\begin{enumerate}[(i)]
		\item  $\pi_S(\Lambda^{-1}(x,y))= y$, for all $(x,y)\in \pi_S^{-1}(x_0)\times  W$.
		\item\label{item_sigma_x0}   $\Sigma_{x_0}:=\{  \pi_S^{-1}(x_0)\cap Y:Y\in \Sigma\} $ is a stratification of $ \pi_S^{-1}(x_0)$,  and $\Lambda(\pi_S^{-1}(W)\cap Y)=(\pi_S^{-1}(x_0)\cap Y)\times W$, for all $Y\in \Sigma$.
\end{enumerate}
\end{dfn}
We now can derive the following consequence of Theorem \ref{thm existence des triangulations}:

\begin{cor}\label{cor_stra_bili_triv}
	Being locally definably bi-Lipschitz trivial is a local stratifying condition. Consequently, every stratification can be refined into a stratification having this property.
\end{cor}
\begin{proof}This condition is obviously local.  Note that if $S$ is a stratum of a stratification $\Sigma$ of a set $X$ and $(V_S,\pi_S,\rho_S)$ is a tubular neighborhood of $S$ then the family $\pi_S^{-1}(t)\cap X$, $t\in S$, is definable.    By Corollary \ref{cor_hardt}, it must be definably bi-Lipschitz trivial along each element of a definable partition $\Pa$ of $S$. The trivialization can be assumed to be a trivialization of every stratum (see Remark \ref{rem_triv_subsets}).  By Proposition \ref{pro_lipschitz_parametres}, the trivialization is locally bi-Lipschitz with respect to $t\in S$, after a possible refinement of the partition $\Pa$.  By Sard's Theorem, there is a nowhere dense definable subset $Z$ of $S$ such that $S\setminus Z$ contains no singular value of the restriction of $\pi_S$ to the strata (i.e. no point $x_0$ of type $x_0=\pi_S(z)$ with $z\in Y\in \Sigma$ such that $\pi_{S|Y}:Y\to S$ is not submersive at $z$), which means that $\Sigma_{x_0}$ (see Definition \ref{dfn_stra_bil_triv} (\ref{item_sigma_x0})) is a stratification of $\pi^{-1}(x_0)$ for every $x_0\in S\setminus Z$.   The elements of $\Pa$ that are open in $S$ constitute together an open dense subset $U$ of $S$, and the stratification $\Sigma$ is  locally definably bi-Lipschitz trivial at every point of $U\setminus cl(Z)$, which yields the first sentence. The last sentence of the corollary is due to Proposition \ref{pro_existence}.
\end{proof}

\begin{rem}\label{rem_lipstrat}
 In Definition \ref{dfn_stra_bil_triv}, as well as in Corollary \ref{cor_hardt}, we say {\it definably} bi-Lipschitz trivial in order to emphasize that the trivialization is definable. The results of \cite{m,parusinskiprep, lipsomin, halupczok}  provide  Lipschitz stratifications in the sense of Mostowski, which are bi-Lipschitz trivial. Although the trivializations of Mostowski's stratifications are not necessarily definable, they are satisfactory for many applications.
\end{rem}

\section{Metric triangulations: a local version}\label{sect_metric_triang_local}
We are going to prove a stronger version of our metric triangulation theorem for germs of definable sets (Theorem \ref{thm_triangulations_locales}). In section \ref{sect_lcs}, we will rely on this so as to describe the ``Lipschitz conic structure'' of definable germs (Theorem \ref{thm_local_conic_structure}), which will yield their  ``Lipschitz contractibility'' (Corollary \ref{cor_Lipschitz_retract}) as well as the definable bi-Lipschitz  invariance of their link (Corollary \ref{cor_unicite_du_link}).  

The idea is that, since $X\in \s_n$ is definably bi-Lipschitz homeomorphic to 
 \nomenclature[bv]{$\check{X}$}{  \nomrefpage}
\begin{equation}\label{eq_Xcheck}
	\check{X}:=\{(t,x) \in \R \times X: t=|x|\},
\end{equation}  which is a subset of $\C_{n+1}(1)$ (see (\ref{eq_ccn}) for  $\C_n(R)$), the study of germs of subsets of $\R^n$ reduces to that of germs of subsets of $\C_{n+1}(1)$ (see the proof of Theorems \ref{thm_local_conic_structure} and \ref{thm_g}). We thus will focus in this section on germs at $0$ of subsets of  $\C_n(R)$, $R>0$. This amounts to  replace the function $x\mapsto |x|$ with the function $\rho(x_1,\dots,x_n)=x_1$, which fits better to  triangulation problems for it is a linear function.

In the case of germs of definable subsets of $\C_n(R)$, $R>0$, we are going to construct some definable metric triangulations having specific properties.
It was already pointed out in Remark \ref{rem sur les triangulations} that when the triangulated set is bounded, it is not necessary  to involve contraction functions $\varphi_{\sigma,i}$ that depend on both $q$ and $q'$. It is therefore not surprising that in the case of triangulations of  germs of definable subsets of $\C_n(R)$, we will just need contractions that are standard simplicial functions on $\sigma$, for every simplex $\sigma$. But
the main improvement provided by this local version is that these contractions will be decreasing at least as fast as the distance to the origin  as we are drawing near this point (up to some constant, see Definition \ref{dfn_radially} and Theorem \ref{thm_triangulations_locales}).  These facts will be essential  in section \ref{sect_lcs}.

\begin{dfn}\label{dfn_radially}Let $\sigma$ be an open simplex of $\R^n$ with $0\in cl(\sigma)$.
	A  nonnegative function $\varphi$ on $\sigma$ is {\bf subhomogeneous}\index{subhomogeneous} if there is a constant $C$ such that  for all    $s\in (0,1]$ and  $q\in \sigma$ we have
	$\varphi(sq)\le C  s\varphi(q)$.
	
	If  $\mu=(\mu_1,\dots,\mu_n)$ is a tame system of coordinates on $\sigma$ and $i\in \{ 1,\dots,n\}$,  we say that $\mu_i$ is {\bf radially constant}\index{radially constant} if 
	$\mu_i(sq)=\mu_i(q)$, for all $s\in (0,1]$ and $q\in \sigma$.
\end{dfn}

In the theorem below, $X_{[0,\ep]}$ stands for the restriction of $X$ to $[0,\ep]$ (see (\ref{eq_restriction})).

\begin{thm}\label{thm_triangulations_locales}
	Let $X$ be  a definable subset of $\C_n(R)$, $R>0$. For $\ep >0$ small enough, there is a metric triangulation $\Psi:(|K|,0) \to (X_{[0,\ep]} ,0)$ satisfying  (\ref{eq_h_contrations_local}) and
	such that:
	\begin{enumerate}[(i)]
		\item $\Psi$ is a vertical Lipschitz mapping.
		\item\label{item_homogeneous} For each $\sigma \in K$ and each $\,2\le i\le n$, the contraction function $\varphi_{\sigma,i}$ (see (\ref{eq_h_contrations_local})) is a bounded subhomogeneous standard simplicial function on $\sigma$ and the  tame coordinate
		$\mu_{\sigma,i}$ is radially constant. Moreover, $\varphi_{\sigma,1}\equiv 1$.
	\end{enumerate}
	Furthermore, we may require this triangulation to be compatible with finitely many given germs of definable subsets of $X$.
\end{thm}
The proof of this theorem occupies the remaining part of this section.

\paragraph{A preliminary lemma.}We denote by $\s_{n,0}$\nomenclature[by]{$\s_{n,0}$}{set of germs at $0$ of subanalytic subsets of $\R^n$\nomrefpage} the set constituted by all the germs at the origin of definable subsets of $\R^n$.
In the lemma below, all the considered germs are germs  at the origin. 

\medskip

\begin{lem}\label{prop proj reg}
	Let  $A_1,\dots,A_\kappa $ be germs of definable subsets of $\C_{n}(R)$,  $R>0$, and
	$\eta_1,\dots,\eta_l$  be   germs of nonnegative definable functions on $\C_{n}(R)$. There exist a germ of definable
	vertical bi-Lipschitz
	homeomorphism (onto its image) $H:(\C_n(R),0) \to (\C_n(R),0)$ and a cell decomposition $\mathcal{D}$ of $\R^n$ such
	that:
	\begin{enumerate}[(i)]
		\item\label{item_compatible}  $\D$ is compatible with (some representatives of the germs of) $H(A_1),\dots, H(A_\kappa)$.
		\item\label{item_graphe} Every cell of $\D$ is
		either a band or the graph of a Lipschitz function.
		\item\label{item_eq} On each
		cell $D$ of $\mathcal{D}$, every germ  $\eta_i\circ H^{-1}$ is $\sim$
		to a function of the form:
		\begin{equation}\label{eq prep}
							|x_n-\theta(\xt)|^r a(\xt),\qquad (\xt,x_n) \in D\subset \R^{n-1}
		\times \R,\end{equation}
		where
		$a$ is a  definable  function of constant sign, $\theta$ is a Lipschitz definable function  (on the basis  of $D$), and $r \in \Q$.
	\end{enumerate}
\end{lem}
\begin{proof}
	Apply Lemma \ref{lem_eq_dist_cor_th_prep} (with $m=0$ and $A=\C_{n}(R)$) to each of the
	functions $\eta_j$, $j=1,\dots,l$ and take a common refinement of the obtained partitions. This provides a finite partition
	$V_1,\dots,V_b$   of $\C_{n}(R)$ together with some definable subsets $W_1,\dots,W_c$ of
	$\C_{n}(R)$,  such that on each $V_i$ each  function $\eta_j$ is equivalent to a product
	of powers of functions of type  $x \mapsto d(x,W_k)$, $k \leq c$.
	
	Possibly  refining
	the partition $V_1,\dots,V_b$, we may assume that the $W_{k}$'s are unions of elements of this partition. The function $d(x,W_k)$ is then on  $V_i$ for each $k$ and $i$ either identically $0$ or equal to  $d(x,\delta W_k)$. We therefore can suppose that  $d(x,\delta W_k)\equiv d(x, W_k)$ on the $V_i$'s (if $\eta_j\equiv 0$ on some
	$V_i$ then (\ref{item_eq}) is trivial on this set for this $j$).
	
	Apply now Theorem \ref{thm_proj_loc} to the union of the  $\delta A_i$'s,
	 the $\delta V_i$'s, and  the $\delta W_i$'s. This provides a germ of vertical bi-Lipschitz homeomorphism onto its image $H:(\C_{n}(R),0) \to (\C_n(R),0)$ such that $e_n$ is regular for the respective images of these sets under $H$.  It means that
	these sets are sent by $H$
	into the union of the graphs of some definable  Lipschitz functions $\theta_1\leq \dots \leq \theta_d $ defined on $\R^{n-1}$ (see Remark \ref{rem_famille_lipschitz_ordonnees}).

	Let $\pi:\R^{n}\to \R^{n-1}$ denote the canonical projection. Thanks to Lemma \ref{lem_graphes_en_plus} (applied to the $\theta_i$'s and to all
	the $(n-1)$-variable functions $\xt \mapsto d(\xt,\pi(H( \delta W_k)\cap
	\Gamma_{\theta_i}))$ with $m=0$), we know that there exist some definable
	functions $\xi_1\leq\dots\leq \xi_p$ and a cell decomposition of $\R^{n-1}$, say $\E$,  such that for every $E \in \E$ and over each
	$[\xi_{i,|E},\xi_{i+1,|E}]$, $i<p$,  the family of functions
	 $$ |x_n-\theta_\nu(\xt)|, \; d(\xt,\pi(
	H(\delta W_k)\cap \Gamma_{\theta_\nu})),\;( \theta_\nu -\theta_{\nu'})(\xt), \quad  \nu ' < \nu \leq d,\; k\le c,$$ where $(\xt,x_n)\in [\xi_{i,|E},\xi_{i+1,|E}]\subset \R^{n-1}\times \R$,  is totally ordered. Adding some graphs if necessary, we can assume that $\bigcup_{i=1}^p \Gamma_{\xi_i} \supset \bigcup_{i=1}^d \Gamma_{\theta_i} $. Moreover, refining $\E$ if necessary, we can also assume  it to be compatible with the cells of $\pi(\F)$, where $\F$ is some  cell decomposition of $\R^n$ compatible with the $H(A_i)$'s and the $\Gamma_{\xi_i}$'s.
	
	The respective graphs of the
	restrictions of the $\xi_i$'s to the cells of
	$\E$ now induce a cell decomposition of $\R^n$
	compatible with the $H(A_i)$'s that we will denote by $\D$.
	Since the $\xi_i$'s are
	Lipschitz functions, we already see that (\ref{item_compatible}) and (\ref{item_graphe})  hold.

	To prove (\ref{item_eq}), fix a cell $D$ of $\D$. If this cell is a graph, (\ref{eq prep}) is obvious. Otherwise, since $H^{-1}(D)$ is included in $V_j$ for some $j \le b$ (for the $H(\delta V_j)$'s are included in the  $\Gamma_{\theta_i}$'s, which are included in the $\Gamma_{\xi_i}$'s, which define the cells of $\D$), we know that for each $k$ the function $\eta_k $ is $\sim$ on $H^{-1}(D)$ to a product of powers of functions of the finite family   $x \mapsto d(x,\delta W_k)$, $k \leq c$. As $H$ is bi-Lipschitz, this entails
	that for each $k$ the function $\eta_k\circ H^{-1}$ is $\sim$ on $D$ to a product of powers of functions of the  family   $x \mapsto d(x,H(\delta W_k))$, $k \leq c$. As a matter of fact, it is enough to check that each function  $x \mapsto d(x,H(\delta W_k))$, $k\in \{1,\dots,c\}$, admits an estimate like displayed in (\ref{eq_prep_eta_j}), for some function $\theta$ independent of $k$.
	
	Fix $k \le c$. As the  $\theta_\nu$'s
	are Lipschitz functions, we have for any $\nu \in \{1,\dots,d\}$:
	\begin{equation}\label{eq fin triang}
		d(x, H(\delta W_k) \cap \Gamma_{\theta_\nu} ) \eqr |x_n
		-\theta_\nu(\xt)|+d(\xt,\pi(H( \delta W_k) \cap \Gamma_{\theta_\nu} )),
	\end{equation}
	for $x=(\xt,x_n) \in   \R^{n-1} \times \R$.
	The terms of the  right-hand-side are nonnegative and comparable with
	each other (for partial order relation $\leq$) over the cell $D$
	(by choice of the
	$\xi_i$'s).
	The left-hand-side is therefore $\sim$   to one
	of them on $D$.
	
	Note that, as the $H(\delta W_k)$'s are included in the graphs of the $\theta_\nu$'s we have:
	$ d(x, H(\delta W_k))
	=\underset{1 \leq \nu\leq d}{\min}
	d(x,H( \delta W_k )\cap \Gamma_{\theta_\nu} ).$
	Hence, by $(\ref{eq fin triang})$,  $d(x,H(\delta W_k))$
	is for each $k$
	equivalent over $D$ either to one of the functions $\xt \mapsto  d(\xt,\pi(H( \delta W_k) \cap \Gamma_{\theta_\nu} )) $ (which is an $(n-1)$-variable function)  or to
	 $x=(\xt,x_n)\mapsto  |x_n -\theta_\nu(\xt)|$, for some $\nu\in\{1,\dots,d\}$. It thus only remains to check that on each cell, the same $\nu$ can be chosen for all $k$.
	But, since the finite family constituted by the functions  $ |x_n-\theta_\nu(\xt)| , \,  \nu \leq d$, together with  the  functions $(\theta_\nu-\theta_{\nu'})$,  $\nu' < \nu \leq d$, is totally ordered on each cell, there is $\nu_0\le d$ (for each cell) such that each of the functions $|x_n-\theta_\nu(\xt)|$, $\nu\in \{ 1,\dots,d\}$, is either equivalent to $|x_n-\theta_{\nu_0}(\xt)|$ or to an $(n-1)$-variable function  (see (\ref{eq_reduced_theta12}) and (\ref{eq_reduced_theta21})). 
	
	 That $a$ has constant sign (in (\ref{eq prep})) may always be obtained up to a refinement of the cell decomposition.
\end{proof}

\begin{proof}[Proof of Theorem \ref{thm_triangulations_locales}]
	We shall use an argument which is similar to the one we used in the proof of Theorem \ref{thm existence des triangulations}, proving
	inductively  the following statements:

	\noindent $\mathbf{(H_n)}$  Let $X$ be a definable subset   of $\mathcal{C}_n(R)$, $R>0$, and  $A_1,\dots , A_\kappa$ some definable subsets  of $X$.
	Given finitely many bounded nonnegative definable function-germs (at the origin) $\eta_1,\dots,\eta_l$ on $\mathcal{C}_n(R)$,
	there exist $\ep>0$ and a  metric triangulation $$\Psi:(|K|,0) \to (X_{[0,\ep]},0)$$ of $X_{[0,\ep]}$  compatible with   $A_{1_{[0,\ep]}},\dots, A_{\kappa_{[0,\ep]}}$, satisfying properties $(i)$ and $(ii)$ of the theorem, and such that for each $j\le l$ the function  $\eta_j\circ \Psi$ is $\sim$ to the germ of a standard simplicial function on each simplex.    Moreover, if $\eta_j(x)\lesssim x_1$ for $x=(x_1,\dots,x_n)\in X_{[0,\ep]}$, then we can require  $\eta_j\circ \Psi$  to be subhomogeneous.

	The assertion  $\mathbf{(H_1)}$ being trivial ($\Psi$ is then the identity map),  let us prove  $\mathbf{(H_n)}$, assuming $\mathbf{(H_{n-1})}$, $n>1$. Fix $X, A_1\dots,A_\kappa$ and $\eta_1,\dots,\eta_l$ as in   $\mathbf{(H_n)}$. We denote by $\pi: \R^{n} \rightarrow \R^{n-1}$ the
	projection onto the $(n-1)$ first coordinates.

	\noindent {\bf Step 1.} We define a $\ccc^0$ triangulation $\pc :|\widehat{K}|\to X$. 
	
	Apply Lemma \ref{prop proj reg} to the family constituted by  $X$, the
	$A_i$'s, and the set $\ccn$ itself, together with the functions $\eta_1,\dots,\eta_l$.  We get a (germ of)
	vertical bi-Lipschitz map  $H:\ccn\to \ccn$ and a
	cell decomposition $\D$ such that $(i)$, $(ii)$, and $(iii)$ of the
	latter lemma hold. As we may work up to a vertical
	bi-Lipschitz map, we will identify $H$ with the identity map.

	By (\ref{item_graphe}) of Lemma
	\ref{prop proj reg}, every  cell of $\D$ included in $\C_n (R)$ which is not a band is the graph of  a Lipschitz function. We thus can include all such cells  in the union of the respective graphs of some definable Lipschitz functions $\xi_1\leq \dots
	\leq \xi_N$ defined on  $\C_{n-1}(R)$ (see Remark \ref{rem_famille_lipschitz_ordonnees}).

	Refine the cell decomposition $\pi(\D)$ into a cell decomposition $\F$ compatible with the  zero loci of the functions $(\xi_j-\xi_{j+1})$, $j<N$, and apply the induction hypothesis to the family constituted by the
	cells of $\F$ to get
 a  homeomorphism  $\Psi:(|K|,0)\to (\C_{n-1}(R)_{[0,\ep]},0)$, with $\ep>0$ and $K$ simplicial complex of $\R^{n-1}$.

	We are going to lift $\Psi$ to a homeomorphism
	$\widehat{\Psi}:|\widehat{K}| \to \C_{n}(R) $. We first define 
	the simplicial complex $\widehat{K}$ in a similar way as in the proof of Theorem \ref{thm existence des triangulations}.
	
	Let $\zeta_1 \leq \dots
	\leq \zeta_N$ be piecewise linear functions over $|K|$ such that for each $i$
	$\zeta_i \equiv \zeta_{i+1}$ on the set  $\{\xi_{i}\circ \Psi=\xi_{i+1} \circ
	\Psi\}$ (this set is a subcomplex of $K$).  Since the graphs of the $\xi_i$'s are subsets of $\C_n(R)$, all these functions vanish at the origin, and we can assume (up to a translation) that so do the $\zeta_i$'s. Let also 
	$$ Z:= \{ (p,y) \in  |K| \times \R \; : \;
	\zeta_{1}(p) \le y \le \zeta_{N}(p) \}.$$
		We obtain a polyhedral decomposition of $Z$ by taking the respective inverse
	images by $\pi_{|Z}$ of the simplices of $K$  as well as all the images of the simplices of $K$ by the
	mappings $p \rightarrow (p,\zeta_i(p))$, $1\le i \le N$. After a
	barycentric subdivision of this polyhedron, we get a simplicial
	complex $\widehat{K}$.

	Define now over $\widehat{K}$ a mapping 
	$\widehat{\Psi}:
	|\widehat{K}|\rightarrow \R^n$ in the following way:
	$$\widehat{\Psi}(p,s \, \zeta_i(p)+(1-s)\zeta_{i+1}(p)):=
	(\Psi(p),s\, \xi_i\circ \Psi(p)+(1-s)\,\xi_{i+1}\circ\Psi(p)),$$
	for $ 1 \leq i < N$, $p \in |K|$, and $s \in [0,1]$.
	By construction, this mapping is a homeomorphism onto its image.  The cells of $\D$ that lie in $\ccn$ are unions of images under $\widehat{\Psi}$ of open simplices. Since $\D$ is compatible with $X$ and the $A_i$'s, the restriction of $\pc$ to $\pc^{-1}(X_{[0,\ep]})$ is a triangulation of $X_{[0,\ep]}$ compatible with  $A_{1_{[0,\ep]}},\dots, A_{\kappa_{[0,\ep]}}$.


Let us now fix an open simplex $\sigma \in \widehat{K}$ and denote by $\tau$ the simplex of $K$ that contains $\pi(\sigma)$. Let $i<N$ be such that $\sigma\subset [\zeta_i, \zeta_{i+1}]$.
	
	\noindent {\bf Step 2.} We  check that over  $\sigma$  the mapping
	$\widehat{\Psi}$ satisfies an inequality of type  (\ref{eq h dist ds triang}), for some subhomogeneous standard simplicial functions $\varphi_{\sigma,1},\dots,\varphi_{\sigma,n}$ and some tame system of coordinates $\mu_\sigma$ that we shall introduce.

	Thanks to the induction hypothesis, we may find standard simplicial functions
	$\varphi_{\tau,1},\dots,\varphi_{\tau,n-1} $ and a tame system
	of coordinates $(\mu_{\tau,1},\dots,\mu_{\tau,n-1})$ such that for
	 $p$ and $p'$ in $\tau$:
	\begin{equation}\label{eq pr sigma prime_loc}
		|\Psi(p)-\Psi(p')| \eqr \sum_{j=1} ^{n-1} \varphi_{\tau,j}(p)
		|\mu_{\tau,j}(p)-\mu_{\tau,j}(p')|.
	\end{equation}
			If $\sigma \subset \Gamma_{\zeta_i}$ or $\sigma \subset \Gamma_{\zeta_{i+1}}$ then the result follows from (\ref{eq pr sigma prime_loc}) and the Lipschitzness of the $\xi_i$'s. Otherwise,
	let $q$ and
	$q'$ be two points of $\sigma$.  Such points may be expressed  
	$$q=(p,s\zeta_i (p)
	+(1-s)\zeta_{i+1}(p)) \quad \eto\quad  q'=(p',s'\zeta_i(p')
	+(1-s')\zeta_{i+1}(p'))$$
	 for some  $(p,p')$ in $\tau \times \tau$ and  $(s,s')$ in
	$(0,1)^2$, and we then can define (see Figure \ref{fig3}, in the proof of Theorem \ref{thm existence des triangulations}) $$q'':=
	(p,s'\zeta_i(p)+(1-s')\zeta_{i+1}(p)).$$
	We will consider $s,s',p$, $p'$, and
	$p''$ as functions of $q$ and $q'$.  By definition of $\widehat{\Psi}$, since $\xi_{i}$ and $\xi_{i+1}$
	are Lipschitz functions, we have over $\sigma \times \sigma$:
	\begin{equation}\label{htilde dist ds preuve_loc}
		|\pc(q)-\pc(q')| \eqr |\pc(q)-\pc(q'')|+|\Psi(p)-\Psi(p')|.
	\end{equation}
		As $\pi(q)=\pi(q'')$, by definition of $\widehat{\Psi}$, we have:
	$$|\pc(q)-\pc(q'')|
	\eqr (\xi_{i+1}(\Psi(p))-\xi_{i}(\Psi(p)))
	\cdot |s-s'|.$$
		Thanks to the induction hypothesis, we can assume that the triangulation $(K,\Psi)$ is such  that  $(\xi_{i+1}-\xi_i) \circ \Psi$  is  $\eqr$ to a
	subhomogeneous standard simplicial function on $\tau$.  The composite
	$(\xi_{i+1}-\xi_i) \circ \Psi \circ \pi$ is thus $\sim$ to a subhomogeneous standard simplicial function on
	$\sigma$ that we will denote by $\varphi_{\sigma,n}$. The functions $\zeta_i$ and $\zeta_{i+1}$ define a
	tame coordinate of $\R^{n}$ (as in  (\ref{eq def pl coord})) that we will denote by $\mu_{\sigma,n}$.  Observe that $s=\mu_{\sigma,n}(q)$ and $s'=\mu_{\sigma,n}(q')$. The preceding estimate can therefore be rewritten as:
	\begin{equation}\label{preuve triang n+1 eme coord_loc}
		|\pc(q)-\pc(q'')| \eqr \varphi_{\sigma,n}(q) \, \cdot
		\,| \mu_{\sigma,n}(q)-\mu_{\sigma,n}(q') \: |.
	\end{equation}
	Define then for $j<n$ and $q\in \sigma$:
	$$\varphi_{\sigma,j}(q):=\varphi_{\tau,j}(\pi(q))\;\; \mbox{ and }\;\; \mu_{\sigma,j}(q):=\mu_{\tau,j}(\pi(q)).$$
	By  (\ref{eq pr sigma
		prime_loc}), (\ref{htilde dist ds preuve_loc}), and (\ref{preuve triang n+1 eme coord_loc}), we get the desired
	equivalence. 	Since $\zeta_i(0)=\zeta_{i+1}(0)=0$, it is clear that $\mu_{\sigma,n}$ is radially constant (see (\ref{eq def pl coord})).


	\noindent {\bf Step 3.} We check that for each $j$ the function $\eta_j\circ \widehat{\Psi}$ is $\sim$ to a standard  simplicial function on $\sigma$ (we recall that $\sigma$ is a fixed element of $ \widehat{K}$).

	For this purpose, fix a positive integer $j \le l$.    If $\pc(\sigma) \subset  \Gamma_{\xi_\iota}$, with $\iota=i $ or $i+1$ (we recall that $\sigma\subset [\zeta_i ,\zeta_{i+1}]$, which entails $\pc(\sigma) \subset [\xi_i ,\xi_{i+1}]$), then we are done since, thanks to the induction hypothesis, we can assume that  $\tau \ni \qt \mapsto \eta_j\big{(}\Psi(\qt),\xi_\iota\circ \Psi(\qt)\big{)}$ is $\sim$ to a standard simplicial function. Otherwise, by definition of $\pc$, we must have $\pc(\sigma) \subset (\xi_i ,\xi_{i+1})$.

	By construction, $\pc(\sigma)$ is included in a cell $D\in \D$. By (\ref{eq prep}) we know that there are functions $a$ and $\theta$ on the basis $E$ of $D$, as well as $\alpha \in \Q$, such that for  $x=(\xt,x_n) \in \widehat{\Psi}(\sigma)$
	\begin{equation}\label{eq_prep_eta_j}\eta_j(x)\sim |x_n-\theta(\xt)|^{\alpha} a(\xt). \end{equation}
	Thanks to $\mathbf{(H_{n-1})}$, we can assume that $a\circ \Psi$ is $\sim$ to a standard simplicial function. We thus merely have to check that the function $(\qt,q_n)\mapsto |\widehat{\Psi}_n(q)-\theta(\Psi(\qt))|$  is equivalent to a standard simplicial function (here we set $\widehat{\Psi}=(\Psi,\widehat{\Psi}_n)$).

	As $\Gamma_{\theta} \subset \bigcup_{k=1}^N  \Gamma_{\xi_k}$, we have on
	$\pi(\widehat{\Psi}(\sigma))$  either  $\theta \geq \xi_{i+1}$ or $\theta \leq \xi_{i}$. For simplicity, we will assume that the latter inequality holds.
	On $ \sigma $ we have: \begin{equation}\label{eq_fin_preuve_triangulation_loc}|\widehat{\Psi}_{n}-\theta \circ \Psi|=(\widehat{\Psi}_{n}-\xi_i \circ
		\Psi)+(\xi_i \circ \Psi-\theta \circ \Psi).\end{equation}
		By (\ref{preuve triang n+1 eme coord_loc}), we have over $\sigma$ for $q=(\qt,q_{n})$:
	\begin{equation}\label{eq_psi_n_moins_xi_loc}
		\pc_{n}(q)-\xi_{i}(\Psi(\qt)) \eqr \mu_{\sigma,n}(q) \, . \, \varphi_{\sigma,n}(q).
	\end{equation}
	The function $\mu_{\sigma,n}(q)$ is obviously $\eqr$ to a
	standard  simplicial  function.
	Thus, as by induction   $|\xi_{i}\circ \Psi-\theta
	\circ \Psi|$ can be  assumed to be also equivalent to a standard
	simplicial  function,  the required fact follows from (\ref{eq_fin_preuve_triangulation_loc}) and (\ref{eq_psi_n_moins_xi_loc}).
	
	\medskip

	\noindent{\bf Step 4.} We check that  if $\eta_j(x)\lesssim x_1$ for  $x=(x_1,\dots,x_n)\in \pc(\sigma)$, $j\le l$, then we can require  in addition the function $\eta_j\circ \pc$ to be subhomogeneous on $\sigma$.

	For simplicity, we will say that a function $f$ on $\pc(\sigma)$ (resp. $\Psi(\tau)$)  is {\bf $\pc$-subhomo\-geneous}\index{subhomogeneous!$\Psi$-subhomogeneous} (resp. {\bf $\Psi$-subhomogeneous}) if $f\circ \pc$ (resp. $f\circ \Psi$) is subhomo\-geneous.
		The induction hypothesis thus allows us to assume that  finitely many given definable $(n-1)$-variable functions that are $\lesssim x_1$ near the origin are $\Psi$-subhomo\-geneous.

 If  $\pc(\sigma) \subset \Gamma_{\xi_i}$ or $\pc(\sigma) \subset \Gamma_{\xi_{i+1}}$ then the result can easily be deduced from the induction hypothesis. Otherwise,
	 recall that we have assumed  that $\theta \le \xi_i $ (right before (\ref{eq_fin_preuve_triangulation_loc}), see  (\ref{eq_prep_eta_j}) for the definition of $\theta$) on the basis of $D$, which entails that for $(\xt,x_n)\in D \subset \R^{n-1}\times \R$:
	\begin{equation}\label{eq min loc}\eta_j(\xt,x_n) \sim \min \bil|x_n-\xi_i(\xt)| ^{\alpha} a(\xt)
		,|\xi_i(\xt)-\theta(\xt)| ^{\alpha}  a(\xt)\bir,\end{equation}  if
	$\alpha$ is negative, and
	\begin{equation}\label{eq max loc}\eta_j(\xt,x_n) \sim \max \bil|x_n-\xi_i(\xt)| ^{\alpha} a(\xt)
		,|\xi_i(\xt)-\theta(\xt)| ^{\alpha}  a(\xt)\bir,\end{equation} in the
	case where $\alpha$ is nonnegative.
	
	Note that  $\eta_j(x)\lesssim x_1$ entails $\eta_j(x)\sim \min(\eta_j(x),x_1)$, which means that
	it is enough to check that   $\min(\eta_j(x),x_1)$ is $\pc$-subho\-mogeneous. Thanks to the induction hypothesis, we can assume that $\tau \ni \xt \mapsto \min(|\xi_i-\theta|(\xt) ^{\alpha} a(\xt),\xt_1)$, $\xt=(\xt_1,\dots,\xt_{n-1})$,  is $\Psi$-subhomogeneous.
	Hence, in virtue of (\ref{eq min loc}) and (\ref{eq max loc}),
	it is enough to show that the function $\min(|x_n-\xi_i(\xt)|^{\alpha} a(\xt),x_1)$ is $\pc$-subhomogeneous (the $\min$ and $\max$ of $\pc$-subhomogeneous functions are $\pc$-subhomogeneous - note also that $\min(\max(u,v),w)=\max(\min(u,w),\min(v,w))$). 
	
	For simplicity, we define a function on $D$ by setting for $x=(\xt,x_n) \in D$:  $$F(x):=|x_n-\xi_i(\xt)|^{\alpha}\cdot a(\xt),$$
	and a function on the basis $E$ of $D$ by setting for $\xt \in E$ $$G(\xt):=|\xi_{i+1}(\xt)-\xi_i(\xt)|^{\alpha} \cdot a(\xt).$$
	Observe that if we set for $x=(\xt,x_n)\in D$ $$\nu(x) := \frac{x_n-\xi_i(\xt)}{\xi_{i+1}(\xt)-\xi_i(\xt)}$$ then we have:
	$$F(x)=\nu(x)^{\alpha} \cdot G(\xt).$$
	Notice also that  $\nu(\pc(sq))$ is
	constant with respect to $s$, which
	entails that for $s\in [0,1]$ and $q=(\tilde{q},q_n) \in \sigma$:
	\begin{equation}\label{eq2 F G loc}F(\pc(sq))=\nu(\pc(q))^{\alpha} \cdot G(\Psi(s\tilde{q})).\end{equation}

	We first suppose that $\alpha$ is negative. Thanks to the
	induction hypothesis, we can assume that $\xt \mapsto \min (G(\xt),\xt_1)$, $\xt=(\xt_1,\dots,\xt_{n-1})$   is $\Psi$-subho\-mogeneous. This implies (multiplying by $\nu^{\alpha}(x)$ and applying
	(\ref{eq2 F G loc})) that
	$x=(x_1,\dots,x_n)\mapsto \min(F(x),\nu^{\alpha}(x)x_1) $ is $\pc$-subhomogeneous, which entails that so is the function $ \min(F(x),\nu^{\alpha}(x)x_1,x_1) $. But, as $\alpha$ is negative, 
	$$\min(F(x),\nu^{\alpha}(x)x_1,x_1)=\min(F(x),x_1),$$ so that we can conclude that $\min(F(x),x_1)$ is $\pc$-subhomogeneous, as required.

	We now suppose that $\alpha$ is nonnegative. As $\eta_j (x) \lesssim x_1$,  (\ref{eq max loc}) then implies that $F(x) \lesssim x_1$  on $\sigma$, which entails that $G(\xt) \lesssim \xt_1$, and therefore $G(\xt)\sim \min (G(\xt),\xt_1)$ on $\pi(\sigma)$, which, thanks to the induction hypothesis, can be assumed to be $\Psi$-subhomogeneous.  By (\ref{eq2 F G loc}), this entails that $F$ is $\pc$-subhomogeneous.
\end{proof}
\begin{rem}\label{rem_eta_j_subh}
 It is worthy of notice that the proof has established (see the induction assumption) that, given some nonnegative definable functions $\eta_1,\dots,\eta_l$ on $\ccn$ satisfying $\eta_j(x)\lesssim x_1$ for all $j$ (if $x=(x_1,\dots,x_n)\in\ccn$), we can construct our metric triangulation $\Psi$  in such a way that $\eta_j\circ \Psi$ is subhomogeneous for all $j$.
\end{rem}


\section{Local conic  structure}\label{sect_lcs}
Since  definable sets can be triangulated \cite{lojasiewicz64a}, their germs are homeomorphic to cones over their links. Although this homeomorphism cannot be chosen bi-Lipschitz, we are going to see that metric triangulations contain information on the way this homeomorphism affects the Lipschitz geometry (Theorem \ref{thm_local_conic_structure} just below). This result recently turned out to be helpful to compute the cohomology of $L^p$ forms of subanalytic varieties \cite{linfty, gvpoincare} as  well as to investigate the theory of Sobolev spaces of these varieties \cite{trace, lprime, laplace}.

Given a point $x_0\in \R^n$ and  $A\in \s_n$, we denote by $x_0*A$\nomenclature[bym]{$x_0*A$}{cone over $A$ at $x_0$\nomrefpage} the cone over $A$ \index{cone over a set} with vertex at $x_0$, i.e., we set:
$$x_0*A:=\{tx+(1-t)x_0:x\in A \mbox{ and } t\in [0,1]\}.$$
A {\bf  retraction by deformation of $X\subset \R^n$ onto $x_0\in X$}\index{retraction by deformation} is a continuous  map $r:[0,1]\times X \to X, (s,x)\mapsto r_s(x),$ such that $r_1:X \to X$ is the identity map, $r_s(x_0)= x_0$ for all $s$, and $r_0\equiv x_0$.

\begin{thm}\label{thm_local_conic_structure}
	Let  $X\in \s_n$ and $x_0\in X $. 
	For $\ep>0$ small enough, there exists a definable homeomorphism
	$$H: x_0* (\sph (x_0,\ep)\cap X)\to  \Bb(x_0,\ep) \cap X,$$  
	satisfying $H_{| \sph (x_0,\ep)\cap X}=Id$, preserving the distance to $x_0$, and having the following Lipschitzness properties:
	\begin{enumerate}[(i)] 
		\item\label{item_H_bi}     $H$ is Lipschitz and the natural retraction by deformation onto $x_0$ $$r:[0,1]\times  \Bb(x_0,\ep)\cap X \to \Bb(x_0,\ep)\cap X,$$ defined by $$r(s,x):=H(sH^{-1}(x)+(1-s)x_0),$$ is also Lipschitz.  
		Moreover, there is a constant $C$ such that  for every fixed $s\in [0,1]$, the mapping $r_s$ defined by $x\mapsto r_s(x):=r(s,x)$, is $Cs$-Lipschitz.
		\item  For each $\eta>0$,
		the restriction of $H^{-1}$ to $\{x\in X:\eta \le |x-x_0|\le \ep\}$ is Lipschitz and, for each $s\in (0,1]$, the map  $r_s^{-1}:\Bb(x_0,s\ep) \cap X\to \Bb(x_0,\ep) \cap X$ is Lipschitz. 
	\end{enumerate}
\end{thm}
\begin{proof}We may assume $x_0=0$.
	We recall that $\check{X}=\{(t,x) \in \R \times X: t=|x|\}.$
	Applying Theorem \ref{thm_triangulations_locales} to this set  (which is a subset of $\C_{n+1}(1)$) provides a vertical Lipschitz definable homeomorphism
	$\Psi:|K|\to \xxc_{[0,\ep]}$, with $\ep$ positive real number and $K$ simplicial complex of $\R^{n+1}$.
	
	Every point of $ 0* (\sph (0,\ep)\cap X)$ can be written $tq$ with $t\in [0,1]$ and $q\in \sph (0,\ep)\cap X$.  For such $t$ and $q$ we set: 
	$$H(tq):=\pi\circ \Psi(t\Psi ^{-1}(\ep,q)),$$
	where $\pi:\check{X}\to X$ is induced by the projection omitting the first coordinate. 
	This defines a homeomorphism on $0*\sph (0,\ep)\cap X$ which, since $\Psi$ is vertical, satisfies $|H(tq)|=t |q|$, showing that $H$ preserves the distance to the origin.  The statements about the Lipschitzness properties of $H$ and $H^{-1}$ directly follow from (\ref{eq_h_contrations_local}) together with  $(i)$ and $(ii)$ of Theorem \ref{thm_triangulations_locales}.

	Moreover, if for $x\in  \Bb(0,\ep) \cap X$ we set  $r(s,x)=H(sH^{-1}(x)) $ then by definition of $H$ we have $$r(s,x)=\pi \circ \Psi(s\,\Psi ^{-1}\circ \pi^{-1}(x))$$
	(to see this, observe that if we denote by $\tilde{r}_s$ the mapping that sits on the right-hand-side of this equality then $\tilde{r}_s=r_s$ on $\sph (0,\ep)\cap X$, for all $s$, and since $r_s\circ r_t=r_{st}$ as well as  $\tilde{r}_s\circ \tilde{r}_t=\tilde{r}_{st}$, this equality must continue to hold on $\bou (0,\ep)\cap X$).
As $\pi$ is a bi-Lipschitz homeomorphism, the  Lipschitzness properties of $r$ thus follow from (\ref{eq_h_contrations_local}) together with  $(i)$ and $(ii)$ of Theorem \ref{thm_triangulations_locales}.
\end{proof}
\begin{rem}\label{rem_conical_same} It follows from the last sentence of Theorem \ref{thm_triangulations_locales} that, given finitely many definable set-germs $X_1,\dots,X_k$ at $x_0\in\cap_{i=1}^k X_i$, the respective  homeomorphisms of the Lipschitz conic structure of the $X_i$'s provided by Theorem \ref{thm_local_conic_structure} can be required to be induced by the same  homeomorphism $H:x_0*\sph (x_0,\ep)\to \Bb(x_0,\ep)$.
\end{rem}
\begin{rem}\label{rem_lcs2}
	The Lipschitz constant of $r_s^{-1}$ (see $(ii)$) is bounded away from infinity if $s$ stays bounded away from $0$. Indeed, if $s\ge \delta >0$ then $r_\delta =r_{\frac{\delta}{s}}\circ r_s$ entails $r_s^{-1}=r_{\frac{\delta}{s}}\circ r_\delta^{-1}$,  and the  Lipschitz constant of  $r_{\frac{\delta}{s}}$  is bounded independently of $s\ge \delta$. This also can be deduced from the proof of the above theorem and  (\ref{eq_h_contrations_local}).
\end{rem}
\begin{rem}
 By Remark \ref{rem_eta_j_subh}, if $\eta_1,\dots,\eta_l$ are definable nonnegative function-germs on $X$ satisfying $\eta_j(x)\lesssim |x|$ then we can require that $\eta_j(r_s(x)) \le Cs \eta_j(x)$, for all $s\in [0,1]$, $x\in X$, $j\le l$, for some constant $C$ independent of $x$ and $s$.
\end{rem}
\begin{exa}\label{exa_lcs_parab}  We wish to illustrate  Theorem \ref{thm_local_conic_structure} by giving an explicit formula for the mapping $H$ that this theorem provides on the example of a cusp $$X:=\{(x,y)\in [0,1]\times \R:|y|\le x^2\}$$ with $x_0=(0,0)$. For each $(x,y)\in x_0*(S(0,1)\cap X)$, let  $$H(x,y):=(t(x,y) x,\,t^2(x,y)xy),$$ where $t$ is the only positive real solution of $t^2x^2+t^4x^2y^2=x^2+y^2$, i.e.
$$t(x,y)=\left(\frac{2x^2+2y^2}{x^2+\sqrt{x^4+4x^2y^2(x^2+y^2)}}\right)^{1/2}.$$

The choice that we made for $t(x,y)$ ensures that this mapping preserves the distance to the origin. A straightforward computation yields that on $x_0*(S(0,1)\cap X)$ we have $|\pa t(x,y)|\le \frac{C}{x}$ for some positive constant $C$, from which it follows that $H$ has bounded first order partial derivatives (the function $t$ is bounded away from zero and infinity on $x_0*(S(0,1)\cap X)$). Clearly, the mapping $H$ of the above theorem cannot be required to be bi-Lipschitz.

Let us here emphasize that if we drop the condition that $H$ preserves the distance to the origin but simply require $\frac{|(x,y)|}{C}\le  |H(x,y)| \le C|(x,y)| $ for some constant $C$ (which is satisfactory for many purposes) then it suffices to set $H(x,y)=(x,xy)$.
\end{exa}

It is also worthy of notice that Theorem \ref{thm_local_conic_structure} is not true on non polynomially bounded o-minimal structures \cite{vdd_omin,costeomin}, as shown by the following example (see also Example \ref{exa_xt}).

\begin{exa}\label{exa_lcs_exp}
Let $$Y:=\{(x,y)\in (0,1]\times \R: |y|\le e^\frac{-1}{x} \}\cup \{(0,0)\}. $$
If there were such a mapping $H$, it is not difficult to see that the points $z:=(x,e^\frac{-1}{x})$ and $z':=(x,-e^\frac{-1}{x})$, for $x\in (0,1]$, would satisfy for $s\in (0,1]$:
$$|r_s(z)-r_s(z')|\le C e^\frac{-1}{sx}, $$
for some constant $C$ independent of $x$ and $s$.
As $|z-z'|=2e^\frac{-1}{x}$, we see that  $r_s^{-1}$ would fail to be Lipschitz for each $s\in (0,1)$.  This example enlightens the key role played by \L ojasiewicz's inequality in the theory.
 \end{exa}

In particular, we have established the following
\begin{cor}\label{cor_Lipschitz_retract}
	Let $X \in \s_n$ and $x_0\in X$. For every  $\ep>0$
	small enough there exists a definable Lipschitz retraction by deformation $r: [0,1] 
	\times X\cap \bou (x_0,\ep)\to
	X\cap \bou (x_0,\ep), (s,x)\mapsto r_s(x),$ onto $x_0$. 
		Moreover,  $r_s$ is  $Cs$-Lipschitz for some $C>0$ independent of $s$,  and can be required to preserve finitely many given definable  subset-germs of $X$ for all $s\in(0,1]$.
\end{cor}

The second consequence that we would like to point out will lead us to the notion of link that will be studied in  section \ref{subsection_uniqueness_of_the_link} (this actually could also be derived from Corollary \ref{cor_hardt}).
\begin{cor}\label{cor_link_bien_defini}
	Let $X\in \s_n$ and let  $x_0 \in \R^n$. Up to a definable bi-Lipschitz homeomorphism,  the set $\sph (x_0,\ep)\cap X$, $\ep>0$,  is independent of $\ep>0$ small enough. 
\end{cor}

	The {\bf link}\index{link} of $X\in \s_n$ at $x_0 \in \R^n$, denoted $lk(X,x_0)$\nomenclature[byp]{$lk(X,x_0)$}{link of $X$ at $x_0$\nomrefpage}, will  be the subset $\sph (x_0,\ep)\cap X$, $\ep>0$ small (by the just above corollary, it is well defined {\it up to a definable bi-Lipschitz homeomorphism}).

\subsection{Some lemmas about  definable set and map germs}
Theorem \ref{thm_triangulations_locales} makes it possible to establish that the link of a globally subanalytic set is invariant under globally subanalytic bi-Lipschitz mappings (Corollary \ref{cor_unicite_du_link}), which requires some preliminaries that we carry out in this section.

Given a definable Lipschitz map-germ $F$ at $0_{\R^n}$ the limit  $\lim_{t \to 0} \frac{1}{t}(F(tx)-F(0))$ exists for all $x$ and clearly defines an $L_F$-Lipschitz mapping, which means that definable Lipschitz maps are Gateau differentiable.  Indeed:
\begin{lem}\label{lem_gateau}
	Let $A$ be a definable subset of $\,\sph(0_{\R^n},\ep)$, $\ep>0$, and let $F:0_{\R^n}*A\to B$ be a Lipschitz definable map. If we set $G(x)=\lim_{t \to 0} \frac{1}{t}(F(tx)-F(0))$ then for all $(t,x)\in [0,1]\times A$
	$$F(tx)= F(0)+tG(x)+t\mu(tx), $$
	for some  definable continuous mapping $\mu$ tending to zero at the origin.
\end{lem}
\begin{proof} Possibly replacing $A$ with  $cl(A)$, we may assume  that this set is closed.
	Notice that  $\tilde{F}_t(x):= \frac{1}{t}(F(tx)-F(0))$, $x\in  A$,  is $L_F$-Lipschitz for each $t\in (0,1]$, and hence, so is $G$. It  suffices to show that $(t,x)\mapsto \tilde{F}_t(x) $ extends continuously on $[0,1]\times A $. Let $(t_i,x_i)$ be a sequence in $(0,1]\times A$ tending to some point $(0,a)$ in this set. By Ascoli's Theorem, extracting a sequence if necessary, we can assume that  $\tilde F_{t_i}$ uniformly converges to $G$. Hence,  as $(\tilde F_{t_i}(x_i)-G(x_i))$  and $(G(x_i)-G(a))$ both tend to zero,  $\tilde F_{t_i}(x_i)$ must tend to $G(a)$.
\end{proof}

  Given $a$ and $b$ in $\R^n$, we denote by $ab$ the line segment joining $a$ and $b$.\nomenclature[can]{$ab$}{line segment joining $a$ and $b$\nomrefpage}
\begin{lem}\label{lem_triang_secantes}Let  $K$ be a simplicial complex of $\R^n$ and 
	let  a germ of Lipschitz definable map $H:(|K|,0)\to (\R^k,0)$ satisfy $|H(q)|\eqr |q| $. The angle at $H(q)$ between the ray $0H(q)$ and the tangent half-line  to the arc $H(0q)$ (the image of the ray $0q$ under $H$) at $H(q)$  tends to zero as $q\ne 0$ goes to the origin. 
\end{lem}
\begin{proof}
	Let  $\pi_1 :\R^n \times \R^k \to \R^n$ (resp. $\pi_2 :\R^n \times \R^k \to \R^k$) denote the projection  onto the  $n$ first (resp. $k$ last)  coordinates.
	Take  a definable arc $\alpha:(0,\ep) \to |K|$ tending to the origin at $0$, let
	$$l:=\lim_{s\to 0} \frac{H(\alpha(s))}{|H(\alpha(s))|}\, ,$$
	 and let $l'$ denote the limit of the unit tangent vector  at $H(\alpha(s))$ to the image of the  ray that stems from $\alpha(s)$, i.e.: 
	 $$l':=\lim_{s \to 0}\lim_{t \to 1_-} \frac{\frac{d}{dt}\beta(s,t)}{|\frac{d}{dt}\beta(s,t)|}\, , \quad \mbox{ where } \;\;\beta(s,t)=H(t \cdot \alpha(s)).$$   
	 Thanks to Curve Selection Lemma (Lemma \ref{curve_selection_lemma}), it suffices to show $l=l'$.
	
	Let then $\alpha(s,t):=t \cdot \alpha(s)$ and define  $\gamma(s,t):=(\alpha(s,t),\beta(s,t)) \in \Gamma_{H}$.   Set $u:=\lim_{s \to 0} \frac{\gamma(s,1)}{|\gamma(s,1)|}$ and $u':= \lim_{s\to 0}\lim_{t \to 1_-} \frac{\frac{d}{dt}\gamma(s,t)}{|\frac{d}{dt}\gamma(s,t)|}$.  We  claim that  $u=u'$.
	
	Since for every $s$, $\{\alpha(s,t): t\in \R\}$ is a line, we clearly have $\frac{\pi_1(u)}{|\pi_1(u)|}=\frac{\pi_1(u')}{|\pi_1(u')|}$. Take a Whitney $(b)$ regular stratification of $\Gamma_H$ (see Propositions  \ref{pro_existence} and \ref{pro_w_stratifying}) compatible with $\{0_{\R^{n+k}}\}$.  Let $S$ be the stratum that contains $\gamma(s,t)$ for $t<1$ close to $1$ (for $s>0$ small enough it is independent of $s$). As $H$ is Lipschitz, $\pi_1$ must induce a one-to-one map on $$\tau:=\lim_{s\to 0}\lim_{t \to 1_-} T_{\gamma(s,t)} S,$$ which contains $u$ (by Whitney's $(b)$ condition) and $u'$ (by definition).  Hence, $\frac{\pi_1(u)}{|\pi_1(u)|}=\frac{\pi_1(u')}{|\pi_1(u')|}$ entails $u=u'$, yielding our claim.
	
	As $|H(q)|\sim |q|$, the vector $u$ cannot be included in  $\ker \pi_2$, so that  by definition of $u$ and $l$ we must have  $l=\frac{\pi_2(u)}{|\pi_2(u)|}$ (since the curve $\gamma(s,1)$ lies on $\Gamma_H$). But  $u'=u$ is not included in the kernel of $\pi_2$ either, so that by definition of $u'$ and $l'$, we also have    $l'=\frac{\pi_2(u')}{|\pi_2(u')|}$.
	We thus get $l=\frac{\pi_2(u)}{|\pi_2(u)|}=\frac{\pi_2(u')}{|\pi_2(u')|}=l'$, as required.  
\end{proof}

\begin{lem}\label{lem_lips_spheres}
	Let $A\in \sno$ and let $f:(A,0) \to (\R,0)$ be a germ of definable function   satisfying
	 \begin{equation}\label{eq_f_lips_spheres}|f(x)|\lesssim |x|, \end{equation}
	  for  $x\in A$  close to the origin. If  $f_{|\sph (0,r)\cap A}$ is $L$-Lipschitz for every $r>0$ small, with $L \in \R$ independent of $r$, then $f$ is the germ of a  Lipschitz function.
\end{lem}
\begin{proof}
	By Curve Selection Lemma, it suffices to check the Lipschitz condition along two definable arcs $x:(0,\ep)\to A$ and $y:(0,\ep)\to A$ tending to the origin at $0$. We may assume that $x$ is parametrized by its distance to the origin. 	Let for simplicity $t_r:=|y(r)|$.
  As $x$ is a  Puiseux arc  satisfying $|x(r)|=r$, its expansion starts like $x(r)=ar+\dots$, with $a \in \sph^{n-1}$, and therefore:
	\begin{equation}\label{eq_p}
		|x(r)-x(t_r)|\lesssim |r-t_r|\leq |x(r)-y(r)|.
	\end{equation}
 By Proposition \ref{pro_puiseux}, $f(x(r))$ is a Puiseux arc. Since $|f(x(r)) |\lesssim r$ for $r$ positive small, this entails that $f$ is Lipschitz along 
	$x(r)$.  Hence, \begin{equation}\label{eq_lips_sphere2}|f(x(r))-f(x(t_r))|\lesssim |x(r)-x(t_r)| \overset{(\ref{eq_p})}{\lesssim} |x(r)-y(r)|.
	\end{equation}
	As $f$ is $L$-Lipschitz on the spheres  $$      |f(x(t_r))-f(y(r))|\le L|x(t_r)-y(r)|\le L|x(t_r)-x(r)|+L |x(r)-y(r) |, $$
	which, together with (\ref{eq_p}) and (\ref{eq_lips_sphere2}), implies the desired inequality.
\end{proof}

\begin{dfn}
	Let $A\in \sno$. A nonnegative Lipschitz definable function-germ $\rho:(A,0) \to (\R,0)$   satisfying
	\begin{equation}\label{eq equiv rho a r}
		\rho(x) \sim |x|
	\end{equation}
	is called {\bf a radius function}\index{radius function}.
\end{dfn}

\begin{lem}\label{lem angle des secantes}
	Let $A\in \sno$ and let $\rho:(A,0) \to (\R,0)$ be a radius function.
	Let in addition $x:[0,\ep)\to A$ and $y:[0,\ep)\to A$ be two $\ccc^0$ definable arcs such that $$x(0)=y(0)=0_{\R^n}\quad \eto \quad \rho(x(r))=\rho(y(r)),\quad \mbox{for all $r\in (0,\ep)$. }$$
	If $x$ and $y$ have   the same tangent  half-line at the origin then  $\frac{x-y}{|x-y|}$  and  $\frac{y}{|y|}$ have different limits  as $r$ goes to zero.
\end{lem}
\begin{proof}
	We claim that we can assume $A=\R^n$. Indeed, for every $r$, the restriction $f_r:=\rho_{|\sph (0,r)\cap A}$ may then be extended to a definable $L_\rho$-Lipschitz function  by setting 
	\begin{equation}\label{eq_ext_radius}
		\tilde{f}_r(q):=\inf \{\rho(p)+L_\rho |p-q|: p \in \sph (0,r)\cap A\}.
	\end{equation}
	 Clearly, $\tilde{\rho}(q):=\tilde{f}_{|q|}(q)$  extends $\rho$ and satisfies (\ref{eq equiv rho a r}).  Moreover, as it is $L_\rho$-Lipschitz on every sphere $\sph (0,r)$, by Lemma \ref{lem_lips_spheres}, it is Lipschitz on a neighborhood of the origin in $\R^n$, yielding that we can assume $A=\R^n$.

	We now shall show that there is $\ep>0$ such that for almost every $q$  in a neighborhood of the origin (as $\rho$ is definable, its derivative exists almost everywhere)
	\begin{equation}\label{eq_inner_prod_gradient}
		\partial_q \rho\cdot \frac{q}{|q|}\ge \ep .
	\end{equation}
 Thanks to Curve Selection Lemma, it suffices to show that for every definable arc $\gamma(r)$  the limit $l:=\lim_{r\to 0}\partial_{\gamma(r)} \rho\cdot \frac{\gamma(r)}{|\gamma(r)|}$ is positive. As $\gamma$ is a Puiseux arc (see Proposition \ref{pro_puiseux}), we may assume that it is parametrized by its distance to the origin. If $\gamma(r)=ar+\dots$, with $a\in \sph^{n-1}$, then $\gamma'(r)=a+\dots$, so that
	\begin{equation}\label{eq_der_rho}\lim_{r\to 0}\frac{d \rho (\gamma(r))}{dr}=\lim_{r\to 0}\partial_{\gamma(r)} \rho\cdot \gamma'(r)\end{equation}
		 has the same sign as $l$ (positive, negative, or zero).  Since $\rho$ is positive and vanishes at the origin, $\frac{d \rho (\gamma(r))}{dr}$ must be positive  for $r$ positive, which means that $l$ must be nonnegative.  If $l=0$ then, by (\ref{eq_der_rho}), we get $\rho(\gamma(r))\ll r$, in contradiction with (\ref{eq equiv rho a r}), yielding (\ref{eq_inner_prod_gradient}).

	Now, we assume that the conclusion of the lemma  fails, i.e.,  that the limit of the segment $xy$ (in the projective space) is the same as the limit of the segment $0y$.
	Let us move slightly $y$
	to some close point $z$ such that $\rho$ is almost
	everywhere differentiable on the segment $xz$ (almost every $z\in \R^n$ has this property). By Definable Choice (Proposition \ref{pro_globally subanalytic_choice}), we can assume that $z$ is a definable arc.
	
	By assumption, $x$ and $y$ share the same half-tangent at the origin. If $z$ is sufficiently close to $y$ then the arcs $\frac{x}{|x|}$ and $\frac{z}{|z|}$  have the same limit in $\sph^{n-1}$, say $v$. Notice that $v=\lim_{r\to 0}\frac{q(r)}{|q(r)|}$, for every arc $q:(0,\ep)\to \R^n$ such that $q(r)$ belongs to the segment $x(r)z(r)$ for all $r$.
	By (\ref{eq_inner_prod_gradient}), we thus have $\pa_q \rho \cdot v\ge \frac{\ep}{2} $ for any point $q$ in the segment $x(r)z(r)$, $r>0$ small enough (since $v$ is close to $\frac{q}{|q|}$).
	
	As we have assumed that $\frac{x-y}{|x-y|}$ tends in the projective space to the same limit as $y$,   if $z$ is sufficiently close to $y$ then $\frac{x-z}{|x-z|}$ must converge to $\pm v$, say $v$ for simplicity.  By the above paragraph, this implies   that $\pa_q \rho .\frac{x-z}{|x-z|}\ge \frac{\ep}{4}$, for any $q$ in the segment $x(r)z(r)$, $r>0$ small enough.
	Hence,  $\rho$ is strictly monotonic on the segment $xz$, with a derivative bounded away from zero by $\frac{\ep}{4}$.
	This contradicts the fact that the points $x(r)$ and $y(r)$ belong to the
	same level surface of $\rho$ (and that $z$ is very close to $y$).
\end{proof}

We now can extend Lemma \ref{lem_lips_spheres} to all the radius functions:

\begin{lem}\label{lem_lips_spheres2}
	Let $A\in \sno$ and let $\rho:A\to \R$ be a radius function. Let $f:(A,0) \to (\R,0)$ be a definable function-germ  satisfying (\ref{eq_f_lips_spheres})
 near the origin. If  $f_{|\{\rho=r\}}$ is $L$-Lipschitz for every $r>0$ small, with $L \in \R$ independent of $r$, then $f$ is the germ of a  Lipschitz function.\end{lem}
\begin{proof}
	One more time, it is enough to check the Lipschitz condition along two definable curves $x$ and $y$ ending at the origin, and we may assume that $x$ is parametrized in such a way that $\rho(x(r))=r$.  
	
	Set for simplicity $t_r:=\rho(y(r))$. If $x$ and $y$ are not tangent to each other at the origin then $|x(r)|\lesssim |x(r)-y(r)|$ and   $|y(r)|\lesssim |x(r)-y(r)|$, so that the result follows from assumption (\ref{eq_f_lips_spheres}).  If they have the same half tangent at the origin, then, by Lemma \ref{lem angle des secantes}, the angle between the vectors $x(t_r)$ and $\pm(x(t_r)-y(r))$ does not go to zero. This implies (since $x$ is definable arc) that the angle between $(x(t_r)-x(r))$ and $(x(t_r)-y(r))$ is bounded below away from zero (for $r\ne t_r$; if $r\equiv t_r$ then the needed fact follows from the assumption ``$f_{|\{\rho=r\}}$ is $L$-Lipschitz'').
	We thus have:
	\begin{equation}\label{eq_p2}|x(r)-x(t_r)|\lesssim|x(r)-y(r)|\end{equation} (since $|x(r)-y(r)|\ll |x(r)-x(t_r)|$ would entail that this angle goes to zero).
Since $f(x(r))$ is a Puiseux arc satisfying $|f(x(r)) |\lesssim r$, the function $f$ is Lipschitz along 
	the arc $x(r)$. We thus can finish the proof with exactly the same computation as in the proof of Lemma \ref{lem_lips_spheres}, writing  (\ref{eq_lips_sphere2}) and replacing (\ref{eq_p}) with (\ref{eq_p2}).
\end{proof}

We also can derive a bi-Lipschitz version of this lemma.

\begin{lem}\label{lem_lips_spheres_bi}
	Let $A \in \sno$ and let $f:(A,0) \to (B,0)$ be a germ of definable map. 
	Let $\alpha:(B,0)\to \R$ be a radius function such that $\rho(x):=\alpha(f(x))$ defines a radius function on $A$. If the function $f_{|\{\rho=r\}}$ is $L$-bi-Lipschitz for every $r>0$ small, with $L \in \R$ independent of $r$, then $f$ is the germ of a bi-Lipschitz mapping.
\end{lem}
\begin{proof} 
	As $\alpha$ and $\rho$ are both radius functions, we  have near the origin $|x|\sim\rho(x)= \alpha(f(x))\sim |f(x)|$.
	Hence, by Lemma \ref{lem_lips_spheres2}, as $f_{|\{\rho=r\}}$ is $L$-Lipschitz for every $r>0$ small, with $L \in \R$ independent of $r$,  $f$ is the germ of a  Lipschitz mapping.
	
	We now check that $f$ is one-to-one. If $p$ and $q$ are two points of $A$ such that $f(p)=f(q)$ then $\rho(p)=\alpha(f(p))=\alpha(f(q))=\rho(q)$. But, as $f$ is one-to-one on the level surfaces of $\rho$, this implies $p=q$. 

It remains to show that $f^{-1}$ is Lipschitz.	The restriction of $f^{-1}$ to the set $$\{x \in f(A): \alpha(x)=r\}=f(\{x\in A:\rho(x)=r\})$$ is $L$-Lipschitz by assumption, for all $r>0$ small enough. Therefore, again due to Lemma \ref{lem_lips_spheres2}, $f^{-1}$ must be a Lipschitz mapping.   
\end{proof}


\subsection{Uniqueness of the link}\label{subsection_uniqueness_of_the_link}

\begin{thm}\label{thm_g}
	Let $X \in \s_{n,0}$  and let $\rho:(X,0) \to (\R,0)$ be a radius function. There exists a germ of definable bi-Lipschitz homeomorphism $g:(\R^n,0)\to (\R^n,0)$ preserving  $X$ and such that $|g(x)|=\rho(x)$ for all $x\in X$.
\end{thm}
\begin{proof} We may extend $\rho$ to a radius function on a definable open neighborhood of the origin $U$ in $\R^n$ (see (\ref{eq_ext_radius})).
We recall that $\check{U}:=\{(t,x) \in \R \times U: t=|x|\}$ (see (\ref{eq_Xcheck})), which is a subset of $\C_{n+1}(1)$.
	As the mapping $P:\check{U} \to U $, induced by the restriction of the projection omitting the first coordinate, is a bi-Lipschitz homeomorphism, it is enough to construct a germ of definable bi-Lipschitz homeomorphism $g:(\check{U},0)\to (\check{U},0)$ preserving $\xxc$  and satisfying $g_1(x)=\check{\rho}(x)$, for all $x\in \check{U}$, where $g_1(x)$ is the first coordinate of $g(x)$ in the canonical basis and  $\check{\rho}(x):=\rho( P(x))$.
	
	Let $\Psi:(|K|,0) \rightarrow (\check{U}_{[0,\ep]},0)$ be the vertical map given by Theorem
	\ref{thm_triangulations_locales} applied to $\xxc$ and  $\check{U}$.
	Define a mapping $\Lambda :|K|\to |K|$ 
	by setting $$\Lambda(q):=\frac{\check{\rho}(\Psi(q))}{q_1 }\cdot q,$$ if $q=(q_1,\dots,q_{n+1})\in |K|$ is nonzero, and $\Lambda(0):=0$.
	Note that $\Lambda$
	preserves the germs of the open simplices that have the origin in their closure.
	Define then a mapping $g$ from $\check{U}$ to itself by:
	$$g (x):=\Psi \circ\Lambda \circ \Psi^{-1}(x), \quad x\in \check{U}.$$  We first check $g$ gives rise to an onto map-germ, i.e., that $g(\check{U})=\check{U}$ as germs, which amounts  to check that  $\Lambda(\Psi^{-1}(\check{U}))=\Psi^{-1}(\check{U})$, as germs. Take $y=(y_1,\dots,y_{n+1})\in \Psi^{-1}(\check{U})$. As our problem is local at $0$, we will assume  that $y_1<\ep$, where $\ep>0$ is sufficiently small for $\{q\in \Psi^{-1}(\check{U}):\check{\rho}\circ \Psi(q)\le \ep\}$ to be compact (note that, since $\Psi$ is vertical and $\rho$ is a radius function, $\check{\rho}\circ \Psi(q)\sim q_1\sim |q|$ on $\Psi^{-1}(\check{U})$). A preimage of $y$ under $\Lambda$ is given by a point $q\in \mathscr{L}_y:=\{t y : t\in(0,\infty)\}$, satisfying $\check{\rho}\circ \Psi(q)=y_1$. If $y_1 \in [0, \check{\rho}\circ \Psi(y)]$ then, as $\rho$ is continuous and vanishes at $0$, it is clear that there is $q$ on $\mathscr{L}_y$ between $0$ and $y$ satisfying $\check{\rho}\circ \Psi(q)=y_1$. If $y_1\in [\check{\rho}\circ \Psi(y),\ep)$ then, as $\{\check{\rho}\circ \Psi\le \ep\}$ is compact, the half-line $\mathscr{L}_y$ cannot meet the frontier of $\Psi^{-1}(\check{U})$ before $\check{\rho}\circ \Psi$ reaches $\ep$, which means that there is $q\in \Psi^{-1}(\check{U})$ on $\mathscr{L}_y$ such that $\check{\rho}\circ \Psi(q)=y_1$, as required.

	Observe also that, as $\Psi$ is vertical, we have for $x\in \check{U}$: $$g_1(x_1,\dots,x_{n+1})=\frac{\check{\rho}(x)}{x_1}\cdot  x_1=\check{\rho}(x),$$
	as required. 
	To finish the proof, we have to check the bi-Lipschitzness of $g$.
	
	By Lemma \ref{lem_lips_spheres_bi}, it is enough to prove that for every $t$ the restriction of $g$ to $\{\check{\rho}=t\}$ is $L$-bi-Lipschitz with $L\in \R$ independent of $t$. 
	To show this, take two definable curves $x:(0,\ep)\to \check{U}$ and $y:(0,\ep)\to \check{U}$ such that $\check{\rho}(x(t))=\check{\rho}(y(t))=t$  for any $t>0$ small enough, and let $p(t)=(p_1(t),\dots,p_{n+1}(t))$ and $q(t)=(q_1(t),\dots,q_{n+1}(t))$ be the respective preimages of these two arcs under $\Psi$.  
	Set $\nu(t) :=\frac{p_1(t)}{q_1(t)}$, and note that, since $\rho$ is a radius function and $\Psi$ is vertical, we have $\nu \sim 1$ (for simplicity, and because we can interchange $x$ and $y$, we will assume that $\nu(t)\le 1$).  


	 \noindent {\bf Claim.} We have:
	\begin{equation}\label{eq_claim_h_unicite_link}|\Psi(p)-\Psi(q)| \sim |\Psi(\nu q)-\Psi(p)|.\end{equation}
			Let $\alpha(s,t)=\Psi (sq(t))$, $s\in [\nu(t),1]$ (if $\nu \equiv 1$ then the above claim is trivial, we will thus suppose $\nu<1$).
By Lemma \ref{lem_gateau}, as  $\Psi$ is Lipschitz and satisfies $|\Psi(z)|\sim |z|$ for $z$ close to zero,	the angle  between $\alpha(s,t)$ and $\Psi(q(t))$
	tends to zero as $t\to 0$ (uniformly in $s\in [\nu(t),1]$).
	Furthermore,
	 by Lemma \ref{lem_triang_secantes}, we know that the angle between $\frac{\pa \alpha}{\pa s}(s,t)$  and  $\alpha(s,t)$ tends to zero as $t\to 0$  (uniformly in $s\in [\nu(t),1]$).  Hence, since $\Psi( q(t))-\Psi(\nu(t)\cdot q(t))=\int_{\nu(t)}^1 \frac{\pa \alpha}{\pa s}(s,t)ds$, we conclude that
	\begin{equation}\label{eq_nuq_moins_q}
		\lim_{t\to 0} \frac{\Psi( q(t))-\Psi(\nu(t)\cdot q(t))}{|\Psi(q(t))-\Psi(\nu(t) \cdot q(t))|}=\lim_{t\to 0} \frac{\Psi(q(t))}{|\Psi(q(t))|}.
	\end{equation}	
	Observe also that, in virtue of (\ref{item_homogeneous}) of Theorem \ref{thm_triangulations_locales} and (\ref{eq_h_contrations_local}), we must have:
	$$ |\Psi(\nu q)-\Psi(p)|\lesssim |\Psi(p)-\Psi(q)|.$$
	Therefore, if (\ref{eq_claim_h_unicite_link}) fails then the angle between
	$ (\Psi( q)-\Psi(\nu q))$ and
	$(\Psi(q)-\Psi(p))$ tends to $0$,
	 which, by (\ref{eq_nuq_moins_q}), means that so does the angle between  $\Psi(q)-\Psi(p)=y-x$ and $\Psi(q)=y$, in contradiction with Lemma \ref{lem angle des secantes},      yielding (\ref{eq_claim_h_unicite_link}).
	
	Observe now that since $\check{\rho}(x(t))=\check{\rho}(y(t))=t$, by definition of $\Lambda$ we must have $$\Lambda(q(t))=\frac{t}{p_1(t)}\nu (t)q(t)\quad \mbox{ and }\quad \Lambda(p(t))=\frac{t}{p_1(t)}p(t),$$ 
	so that, since the homothetic transformation $|K|\ni z \mapsto \frac{t}{p_1(t)}z$ preserves the tame coordinates $\mu_{\sigma,i}$ for $i \geq 2$  (see (\ref{item_homogeneous}) of Theorem \ref{thm_triangulations_locales}) and $\frac{t}{p_1(t)}\sim 1$, we have:  $$|\Psi(\Lambda(q))-\Psi(\Lambda(p))|\overset{\mbox{(\ref{eq_h_contrations_local})}}{\sim} |\Psi(\nu q)-\Psi(p) |\overset{\mbox{(\ref{eq_claim_h_unicite_link})}}{\sim}  |\Psi(q)-\Psi(p)|,$$
	establishing  the bi-Lipschitz character of $g$. 
\end{proof}

\begin{rem}The just constructed mapping $g$ only depends on the triangulation $\Psi$. As a matter of fact,
 if we have several definable set-germs $X_1,\dots, X_l$ at $\orn$, we may demand $g$ to be the same for all the $X_i$'s (taking a triangulation $\Psi$ compatible with all the $\check{X_i}$).
\end{rem}

Given two definable sets $X$ and $Y$ (resp. germs of definable sets), we write $X\approx Y$\nomenclature[cb]{$X\approx Y$}{definable bi-Lipschitz equivalence of sets or germs\nomrefpage} if there is a definable  bi-Lipschitz homeomorphism (resp. a germ of definable bi-Lipschitz homeomorphism) sending $X$ onto $Y$. 
An immediate consequence of the preceding theorem is:

\begin{cor}
	Let $X \in \sno$  and let $\rho:(X,0) \to (\R,0)$ be a radius function.  For any $r>0$ small, $\{x\in X: \rho(x)=r\} \approx lk(X,0)$.
\end{cor}

Another consequence of Theorem \ref{thm_g} is the following:

\begin{thm}\label{thm_unicite_du_link}
	Let $X$ and $Y$ be two  germs  of definable sets at the origin. If $X \approx Y$  then there exists a germ of definable  bi-Lipschitz homeomorphism $\Phi$ sending $X$ onto $Y$ and preserving the distance to the origin.
\end{thm}
\begin{proof}
	Let $h:(X,0) \to (Y,0)$ be a germ of definable bi-Lipschitz homeomorphism and define  a radius function on $X$ by setting for $x\in X$, $\rho(x):=|h(x)|$.  By Theorem \ref{thm_g}, there is a  germ of definable homeomorphism $g:(X,0)\to (X,0)$ such that $|g(x)|=\rho(x)$ for all $x\in X$, which implies that the mapping $\Phi:=h\circ g^{-1}$ is bi-Lipschitz and preserves the distance to the origin.
\end{proof}

Observe that   Lemma \ref{lem_lips_spheres_bi}  and  Theorem \ref{thm_unicite_du_link} establish that the metric type of a definable  set-germ $X$ is characterized by the metric types of all the sections $\sph (0,r)\cap X$, $r>0$ small, and vice-versa. In particular, the above theorem yields the definable bi-Lipschitz invariance of the link:
\begin{cor}\label{cor_unicite_du_link}
	Let $X$ and $Y$ be two definable set-germs at the origin. If $X \approx Y$ then $lk(X,0)\approx lk(Y,0)$.
\end{cor}

\begin{rem}In the just above corollary and in Theorem \ref{thm_unicite_du_link}, $X$ and $Y$ do not necessarily lie in the same euclidean space. Moreover,
	if we define the link as the generic fiber of the distance to the origin (in the field of convergent Puiseux series), a converse is possible \cite{pams}. This fact can indeed be derived from Lemma \ref{lem_lips_spheres_bi} and standard arguments of algebraic geometry. We also stress the fact that, in  Theorem \ref{thm_unicite_du_link} and Corollary \ref{cor_unicite_du_link}, if  the  homeomorphisms between $X$ and $Y$ are given by homeomorphisms of the ambient space, then the provided mappings can also be required to be  induced by such homeomorphisms (since Theorem \ref{thm_g} itself provides such a homeomorphism).
\end{rem}

\paragraph{Historical notes.}  
 The Lipschitz cell decomposition theorem (Theorem \ref{thm_lipschitz_cells}) is due to K. Kurdyka and A. Parusi\'nski \cite{kurdykawhitney, kp}. It is called in \cite{birmos} the ``Pancake Decomposition Theorem''. Statement (\ref{item_system}) of this theorem is however  an addendum to the latter works that was  obtained in \cite{vlt} in order to show Lemma \ref{lem_eq_dist_cor_th_prep}. Let us mention that this fact was upgraded by W. Paw\l ucki  \cite{pawluckilcell} who showed that the linear changes of coordinates may always be expressed as a permutation of the vectors of the canonical basis.   Proposition \ref{pro_holder_normal} seems to be new and has been added because it is useful to the study of Sobolev spaces of subanalytic manifolds (via Morrey's embedding) recently investigated by the author of these notes \cite{poincwirt,trace, lprime,laplace}.
 
The study of bi-Lipschitz equivalence of analytic singularities  was initiated by T. Mostowski \cite{m} who focused on complex analytic sets and continued by A. Parusi\'nski \cite{parusinskiprep} who explored the real case (see also \cite{lipsomin, halupczok}).
The material of sections \ref{sect_metric triangulations}, \ref{sect_metric_triang_local}, and \ref{sect_lcs} is however due to the author of  these notes. 
 Metric triangulations were introduced in \cite{vlt} to establish
Corollary \ref{cor_hardt} (and where they are called Lipschitz triangulations), which is the Lipschitz counterpart of a result about topological stability sometimes referred as Hardt's Theorem  \cite{hardt}. Theorem \ref{thm_triangulations_locales} and uniqueness of the link (Corollary \ref{cor_unicite_du_link}) were established in \cite{pams} ((\ref{item_homogeneous}) of Theorem \ref{thm_triangulations_locales} was added in \cite{gvpoincare}). In the case of curves, a first insight had been made earlier by L. Birbrair and his  H\"older complex decomposition \cite{bir1}. Corollary \ref{cor_sullivan} was established in \cite{jsl} while Corollary \ref{cor_hardt} was already present in \cite{vlt}. Corollary \ref{cor_Lipschitz_retract} was proved in \cite{linfty} and the description of the Lipschitz conic structure of globally subanalytic sets (Theorem \ref{thm_local_conic_structure}) may be found in \cite{gvpoincare}. Existence of $\ccc^0$ triangulations and local $\ccc^0$ retracts however goes back to as far as the original work of S. \L ojasiewicz \cite{lojasiewicz64a,lojasiewicz64b}.

             \markboth{G. Valette}{Geometric measure theory}

\chapter{Geometric measure theory}\label{chap_Geometric measure theory}\label{chap_gmt}
We study the Hausdorff measure of globally subanalytic sets as well as integrals of globally subanalytic functions. 
It is easy to see that the antiderivatives of a globally subanalytic function are not necessarily globally subanalytic.
After recalling some basic formulas, we will show that functions defined by integrals of globally subanalytic families of functions may be described by polynomials in some globally subanalytic functions and their logarithms (Corollary \ref{cor_integration_fonctions}), that we call {\it $\fln$-functions} (Definition \ref{dfn_log_functions}). This will yield that if $A$ is subanalytic then $\hn^l(A\cap \bou(0,r))$ ($\hn^l$ being the Hausdorff measure) has an analytic expansion in $r$ and $\ln r$, and will lead us to the notion of density that we will study on stratified sets (section \ref{sect_densite}). On the way, we give several results of measure theory of globally subanalytic sets and families that are of their own interest (sections \ref{sect_measure} and \ref{sect_volume_alpha_approx}).
We end this chapter (section \ref{sect_stokes}) by establishing Stokes' formula on globally subanalytic sets (possibly singular). 

We recall that the word ``{\bf definable}'' is used as a shortcut of ``globally subanalytic'' (see right after Proposition \ref{pro_csq_qe}).

\section{Cauchy-Crofton's formula}
As this will be the central tool of our study, we enclose a proof of this formula.
   For simplicity, we will just focus on proving a formula that relates the Hausdorff measure of a definable set to the cardinal of its sections by affine spaces of complementary dimension (Theorem \ref{thm_cc_formula}). There exist much more general versions  \cite{federer} but this will be enough for our purpose. We first recall some  basic techniques of integration theory.

   Given $A \in \s_n$, we denote by $\hn^k(A)$\nomenclature[cc]{$\hn^k$}{$k$-dimensional Hausdorff measure\nomrefpage} the $k$-dimensional Hausdorff measure of $A$. We recall that   $\pi_P:\R^n \to \R^n$ stands for the  orthogonal projection onto $P\in \G_k^n$.

\paragraph{The generalized Jacobian.} We will often  apply the so-called {\it  co-area formula} (see below).  This requires the following notion.

\begin{dfn}\label{dfn_gen_jacobian}
Let $f:A \to \R^m$ be a definable mapping  with $\dim A \ge m$. Take a point $x \in reg(f)$ at which $A$ is a manifold of dimension $\dim A$
and  denote by $M_x(f)$ the Jacobian matrix of $f$ with respect to  some orthonormal bases of $T_x A$ and $ \R^m$.
The {\bf generalized Jacobian of $f$ at $x$}\index{generalized Jacobian} is defined as  $$J_x(f):= \sqrt{\det \big{(}M_x(f)^t\cdot  M_x(f)\big{)}  }.$$ \nomenclature[cca]{$J_x(f)$}{generalized Jacobian of $f$\nomrefpage}
Here  $M_x (f)^t$ stands for the transpose of the matrix $M_x (f)$.
The generalized Jacobian coincides with the square root of  the sum of the squares of the minors of order $m$ of  $M_x(f)$. It is of course independent of the choice of the orthonormal bases. 
\end{dfn}

It is worthy of notice that if $\dim A=m$ then the generalized Jacobian equals the usual Jacobian, the absolute value of the determinant of the Jacobian matrix  of $f$.  Note also that if $m=1$, i.e., if $f$ is a function, then $J_x(f)$ is merely the norm of the gradient of $f$ at $x$.


\paragraph{The co-area formula.} This formula will be useful in section \ref{sect_densite} to estimate the variation of the measure of globally subanalytic sets. A proof  can be found  in \cite{federer, krantz}. These two books actually provide it in a much more general context than globally subanalytic sets.

 Let $f:A \to \R^m$ be a definable mapping, $E \subset A$ be a definable subset, and  set $l:=\dim A$.
  If $l \ge m$ then we have: 
\begin{equation}\label{eq_coarea_formula}
\int_{y \in \R^m}  \hn^{l-m}(f^{-1}(y)\cap E) \, d \hn^m(y)  =\int_{x\in E } J_x(f)\, d \hn^l(x) . 
\end{equation} \index{coarea formula}

\begin{proof}Generally, a smoothness assumption is put on $f$. As $f$ is here just assumed to be definable, let us explain how to derive this formula from the classical one. There are  stratifications $\Sigma$ of $E$ and $\Sigma'$ of $\R^m$ such that  $f$ induces a smooth submersion on every $S\in \Sigma$ to a stratum $S'\in\Sigma'$ (that depends on $S$). Fix such $S$ and $S'$, and let us emphasize that for $x\in S$, $J_x(f)$ is not the Jacobian of the mapping induced by $f$ from $S$ to $S'$ that, to avoid any confusion, we will denote by $g:S\to S'$. As strata are disjoint, it suffices to show:
\begin{equation*}\label{eq_coarea_formula_p}
\int_{y \in S'}  \hn^{l-m}(f^{-1}(y)\cap S) \, d \hn^m(y)  =\int_{x\in S } J_x(f)\, d \hn^l(x) .
\end{equation*}
Note first that if $\dim S<l$ and $\dim S'<m$ then $\hn^l(S)=\hn^m(S')=0$, which  means that both sides vanish. If $\dim S<l$ and $\dim S'=m$ then $\hn^l(S)=0$ and (since $g$ is a submersion) $\dim (f^{-1}(y)\cap S)<l-m$ which means that $ \hn^{l-m}(f^{-1}(y)\cap S)=0$, and hence that both sides are also zero. If $\dim S=l$ and $\dim S'<m$ then $J_x(f)=0$ on $S$ (since the derivative of $f$ has rank less than $m$) and $\hn^m(S')=0$, in which case both sides are also zero. Finally, in the case $\dim S=l$ and $\dim S'=m$, as $J_x(f)=J_x(g)$ on $S$, this follows from the classical coarea formula \cite{krantz}.
\end{proof}

\paragraph{The measure $\gamma_{l,n}$.} Cauchy-Crofton's formula involves integration on the Grassmannian manifold $\G_l^n$, which requires to introduce a measure on this set. We sketch the classical construction of what is usually called the {\it Haar measure}.

Denote by $\mathcal{O}_n$ the group of linear isometries of $\R^n$.
This set may be identified with a definable compact subset of $\R^{n^2}$. Denote by $d_n$ its dimension. We first define a measure $\theta_n$ on $\mathcal{O}_n$ satisfying $\theta_n(\mathcal{O}_n)=1$ by setting, for $W \subset \mathcal{O}_n$,  $\theta_n(W):= \frac{\hn^{d_n}(W)}{\hn^{d_n}(\mathcal{O}_n)}$.

 Fix now  any $V \in \G_l^n$ and set for $U \subset \G_l^n$\nomenclature[cd]{$\gamma_{l,n}$}{measure on the Grassmannian\nomrefpage}:
$$\gamma_{l,n}(U):=\theta_n (\{L \in \mathcal{O}_n:  LV \in U\}).$$
It is not difficult to derive from the definitions that the number $\gamma_{l,n}(U)$ is independent of  $V \in \G_l^n$, and that $\gamma_{l,n}$ is a measure on $\G_l^n$. As integration over $\G_l^n$ will always be considered with respect to this measure, we will not specify the considered measure when integrating on $\G_l^n$ (simply writing $dP$ if $P\in \G_l^n$ is the variable of integration).

\paragraph{Cauchy-Crofton's formula.}
Given a set $E$, let  $\card\, E $ denote its {\bf cardinal}, i.e., the number of elements of  $E$, with the convention that $\card \, E :=\infty$ \nomenclature[ce]{$\card \, E$}{cardinal of $E$\nomrefpage}\index{cardinal}  if $E$ is infinite.

Given $l \in \{1,\dots, n\}$,   $P \in \G_l^n$, and $y \in P$, we denote by $N_P^y$\nomenclature[cf]{$N_P^y$}{affine space passing through $y$ and directed by the orthogonal complement of $P \in \G_l^n$\nomrefpage} the $(n-l)$-dimensional affine space passing through $y$ and directed by the orthogonal complement of $P$ in $\R^n$.

\begin{thm}\label{thm_cc_formula}(Cauchy-Crofton's formula)\index{Cauchy-Crofton's formula} There exists a positive constant $\beta_{l,n}$ such that for any $A \in \s_n$  we have:
\begin{equation}\label{cauchy_crofton_formula}
\hn^l (A)=\beta_{l,n}\int_{ P\in \G_l^n} \int_{y \in P} \; \card \,  (N_P ^y \cap A)\; d \hn^l(y) \;d P .
\end{equation}
\end{thm}

\begin{proof}
We may assume that  $\dim A=l$ since otherwise both sides of the equality are $0$ or $\infty$. Thanks to Corollary \ref{cor_cc_families}, we know that for each   $A \in \s_n$ there is  $m_A\in \N$ such that $\card (N^y_P\cap A)\ne j$ for all integers $j>m_A$,  all $P \in \G_l^n$, and all $y\in P$.

\noindent {\bf Step 1.} We  establish the desired formula in the case where  $A$ is a definable subset of a vector space $Q \in \G_l^n$.

Let $\lambda_P:Q\to P$ be the restriction of $\pi_P$.   Take its matrix with respect to  orthonormal bases of $Q$ and $P$, denote by $\alpha_{l,n}(P)$  the absolute value of its determinant, and define the desired constant $\beta_{l,n}$ by setting:  $$\beta_{l,n}:= \big{(}\int_{P\in  \G_l^n} \alpha_{l,n}(P)dP \big{)} ^{-1}.$$
 For almost every $P$, $\lambda_P$ is a linear isomorphism, which entails that we have $ \card\,  (\pi_P ^{-1}(y) \cap A) =1$ for all $y \in \lambda_P(A)$, and since $\hn^l(\lambda_P(A))=\alpha_{l,n}(P)\hn^l(A),$
 we can  evaluate the  right-hand-side of (\ref{cauchy_crofton_formula}) as follows:
$$\beta_{l,n}\int_{ \G_l^n} \int_{\lambda_P(A)}  d \hn^l (y) dP=  \beta_{l,n}\left(\int_{ \G_l^n} \alpha_{l,n}(P) \hn^l(A)dP \right)  =\hn^{l}(A).$$

\noindent {\bf Step 2.} We prove that there exists a positive constant $C$ such that  for any   $\alpha\in(0, \frac{1}{2}]$ and any
$\alpha$-flat set $A \in \s_n$,  we have:   $$ |\hn^l(A)-\mu(A)| \leq C \alpha \hn^l(A),$$
  where $\mu(A)$ stands for the right-hand-side of (\ref{cauchy_crofton_formula}).

 Let $\alpha\in (0,\frac{1}{2}]$ and let  $A \in \s_n$ be $\alpha$-flat, which implies that  we can find $Q \in\G_l^n$ such that $\angle (P,Q)<\alpha$, for all $P\in\tau (A)$ (see the beginning of section \ref{sect_regular_vector} for $\tau(A)$).
We will assume that $Q=\R^l\times \{0_{\R^{n-l}}\}$ for simplicity (this is possible thanks to the definition of the measure on the Grassmannian). This means that the set $A$ coincides with  the disjoint union of the graphs of some definable mappings,  $\xi_i:A_i \to \R^{n-l}$, $i=1,\dots,s$, $A_i\subset \R^l$, with $|d_x \xi_i| \leq \alpha$ almost everywhere. Fix $i \le s$ and observe that, by the coarea formula (applied to  $\pi_{Q|\Gamma_{\xi_i}}$), we have
\begin{equation}\label{eq_diff_mesure_cc}
	|\hn^l(A_i)-\hn^l (\Gamma_{\xi_i})| \leq  \alpha \hn^l(A_i). \end{equation}
For simplicity, set for $E \in \s_n$ and $P \in \G_l^n$:
$$\nu_P (E):=  \int_{y \in P} \; \card \,  (N_P ^y \cap E)\; d \hn^l(y).$$
Since  $\angle (T_x \Gamma_{\xi_i},Q)<\alpha$ at every $x \in \Gamma_{\xi_{i,reg}}$, it is an easy exercise of linear algebra to establish that for $x \in Q$ we have $|J_x(\pi_{P|Q})-J_{(x,\xi_i(x))} (\pi_{P|\Gamma_{\xi_i}})|\le \alpha$.
By the coarea formula (\ref{eq_coarea_formula}), this implies that we have (using again $|d_x \xi_i| \le \alpha$):
\begin{equation}\label{eq_diff_K_P}
|\nu_P(\Gamma_{\xi_i})-\nu_P (A_i)|\leq 2 \alpha \hn^l (\Gamma_{\xi_i}).
\end{equation}
Integrating with respect to $P$ we get:
\begin{equation}\label{eq_diff_theta}
|\mu(\Gamma_{\xi_i})-\mu (A_i)|\leq C \alpha \hn^l (\Gamma_{\xi_i}),
\end{equation}
for some constant $C$. 
 As the $A_i$'s are  subsets of an $l$-dimensional vector subspace of $\R^n$, thanks to step 1, we know that Cauchy-Crofton's formula must hold for each of them.  Making use of  this formula, we immediately derive from (\ref{eq_diff_theta}):
\begin{equation}\label{eq_mu_xia}
 |\mu(\Gamma_{\xi_i}) -\hn^l (A_i)|\leq  C \alpha \hn^l (\Gamma_{\xi_i}).
\end{equation}
Note that as $\alpha\le \frac{1}{2}$, (\ref{eq_diff_mesure_cc}) entails that  we have $\hn^l (A_i)\le 2\hn^l (\Gamma_{\xi_i})$, and consequently
$$ |\mu(\Gamma_{\xi_i}) -\hn^l (\Gamma_{\xi_i})|\le |\mu(\Gamma_{\xi_i}) -\hn^l (A_i)|+|\hn^l (A_i)-\hn^l (\Gamma_{\xi_i})|  \le    (C+2) \alpha \hn^l (\Gamma_{\xi_i}), $$
by (\ref{eq_diff_mesure_cc}) and (\ref{eq_mu_xia}).
  Adding-up these inequalities for all $i$  gives the needed estimate.

\noindent {\bf Step 3.} We prove Cauchy-Crofton's formula.

 Given $A$ in  $\s_n$ and $\alpha>0$, we can find an $\hn^l$-negligible definable set $E\subset A$ such that $A\setminus E$ can be covered by $\alpha$-flat disjoint definable sets   $B_1,\dots, B_k$  (see Lemma \ref{lem_partition_tangents} and Proposition \ref{pro_delta_A}). Applying step 2 to all the  $B_i$'s and adding the resulting  inequalities provides for any  $\alpha\in (0,\frac{1}{2}]$:
$$|\mu(A)-\hn^l(A)|\leq C \alpha \hn^l(A)$$
 ($C$ being given by step 2), and therefore $\mu(A)=\hn^l(A)$.
\end{proof}

Set for $j\in\N \cup \{\infty\}$, $P\in \G_l^n$, and $A \in \s_n$:
$$K_j ^P(A):=\{x \in P : \card\,  (N_P^x \cap A )=j  \}.$$\nomenclature[ch]{$K_j^P(A)$}{points $x\in P \in \G_l^n$ at which $\pi_P^{-1}(x)\cap A$ has cardinal $j$\nomrefpage}
As emphasized at the beginning of the above proof, Corollary \ref{cor_cc} yields that for every set  $A \in \s_n$, there is an integer $k$\nomenclature[ci]{$m_A$}{multiplicity of $A$\nomrefpage} such that $K_j ^P(A)$ is
 empty for all integers $j>k$ and all $P \in \G_l^n$. We will call the smallest integer having this property  the {\bf multiplicity}\index{multiplicity} of
$A$ and will denote it by $m_A$.

Since $K_j^P ( A)=\emptyset$ for any integer $j>m_A$, and because $\hn^l(K_\infty^P(A))=0$ if $l\ge \dim A$,  Cauchy-Crofton's formula may be rewritten for $l\ge \dim A$ as:
\begin{equation}\label{cauchy_crofton_formula_rewritten}
\hn^l (A)=\beta_{l,n}\,\sum_{j=1} ^{m_A}\,j\int_{\G_l^n}  \,   \hn^l( K_j ^P (A))\, dP.
\end{equation}
This formula, together with the uniform finiteness properties of definable sets, provides many bounds for the $\hn^l$-measure of these sets. We illustrate this fact with a result about definable families that will be useful later on.
\begin{pro}\label{borne unif pour les volumes}
Let  $A  \in \s_{m+n}$ and let  $l:=\max_{\tim}\dim A_t$. There exists a
constant $C$ such that for all  $\tim$ and all 
$  r \ge 0 $:
\begin{equation}\label{eq_borne_volume}
 \hn^l(A _t \cap \bou (0,r)) \leq Cr^{l}.
\end{equation}
\end{pro}
\begin{proof}
In the case  $l=n$, since $\hn^l( \bou (0,r)) = \hn^l(\bou (0,1))r^{l}$, the result is  clear. For the same reason, (\ref{eq_borne_volume}) also holds when $A_t$ is for each $t$ a subset of some $l$-dimensional vector subspace of $\R^n$, which establishes this estimate for  the definable family $(K_j^P(A_t))_{P\in \G_l ^n, \tim}$. Since Corollary \ref{cor_cc_families} yields $\sup \{ m_{A_t} : \tim\}<\infty$,
the case $l<n$ thus easily comes down from  (\ref{cauchy_crofton_formula_rewritten}).
\end{proof}

As a matter of fact, for any   $A \in \s_{m+n}$ such that $\dim A_t \le l$ for all $\tim$ and $\sup_{\tim} diam (A_t)<\infty$ (see (\ref{eq_diameter_def}) for $diam$), we have:\begin{equation}\label{eq_borne_uniforme}
\sup\{ \hn^l(A_t):\tim\} <\infty.  
                                                 \end{equation}

\section{On integration of definable functions}
In this section, we introduce the class of $\fln$-functions and show that the integrals of definable functions give rise to $\fln$-functions (Corollary \ref{cor_integration_fonctions}).

 Let  $X \in \s_{n}$ and  $k \le n$.  For any $f:X \to \R$,  we set $|f|_{1,\hn^k}:=\int_{X} |f| d\hn^k$ \nomenclature[ck]{$\vert f\vert_{1,\hn^k}$}{$\;L^1$ norm of $f$ with respect to  $\hn^k$\nomrefpage} (possibly infinite).
 We denote by $L^1_{\hn^k}(X)$ the set of functions $f:X \to \R$ that are $L^1$ for the measure $\hn^k$\nomenclature[cl]{$L^1_{\hn^k}(X)$}{$L^1$ space with respect to $\hn^k$\nomrefpage} (and  then say that  $f$ is $L^1_{\hn^k}$).  
 

\begin{pro}\label{pro_familles_L_1}
Let $f:A \to \R$ be a definable function, $A\in \s_{m+n}$. For each $l \le n$, the set 
\begin{equation}\label{eq_int_locus}
 \{\tim : f_t \in L^1_{\hn^l} (A_t)\} 
\end{equation}
is definable.
\end{pro}
\begin{proof}
Let $B:= \{(t,x,y) \in   A \times \R : 0 \leq y \leq |f_t(x)|\} .$  The function  $f_t$ is $L^1_{\hn^l}$ if and only if $B_t$ has finite
$\hn^{l+1}$-measure.  The family $(B_t)_{t \in \R^m}$ being definable, 
 it follows from Corollary \ref{cor_hardt} that there exists a definable partition $\Pa$ of  $\R^m$ such
that for every  $C\in \Pa$ and any $(t,t')\in C\times C$, the sets $B_t$ and $B_{t'}$ are bi-Lipschitz homeomorphic. The set appearing in (\ref{eq_int_locus}) being the union of some elements of $\Pa$, it must be definable.
\end{proof}

 In the situation of the above proposition, the function $g(t):= \int_{A_t} f_t\, d\hn^l$,  defined on the set appearing in (\ref{eq_int_locus}),
 is of course not always definable. For instance, if $f_t(x)=1/x$ and $A_t=[1,t]$, for every $t >1$,  then $g(t)=\ln t$,  which is not a definable function. We are going to explain (Theorem \ref{thm_integration_des_log}) that the function $g$ is always a polynomial combination of definable functions and of their logarithms, which leads us to the following definition.

We denote by $x \mapsto \ln x$ the natural logarithm function, that we will consider as defined on $\R$, with $\ln (-x)=\ln x$ \nomenclature[clm]{$\ln$}{ logarithm (extended for $x$ negative or $0$) \nomrefpage} and $\ln 0=0$.

\begin{dfn}\label{dfn_log_functions}
 A {\bf $\fln$-function on a definable set $X$}\index{ln-function@$\fln$-function} is a function $f$ of type
 \begin{equation}\label{eq_log_function_dfn}
f=P(a_1,\dots,a_k,\ln a_1,\dots,\ln a_k),                 
                 \end{equation}where $P$ is a polynomial and the $a_i$'s are definable functions on $X$.
\end{dfn}

\begin{rem}\label{rem_dfn_log_functions}
If $C_1,\dots,C_ l$ is a definable partition of $X \in \s_n$ and if $f:X \to \R$  is a function such that $g_i:=f_{|C_i}$ is a $\fln$-function for each $i$ then   $f$ is itself a $\fln$-function. This follows from the definition.
\end{rem}

The following proposition gives a description of  $\fln$-functions which is derived from the Preparation Theorem and can be regarded as a Preparation Theorem for $\fln$-functions.

\begin{pro}\label{pro_fln_function_2nd_form}
Given a $\fln$-function $f:X \to \R$, $X\in \s_n$, there exists a cell decomposition of $\R^n$ compatible with $X$ such that on every cell $C\subset X$ we have for  $x=(\xt,x_n)\in X\subset \R^{n-1}\times \R$:
\begin{equation}\label{eq_fln_fn_2nd_form}
 f(x)=\sum_{i,j,k=0}^N \mu_i(x)\,c_j(\xt)\,\ln^k (x_n-\theta(\xt)),
\end{equation}
 where $N\in \N$, $\theta$ is a definable function on the basis $D$ of $C$ (independent of $i,j,$ and $k$), the $\mu_i$'s are reduced functions on $C$ with translation $\theta$, and  the $c_j$'s  are $\fln$-functions on $D$.
\end{pro}

\begin{proof}
 Consider a function $f$  as in (\ref{eq_log_function_dfn}) and apply the Preparation Theorem to the $a_i$'s. This provides
a cell decomposition compatible with $X$
such that on every cell $C\subset X$, every $a_i$ can be written  (see Lemma \ref{lem_meme_morphisme}) \begin{equation}\label{eq_prep_proof_integration}
                                                        a_i(\xt,x_n)=|x_n -\theta(\xt)|^{r_i}b_i
(\xt)U_i(\xt,x_n-\theta(\xt)),               \end{equation}
where   $b_i$ and $\theta$ are analytic definable functions, $r_i\in \Q$, and $U_i$ is an $\la$-unit of the cell $C$.
  Let us fix such a cell $C$ and observe that the latter equality implies:
  $$\ln a_i(x)=r_i \ln |x_n-\theta(\xt)| +\ln b_i(\xt) +\ln U_i(\xt,x_n-\theta(\xt)). $$
 Since   $U_i$ is a unit, $\ln U_i$ is a
 definable function. By (\ref{eq_log_function_dfn}) and (\ref{eq_prep_proof_integration}), we thus can see that $f$ can be expressed as a  sum of type:
\begin{equation}\label{eq_fln_fonction_terme}
\sum_{i,j,k=0}^N \mu_i(x) c_j(\xt) \ln^k |x_n -\theta(\xt)|,
 \end{equation}
where each $\mu_i$ is a definable function,   each $c_j$ is a $\fln$-function on the basis of $C$, and $N \in \N$.

The expression appearing in (\ref{eq_fln_fonction_terme}) is not completely satisfying because we are not sure that each $\mu_i$ can be reduced with translation $\theta$.    To overcome this problem, we are going to make use of an argument which is similar to the one that we used in the proof of Lemma \ref{lem_meme_morphisme} (see (\ref{eq_reduced_theta12}) and (\ref{eq_reduced_theta21})).
 Since each $\mu_i$ is definable, refining the cell decomposition if necessary, we can assume  it to be reduced on $C$. Denote by $\theta'$ the translation of the reduction on $C$ (for all $i$, see Lemma \ref{lem_meme_morphisme}).

Up to a  refinement, we can assume that $(\theta'-\theta)$, $(x_n-\theta')$, and $(x_n-\theta)$ are of constant sign on $C$ and that their respective absolute values are comparable with each other. If $|x_n-\theta| \le |x_n-\theta'|$  then (see (\ref{eq_reduced_theta12}) and (\ref{eq_reduced_theta21}))   $|x_n-\theta'|$ is reduced with translation $\theta$. This implies that  $\mu_i$ is itself reduced with translation $\theta$ and we are done.
So, we can assume that $|x_n-\theta'| \le |x_n-\theta|$. We now distinguish two cases:

\underline{Case $1$}: $|x_n - \theta'| \leq |\theta' - \theta|$ on $C$.

One can easily see that in this case  $\frac{x_n - \theta'}{\theta'- \theta}$ takes values in $(-1,+\infty)$  and is  bounded away from $-1$ and $+\infty$ on $C$ (using that $|x_n-\theta'| \le |x_n-\theta|$). Since $u(h):=\frac{\ln (1+h)}{h}$ extends to a nowhere vanishing analytic function on $(-1,+\infty)$, we thus see that
$$ \ln (1+\frac{x_n - \theta'}{\theta'- \theta})=\frac{x_n - \theta'}{\theta'- \theta}\,u\left(\frac{x_n - \theta'}{\theta'- \theta}\right)$$ 
is reduced  with translation $\theta'$. As a matter of fact, if we write 
\begin{equation*}\label{eq_reduction_expansion}
  \ln (x_n-\theta) = \ln (\theta'-\theta)+ \ln (1+\dfrac{x_n - \theta'}{\theta'- \theta}),     
        \end{equation*}
 and plug this equality into (\ref{eq_fln_fonction_terme}), we see that $f$ has the desired form (with translation $\theta'$ in this case).

\underline{Case $2$}: $|x_n - \theta'| \geq |\theta' - \theta|$ on $C$. This case is addressed analogously (see (\ref{eq_reduced_theta21})).
\end{proof}

\begin{pro}\label{pro_limite}
  Let $f:X \times \R \to \R$ be a  $\fln$-function, $X\in \s_n$. If $f$ is bounded then $\lim_{\ep \to 0^+} f(x,\ep)$ exists for all $x\in X$ and defines a $\fln$-function on $X$.
\end{pro}
\begin{proof}
Let $\D$ be the cell decomposition of $\R^{n+1}$ compatible with $X\times \R$ provided by Proposition \ref{pro_fln_function_2nd_form} (applied to $f$).
We may assume $\D$ to be compatible with $X \times \{0\}$. Fix a cell $D \in \pi(\D)$, where $\pi:\R^{n+1}\to \R^n$ is the canonical projection. There is a unique cell $C$ of $ \D$  which is a band  $(0,\zeta)$, with $\zeta:D \to (0,+\infty]$   definable function.

By  Proposition \ref{pro_fln_function_2nd_form}, there are $\fln$-functions $c_0,\dots,c_N$  on $D$,  a definable function $\theta$  on $D$, as well as some reduced functions $\mu_0,\dots,\mu_N$ on $C$ (with translation $\theta$) such that we have on $C$ \begin{equation}\label{eq_log_ep}
                     f(x,\ep)=\sum_{i,j,k=0} ^N \mu_i(x,\ep)c_j(x)\ln^k(\ep-\theta(x)) .                                                                                                                                                                                                                                \end{equation}
 Up to a
refinement of our cell decomposition,
we may assume that either $\theta\equiv 0$ or $\theta$ does not vanish on $D$.
 If  $\theta(x)$
is
nonzero then $ \ln (\ep-\theta)$ tends to $\ln \theta$ as $\ep \to 0^+$ and $\lim_{\ep \to 0^+} \mu_i(x,\ep)$ gives rise to a definable function (this limit is finite for $\mu_i$ is reduced with translation $\theta$). The result is thus clear in this case and we will suppose $\theta\equiv 0$.

 By Proposition \ref{pro_puiseux_non_cont}, there is a definable partition $\Pa$ of $D$ into definable manifolds  such that for every $B \in \Pa$  and each $i$, $\mu_i(x,\ep)$ coincides with a Puiseux series in $\ep$ with analytic coefficients on  $(0,\xi)$, where $\xi<\zeta_{|B}$ is a positive continuous definable function on $B$.  Fix $B\in \Pa$. By (\ref{eq_log_ep}), $f$ itself then may be expressed on $(0,\xi)$ as a convergent series  $\sum_{j=\nu}^\infty \sum_{k=0} ^N \alpha_{jk}(x) \ep ^{\frac{j}{p}} \ln^k \ep$, where the $\alpha_{jk}$'s are $\fln$-functions on $B$, $p \in \N^*$, $N \in \N$, and $\nu \in \Z$. Every monomial $\alpha_{jk}(x) \ep ^j \ln ^k \ep$ for which $(j, k)$ is nonzero tends to $0$ or $\pm\infty$ as $\ep$ tends to $0$ (for each $x$). As $f$ is bounded, we see that $\alpha_{jk}(x)\equiv 0$ for all $(j,k) \in (\Z\setminus \N)\times \N$ and all $(j,k) \in \{0\} \times \N^*$ (since any two distinct such monomials cannot go to $\pm \infty$ at the same speed).
We deduce that $$\lim_{\ep \to 0^+} f(x,\ep)=\alpha_{00} (x),$$ which is a $\fln$-function on $B\in \Pa$, and consequently $\lim_{\ep \to 0^+} f(x,\ep)$ is a $\fln$-function (see Remark \ref{rem_dfn_log_functions}).
\end{proof}

\begin{thm}\label{thm_integration_des_log}
 If
$f:A\to \R$, $A \in \s_{m+n}$,   
is a  $\fln$-function and if $l \le n$  is such that $f_t(x):=f(t,x) \in L^1_{\hn^l}(A_t)$, 
 for all $t \in \R^m$, then  
 the
function $$g(t):=\int_{x\in A_t} f_t(x)\; d \hn^{l}(x),\quad  t\in \R^{m},$$
is a  $\fln$-function as well.  
\end{thm}
\begin{proof} 
The strategy of the proof will go as follows.  The description of $\fln$-functions given in (\ref{eq_fln_fn_2nd_form}) will provide a convergent expansion of $f(t,x)$ in $x_n$ and $\ln x_n$ (it will suffice to integrate with respect to the last variable $x_n$ for we will argue by induction). Integrating every term of this convergent series will provide an expansion of the same type for $g$,  yielding that $g$ is a $\fln$-function.

  Up to a definable diffeomorphism, we can assume that $A$ is included in  $[0,1]^{m+n}$. Take  a cell decomposition of $\R^{m+n}$ compatible with $A$  such that $f$ is continuous on every cell (see Remark \ref{rem_globally subanalytic_implies_analytic}).  It is enough to show the result for a cell $C\subset A$ (see Remark \ref{rem_dfn_log_functions}).

 If $ \dim C_t<l$ (for some  and hence for all $t$)  then  the result is obvious. If $ \dim C_{t}>l$  (for all $t$ in the basis of $C$)  then $f \equiv 0$ on $C$ (for if $f$ were nonzero at some $(t_0,x_0) \in C$ then $f_{t_0}$, which is continuous, would be bounded below away from zero near $x_0$ and thus could not be $L^1_{\hn^l}$, since every open neighborhood of $x_0$ in $C_{t_0}$ has dimension bigger than $l$). We thus can assume $\dim C_t=l$ for all $t$ in the basis of $C$. 
 Moreover, there is a linear map $\pi:\R^{m+n}\to \R^{m+l}$ preserving the $m$ first coordinates and inducing a diffeomorphism on  $C$ onto a cell of $\R^{m+l}$.  It means that this is no loss of generality to simply address the case $l=n$.

 Thanks to Proposition \ref{pro_fln_function_2nd_form}, we can assume that   $f(t,x)$ can be decomposed for $x=(\xt,x_n) \in C_t \subset \R^{n-1} \times \R$ as:
\begin{equation}\label{eq_aln} f_t(x)=\sum_{i,j,k=0}^N \mu_i(t,x)\cdot c_j(t,\xt)\cdot \ln^k  (x_n-\theta(t,\xt)) ,\end{equation}
where $\theta$ is a definable function on the basis of $C$,  the $\mu_i$'s are reduced functions on $C$ (with translation $\theta$), and the $c_j$'s are $\fln$-functions.

Fix $i \le N$. Up to a change of variables  (without changing notations) of type $(t,x)
\mapsto (t,\xt, x_n +\theta(t,\xt))$, we can  assume that $\theta\equiv 0$. As $\mu_i$ is a reduced function on $C$ with translation $\theta\equiv 0$, it can be written for $(t,\xt,x_n)$ in $C$
\begin{equation}\label{eq_mui}\mu_i(t,x)= b(t,\xt)\, x_n^{r}\, U (t,x),\quad r \in \Q,\end{equation}
where $b$ is a definable function, $r\in \Q$, and $U$ is a unit ($b,r$, and $U$ of course depend on $i$), i.e., a function that can be written $\psi \circ W$ with $W$ bounded mapping of type
$$W(t,x)=(u_1(t,\xt),\dots,u_p(t,\xt), v(t,\xt)x_n^\frac{1}{q},w(t,\xt)x_n^{-\frac{1}{q}}), $$
 with $u_1,\dots,u_p,v,w$  definable functions, $q \in \ns$, and $\psi$ analytic on a neighborhood of $cl(W(C))$. 

Before integrating $f_t$, we need to break the unit $U$ into two  series, one with negative powers in $x_n$ and one with nonnegative powers. This is the purpose of the claim below.
 Set first for simplicity for $(t,x)=(t,\xt,x_n)$ in $C$:$$W_1(t,x):=(u_1(t,\xt),\dots,u_p(t,\xt),v(t,\xt)\,x_n^\frac{1}{q},v(t,\xt)w(t,\xt))$$
 as well as $$W_2(t,x):= (u_1(t,\xt),\dots,u_p(t,\xt),w(t,\xt) \, x_n^{-\frac{1}{q}},v(t,\xt)w(t,\xt)).$$

\noindent{\bf Claim.} There is a partition of $C$ into cells such that on each element of this partition, $U$ can be written
 \begin{equation}\label{eq_splitting_pr_integration}
 U(t,x)=\psi(W(t,x))=\Psi_1(W_1(t,x))+\Psi_2(W_2(t,x)) 
\end{equation}
where $\Psi_1$ and $\Psi_2$ are two analytic functions on a neighborhood of $cl(W_1(C))$ and $cl(W_2(C))$ respectively. 

To prove this, we shall distinguish three cases. Given $\eta>0$, splitting $C$ into several cells, we can assume that one of the following situations occurs on $C$:

\noindent \underline{{\bf Case 1:}} $|v(t,\xt) x_n^{\frac{1}{q}}|\ge \eta$. In this case,  we can write $$w(t ,\xt)x^{-\frac{1}{q}}=\Lambda(v(t,\xt)x_n^\frac{1}{q},v(t,\xt)w(t ,\xt)),$$ with $\Lambda(y,z)=\frac{z}{y}$. Hence, since $U=\psi\circ W$, it is enough to set $\Psi_2=0 $ and $\Psi_1(u,y,z):=\Psi(u,y,\Lambda(y,z))$, which is analytic  on $cl(W_1(C))$ (since  $s\mapsto \frac{1}{s}$ is analytic on the complement of the origin).

\noindent\underline{{\bf Case 2:}} $|w(t ,\xt)x_n^{-\frac{1}{q}}| \ge \eta$. This case is addressed completely analogously. We set $\Psi_1=0$ and $\Psi_2$ is defined in a similar way as $\Psi_1$ in case $1$.

\noindent \underline{{\bf Case $3$:}} We suppose that $|v(t,\xt) x_n^{\frac{1}{q}}|$ and $|w(t ,\xt)x_n^{-\frac{1}{q}}| $ are both smaller than $\eta$. If $\eta$ is chosen small enough  then the result directly follows from
 Lemma
 \ref{lem_splitting} (applied to  the function $\psi$).  This completes the proof of the claim.

  Write now $C$ as $(\zeta,\zeta')$, where
 $\zeta$ and $\zeta'$ are two functions on the basis $D$ of $C$ satisfying $\zeta<\zeta'$. Set for simplicity $\xi_{t,\ep}:=\zeta_{t}+\ep(\zeta_{t}'-\zeta_{t})$ as well as $\xi_{t,\ep}':=\zeta_t'-\ep(\zeta_t'-\zeta_t)$, for $\ep>0$, and observe that  $[\xi_{t,\ep},\xi'_{t,\ep}]\subset C_t$.  By Proposition \ref{pro_limite} and Lebesgue's Dominated Convergence Theorem, it is enough to show that
$$\lambda(t,\ep):=\int_{[\xi_{t,\ep},\xi_{t,\ep}']}
   f_t \;d \hn^n$$
 is a $\fln$-function.
    Notice that $\lambda(t,\ep)=\int_{\xt \in D_t} h(t,\ep,\xt) d\hn^{n-1}(\xt) $ where we have set for $\xt \in D_t$
  $$h(t,\ep,\xt):= \int_{\xi_{t,\ep}(\xt)} ^{\xi_{t,\ep}'(\xt)} f_t(\xt,x_n) dx_n. $$
  Since we can argue by induction on $n$, it is enough to establish that $h$ is a $\fln$-function.
  The above claim (see (\ref{eq_splitting_pr_integration})) implies that $f_t$  can be decomposed on $C$ as the sum of two convergent series (via (\ref{eq_aln}) and (\ref{eq_mui})). As we can integrate  these two series by integrating every term,   it is enough to
deal with each monomial $ x_n^{j/p} \cdot\ln^k  x_n
 $,  $j\in \Z,k\le N$, $p \in \ns$,  appearing in the convergent expansion.

 These monomials may easily be integrated by finitely many integrations by parts. 
   Namely, for $k=0$ or $\frac{j}{p}=-1$, a straightforward computation of antiderivative yields that $\int_{\xi_{t,\ep}(\xt)} ^{\xi_{t,\ep}'(\xt)}  x_n^{j/p} \cdot\ln^k  x_n
 \, dx_n$ is  a
$\fln$-function. For $k$ positive integer or $\frac{j}{p}\ne -1$, after a suitable integration by 
 parts, one gets a new integral of the same type with a lower
exponent in $\ln x_n$.
\end{proof}


 Together with Proposition \ref{pro_familles_L_1}, this theorem implies:
 
 \begin{cor}\label{cor_integration_fonctions}
 If $f:A \to \R$ is definable, with $A\in \s_{m+n}$,  then for each  $l \le n$
 $$g(t):=\int_{x\in A_t} f_t(x)\; d \hn^{l}(x),$$
 defined on $E:=\{\tim: f_t \in L^1_{\hn^l}(A_t)\},$
 is a  $\fln$-function. 
 \end{cor}

In particular, in the case of the constant function $f:A\to \R$, $f\equiv 1$, we get: 

\begin{cor}\label{cor_volume_lionrolin}
	Let $A \in \s_{m+n}$ and $l \le n$. The function $g(t):=  \hn^l(A_t)$, defined on the definable set
	$\{\tim: \hn^l(A_t)<\infty \} ,$
	is a $\fln$-function.
\end{cor}

\section{On the measure of definable sets and families}\label{sect_measure}\subsection{The function $\psi(X,r)$ and the density $\theta_X$}
 Given a set  $X \in \s_n$, we set for $x\in \R^n$ and $r\ge 0$: $$\psi(X,x,r):=\hn^l(X\cap \bou (x,r)),$$
where $l=\dim X$. When $x=0$, we will shorten\nomenclature[cm]{$\psi(X,x,r)$}{germ of the volume of $X$ near $x$\nomrefpage} \nomenclature[cn]{$\psi(X,r)$}{germ of the volume of $X$ near $0$\nomrefpage}
$ \psi(X,x,r)$ into $\psi(X,r)$. 

In this section, we describe some properties of  $\psi$ and introduce the notion of density, sometimes called  the Lelong number.  It is easy to see that if $A \in \St_{m+n}$ then $E^l:=\{\tim :\dim A_t=l\}$ also belongs to $\s_{m+n}$ for every $l$.
Hence, applying Proposition \ref{borne unif pour les volumes} and Corollary \ref{cor_volume_lionrolin} to $E^l$ for each $l$, we see:

\begin{pro}\label{pro_volume_lionrolin}
For any  $A \in \St_{m+n}$, the function $(t,x,r) \mapsto \psi(A_t,x,r)$ is a $\fln$-function satisfying $\psi(A_t,x,r)\lesssim r^{l_t}$, where $l_t:=\dim A_t$, for $(t,x,r)\in \R^{m}\times \R^n\times [0,+\infty)$.
\end{pro}

In particular,  in the case $m=0$, thanks to Proposition
\ref{pro_puiseux_avec_parametres}, this entails that for every $X \in\s_n$ there exist positive
 integers $p$ and $N$ as well as real numbers $a_{i,j}$, $i \in \N$, $j \leq N$, such that for $r>0$
 small enough:
 \begin{equation}\label{eq_volume}\psi(X,r) =\sum_{i=0}^\infty \sum_{j=0} ^N
 a_{i,j}\, r^{\frac{i}{p}}\ln^{j} r.\end{equation}
It is worthy of notice that the above proposition entails that the first term of this expansion is of order at least $l:=\dim X$.

\begin{pro}\label{pro_density}
Let $X \in \St_n$ be of dimension $l$. Given $x\in \R^n$, the limit
$$\theta_X (x):=\lim_{r \to 0} \frac{\psi(X,x,r)}{ \hn^l(\bou (0_{\R^l},1))\cdot r^l}$$
 exists and is finite. 
It is called the {\bf density of $X$ at $x$}.\index{density}\nomenclature[cp]{$\theta_X$}{density of $X$ (Lelong number)\nomrefpage}
\end{pro}
\begin{proof}
 By (\ref{eq_volume}), the  limit exists, and, by Proposition \ref{pro_volume_lionrolin}, it is finite. 
\end{proof}
 The density is sometimes called the {\bf Lelong number}\index{Lelong number}. It is easily checked that if $X$ is a
smooth manifold then $\theta_X\equiv 1$ on $X$.  If $X$ is a complex analytic subset of $\mathbb{C}^n$ (that we can regard as a
definable subset of $\R^{2n}$) and $x\in X$ then $\theta_X(x)$ is equal to the multiplicity of
$X$ at $x$ \cite{dra69}. The notion of density may thus be considered as a real counterpart of the complex notion
of multiplicity. 

\begin{thm}
 The function $\theta_X$ is a $\fln$-function on $X$, for all $X\in \s_n$. 
\end{thm}
\begin{proof}
 Let $X \in \St_n$. 
 Since  Proposition \ref{pro_volume_lionrolin}  yields that $f(x,r):=\frac{\psi(X,x,r)}{\hn^l(\bou(0_{\R^l},1))\cdot r^l}$ is a bounded
$\fln$-function (where $l:=\dim X$),  Proposition \ref{pro_limite}  gives the result.
\end{proof}

 \subsection{Uniform bounds}\label{sect_uniform_bounds} We give some estimates of the $\hn^l$-measure of germs of globally subanalytic sets that will be needed in the proof of Theorem \ref{thm_alpha_approx_invariance_des_vol}.   It will be very important  that
the constants given by  the following  propositions do not depend
on the parameters.

Given $X \in \s_n$ and  $\ep\ge 0$, define the \index{ep neighborhood@ $\ep$-neighborhood} {\bf $\ep$-neighborhood of $X$} as:
$$X_{\le \ep}:=\{x\in \R^n: d(x,X)\le \ep  \}.$$\nomenclature[cq]{$X_{\le \ep}$}{$\ep$-neighborhood of $X$\nomrefpage}
We also set:
$$X_{= \ep}:=\{x\in \R^n: d(x,X)= \ep  \}.$$\nomenclature[cr]{$X_{= \ep}$}{level surface $d(x,X)=\ep$\nomrefpage}
The  proposition below enables  us to bound uniformly with
respect to $\varepsilon$ the measure of the $\ep$-neighborhoods of the
fibers of a globally subanalytic family.

\begin{pro}\label{pro_vol_A_t_B_t} Let $A \in \s_{m+n}$ be such that $\sup_\tim diam (A_t)<\infty$ (see (\ref{eq_diameter_def}) for $diam$), and let $k<n$ be an integer.  There is  $C>0$ such that  for all $\tim$ for which $\dim A_t\le k$   and all $\ep>  0$ we have
\begin{equation}\label{eq_ligne_de_niveau}
\hn^{n-1}(A_{t,=\ep}) \le C\ep^{n-k-1}. 
\end{equation}
Moreover, given $B$ in $\s_{m+n}$ and an integer $l>k$,  there is $C>0$ such that  for all $\tim$ for which $\dim B_t\le l$ as well as $\dim A_t\le k$,   and all $\ep>  0$, we have
\begin{equation}\label{eq_vol_A_t_B_t}
\hn^l(A _{t,\leq \varepsilon}\cap B_t) \leq C\,\varepsilon^{l-k}.
\end{equation}
\end{pro}
\begin{proof}
We establish these two estimates simultaneously by induction on $n$, both statements being obvious in the case $n=1$ thanks to Corollary \ref{cor_cc_families} ((\ref{eq_ligne_de_niveau}) follows from the fact that $(n-k-1)\le 0$, and (\ref{eq_vol_A_t_B_t}) from the fact that the assumptions then force either $ A_t$ to be a finite set or $l$ to be bigger than $n$).  We start with the induction step of (\ref{eq_ligne_de_niveau}).  Given $P\in \G_{n-1}^n$ and a positive integer $j$, we have  $K_j^P(A_{t,=\ep})\subset \pi_P(A_t)_{\le \ep}$. Moreover,  the family of sets
$\pi_P(A_{t})$, $\tim, P\in \G^n_{n-1}$, is  definable.  Hence, thanks to the induction hypothesis (identifying $P$ with $\R^{n-1}$ and applying (\ref{eq_vol_A_t_B_t})), we can conclude that for $\tim$ and $\ep\ge 0$
$$\hn^{n-1}(K_j^P(A_{t,=\ep}))\le \hn^{n-1}(\pi_P(A_t)_{\le \ep})\lesssim \ep ^{n-k-1} ,$$
which means that the result is a direct consequence of (\ref{cauchy_crofton_formula_rewritten}).

 We now turn to perform the induction step of (\ref{eq_vol_A_t_B_t}), 
 starting with the case $l=n$, for which we can assume $B_t=\R^n$, for all $t$. Let, for $x\in \R^n$, $\rho_t(x):=d(x,A_t)$.  Observe that  $|\pa_x\rho_t|=1$
 at each $x$ where
$\rho_t$ is  differentiable (as $\rho_t$ is $1$-Lipschitz, we have $|\pa_x\rho_t|\le 1$; on the other hand, if $y$ is a point of $cl(A_t)$ that realizes the distance $\rho_t(x)$ then $\rho_t(z)$ is decreasing as fast as $|z-y|$ when $z$ goes to  $y$ from $x$, which shows  $|\pa_x\rho_t|\ge 1$), so that $J_x (\rho_t)\equiv 1$ on a definable dense subset of $\R^n$, which, by (\ref{eq_coarea_formula}), yields:
$$ \hn^n (A_{t,\le \ep})\le \int_0 ^\ep \hn^{n-1}(A_{t,=\alpha})d\alpha \overset{(\ref{eq_ligne_de_niveau})}{\lesssim}   \ep^{n-k}, $$
as required. 

It remains to show the result in the case $l<n$. Observe that in this case, $K_j^P(A_{t,\le \ep}\cap B_t)$ is included in $\pi_P(A_{t})_{\le \ep}\cap B_t$, for every  $j\in \N^*$ and $P\in \G^n _l$.  Hence, thanks to (\ref{cauchy_crofton_formula_rewritten}) and the induction hypothesis  (identifying $P$ with $\R^l$), we see that the desired estimate holds.
\end{proof}

\begin{pro}\label{volume des voisinges}
 Given $A\in \s_{m+n}$,  there exists a constant $C$ such that for all $\tim$,   all  $\varepsilon \in (0,1]$, and  all $ r\ge 0$, we have:$$\psi(A _{t,\leq \varepsilon},r) \leq C\, r^{n-1}\varepsilon + \hn^n(A _t\cap \bou (0_{\R^n},r)).$$
In particular, if $\dim A_t<n$, we then have for all such $r$, $t$, and $\ep$:
\begin{equation}\label{inegalit volume des voisinges}
\psi(A _{t,\leq \varepsilon} ,r) \leq Cr^{n-1}\varepsilon.
\end{equation}
\end{pro}

\begin{proof} We will use a method which is similar to the one we used in the proof of the preceding proposition.  We first establish the desired estimate assuming $\dim A_t<n$, for all $\tim$.
  As explained in the proof of the preceding proposition,  $J_x (d(x,A_t))\equiv 1$ on a definable dense subset of $\R^n$.
Moreover, since $\dim A_{t,=\alpha} \le n-1$ for  every $\alpha>0$ and $\tim$, by Proposition \ref{borne unif pour les volumes}, there exists $C>0$ such that
for all such $t$ and $\alpha$ we have
$\hn^{n-1}(A_{t,=\alpha}\cap \bou(0,r)) \leq Cr^{n-1}$. We thus can write:
\begin{eqnarray*}
\psi(A _{t,\leq \varepsilon},r) &=& \int_{A _{t,\leq
\varepsilon} \cap \bou (0,r)}d \hn^n\\
&\overset{(\ref{eq_coarea_formula})}\leq& \int_0^\varepsilon
\hn^{n-1}(A_{t,=\alpha}\cap \bou(0,r))\;d \hn ^1(\alpha) \quad\; \mbox{(since $J_x(d(x,A_t))\equiv 1$)}\\ &\leq& Cr^{n-1} \varepsilon,
\end{eqnarray*}
 establishing (\ref{inegalit volume des voisinges}).  To prove the result in general, let us set 
$E_t:=\delta(A _t)$.  Since  $A _{t,\leq \varepsilon}
\subset E_{t,\leq\varepsilon} \cup A _t$ and $\dim E_t<n$ for all $t$, the
result follows from  (\ref{inegalit volume des voisinges}) for $E_t$.
\end{proof}

\begin{pro}\label{psit est lips}
Let $l\in \N$ and  $A \in \s_{m+n}$ be such that $\dim A_t=l$ for all $\tim$.
There is $C>0$  such that for each $\tim$, we have for all $r$ and  $r'$ small enough positive real numbers satisfying $r'\le r$:
$$|\psi(A _t,r)-\psi(A _t,r')| \leq Cr^{l-1}|r-r'|.$$
\end{pro}

\begin{proof}
For $\tim$ let $ \lambda_t$ denote  the restriction to $A_{t,reg}$ of the function $\rho:\R^n \to \R$, defined by $\rho(x):= |x|$.  We wish to estimate $\psi(A_t,r)$
by means of (\ref{eq_coarea_formula}), which requires to 
  establish that 
$J_{\lambda_t}(x)$ goes to $1$ as $x \in A_{t,reg}$ tends to $0$ for every $\tim$.

Since
  $J_x(\lambda_t)=|\pa_x \lambda_t|$,
and since, for almost each  $x\in A_{t,reg}$, the vector $\pa_x \lambda_t$ is the projection of $\pa_x \rho=\frac{x}{|x|}$ onto $T_x A_{t,reg}$, it is enough to check that the angle between $x$ and $T_x A_{t,reg}$ tends to zero as $x \in A_{t,reg}$ tends to zero (for each $t$). Indeed, if otherwise, by Curve Selection Lemma (Lemma \ref{curve_selection_lemma}), we could find an analytic arc $\gamma(s)$ in $A_{t,reg}$ tending to zero (as $s\to 0^+$) such that the angle between $\gamma(s)$ and $T_{\gamma(s)} A_{t,reg}$ is bounded below away from zero. Since  $\lim_{s\to 0^+}  \frac{\gamma(s)}{|\gamma(s)|}=\lim_{s\to 0^+} \frac{\gamma'(s)}{|\gamma'(s)|}$ and $\gamma'(s)\in T_{\gamma(s)} A_{t,reg}$, this is a contradiction which yields that $J_{\lambda_t}(x)$ goes to $1$ when $x \in A_{t,reg}$ tends to $0$.

  Fix now $\tim$ and choose a positive real number $r_0$ (depending on $t$) sufficiently small for the sets  $\sph (0_{\R^n},r) \cap A_t$ to be of dimension smaller than $l$ for all $r\le r_0$.
 As $A_t\cap \sph (0_{\R^n},r)\subset A_t\cap \bou (0_{\R^n},2r),$ by Proposition \ref{borne unif pour les volumes}, we know that there is a constant $C$ independent of $t$ and $r$ such that for all $r\le r_0$:
 \begin{equation}\label{eq_norm_link}
  \hn^{l-1}(A_t\cap \sph (0_{\R^n},r) )\le C r^{l-1}.
 \end{equation}
We thus can write for $r' \leq r\le r_0 $:
\begin{eqnarray*}
|\psi(A _t,r)-\psi(A _t,r')| &=& |\int_{A _t\cap
\bou (0,r)\setminus \bou (0,r')}d\hn ^l|\\ &\leq&2\int_{A_t \cap
\bou (0,r)\setminus \bou (0,r')}J_x(\lambda_t)\, d\hn ^l(x)\;
(\mbox{since }J_x(\lambda_t) \; \mbox{tends to} \, 1)\\
&\overset{(\ref{eq_coarea_formula})}=&
2\int^r_{r'}\hn^{l-1}(A _t \cap \sph (0_{\R^n},s)) \,d\hn ^1(s),
\end{eqnarray*}
which, by (\ref{eq_norm_link}), yields the claimed estimate.
\end{proof}
\begin{rem}\label{rem_small_enough}
 In the above proposition, by ``for each $\tim$, we have for all $r$ and  $r'$ small enough positive real numbers satisfying $r'\le r$'', we mean that for each $\tim$, there is $r_0(t)>0$ such that the claimed estimate holds for all $r$ and  $r'$  satisfying $0<r'\le r\le r_0(t)$, i.e. it should be understood that this ``small enough'' depends on $t$. This will be the same in Definition \ref{def alpha approx}, Theorem \ref{thm_alpha_approx_invariance_des_vol}, and Proposition \ref{pro_multiplicites_b}.
\end{rem}

\section{Measure and $\alpha$-approximations of the identity}\label{sect_volume_alpha_approx}
 We are going to show  that families  of homeomorphisms that are close to the identity can induce some stability of the measure even if they are not Lipschitz.  This fact is possible because we work with globally subanalytic families of sets.  The considered homeomorphisms will however not be assumed to be globally subanalytic.

\begin{dfn}\label{def alpha approx}
Let  $A$ and  $B$ in $\s_{m+n}$ and let $\alpha:\; (0,\eta)\times \R^m\rightarrow \R$ be a positive definable function. 
We call \textbf{$\alpha$-appro\-ximation of the identity}   a 
family of germs of homeomorphisms  (not necessarily definable) $ h_t  : (A_t , 0) \rightarrow (B_t,0)   $, $\tim$, 
 such that for each $\tim$ we have for all $r>0$ small enough
\begin{equation}\label{eq ineg def alpha approx}
 |h_t(x)-x| \leq \alpha(r,t),
 \end{equation}
for all $ x \in \bou (0,r)\cap A_t$, and
 \begin{equation}\label{eq ineg def alpha approx h^-1}
 |h_t^{-1}(x)-x| \leq \alpha(r,t),
 \end{equation}
for all $ x \in \bou (0,r)\cap B_t$.
\end{dfn}

Again, we stress the fact that in the above definition, as well as in Theorem \ref{thm_alpha_approx_invariance_des_vol} and Proposition \ref{pro_multiplicites_b} below, the requirement $r>0$ small enough ``depends on $t$'' as explained in Remark \ref{rem_small_enough}.
The main purpose for introducing the notion of $\alpha$-approximation of the identity is the study of the variation of the density on a definable set (section \ref{sect_densite}). The theorem below that compares the measure of the fibers of two families that are related by an $\alpha$-approximation of the identity is however of its own interest. It is  easy to produce examples of non subanalytic families for which this theorem fails.

\begin{thm}\label{thm_alpha_approx_invariance_des_vol}
Let  $(A_t)_{\tim} $ and  $(B_t)_{\tim}$ be two definable families of $l$-dimen\-sional subsets  of $\R^n$ and
let $h_t: (A_t,0) \rightarrow (B_t,0)$, $\tim$, be an $\alpha$-approximation of the identity,  with $\alpha:\; (0,\eta) \times \R^m\rightarrow \R$ definable positive function.
There is  $C>0$ (independent of $\alpha$) such that for every $t \in \R^m$ we have for all $r>0$ small enough:
$$|\psi(A_t,r)-\psi(B_t,r)| \leq C \alpha(r,t)\cdot r^{l-1}.$$
\end{thm}

The idea  of the proof of this theorem is that since Cauchy-Crofton's formula reduces the computation of the Hausdorff measure of a set to the computation of the cardinal of generic fibers of generic projections (restricted to the considered set), comparing the $\hn^l$-measures of  $A_t \cap \bou (0,r)$ and   $B_t \cap \bou (0,r)$ reduces to compare the respective cardinals of  $\pi_P^{-1}(x)\cap A_t \cap \bou (0,r)$ and $\pi_P^{-1}(x) \cap B_t \cap \bou (0,r)$, for $P \in \G_l^n$ and $x\in P$ generic.
We therefore start with a proposition which, assuming that we are in the situation of Theorem \ref{thm_alpha_approx_invariance_des_vol}, yields that the respective cardinals of generic fibers of projections are the same on the complement of a set whose measure is bounded in terms of $\alpha$.

\begin{pro}\label{pro_multiplicites_b}
Let  $(A_t)_{\tim} $ and  $(B_t)_{\tim}$ be two definable families of $l$-dimen\-sional subsets  of $\R^n$  and
let  $h_t: (A_t,0) \rightarrow (B_t,0)$ be an $\alpha$-approximation of the identity,  with $\alpha:\; (0,\eta) \times \R^m\rightarrow \R$ definable positive function.
There are a constant  $C$ (independent of $\alpha$) and a definable family of sets $Z_t(P,r)$, $P \in
\G_{l}^n$, $r>0$, $  \tim$,  satisfying for each such $P $ and $t$, for all $r>0$ small enough: \begin{enumerate}
                                    \item   $\;\hn^l(Z_t(P,r)) \leq C \alpha(r,t)
\cdot r^{l-1}.$
\item For any $ x \in
  P\cap \bou (0,r) \setminus Z_t(P,r)$:
\begin{equation}\label{claim cardinal des fibres}
  \card\,  (\pi_P^{-1}(x) \cap A _t \cap \bou (0,r))= \card\,  (\pi_P^{-1}(x)
 \cap B_t \cap \bou (0,r)).
\end{equation}
                                   \end{enumerate}
\end{pro}

\begin{proof}
We first define $Z_t(P,r)$ (see (\ref{eq_KPRT})). In this proof, we will sometimes identify an element $P$ of $\G_{l}^n$ with $\R^l$ (mentioning it) so that, given a subset  $X\subset P$ and $\ep >0$, we write $X_{\le \ep}$ for the set of points  $x\in P$ satisfying $d(x,X)\le \ep$.

  Let $A'_t$ be the set of points of $A_t$ at which $A_t$ fails to be an analytic manifold of dimension $l$.  For every $P \in \G_l^n$ and $\tim$ let $S_t(P)$ be the union of $A_{t}'$ together with the set of  points of $A_{t}\setminus A'_t$ at which $\pi_P$ fails to be a submersion.
Let then
$$H_t(P,r):= \pi_P (S_t(P))\cap \bou (0,r). $$
 For every $t$ and $P$, $\dim H_t(P,r)\le \dim \pi_P\lepa S_t(P)\ripa<l$, by Sard's Theorem. As the family   $H_t(P,r)$, $P \in \G_l^n,\, r\in \Qp,\,\tim$, is definable, by  Proposition \ref{volume des voisinges}
(identify $P$ with $\R^l$) there is a constant $C$ such that for all $\tim$:
\begin{equation}\label{mesure de H(p,j,A)}
\hn^l(H_t(P,r)_{\leq 2\alpha(r,t) }) \leq C\alpha(r,t) r^{l-1}.
\end{equation}
Let now  $$\M_t(P,r):=\pi_P(A _t\cap \bou (0,r)\setminus \bou (0,r-2\alpha(r,t)
)),$$
and notice that  by  Proposition \ref{psit est lips},  there is  $C>0$  such that for all $\tim$:
$$\hn^l(\M_t(P,r)) \leq C \alpha(r,t)\cdot  r^{l-1}.$$
We then easily derive from
Proposition~\ref{volume des voisinges} (again identifying $P$ with
$\R^l$) that:
$$\hn^l(\mathcal{M}_t(P,r)_{\leq 3\alpha(r,t) }) \leq C \alpha(r,t)\cdot  r^{l-1},$$
for some  $C$ independent of $r$ and $t$. 
 Let us  now
define  $Z_t(P,r)\subset P$ as the union 
\begin{equation}\label{eq_KPRT}
 H_t(P,r)_{\le 2\alpha(r,t)} \cup H'_t(P,r)_{\le 2\alpha(r,t)} \cup
\mathcal{M}_t(P,r)_{\leq 3\alpha(r,t) }\cup \mathcal{M}'_t(P,r)_{\leq
	3\alpha(r,t) },
\end{equation}
where $H_t'(P,r)$ and $\mathcal{M}'_t(P,r)$ are defined in the same way as $H_t(P,r)$ and $\mathcal{M}_t(P,r)$ but replacing $A_t$ with $B_t$.
By the above estimates (the estimates obtained for $H_t(P,r)$ and $\mathcal{M}_t(P,r)$ clearly hold for  $H'_t(P,r)$ and $\mathcal{M}'_t(P,r)$ as well) we see that $(i)$ holds.

Let us now prove (\ref{claim cardinal des
fibres}). Fix $P \in \G_l^n$, $r>0$, $\tim$, as well as an element $x \in P \cap \bou (0,r) \setminus
Z_t(P,r) $, and set for simplicity $$j:=\card \; \pi_P^{-1}(x) \cap A_{t}\cap \bou(0,r)
 \quad \mbox{ and } \quad j':=\card \;
\pi_P^{-1}(x) \cap B_t\cap \bou(0,r).$$
We have to check that $j=j'$.
Remark  that by   definition of
$Z_t(P,r)$ we have
$$d(x,\pi_P(S_t(P))) > 2 \alpha(r,t),$$
which means that $\pi_P^{-1}(\bou (x,2\alpha(r,t)))\cap A_t\cap \bou (0,r)$ is exclusively constituted by nonsingular points at which $\pi_{P|A_t}$ is submersive. By the symmetry of the roles of $A_t$ and $B_t$, the analogous fact holds for $\pi_P^{-1}(\bou (x,2\alpha(r,t)))\cap B_t\cap \bou (0,r)$.

By Ehresmann's Theorem\footnote{A covering map above a simply connected set is always globally trivial.}, the intersection of   $\pi^{-1}_P( \bou (x,2\alpha(r,t)))$ with the set $A_{t,reg}\cap \bou (0,r)$ (resp. $B_{t,reg}\cap \bou (0,r)$)
 is thus the  union  of $j$ (resp. $j'$) connected components
 $C_1,\dots,C_j$ (resp. $D_1,\dots, D_{j'}$)
  and the restriction of   $\pi_P$ to every  $C_i$ (resp. $D_i$)
 is  a  homeomorphism  onto $\bou (x,2\alpha(r,t))\cap P$.

  Since  $x$ does not belong to
 $\M_{t}(P,r)_{\leq 3\alpha(r,t)}$, the ball 
of radius   $\alpha(r,t)$ centered at  $x$ does not intersect 
$\M_t(P,r)_{\leq2\alpha(r,t)}$. Hence, due to the definition of
$\M_t(P,r)$,
  every point  $y$  of $C_{j_0} \cap \pi_P^{-1}( \bou (x,\alpha(r,t)))$, with $j_0 \in \{1,\dots,j\}$, must belong to
$\bou (0,r-2\alpha(r,t))$,
 so that, by (\ref{eq ineg def alpha approx}), 
   the point $h_t(y)$ must belong to $
\bou (0,r)$.
 Moreover, again due to (\ref{eq ineg def alpha approx}),  we have $\pi_P(h_t(y)) \in \bou (x,2\alpha(r,t))$
 so that  $h_t(y)$
actually belongs  to 
 one of the $D_i$'s.  As  $C_{j_0}$ is  connected  and  the  $D_i$'s are  disjoint,
the integer $k\le j'$ for which  $h_t(y)\in D_k$ just depends on $j_0$ and not on the  point $y$ in $ C_{j_0}
\cap \pi_P^{-1}( \bou (x,\alpha(r,t)))$. Let us thus denote by $\sigma(j_0)$ this
integer.

In this way, we have defined a mapping $\sigma$  from $\{1,\dots
,j\}$ to $\{1,\dots,j'\}$. In order to show   $j'\leq j$,  it
suffices to establish that  $\sigma$ is  surjective  (by the symmetry  of the  roles
of $j$ and $j'$, it is enough to check $j' \leq
j$. Here, $j$ and $j'$ might be zero but the argument just above has shown that if $j\ne 0$ then $j' \ne 0$ and the argument below will show that if $j'\ne 0$ then $j\ne 0$).

  Let  $i$ be an integer between   $1$ and  $j'$, take a point
  $z\in \pi_P^{-1}(x) \cap D_i$, and  set
$y:=h_t^{-1}(z)$. Since  $x \notin \M'_t(P,r)$, the point $z$ belongs to
$\bou (0,r-2\alpha(r,t))$;
 this implies, via (\ref{eq ineg def alpha approx h^-1}),
that the point  $y$ belongs to  $\bou (0,r)$.
By (\ref{eq ineg def alpha approx h^-1}), it is clear that  $\pi_P(y)\in \bou (x,2\alpha(r,t))$. Thus,
 $y \in C_{i_0}$ for some  $i_0$, which
implies that $\sigma(i_0)=i$. 
\end{proof}

\begin{proof}[Proof of Theorem \ref{thm_alpha_approx_invariance_des_vol}]
Since $(A_t)_{\tim} $ and  $(B_t)_{\tim}$ are two definable families, $m_{A_t}$ and $m_{B_t}$ are bounded independently of $t$ (see Corollary \ref{cor_cc_families}).  Moreover, by
$(i)$ and $(ii)$ of Proposition   \ref{pro_multiplicites_b},  there is a positive constant $C$ such that  for each $\tim$,  $j\in \N$, and  $P\in \G^n_l$, we have for all $r>0$ small enough:
$$|\hn^l(K^P_j(A_t \cap \bou (0,r)))- \hn^l(K^P_j(B _t
\cap \bou (0,r)))|\leq C \alpha(r,t)\cdot  r^{l-1}. $$
 The result thus directly follows from  (\ref{cauchy_crofton_formula_rewritten}).
\end{proof}

\section{Variation of the density}\label{sect_densite}
As we noticed,  $\theta_X\equiv 1$ on the regular locus  of $X\in \s_n$. This raises a natural
question: how is  the density of a
globally subanalytic set affected by the geometry of the singularities?
The theorem below provides information on this issue.

\begin{thm}\label{thm_continuite_densite}
Let a locally closed set $X\in \s_n$ be stratified by a
stratification $\Sigma$.
\begin{enumerate}\item If  $\Sigma$ is $(w)$ regular
then  $\theta_X $ is
locally Lipschitz on the strata of $\Sigma$.
\item If  $\Sigma$ is Whitney $(b)$ regular then  $\theta_X$ is continuous on
the strata of $\Sigma$. 
\end{enumerate}
\end{thm}

\begin{proof}
[Proof of $(i)$]  For simplicity we will do the proof in the case where the stratification $\Sigma$ is reduced to two strata $S$ and $S'$ with $S \subset cl(S')\setminus S'$. The proof in the general case relies on the same idea and the reader is referred to \cite{vdensity}.  Up to a choice of coordinate system, we may assume $S= \R^k \times \{0\}$ and work nearby the origin.  We will carry out the proof in the case $k=1$. In the case where $S$ has higher dimension, one may apply the same argument to establish that the density is locally Lipschitz with respect to each coordinate of $S$ (which yields the local Lipschitzness of the density on $S$).
Given $t\in S$, define \begin{equation}\label{eq_A_t}
                           A_t:=\{x \in \R^n: (x+t) \in X \},
                          \end{equation}
so that the germ of $A_t$ at the origin is the translation of the germ of $X$ at $t$.

 We claim that, for $t$ and $t'$ in a neighborhood of $0$, there is an $\alpha$-approximation of the identity $h_{t,t'}:A_t \to A_{t'}$ , with $\alpha(r,t,t')= C r|t-t'|$ for some constant $C>0$ (here $A_t$ and $A_{t'}$ are regarded as families parameterized by two parameters $t'$ and $t$, constant with respect to $t'$ and $t$ respectively).

We start by defining a vector field. The desired family of homeomorphisms will then be given by the flow of this vector field. 
 Set   for $x$ in $S'$ \begin{equation}\label{eq_proj}
 v (x):=\frac{P_x(e_1)}{<P_x(e_1),e_1>},                                                                                                                                                                                                                                                                                                                                                                                                                                     \end{equation}
where $P_x$ stands for the orthogonal projection onto $T_x S'$, $e_1$ for the first vector of the canonical basis of $\R^n$, and $<,>$ for the euclidean inner product. Extend $v $ to $S$ by setting $v(x):=e_1$ for $x\in S$. It easily follows from the $(w)$ condition that there is $C>0$ such that for any $x \in S'$ and $t \in S$ sufficiently close to $0$:
\begin{equation}\label{eq_rugueux} |v (x)-v (t)|\leq C |x-t|.\end{equation}
Denote by $\phi $ the local flow of this vector field (defined on each stratum).  By (\ref{eq_proj}), we see that if $\pi:\R^n\to \R\times \{0_{\R^{n-1}}\}$ denotes the orthogonal projection then
\begin{equation}\label{eq_phi_pi}\phi(\pi(x),s)=\pi(\phi(x,s)),\end{equation}
for all $x\in S'$ and $s\in \R$ close to zero. Furthermore, by Gr\"onwall's Lemma,  (\ref{eq_rugueux}) implies  that for $x\in S'$ and $t\in S$ close to $0$, and $s$ positive small:
\begin{equation}\label{eq_gronwall}
  |x-t| \exp (-Cs )\le |\phi (x,s)-\phi (t,s)|\leq  \exp (Cs )|x-t|.
\end{equation}
The first inequality and (\ref{eq_phi_pi}) establish that an integral curve starting at $x\in S'$ may not fall into $S$. The second inequality implies that it stays in a little neighborhood of $S$ if $x$ is chosen sufficiently close to $S$. It also yields that $\phi $ is continuous ($\phi$ being smooth on strata, its continuity just needs to be checked at points of $S$).

Fix now $t$ and $t'$ in $S$ close to $0$,
and let for $x \in A_t$ (i.e., $(x+t) \in X$) close to the origin:
$$
                      h_{t,t'}(x):=\phi (x+t,t'-t)-t'=\phi (x+t,t'-t)-\phi(t,t'-t),
      $$
which belongs to $A_{t'}$. Integrating (\ref{eq_rugueux}), we  see (using (\ref{eq_gronwall}) and the Mean Value Theorem) that there is a constant $C'$ such that:
\begin{equation}\label{eq_h_r_approx}
 |h_{t,t'}(x)-x|\le C' |t-t'|\cdot |x|.
\end{equation}
This establishes that $h_{t,t'}$ is an $\alpha$-approximation of the identity with $\alpha(r,t,t'):=C'r\cdot |t-t'|$.
 By Theorem \ref{thm_alpha_approx_invariance_des_vol}, we get that there is a constant $C''$  independent of $t$ and $t'$ in $S$ such that for $r$ positive small enough:$$|\psi(A_t,r)-\psi(A_{t'},r)|\le C'' r^{l} \cdot |t-t'|,$$
 where $l=\dim X$. This implies:
$$|\theta_X(t)-\theta_X(t')|= |\theta_{A_t}(0)-\theta_{A_{t'}}(0)| \leq C'' |t-t'|,$$
which yields the Lipschitzness of the density in the vicinity of $0$.

\noindent {\it Proof of $(ii)$.}   As the argument is very similar, we will just provide a sketch of proof. We also restrict ourselves to the case of a stratification constituted by only two strata $S$ and $S'$, with $S \subset cl(S')\setminus S'$, assuming $\dim S=1$. Using curve selection lemma, one may actually reduce the proof of the key point (see \cite{vdensity} for more details) to the case where the stratum $S$ is one dimensional. 

 As our problem is local, we may identify $S$ with a neighborhood of the origin in $\R\times \{0_{\R^{n-1}}\}$.  Since $S$ is a one-dimensional stratum, by Proposition \ref{pro_cond_reg_dim_1},  $(S',S)$ satisfies Kuo's $(r)$ condition. By \L ojasiewicz's  inequality, it means that there is a rational number $\mu<1$ such that on $S'$ we have on a neighborhood of the origin:
\begin{equation}\label{eq_proof_density_r_e} \angle (T_{\pi(x)} S,T_x S')\lesssim \frac{|x-\pi(x)|}{|x|^\mu},\end{equation}
 where $\pi$ is the orthogonal projection onto $\R\times \{0_{\R^{n-1}}\} $.
 Let $v$ be the vector field on $X$ defined  in the proof of $(i)$  (see (\ref{eq_proj})).
  Of course, because we no longer assume the $(w)$ condition,  (\ref{eq_rugueux}) might fail. Nevertheless,  (\ref{eq_proof_density_r_e}) ensures that for $x\in X$ close to the origin:
$$|v(x)-v(\pi(x))| \lesssim  \frac{|x-\pi(x)|}{|x|^\mu}.$$
Denote by $\phi$ the flow of this vector field. One may show existence and uniqueness of the integral curves by a similar argument as in the proof of $(i)$ (see (\ref{eq_gronwall})).  We then can define a family of mappings $h_t: A_t \to A_{0}$ (where $A_t$ is as in (\ref{eq_A_t})) by $h_t(x):= \phi(x,-t)$. This mapping is an $\alpha$-approximation of the identity, with $\alpha(r,t):= Ct^{1-\mu} \cdot r$ for some positive constant $C$  (by the same argument as to show (\ref{eq_h_r_approx})). Again using Theorem \ref{thm_alpha_approx_invariance_des_vol}, we derive that for $t$ close to $0$:
$$|\psi(A_t,r)-\psi(A_{0},r)|\le C t^{1-\mu}\cdot r^l $$
 (where again $l=\dim X$), which implies
$$|\theta_X(t)-\theta_X(0)|= |\theta_{A_t}(0)-\theta_{A_{0}}(0)| \leq C t^{1-\mu},$$
which yields the continuity of the density at the origin.
\end{proof}

\begin{exa}\label{exa_densite}
 Let $f(x,y,z):=y^4-(x^2+y^2+z^2)x^4$ and let $X:=\{f\le 0\}$. For $q=(x,y,z)\in \R^3$, we have $\pa_q f= (-6x^5-4(y^2+z^2)x^3,4y^3-2yx^4, -2zx^4)$. The set $X$  can be stratified by  three strata: the  $z$-axis $O_z$, the manifold $S:=\{(x,y,z)\in \R^3:f(x,y,z)=0, (x,y)\ne (0,0)\}$, and the open set $U:=\{f<0\}$. Let us check that this stratification is Whitney $(b)$ regular near the origin. As this is trivial when one stratum is an open set, we just have to check it for the couple  $(S,O_z)$. To prove that it holds at the origin (by Proposition \ref{pro_w_stratifying}, we already know that Whitney's $(b)$ condition holds at $(0,0,z)$ for each $z\ne 0$ small), thanks to Proposition \ref{pro_cond_reg_implications}, it suffices to check the $(r)$ condition at $0$, which amounts to show that
 \begin{equation}\label{eq_r_exa}
 \lim_{q=(x,y,z)\to 0,q\in S} \frac{|\pa f/\pa z(q)| }{|\pa_q f|}\cdot \frac{|q|}{|(x,y)|}=0.
 \end{equation}
 Note that for $q=(x,y, z)\in S$, since $|y|= |q|^{1/2} |x|$, we must have
$$\frac{\pa f}{\pa y}(q) = 4y^3-2yx^4=2y(2y^2-x^4)=2yx^2(2|q|-x^2) ,$$
so that  for $q=(x,y, z)\in S$ close to the origin
$$|\frac{\pa f}{\pa y}(q)  |\sim |y|\,x^2\,|q|= |q|^{3/2} |x|^3, $$
which entails
 $$ \frac{|\pa f/\pa z|}{|\pa f/\pa y|}(q)\sim \frac{|zx^4|}{ |q|^{3/2} |x|^3}\le  \frac{|(x,y)|}{|q|^{1/2}}, $$
 which yields (\ref{eq_r_exa}).
 By the above theorem, we conclude that the density must be continuous along the $z$-axis at the origin. The proof of the above theorem has actually established that it is H\"older continuous with exponent $1-\frac{1}{2}=\frac{1}{2}$.

Let us now regard the set $X$  as a family parameterized by $z$. For every $z$ the fiber $X_z$ is the set $\{ (x,y): |y|\le |(x,y,z)|^{1/2}|x|\}$. For $z\ne 0$, this portion of plane contains the area around the $x$-axis delimited by two lines passing through the origin and making an angle at the origin with the $x$-axis not smaller than $c |z|^{1/2}$ (with $c>0$  independent of $z$ small), which shows that for $z$ small $\theta_X(0,0,z)\ge \eta |z|^{1/2}$, for some $\eta>0$. On the other hand, since the tangent cone of $X$ at the origin has positive codimension, $\theta_X(0,0,0)=0$, which means that $\theta_X(0,0,z)$  is not a Lipschitz function at $0$. Kuo-Verdier's $(w)$ condition does not hold; inequality (\ref{eq_w}) fails along the arcs $\alpha(t)\equiv (0,0, 0)\in O_z$ and $\beta(t)=( t,b(t),t)\in S$, $t>0$ small, where $b(t)$ is a solution of $b^4-t^4b^2-2t^6=0$.

We gave a three strata example, although we only did the proof in the two strata case. A revolution of $S$ about the $(x,z)$-plane however provides (together with the $z$-axis) a two strata example (in $\R^4$) that has the same properties. It is also worthy of notice that $\Sigma:=\{O_z\times \R, S\times \R, U\times \R\}$ is a Whitney $(b)$ regular stratification of $X\times \R $ (near the origin) which is not $(r)$ regular. The above theorem nonetheless still ensures the continuity (not necessarily H\"older's) of the density  along the strata of $\Sigma$ near $0$.
\end{exa}

\section{Stokes' formula}\label{sect_stokes}
   We end this chapter by proving Stokes' formula on globally subanalytic (possibly singular) sets (Theorem \ref{thm_stokes_leaves}). The formula that we will give applies to a large class of differential forms, called {\it stratified forms}.  These differential forms are not necessarily continuous but are locally bounded (Proposition \ref{pro_bounded}). What makes them attractive is that the pull-back of a stratified form via a definable Lipschitz (not necessarily differentiable) map is a stratified form (see Definition \ref{dfn_refinement}).


\noindent {\bf Stratified forms.} If $\omega$ is a differential $k$-form on a submanifold $S\subset \R^n$, we denote by $|\omega(x)|$ the norm of the linear form $\omega(x):\otimes^k T_x S \to \R$, where $S$ is equipped with the Riemannian metric inherited from the ambient space. We denote by $d\omega$ the exterior differential of $\omega$. \nomenclature[crm]{$d\omega$}{exterior differential of a differential form $\omega$\nomrefpage}

\begin{dfn}\label{dfn_stratified_form}
Let $X\in \s_n$ and let $\Sigma$ be a  stratification of $X$. 

A {\bf stratified differential $0$-form on $(X,\Sigma)$}\index{stratified form} is a collection of functions $\omega_S:S\to \R$, $S \in \Sigma$, that glue together into a continuous function on $X$.  

A {\bf stratified differential $k$-form on $(X,\Sigma)$}, $k>0$, is a collection $(\omega_S)_{S \in \Sigma}$  where, for every $S$, $\omega_S$ is a continuous differential $k$-form on $S$ such that for any $(x_i ,\xi_i)\in \otimes^k  TS$, with $x_i$ tending to $x\in S'\in \Sigma$ and $\xi_i$ tending to $\xi \in \otimes^k  T_xS'$, we have $$\lim \omega_S(x_i,\xi_i)=\omega_{S'}(x,\xi).$$ 

The {\bf support of a stratified form}\index{support of a stratified form} $\omega$ on $(X,\Sigma)$ is the closure in $X$ of the  set
$$\bigcup_{S\in \Sigma}\{x\in S: \omega_S(x)\ne 0 \}.$$
When this set is compact, $\omega$ is said to be {\bf compactly supported}\index{compactly supported}.

 We say that a stratified form $\omega=(\omega_S)_{S \in \Sigma}$ is {\bf differentiable}  if $\omega_S$ is $\ccc^1$ for every $S\in \Sigma$ and if $d\omega:=(d\omega_S)_{S\in \Sigma}$ is a stratified form.
\end{dfn}

\begin{pro}\label{pro_bounded} Let $(X,\Sigma)$ be a stratified subset of $\R^n$ and let
$\omega=(\omega_S)_{S \in \Sigma}$ be a stratified form.  If the support of $\omega$ is closed (in $\R^n$) then, for every $S\in \Sigma$, $|\omega_S(x)|$ is bounded on every bounded subset of $S$.
\end{pro}
\begin{proof}
 If $\omega$ is a  $0$-form, this is clear since $\omega_S$ is the restriction of a continuous function on $X$. Take a stratified $k$-form  $\omega$ that has closed support (in $\R^n$), with $k>0$, and let us assume that the result fails for $\omega$. It means that there is a bounded  sequence  $(x_i ,\xi_i)\in \otimes^k  TS$, $S \in \Sigma$, such that $\omega_S(x_i,\xi_i)$  goes to infinity. Since $x_i $ is a bounded sequence, extracting a subsequence if necessary, we may assume that it is convergent to some element $x\in S'$, $S' \in \Sigma$ (the point $x$ must belong to $X$ for $\omega_S$ is zero near $fr(X)$). Let
 $$\xi'_i:=\frac{\xi_i}{\omega_S(x_i,\xi_i)},$$ so that $\omega_S(x_i,\xi_i')=1$, for all $i$. 
As $\omega_S(x_i,\xi_i)$ is going to infinity and $\xi_i$ is bounded,  $\xi'_i$  goes to zero.    As $\omega$ is a stratified form, this implies that we must have $$\lim \omega_S(x_i,\xi_i')=\omega_{S'}(x,0)=0,$$ in contradiction with $\omega_S(x_i,\xi_i')\equiv1$.
\end{proof}

Given a stratification $\Sigma$, we denote by $\Sigma^{(k)}$\nomenclature[cs]{$\Sigma^{(k)}$}{set of  $k$-dimensional strata of  $\Sigma$\nomrefpage} the collection of all the strata of $\Sigma$ of dimension $k$, and   by $\cup \Sigma^{(k)}$\nomenclature[csa]{$\cup\Sigma^{(k)}$}{union of all the $k$-dimensional strata of $\Sigma$\nomrefpage} the union of all the elements of $\Sigma^{(k)}$.

\begin{dfn}\label{dfn_refinement}
 Let $\omega=(\omega_S)_{S \in \Sigma}$ be a stratified form, $\Sigma'$ be a refinement of $\Sigma$, and take $T \in \Sigma'$. By definition of refinements, there is a unique  $S \in \Sigma$ which contains $T$. Let $\omega_T$ denote the differential form induced by $\omega_S$ on $T$. It is a routine to check that $\omega':=(\omega_T)_{T\in \Sigma'}$ is also a stratified form. We then say that $\omega'$ is a {\bf refinement of $\omega$}\index{refinement! of a stratified form}.

Given  a horizontally $\ccc^1$ stratified mapping $F:(X,\Sigma_1)\to (Y,\Sigma_2)$ and a stratified form  $\omega=(\omega_S)_{S\in \Sigma_2}$  on $(Y,\Sigma_2)$, let us define the {\bf pull-back of the stratified form $\omega$} \index{pull-back of a stratified form} under $(F,\Sigma_1,\Sigma_2)$ as $$F^*\omega:=  (F_{|S}^* \omega_{S'})_{S\in \Sigma_1},$$
where $F^*_{|S} \omega_{S'}$ stands for the pull-back of the differential form $\omega_{S'}$ under the smooth mapping $F_{|S}:S\to S'$, $S\in \Sigma_1, S'\in \Sigma_2$, induced by $F$ on $S$ (see Definition \ref{dfn_stratified_mapping}).
Since 
 $F$ is horizontally $\ccc^1$,
it is easily checked from the definitions that $F^*\omega$ is then also a stratified form. 

  If a  definable mapping (not necessarily smooth)  $h:X\to Y$ is locally Lipschitz  with respect to the inner metric then, by  Proposition \ref{pro_h_hor_C1}, there are stratifications $\Sigma_1$ and $\Sigma_2$ of $X$ and $Y$ respectively that make of $h$ a horizontally $\ccc^1$ stratified mapping. Moreover, if $(\omega, \Sigma)$ is a stratified form, for some stratification $\Sigma$ of $Y$, then we can require $\Sigma_2$ to be compatible with $\Sigma$, so that $\omega$ induces a stratified form on $\Sigma_2$, that we can pull-back by means of $h$ as explained just above. In other words,
  every stratified form on $Y$ can be pulled-back to a stratified form on $X$.
In particular, if $Y$ is a manifold (that we can endow with the one-stratum stratification), then every smooth form on $Y$ can be pulled-back to a stratified form by such a mapping $h$.

\noindent {\bf Integration of stratified forms.} Let $(X,\Sigma)$ be a stratified set, $X \in \St_n$.
Take a compactly supported stratified $k$-form  $\omega=(\omega_S)_{S \in \Sigma}$ on $(X,\Sigma)$, $k\in \N$, and let $Y\subset X$ be a definable subset  of dimension $k$  such that $Y_{reg}$ is oriented.
We are going to define the integral of $\omega$ on $Y$, denoted $\int_Y \omega$. 

 Let $\Sigma'$ be a refinement of $\Sigma$ compatible with $Y_{reg}$.  This refinement inducing a refinement $\omega'$ of $\omega$ (as explained  just above),
 we may naturally set $$\int_Y \omega:=\sum_{S\in \Sigma'^{(k)}, S\subset Y  }\int_S \omega_S, $$
where every stratum is endowed with the orientation induced by $Y_{reg}$. That $\omega_S$ is integrable on $S$ follows from the fact that $\omega$ is compactly supported (and hence bounded by Proposition \ref{pro_bounded}).

 If  $\Sigma''$ is another  refinement of $\Sigma$ compatible with $Y_{reg}$  then  $Y_{reg}\cap  (\cup \Sigma'^{(k)})$ and $Y_{reg}\cap( \cup \Sigma''^{(k)})$ coincide outside a set of dimension less than $k$ (which is thus $\hn^k$-negligible)  so that
$$\sum_{S\in \Sigma'^{(k)}, S\subset Y  }\int_S \omega_S=\sum_{S\in \Sigma''^{(k)}, S\subset Y }\int_S \omega_S\;,$$
 which shows that the integral is independent of the chosen refinement. 
In the case where $\dim Y<k$, we set this integral to be zero.

\medskip

Given $X \subset \R^n$,  a {\bf definable singular $k$-simplex of $X$}\index{singular simplex} will be a continuous  definable mapping $\sigma: \Delta_k \to X$, $\Delta_k$\nomenclature[ct]{$\Delta_k$}{oriented standard simplex of $\R^k$\nomrefpage} being the $k$-simplex spanned by
$0,e_1,\dots,e_k$, where $e_1,\dots, e_k$ is the canonical basis of
$\R^k$.    We denote by $C_k(X)$\nomenclature[cu]{$C_k(X)$}{singular definable $k$-chains of $X$ with coefficients in $\R$\nomrefpage}
the vector space of definable singular $k$-chains of $X$, i.e., finite linear combinations (with real coefficients) of  singular definable simplices of $X$, and by
 $\partial c$  the boundary of a chain $c$.


 \medskip
 
\noindent {\bf Integration on definable singular simplices.} 
Let again $(X,\Sigma)$ be a stratified set, $X \in \St_n$, and let 
$\omega=(\omega_S)_{S \in \Sigma}$ be a stratified $k$-form on $X$.

 We are going to define the integral of  $\omega$ over an oriented definable singular simplex $\sigma:\Delta_k \to X$. As $\sigma$ is  definable, by Proposition \ref{pro_strat_mapping}, there exist stratifications $\hat{\Sigma}$ of $\Delta_k$ and $\Sigma'$ of $X$ such that for any $S$ in $\hat{\Sigma}$ there is $S' \in \Sigma'$ such that the mapping $\sigma_{|S}:S\to S'$, induced by the  restriction of $\sigma$, is a $\ccc^\infty$  submersion. Moreover, we may assume that $\Sigma'$ is a refinement of $\Sigma$, and hence  that $\omega$ is  a stratified form on $(X,\Sigma')$.

Let us make a point that $(\sigma_{|S}^*\omega_{S'})_{S \in \hat{\Sigma}}$ is not necessarily a stratified form on $\Delta_k$. In particular, $\sigma_{|S}^*\omega_{S'}$ is not necessarily bounded for all $S'$ ($\sigma$ is not assumed to have bounded first derivative). We  however can prove that it is  $L^1_{\hn^k}$, as follows. If $\dim S'<k$ then $\omega_{S'}=0$ and this is clear. Otherwise, $\sigma_{|S}:S\to S'$ is a local diffeomorphism (it is a submersion), and it
follows from existence of cell decompositions (applied  to the graph of $\sigma$) that there is a negligible definable subset $Z$ of $S $ such that $S\setminus Z$ can be covered by finitely many definable open subsets $U_1,\dots,U_l$ on which $\sigma$ induces a diffeomorphism onto its image. Then  $|\sigma_{|S}^*\omega_{S'}| =|\omega_{S'}\circ \sigma| \cdot \jac \sigma$ on each $U_i$, which is $L^1$, since $\omega_{S'}$ is bounded. We now set:
$$\int_\sigma \omega=\sum_{S\in \hat{\Sigma}^{(k)}}\; \int_{S}\; \sigma^*_{|S}\omega_{S'}.$$
Again, since the manifold $\cup \hat{\Sigma}^{(k)}$ is independent of the stratification $\hat{\Sigma}$ up to a negligible set, this definition is clearly independent of the chosen stratifications. 
 The integral over a definable chain $c \in C_k(X)$ is then defined naturally.\end{dfn}

\noindent {\bf Stokes' formula for stratified forms.} Our Stokes' formulas, stated in Theorems \ref{thm_stokes_leaves} (on definable weakly normal manifolds) and  \ref{thm_stokes} (on definable singular simplices), require 
  some preliminaries. The main difficulty of extending this formula to manifolds that admit singularities within their closure is   that there is no nice notion of boundary in this framework. This motivates the following definitions.

Given a  $\ccc^0$ submanifold $M$ of $\R^n$, we define {\bf the $\ccc^i$ boundary} as\index{ci boundary@$\ccc^i$ boundary}\nomenclature[cv]{$\pai M$}{$\ccc^i$ boundary of $M$ \nomrefpage}:
$$\pai M:=\{x\in fr(M): (cl(M),fr(M)) \mbox{ is a $\ccc^i$ manifold with boundary at $x$}\}.$$
The manifold $M$ will be said to be {\bf weakly normal} if there is a definable set $E\subset fr(M)$ along which $M$ is connected  (see Definition \ref{dfn_normal}) and such that $\dim (fr(M)\setminus E)\le \dim M-2$.

Of course, every normal manifold is  weakly normal.  An interesting feature of weakly normal definable manifolds is the following property.

\begin{lem}\label{lem_C_1_manifold_with_boundary}
Let $M$ be a  definable $\ccc^i$ manifold, $i=0$ or $1$. If $M$ is weakly normal then $\dim (fr(M)\setminus \pai M)\le \dim M-2$. 
\end{lem}
\begin{proof} If $M$ is weakly normal,  there is a  definable subset  $E\subset fr( M)$ satisfying $\dim (fr(M)\setminus E)\le\dim M-2$, and such that $(cl(M),E)$ is a $\ccc^0$ manifold with boundary near every point of $E$ (this can be shown using a $\ccc^0$  definable triangulation of $\mba$,  applying Theorem \ref{thm existence des triangulations}).  This already shows the result in the case $i=0$. If $i=1$, the result  then easily follows from Proposition \ref{pro_puiseux_avec_parametres}.
\end{proof}

\begin{lem}\label{lem_stokes_manifold}
Let  $(M,\pa M)$ be a  definable $\ccc^1$ manifold with boundary that we stratify by $\Sigma:=\{M\setminus \pa M,\pa M\}$, and let us take a continuous definable function $\rho:M \to [0,+\infty)$, $\ccc^1$ on $M\setminus \pa M$,  satisfying $\rho^{-1}(0)=\pa M$.  For any compactly supported differentiable stratified $(k-1)$-form $\omega$  on $(M,\Sigma)$, $k=\dim M$, we have \begin{equation}\label{eq_lem_stokes_manifold}\lim_{\ep \to 0^+}\; \int_{\rho =\ep} \omega=\int _{\pa  M} \omega .\end{equation}   \end{lem}
\begin{proof}
Up to a partition of unity we may assume that the support of $\omega$ fits in one chart of $M$ and, up to a coordinate system, we may identify $M$ with $\bou (0_{\R^k},\alpha)\cap \{(x_1,\dots,x_k)\in \R^k:x_k\ge 0\}$, $\alpha >0$. As $\omega$ is a stratified form, the coefficients of the restriction of the multi-linear form $\xi\mapsto \omega(x_1,\dots, x_k)\xi$ to $\otimes^{k-1}\R^{k-1}\times \{0_{\R}\}$ are continuous with respect to $x_k\ge 0$ (and bounded, see Proposition \ref{pro_bounded}). Hence, in the case where  $\rho(x_1,\dots,x_k)= x_k$ for all $(x_1,\dots, x_k)$,  it suffices to pass to the limit inside the integral.  As a matter of fact, it is enough to check that the limit always exists and is independent of the function $\rho$.

 For any function $\rho$  satisfying the assumptions of the lemma, by (the classical) Stokes' formula, we have for relevant orientations and $0<\ep<\ep'$:
$$ \int_{\rho =\ep} \omega- \int_{\rho=\ep'}\omega=\int_{\rho^{-1}((\ep,\ep'))}d\omega.$$
Since $\omega$ is compactly supported and $d\omega$ is bounded (see again Proposition \ref{pro_bounded}), by  Lebesgue's Dominated Convergence Theorem, we see that the right-hand-side goes to zero as $\ep,\ep'\to 0$ (bounded definable sets have finite measure, see Proposition \ref{borne unif pour les volumes}). Consequently, the limit exists for all such function $\rho$.   That the limit is independent of $\rho$ follows from an analogous argument. 
\end{proof}

\begin{lem}\label{lem_stokes_mesure}
	Let $B$ be a definable compact  subset of $ \R^m\times \R^n$.  If  $\dim B_{t}\le l$, for all $\tim$, and $\dim B_{0_{\R^m}} <l$ then there are a constant $C$ and a positive rational number $\theta$ such that for all $\tim$:
	$$\hn^l(B_t)\le C\,|t|^\theta .$$ 
\end{lem}
\begin{proof}
	Let for $\tim$, $f(t):=\sup_{x\in B_t} d(x,B_{0_{\R^m}})$, with $f(t)=0$ if $B_t$ is empty.
	We wish to apply \L ojasiewicz's inequality (Theorem \ref{thm_lojasiewicz_inequality}) to $f$. We  need for this purpose to check that $f(\gamma(s))$ tends to zero for every definable arc $\gamma:(0,\ep)\to \R^m$ satisfying $\lim_{s\to 0^+} \gamma(s)=0$. If $\gamma$ is such an arc, let $x(s)$ be a definable arc  satisfying $x(s)\in B_{\gamma(s)}$ and $ d(x(s),B_{0_{\R^m}})=f(\gamma(s))$, for every $s$ positive small (if $B_{\gamma(s)}=\emptyset$ then the needed fact is clear).   Because $B$ is compact, the arc $x(s)$ must end at a point of $B_{0_{\R^m}}$, which shows that $d(x(s),B_{0_{\R^m}})=f(\gamma(s))$ tends to zero, as required.
	
	We thus can derive from Theorem \ref{thm_lojasiewicz_inequality} that there are $C>0$ and a positive rational number  $\theta$ such that $f(t)\le C|t|^\theta$.
	Since $B_{t}\subset (B_\orm)_{\le f(t)}$,  we have for  $\tim$ and  $P \in \G_l^n$: $$\hn^l(\pi_P(B_t))\le \hn^l \left(\pi_P(B_\orm)_{\le f(t)}\right) \overset{(\ref{inegalit volume des voisinges})}{\lesssim} f(t) \lesssim |t|^\theta.$$In virtue of Cauchy-Crofton's formula, this yields the desired estimate.
\end{proof}

We first establish Stokes' formula for stratified forms on weakly normal manifolds.
\begin{thm}\label{thm_stokes_leaves}
 Let $M$ be an oriented $k$-dimensional weakly normal definable $\ccc^0$ manifold (without boundary) and let $\Sigma$ be a stratification of $cl(M)$. For any compactly supported differentiable stratified $(k-1)$-form  $\omega$ on $(cl(M),\Sigma)$,   we have:
\begin{equation}\label{eq_stokes_stratified_leaf}
 \int_M d\omega =\int_{fr( M)} \omega,
\end{equation}
where $fr(M)_{reg}$ is endowed with the induced orientation.
\end{thm}
\begin{proof}Take such a stratified form $\omega$ and let $V$ be a definable neighborhood  in $cl(M)$ of the support on $\omega$.
	Let $h:|K|\to cl(M)$ be a $\ccc^0$  definable triangulation such that  $V$, $fr(M)$, as well as the elements of $\Sigma$ are unions of images of open simplices. Note that  if $\sigma\in K$ is such that $\dim \sigma=\dim M$ then $h(\sigma)$ is an open subset of a stratum, and hence a smooth manifold. As $h(cl(\sigma))$ is a $\ccc^0$ manifold with boundary,  $h(\sigma)$ is clearly  weakly normal. Since it is enough to prove the result for the sets $h(\sigma)$, $\sigma\in K$ such that $\sigma\subset h^{-1}(V)$, it means that we can assume that $\omega$ refines a form which is $\ccc^1$ on $M$. Moreover, taking a refinement if necessary we can assume $fr(M)$ to be a union of strata.
	
	By  Proposition \ref{pro_approx}, there is a $\ccc^1$ definable function  $\rho$ on $M$ satisfying $|\rho(x) -d(x,fr(M))|<\frac{d(x,fr(M))}{2}$, which means that $\rho$ is positive and extends continuously (by $0$) on $fr(M)$.
Let us then set
$$T_\ep:=\{x\in M: \rho(x)\geq\ep \} .$$
Note that for $\ep>0$ small enough, by Sard's theorem, $T_\ep$ is a $\ccc^1$ manifold with boundary and $\omega_M$ is a smooth form on it. Thus, 
by the classical Stokes' formula:
$$\int_{T_\ep}d\omega_M = \int_{\pa T_\ep}\omega_M. $$  
As $d\omega_M$ is bounded (by Proposition \ref{pro_bounded}) and compactly supported, it easily follows from Lebesgue's Dominated Convergence Theorem (bounded definable sets have finite measure, see Proposition \ref{borne unif pour les volumes}) that:
$$\lim_{\ep\to 0} \int_{T_\ep}d\omega_M = \int_{M}d\omega_M. $$
We thus only have to show that:
\begin{equation}\label{eq_2}
\limsup_{\ep\to 0} \int_{\pa T_\ep}\omega_M=\int_{fr( M)}\omega .
\end{equation}

As $M$ is  weakly normal, by Lemma \ref{lem_C_1_manifold_with_boundary}, there is a  definable set $E\subset fr( M)$ satisfying $\dim (fr( M)\setminus E)\le \dim M-2$ and such that $cl(M)$ is a $\ccc^1$ manifold with boundary at every point of $fr (M)\setminus E$.
Let $(\varphi_\delta,\psi_{\delta})$ denote a $\ccc^\infty$ partition of unity subordinated to the covering of $\R^n$ constituted by the two open sets $int(E_{\le \delta})$ and $\R^n \setminus E_{\le\frac{\delta}{2}}$ (see section \ref{sect_uniform_bounds} for $E_{\le\delta}$).
As $\omega_M$ is bounded (see Proposition \ref{pro_bounded}) and because $\varphi_\delta$ has support in $E_{\le\delta}$,  we can  write for some $C>0$:
\begin{equation*}\label{eq_proof_stokes_vol}\int_{\pa T_\ep}\varphi_{\delta}\,\omega_M \leq C  \hn^{k-1}(E_{\le \delta} \cap \pa T_\ep).  \end{equation*}
 Thanks to Lemma \ref{lem_stokes_mesure} (applied to the definable family $B_{\delta,\ep}:= E_{\le \delta} \cap \pa T_\ep$, $B_{0,0}:=cl(E)$),
this entails that:
$$\limsup_{\delta\to 0} \limsup_{\ep\to 0} \int_{\pa T_\ep}\varphi_{\delta}\,\omega_M=0.$$
As a matter of fact, since for each $\delta>0$ we have $\omega=\varphi_\delta \omega+\psi _\delta\omega$, equality (\ref{eq_2}) reduces to show that:
\begin{equation}\label{eq_proof_stokes_psi}
 \lim_{\delta\to 0} \lim_{\ep\to 0} \int_{\pa T_\ep}\psi_\delta \,\omega_M=\int_{fr(M)} \omega.
\end{equation}
It follows from Lemma \ref{lem_stokes_manifold} (applied to $\psi_\delta \,\omega$ which induces a compactly supported stratified form on the manifold with boundary $cl(M) \setminus E$) that
$$ \lim_{\ep\to 0} \int_{\pa T_\ep}\psi_\delta\, \omega=\int_{fr( M)\setminus E} \psi_\delta\,\omega=\int_{fr( M)} \psi_\delta\,\omega.$$
 Passing to the limit as $\delta>0$ tends to zero and applying Lebesgue's Dominated Convergence Theorem, we get (\ref{eq_proof_stokes_psi}). 
\end{proof}

%
%

We  now give a version of Stokes' formula on definable singular simplices, which generalizes the classical Stokes' formula for smooth simplices.
\begin{thm}\label{thm_stokes}
Let $(X,\Sigma)$ be a  definable stratified set. If $\omega$ is a differentiable stratified $(k-1)$-form on $(X,\Sigma)$ then we have for all  $c \in C_k(X)$:
$$\int_c d\omega =\int_{\partial c} \omega. $$
\end{thm}
 \begin{proof}  It is of course enough to carry out the proof for one single definable simplex $\sigma:\Delta_k \to X$.
 We denote by $\pi_1:\Gamma_\sigma \to \Delta_k$ and $\pi_2:\Gamma_\sigma \to X$ the natural projections (where $\Gamma_\sigma$ is the graph of $\sigma$). By Proposition \ref{pro_h_hor_C1}, there are stratifications $\tilde{\Sigma}$  of $\Gamma_\sigma$ and $\Sigma'$ of $X$  making of $\pi_2$ a horizontally $\ccc^1$ stratified mapping. We may assume that $\Sigma'$ refines $\Sigma$, and hence that $\omega$ is a stratified form on $(X,\Sigma')$.  Note that  $\beta:=\pi_2^*\omega$ is a stratified form on $(\Gamma_\sigma,\tilde{\Sigma})$, which,   by Theorem \ref{thm_stokes_leaves} (applied with $M$ equal to the interior of the $\ccc^0$ manifold with boundary $\Gamma_\sigma$ and therefore $fr(M)$ equal to the boundary of this manifold), entails that
 \begin{equation}\label{eq_beta_stokes}\int_{\Gamma_\sigma} d\beta=\int_{\pa \Gamma_\sigma} \beta. \end{equation}
 As $\pi_{1|S}$ is a diffeomorphism onto its image for all $S\in \tilde{\Sigma}$ and since $\pi_2\circ \pi_1^{-1}=\sigma$,  we have  $\sigma^*\omega_S=\pi_1^{-1*}\beta_{\tilde{S}}$ and $\sigma^*d\omega_S=d\pi_1^{-1*}\beta_{\tilde{S}}$, for every $S\in \Sigma'$ and $\tilde{S} \in \tilde{\Sigma}$ satisfying $\pi_2(\tilde{S})\subset S$.  Hence,
    $$\int_{\Delta_k} \sigma^*d\omega=\int_{\Delta_k} \pi_1^{-1*}d\beta=\int_{\Gamma_\sigma} d\beta \overset{(\ref{eq_beta_stokes})}=\int_{\pa \Gamma_\sigma}\beta=\int_{\pa \Delta_k}  \pi_1^{-1*}\beta= \int_{\pa \Delta_k}\sigma^*\omega,$$
    yielding the formula claimed in the statement of the theorem.
\end{proof}

\paragraph{Historical notes.} The results about integration of globally subanalytic functions are taken from \cite{lionrolin,lionrolinlog,comtelionrolin} (see also \cite{parucamb}). The density was originally introduced for complex analytic sets by P. Lelong, and studied in the subanalytic category by  K. Kurdyka and G. Raby \cite{kr}. The stability of the measure under $\alpha$-approximations of the identity (Theorem \ref{thm_alpha_approx_invariance_des_vol}) was studied in \cite{vdensity} in order to show the continuity of the density on Whitney stratified sets. This problem was actually investigated earlier by G. Comte in \cite{comte} who established the continuity of the Lelong number under the slightly stronger Kuo-Verdier's $(w)$ condition. Stokes' formula on subanalytic sets was proved by W. Paw\l ucki in \cite{pawluckistokes}. The present form is indeed a variation that was established in \cite{vstokes}.




\cleardoublepage
\phantomsection

\cleardoublepage 
\phantomsection
\addcontentsline{toc}{chapter}{List of symbols}

 \printnomenclature
\cleardoublepage
\phantomsection
\addcontentsline{toc}{chapter}{Index}

\printindex

\end{document}